\documentclass[a4paper,11pt]{article}
\usepackage[latin1]{inputenc}
\usepackage[T1]{fontenc}

\usepackage{array}
\usepackage{fancyhdr}
\usepackage{longtable}
\usepackage{tabularx}
\usepackage{rotating}

\usepackage{makeidx}
\usepackage{amsmath}
\usepackage{empheq}
\usepackage{amsthm}
\usepackage{amssymb}
\usepackage[mathscr]{eucal}
\usepackage{enumerate}
\usepackage[dvips]{epsfig}
\usepackage{float}  
\usepackage{ifthen} 
\usepackage{extarrows}  
\usepackage{stmaryrd}   
\usepackage{multirow}
\usepackage{mathtools}
\usepackage[dvips]{geometry} 
\usepackage{bbm} 
\usepackage{xcolor}

\usepackage{listings}
\usepackage{subfig}
\usepackage{dsfont}
\usepackage{xspace}
\usepackage{xfrac}

\usepackage{todonotes}
\usepackage{changes} 
\usepackage{cite}	

\providecommand{\noopsort}[1]{}

\theoremstyle{plain}
\newtheorem{thm}{Theorem}[section]
\newtheorem{prop}[thm]{Proposition}
\newtheorem{lem}[thm]{Lemma}
\newtheorem{cor}[thm]{Corollary}
\theoremstyle{definition}

\theoremstyle{remark}

\newcommand{\Prob}{\mathrm{P}}
\newcommand{\Erw}{\mathrm{E}}

\newcommand{\ind}{\mathbbmss{1}}

\newcommand{\err}{\operatorname{err}}
\newcommand{\ccosts}{\bar{c}}
\newcommand{\Oo}{\mathcal{O}}

\usepackage{hyperref}
\usepackage{bookmark}
\hypersetup{
	pdfauthor={Andreas R\"o{\ss}ler},
	% ...
	bookmarksopenlevel=0,
	bookmarksnumbered,
	colorlinks=true,linkcolor=blue,citecolor=red,urlcolor=red,
}

\newcommand{\hnull}{h_0}	
\newcommand{\hmax}{h}	
\newcommand{\IhN}{\mathcal{I}_{\hmax}^N}	
\newcommand{\Zz}{Z}	
\newcommand{\cc}{c}		
\newcommand{\ccp}{c_1}	
\newcommand{\czwei}{c_3}	
\newcommand{\cMH}{c_4}	
\newcommand{\cYMB}{c_M} 
\newcommand{\cvierM}{c_8}	
\newcommand{\cMXinc}{c_2}	
\newcommand{\cMHdetInc}{c_6}	
\newcommand{\cMHk}{c_5}	
\newcommand{\cXZ}{c_A}	
\newcommand{\cZY}{c_B}	
\newcommand{\cMHnull}{c_7}	
\newcommand{\cIierror}{C}	
\newcommand{\cYtildeZ}{c_{\tilde{A}}}	
\newcommand{\cYtildeZtilde}{c_{\tilde{B}}}	
\newcommand{\cMHtilde}{\tilde{c}_3}	
\newcommand{\cYMBtilde}{\tilde{c}_4} 
\newcommand{\Ceins}{C_1}
\newcommand{\Czwei}{C_2}

\newcommand{\Ii}{\operatorname{I}}	
\newcommand{\Ji}{\operatorname{J}}	
\newcommand{\IiC}{\operatorname{I}^C}	
\newcommand{\JiC}{\operatorname{J}^C}	
\newcommand{\Iihat}{\operatorname{\hat{I}}}	
\newcommand{\Iitilde}{\operatorname{\tilde{I}}}	
\newcommand{\cCost}{\rho}	
\newcommand{\peff}{p_{\operatorname{eff}}}	

\newcommand{\MIL}{\operatorname{MIL}}	
\newcommand{\EM}{\operatorname{EM}}	
\newcommand{\SSBE}{\operatorname{SSBE}}	

\newcommand{\SRIzweiSeins}{\operatorname{SRI2s1}}	
\newcommand{\SRIzweiSzwei}{\operatorname{SRI2s2}}	
\newcommand{\SRSzweiSeins}{\operatorname{SRS2s1}}	
\newcommand{\SRSzweiSzwei}{\operatorname{SRS2s2}}	
\newcommand{\SRICzweiSeins}{\operatorname{SRIC2s1}}	
\newcommand{\SRICzweiSzwei}{\operatorname{SRIC2s2}}	
\newcommand{\SRSCzweiSeins}{\operatorname{SRSC2s1}}	
\newcommand{\SRSCzweiSzwei}{\operatorname{SRSC2s2}}	
\newcommand{\SRIAzweiSeins}{\operatorname{SRA2s1}}	
\newcommand{\SRIAzweiSzwei}{\operatorname{SRA2s2}}	

\newcommand{\SRIeins}{\operatorname{SRI1}}	
\newcommand{\SRICeins}{\operatorname{SRIC1}}	

\newcommand{\SPLI}{\operatorname{SPLI}}	

\allowdisplaybreaks

\title{A Class of Stochastic Runge-Kutta Methods for Stochastic Differential Equations 
	Converging with Order~1 in $L^p$-Norm}

\author{Andreas R\"o{\ss}ler
\thanks{e-mail: roessler@math.uni-luebeck.de \\ Date: June 27, 2025}
\bigskip
\\
\small{Institute of Mathematics, Universit\"at zu L\"ubeck,} \\
\small{Ratzeburger Allee 160, 23562 L\"ubeck, Germany} 
}

\date{}

\begin{document}

\maketitle

\begin{abstract}
\noindent
For the approximation of solutions for It{\^o} and Stratonovich 
stochastic differential equations (SDEs)
a new class of efficient stochastic Runge-Kutta (SRK) methods is developed.
As the main novelty only two stages are necessary for the proposed SRK 
methods of order~$1$ that can be applied to SDEs with non-commutative or 
with commutative noise. 
In addition, a variant of the SRK method for SDEs with additive noise is presented. 
All proposed SRK methods cover also the case of
drift-implicit schemes and general order conditions for the coefficients are calculated 
explicitly. The new class of SRK methods is highly efficient in the sense that it features 
computational cost depending only linearly on the dimension of the SDE and on
the dimension of the driving Wiener process.
For all proposed SRK methods strong convergence with order $1$ 
in $L^p$-norm for any $p \geq 2$ is proved. Moreover, sufficient conditions for 
approximated iterated stochastic integrals are established such that 
convergence with order~$1$ in $L^p$-norm is preserved if they are applied for the
SRK method. The presented theoretical 
results are confirmed by numerical experiments.
\end{abstract}
\tableofcontents
%
%
%
%
\section{Introduction}
\label{Sec:Introduction}
The numerical solution for stochastic (ordinary) differential equations (SDEs) is 
subject to
ongoing research and has been intensively studied in recent years. Much effort 
has been spent to the derivation of various numerical schemes with the focus 
on different objectives like, e.g., good stability properties~\cite{AbCi08,DeKvMa21,KuKvRo15}, 
geometry preserving properties~\cite{Hag19,Ant23,DeKvMa22,HaMaDi19,KoBu13}, low regularity 
assumptions~\cite{Hag25,HaiHutJen15,HuJe11,HiMaoStu02} or high orders of 
strong or weak convergence, see, e.g., \cite{Kom07b,XiTa16,MueRi08}. 
%
%
Based on Taylor expansions there exist stochastic Taylor schemes
of theoretically arbitrarily high order of convergence. However, stochastic 
Taylor schemes need the computation of derivatives and suffer from 
high computational complexity. For example, the often used Milstein
scheme~\cite{KP10,Mil74,Mil95} that is the strong order~$1$ 
stochastic Taylor scheme, has~$\mathcal{O}(d^2 \, m)$ computational 
complexity if it is applied to $d$-dimensional SDE systems
with an $m$-dimensional driving Wiener process and if the computation of 
iterated stochastic integrals is not taken into account.

In recent years, research on derivative-free Runge-Kutta type methods 
for SDEs has produced
various schemes. One of the early papers on stochastic Runge-Kutta methods 
is~\cite{Rue82}, followed by~\cite{Gar88,New91}. 
Further derivative-free Runge-Kutta type schemes for
SDEs can be found in~\cite{KP10,Mil95,MilTret21} where also forward difference
quotients are applied, as well as~\cite{Kan95,BuBu96,BuBu98,KuKvRo15}. 
However, all of these strong order~$1$ schemes suffer 
from a high computational complexity~$\Oo(d \, m^2)$ depending
quadratically on the dimension $m$ of the driving Wiener process of the 
$d$-dimensional SDE system under consideration. This is in contrast 
to the deterministic setting and also in contrast to the strong order~$\sfrac{1}{2}$
Euler-Maruyama scheme with computational complexity~$\Oo(d \, m)$ 
depending linearly on the dimension of the SDE system and of the Wiener process.

A breakthrough for the reduction of computational complexity for higher
order stochastic 
Runge-Kutta (SRK) methods of weak order~$2$ and strong order~$1$ has been
achieved in the seminal works~\cite{Roe07,Roe09,Roe10}, respectively.  These 
highly efficient SRK methods proposed in~\cite{Roe07,Roe09,Roe10}
allow for computational complexity~$\Oo(d \, m)$ that depends only linearly on 
the dimension~$d$ of the SDE system and on the dimension~$m$ of the driving 
Wiener process due to specially tailored stages 
that are applied based on colored rooted tree theory~\cite{Roe06,Roe04,Roe10a}.

In the present paper, we are concerned with the problem of approximations
to solutions of SDEs in $L^p(\Omega)$-norm.
We note that very efficient SRK methods for $d$-dimensional SDE systems
with an $m$-dimensional Wiener process converging with order~$1$ in
$L^2(\Omega)$-norm with computational complexity~$\Oo(d \, m)$ 
requiring~$3$ stages have been proposed in~\cite{Roe10}. On the other hand,
SRK methods using only~$2$ stages based on forward difference approximations 
are known in case of SDEs with a scalar~($m=1$) Wiener process, see, e.~g.,
\cite[(11.1.5)]{KP10}, having computational complexity $\Oo(d)$. However, 
the approach in~\cite{KP10} based on forward difference approximations is not
efficient for
SDE systems with an $m$-dimensional Wiener process leading to computational
complexity~$\Oo(d \, m^2)$, see~\cite[(11.1.7)]{KP10}. Therefore, it has 
been a long standing open question whether there exist order $1$ SRK methods 
with only~$2$ stages featuring computational complexity~$\Oo(d \, m)$. 
The aim of the present paper is to give a positive answer to this question.
%

In the following, we focus on the development of derivative-free approximation
methods of Runge-Kutta type for $d$-dimensional SDE
systems with an $m$-dimensional driving Wiener process. Here, both cases of
It{\^o} as well as Stratonovich SDE systems are considered.
The aim is to derive SRK methods that feature
convergence with order~$1$ in $L^p(\Omega)$-norm for any $p \geq 2$ and 
that need only~$2$ stages. Especially, making use of the innovative trick 
firstly proposed in~\cite{Roe07} and also applied in~\cite{Roe09,Roe10}, 
and inspired by the approach used in~\cite{HaRoe23} for the design
of efficient derivative-free methods, it turns out that there exist specially designed
SRK methods with only~$2$ stages such that the computational complexity 
can be reduced to be~$\Oo(d \, m)$, i.~e., the cost
depend only linearly on the dimensions~$d$ and~$m$, respectively. 

Thus, the newly proposed SRK methods allow for 
approximations with reduced computational cost compared to well known 
schemes. Next to an SRK method for general SDEs, we also introduce an SRK
method for SDEs with commutative noise as well as for additive noise. 
For all of these SRK methods, their convergence in $L^p(\Omega)$-norm 
is analyzed and order conditions for their coefficients are derived. 
As a further novelty, it has to be pointed out that the case of drift-implicit schemes 
is covered as well. This feature may be used for the development of schemes 
with, e.~g., improved stability properties. 
For example, the split-step stochastic backward Euler method considered
in~\cite{HiMaoStu02,MaStuHig02}
turns out to be a special case of the proposed SRK method.
For all proposed SRK methods detailed convergence
proofs including the drift-implicit case as well as the non-autonomous case are given.
This allows for, e.~g., explicit conditions 
for the evaluation time points of the drift and diffusion functions in the SRK method.
Moreover, to the best of the authors knowledge, this is the first time that convergence
for efficient SRK methods is proved in the strong $L^p(\Omega)$-norm with the 
supremum of the errors over all discretization times inside the expectation.
As a side-product of convergence in $L^p(\Omega)$-norm for any~$p \geq 2$
we also get pathwise convergence with nearly the same order. 
Observe that convergence in $L^p(\Omega)$-norm is important also
for weak approximation with irregular functionals, see~\cite{Avi09} for details.

In case of general SDEs with a multi-dimensional driving Wiener process, 
iterated stochastic integrals and the corresponding L\'{e}vy 
areas need to be used by order~$1$ approximation schemes, see
also~\cite{ClCa80,HoMueRi02}. 
Quite recently very efficient approximation algorithms for the approximation
of iterated stochastic integrals have been developed, 
see~\cite{Wik01,MroRoe22,KaRoe23},
that can be applied for the proposed SRK methods.
Therefore, conditions for the necessary precision of 
the approximation of iterated stochastic integrals are given such that the 
order of convergence for the proposed SRK method is preserved. We note
that strong order~$1$ approximation methods involving L\'{e}vy areas are 
important for, e.~g., step size control algorithms as discussed in~\cite{GaLy97}.

The paper is organized as follows: In Section~\ref{Sec:Setting} we fix the setting and
detail all assumptions like the smoothness of coefficients of the SDEs. Then, we
propose an efficient SRK method for general SDEs in
Section~\ref{SubSec:SRK-method-general-SDEs} and present convergence results
in $L^p(\Omega)$-norm in Section~\ref{Sec:Convergence-Result}. Moreover, 
a variant of the SRK method for SDEs with commutative noise is analyzed in
Section~\ref{Sec:SRK-Method-CommNoise} and for the special case of additive
noise in Section~\ref{Sec:SRK-Method-AdditiveNoise}. A discussion of the
computational
complexity together with numerical examples in order to analyze the performance of
the introduced SRK schemes can be found in Section~\ref{Sec:CompCost}. The paper
closes with a discussion of the specially tailored design of the introduced 
SRK methods compared to well known SRK methods in
Section~\ref{Sec:Derivation-SRK-method}. The detailed proofs 
are postponed to Section~\ref{Sec:proofs}.
%
%
%
\section{Setting and Assumptions}
\label{Sec:Setting}
In order to introduce the stochastic Runge-Kutta method for the strong 
approximation of solutions of an SDE system, we first have to fix the setting. 
In the following, let 
$(\Omega, \mathcal{F}, \Prob)$ be a complete probability space with a filtration
$(\mathcal{F}_t)_{t \geq 0}$ that fulfills the usual conditions.
For $0 \leq t_0 < T$ let $\mathcal{I}=[t_0,T]$ be the considered time interval 
and let $X=(X_t)_{t \in \mathcal{I}}$ denote the solution of a $d$-dimensional 
It{\^o} or Stratonovich stochastic differential equation system, respectively, 
such that
\begin{align} \label{SDE-Integral-form}
	X_t &= X_{t_0} + \int_{t_0}^t a(s,X_s) \, \mathrm{d}s
	+ \sum_{k=1}^m \int_{t_0}^t b^k(s,X_s) \, \ast \mathrm{d}W_s^k
\end{align}
for $t \in \mathcal{I}$ with an initial value $X_{t_0} \in L^p(\Omega)$ for some $p \geq 2$,
where $a, b^k \colon \mathcal{I} \times \mathbb{R}^d \to \mathbb{R}^d$ are
measurable mappings 
%
%
and $(W^k_t)_{t \geq 0}$ for $k=1, \ldots, m$ are independent Wiener processes. 
In case of an It{\^o} SDE~\eqref{SDE-Integral-form} we write $\ast \mathrm{d}W_s^k
= \mathrm{d}W_s^k$ whereas in case of a Stratonovich SDE~\eqref{SDE-Integral-form}
we write $\ast \mathrm{d}W_s^k = \circ \mathrm{d}W_s^k$ and we always assume 
that  $b^k \in C^{0,1}(\mathcal{I} \times \mathbb{R}^d,\mathbb{R}^d)$, $k=1, \ldots, m$,
in case of a Stratonovich SDE.
Note that
the drift $a = (a^i)_{1 \leq i \leq d}$ and the diffusion $b^k = (b^{i,k})_{1 \leq i \leq d}$ are 
given by scalar valued functions
$a^i, b^{i,k} \colon \mathcal{I} \times \mathbb{R}^d \to \mathbb{R}$ for $i=1, \ldots, d$
and $k=1, \ldots, m$. We assume
that SDE~\eqref{SDE-Integral-form} possesses a unique strong solution in $L^p(\Omega)$.
\subsection{Assumptions on Drift and Diffusion Functions}
\label{Sub-Sec:Assumptions}
For the convergence analysis of the new approximation schemes, we have to fix some 
notation and assumptions. In the following, let $\| \cdot \|$ denote the euclidean vector 
or matrix norm
if not stated otherwise and let $c>0$ denote some universal constants that may take
different values from line to line.
%
%
We assume that $a \in C^{1,2}(\mathcal{I} \times \mathbb{R}^d, \mathbb{R}^d)$
and $b^k \in C^{2,2}(\mathcal{I} \times \mathbb{R}^d, \mathbb{R}^d)$ for $k = 1, \ldots, m$.
To be precise, we claim that $a$ and $b^k$ fulfill linear growth conditions
%
\begin{align} \label{Assumption-a-bk:lin-growth}
	\| a(t,x) \| \leq \cc ( 1+ \|x \|) \, ,
	\quad \quad \quad
	\| b^k(t,x) \| \leq \cc ( 1+ \|x \|) \,  ,
\end{align}
%
and global Lipschitz conditions
%
\begin{align} \label{Assumption-a-bk:Lip}
	\| a(t,x) - a(t,y) \| \leq \cc \, \| x-y \| \, ,
	\quad \quad \quad
	\| b^k(t,x) - b^k(t,y) \| \leq \cc \, \| x-y \| \, ,
\end{align}
%
with bounded first derivatives
%
	\begin{align} \label{Assumption-a-bk:Bound-derivative-1}
		\Big\| \frac{\partial}{\partial x_q} a(t,x) \Big\| \leq \cc \, ,
		\quad \quad \quad
		\Big\| \frac{\partial}{\partial x_q} b^k(t,x) \Big\| \leq \cc \, ,
	\end{align}
%
%
as well as bounded second derivatives
%
	\begin{align} \label{Assumption-a-bk:Bound-derivative2-x}
		\Big\| \frac{\partial^2}{\partial x_q \partial x_r} a(t,x) \Big\| \leq \cc \, ,
		\quad \quad \quad
		\Big\| \frac{\partial^2}{\partial x_q \partial x_r} b^k(t,x) \Big\| \leq \cc \, ,
	\end{align}
%
for $q,r=1, \ldots, d$, $k=1, \ldots, m$
and a global Lipschitz condition of type
\begin{align} \label{Assumption-bDb-Lip}
	\Big\| \sum_{q=1}^d b^{q,k_1}(t,x) \, \frac{\partial}{\partial x_q} b^{k_2}(t,x)
	- \sum_{q=1}^d b^{q,k_1}(t,y) \, \frac{\partial}{\partial x_q} b^{k_2}(t,y) \Big\| 
	\leq \cc \, \|x-y\| 
\end{align}
for all $k_1, k_2 = 1, \ldots, m$, for all $x,y \in \mathbb{R}^d$ and for all $t \in \mathcal{I}$.
Moreover, we assume the linear growth conditions
%
	\begin{align} \label{Assumption-Bound-derivative-1t-and-2t}
		\Big\| \frac{\partial}{\partial t} a(t,x) \Big\| \leq \cc (1+\|x\|) \, ,
		\quad
		\Big\| \frac{\partial}{\partial t} b^k(t,x) \Big\| \leq \cc (1+\|x\|) \, ,
		%
		\quad
		\Big\| \frac{\partial^2}{\partial t^2} b^k(t,x) \Big\| \leq \cc (1+\|x\|) \, ,
	\end{align}
%
for $k=1, \ldots, m$, for all $x \in \mathbb{R}^d$ and for all $t \in \mathcal{I}$.
Finally, we assume the linear growth condition
%
\begin{align}
	\Big\| \frac{\partial^2}{\partial t \, \partial x_q} b^k(t,x) \Big\| 
	\leq \cc ( 1 + \|x\| )
	\label{Assumption-Bound-derivative-2tx-bk}
\end{align}
for $k=1, \ldots, m$, for $q=1, \ldots, d$, for all $x \in \mathbb{R}^d$ and for 
all $t \in \mathcal{I}$.
%

Only in case of a Stratonovich SDE~\eqref{SDE-Integral-form}, we need 
that $b^k \in C^{2,3}(\mathcal{I} \times \mathbb{R}^d, \mathbb{R}^d)$ and
we additionally assume for the functions
$\underline{b}^k = \sum_{q=1}^d b^{q,k} \, \frac{\partial}{\partial x_q} b^k$ 
and $b^k$ that
%
\begin{align}
	\Big\| \frac{\partial}{\partial t} \underline{b}^{k}(t,x) \Big\|
	\leq \cc (1 + \|x\| ) \, ,
	\quad \quad
	\Big\| \frac{\partial^3}{\partial x_j \, \partial x_q \, \partial x_r} b^k(t,x) \Big\|
	\leq \cc \, ,
	\label{Assumption-LinGrow-derivative-tx-bk-bk-Strato}
\end{align}
for $k=1, \ldots, m$, for $j,q,r=1, \ldots, d$, for all $x \in \mathbb{R}^d$ and for 
all $t \in \mathcal{I}$.
%
%

Note that the assumptions on the functions $a$ and $b^k$ are similar to the
assumptions for the convergence result in $L^2(\Omega)$-norm for the
original Milstein scheme, see~\cite{Mil74}. In contrast to the assumptions 
in~\cite{Mil74} we need some the additional smoothness for $b^k$ that has to 
be twice continuously differentiable with respect to time and has to fulfill
assumption~\eqref{Assumption-Bound-derivative-1t-and-2t}.
Especially, under the above assumptions
SDE~\eqref{SDE-Integral-form} possesses a unique strong solution~\cite{Mao07}.
\section{Stochastic Runge-Kutta Methods}
\label{Sec:SRK-method}
We introduce new SRK methods that can be applied to $d$-dimensional It{\^o} and
Stratonovich SDE systems of type~\eqref{SDE-Integral-form} with an $m$-dimensional 
driving Wiener process. In order to approximate the solution $X$, we discretize the
time interval $\mathcal{I}$. For $N \in \mathbb{N}$ let 
$\IhN = \{t_0, \ldots, t_N\} \subset \mathcal{I}$ with 
$t_0 < t_1 < \ldots < t_N = T$ denote a discretization of $\mathcal{I}$ with 
step sizes $h_n = t_{n+1}-t_n$  and maximum step size $\hmax = \max_{0 \leq n < N} h_n$.
In the following, we derive SRK methods for arbitrary SDE systems, for SDE
systems that fulfill a certain commutativity condition and for SDE systems with
additive noise. For each of the methods, we derive order conditions for their
coefficients and discuss their computational cost.
\subsection{The General Stochastic Runge-Kutta Method}
\label{SubSec:SRK-method-general-SDEs}
First of all, we consider $d$-dimensional SDEs of the form~\eqref{SDE-Integral-form} 
with an $m$-dimensional Wiener process for arbitrary $d,m \in \mathbb{N}$.
For higher order approximations it is necessary to incorporate iterated stochastic 
integrals. Therefore, we consider for $t_0 \leq s<t \leq T$ increments and
iterated It{\^o} stochastic integrals w.r.t.\ the Wiener processes denoted as
\begin{align} \label{Sec:SRK-method:RVs}
	\Ii_{(k),s,t} = \int_s^t \, \mathrm{d}W_u^k = W_{t}^k-W_{s}^k
	\qquad \text{ and } \qquad
	\Ii_{(l,k),s,t} = \int_{s}^{t} \int_s^u \, \mathrm{d}W^l_v \, \mathrm{d}W^k_u
\end{align}
for $k,l \in \{1, \ldots, m\}$. The random variables $\Ii_{(k),s,t}$, $k=1, \ldots, m$,
are stochastically independent and $\mathcal{N}(0,t-s)$-distributed. Further, it 
holds $\Ii_{(k,k),s,t} = \frac{1}{2} (\Ii_{(k),s,t}^2 - (t-s) )$ and $\Ii_{(k,l),s,t} = \Ii_{(k),s,t}
\, \Ii_{(l),s,t} - \Ii_{(l,k),s,t}$ for $k,l=1, \ldots, m$ and $k \neq l$. In contrast to this, 
the distribution of $\Ii_{(l,k),s,t}$ is unknown if $k \neq l$. Therefore, these random variables
need to be approximated and there exist very efficient algorithms that can 
be easily implemented to do this, see, e.g., \cite{KaRoe23,MroRoe22,Wik01} for details.
For Stratonovich calculus we consider the following increments and and iterated
Stratonovich stochastic integrals  w.r.t.\ the Wiener processes denoted as
\begin{align} \label{Sec:SRK-method:RVs-Strato}
	\Ji_{(k),s,t} = \int_s^t \, \circ \mathrm{d}W_u^k = W_{t}^k-W_{s}^k
	\qquad \text{ and } \qquad
	\Ji_{(l,k),s,t} = \int_{s}^{t} \int_s^u \, \circ \mathrm{d}W^l_v \, \circ \mathrm{d}W^k_u
\end{align}
for $k,l \in \{1, \ldots, m\}$. Here, it holds $\Ji_{(k),s,t} = \Ii_{(k),s,t}$ and $\Ji_{(l,k),s,t}
= \Ii_{(l,k),s,t}$ for $k \neq l$, whereas $\Ji_{(k,k),s,t} = \Ii_{(k,k),s,t} + \frac{1}{2} (t-s)$
in case of $k=l$.
To keep notation brief, we denote $\Ii_{(k),n} = \Ii_{(k),t_n,t_{n+1}}$, 
$\Ji_{(k),n} = \Ji_{(k),t_n,t_{n+1}}$, $\Ii_{(l,k),n} = \Ii_{(l,k),t_n,t_{n+1}}$
and $\Ji_{(l,k),n} = \Ji_{(l,k),t_n,t_{n+1}}$ for $k,l  = 1, \ldots, m$ and $n = 0, 1, \ldots, N-1$.

Now, we have all necessary ingredients to introduce the specially tailored
SRK method.
The newly proposed stochastic Runge-Kutta method with $s$ stages is defined 
by $Y_0=X_{t_0}$ and
\begin{subequations} \label{SRK-method}
\begin{align} \label{SRK-method-Main}
	Y_{n+1} &= Y_n + \sum_{i=1}^s \alpha_i \, a(t_n + c_i^{(0)} h_n, H_{i,n}^{(0)}) \, h_n
	+ \sum_{i=1}^s \sum_{k=1}^m \big( \beta_i^{(1)} \, \Ii_{(k),n} + \beta_i^{(2)} \big) \,
	b^k(t_n + c_i^{(1)} h_n, H_{i,n}^{(k)})
\end{align}
for $n=0,1, \ldots, N-1$ with stages
\begin{align}
	H_{i,n}^{(0)} &= Y_n + \sum_{j=1}^s A_{i,j}^{(0)} \, a(t_n + c_j^{(0)} h_n, H_{j,n}^{(0)}) 
	\, h_n
	\label{SRK-method-stage-H0} \\
	H_{i,n}^{(k)} &= Y_n + \sum_{j=1}^s A_{i,j}^{(1)} \, a(t_n + c_j^{(0)} h_n, H_{j,n}^{(0)}) 
	\, h_n
	+ \sum_{j=1}^{i-1} \sum_{l=1}^m B_{i,j}^{(1)} \, b^l(t_n + c_j^{(1)} h_n, H_{j,n}^{(l)}) 
	\, \Iihat_{(l,k),n} 
	\label{SRK-method-stage-Hk}
\end{align}
\end{subequations}
for $i=1, \ldots, s$ and $k = 1, \ldots, m$. For the definition of the random variables
$\Iihat_{(l,k),n}$ we consider various cases in the following. If solutions for a general 
SDE~\eqref{SDE-Integral-form} are to be approximated with SRK 
method~\eqref{SRK-method}, we define $\Iihat_{(l,k),n}$ in case of It{\^o} SDEs
in a different way compared to the case of Stratonovich SDEs, that is
we define
\begin{align} \label{SRK-method-Ito-Strato-Iihat}
	\Iihat_{(l,k),n} = \begin{cases} 
		\Ii_{(l,k),n} & \text{ for approximating It{\^o} SDEs}, \\
		\Ji_{(l,k),n} & \text{ for approximating Stratonovich SDEs},
	\end{cases}
\end{align}
for $n=0,1, \ldots, N-1$ and for $k,l=1, \ldots, m$.

The proposed SRK method~\eqref{SRK-method}
contains coefficients that can be represented as vectors 
$c^{(q)} = (c_i^{(q)})_{1 \leq i \leq s} \in \mathbb{R}^s$, 
$\alpha = (\alpha_i)_{1 \leq i \leq s}, \beta^{(r)} = (\beta_i^{(r)})_{1 \leq i \leq s}
\in \mathbb{R}^s$ and as matrices $A^{(q)} = (A_{i,j}^{(q)})_{1 \leq i,j \leq s}, 
B^{(1)} = (B_{i,j}^{(1)})_{1 \leq i,j \leq s} \in \mathbb{R}^{s \times s}$ for $q \in \{0,1\}$
and $r \in \{1,2\}$. Here, we always define $B_{i,j}^{(1)}=0$ for $j \geq i$ which is in
accordance with~\eqref{SRK-method-stage-Hk}.

Once some coefficients for SRK method~\eqref{SRK-method}
with~\eqref{SRK-method-Ito-Strato-Iihat} are fixed we call 
it a specific SRK scheme. The coefficients of an SRK scheme can be represented in 
an extended Butcher tableau of the form
{
	\renewcommand{\arraystretch}{1.6}
	\begin{align} \label{Butcher-Tableau-SRK}
		\begin{array}{r|c|c|c}
			c^{(0)} & A^{(0)} & \multicolumn{1}{c|}{ } \\
			\cline{1-3}
			c^{(1)}  & A^{(1)} & \multicolumn{1}{c|}{ B^{(1)} } & \\
			\hline
			& \alpha^T & {\beta^{(1)}}^T & {\beta^{(2)}}^T
		\end{array}
	\end{align}
}%
defining the scheme. Each SRK scheme for SDEs can also be applied to some 
deterministic ordinary differential equation (ODE), i.~e., if $b^k \equiv 0$ for all~$k$.
Then, only the coefficients $\alpha$, $c^{(0)}$ and $A^{(0)}$ are relevant and the
SRK method~\eqref{SRK-method} reduces to the well known Runge-Kutta method
for ODEs, see, e.~g., \cite{HNW93}.
However, then the order of convergence may increase as the rough stochastic part
is absent. Following~\cite{Roe06a,Roe07,Roe09,Roe10} we denote by 
$(p_D,p_S)$ the orders of convergence for the SRK scheme where $p_D$ and 
$p_S$ indicate the orders of convergence in $L^p(\Omega)$-norm if the scheme 
is applied to an ODE or SDE, respectively. Thus, it holds $p_D \geq p_S$.
If we simply talk about the order of a scheme then we always refer to $p_S$ in the 
following.

In general, an SRK scheme~\eqref{SRK-method} 
is called explicit 
if $A_{i,j}^{(0)} = 0$ for $j \geq i$ and drift-implicit otherwise. Since 
$B_{i,j}^{(1)}=0$ for $j \geq i$, the considered schemes are always explicit w.~r.~t.\
the diffusion term. This is necessary because 
the random variables $\Iihat_{(l,k),n}$ are unbounded. Thus, if the SRK scheme would be 
implicit in the diffusion term then it can not be assured anymore that there exists 
a solution for the implicit stage equation.
\subsubsection{Convergence of the SRK Method for General SDEs}
\label{Sec:Convergence-Result}
Next, we consider the errors of the approximations based on SRK 
method~\eqref{SRK-method} with~\eqref{SRK-method-Ito-Strato-Iihat} 
at the discretization time
points. So as to ensure a specific order of convergence, the coefficients of the SRK 
method need to fulfill some conditions that result from the convergence analysis. 
Knowing these order conditions, one can try to find coefficients such that, e.~g., a 
minimal number of stages and thus minimal computational effort is needed.
In the following, let $e = (1, \ldots, 1)^T \in \mathbb{R}^s$
and let $Y_n^{\hmax}$ denote the approximation $Y_n$ at time $t_n$ 
for $n=0, \ldots, N$ defined by~\eqref{SRK-method}
based on a grid $\IhN$ with maximum step size $\hmax$. 
Then, the following theorem gives the order conditions for the SRK method in case
of It{\^o} SDEs.
\begin{thm} 
	\label{Sec:Main-Result:Thm-Konv-SRK-allg}
	Let $p \geq 2$, let $a \in C^{1,2}(\mathcal{I} \times \mathbb{R}^d, \mathbb{R}^d)$,
	$b^k \in C^{2,2}(\mathcal{I} \times \mathbb{R}^d, \mathbb{R}^d)$ for $k = 1, \ldots, m$
	and let \eqref{Assumption-a-bk:lin-growth}--\eqref{Assumption-Bound-derivative-2tx-bk}
	be fulfilled. Further, let $(X_t)_{t \in \mathcal{I}}$ be the solution of It{\^o}
	SDE~\eqref{SDE-Integral-form} with $X_{t_0} \in L^{2p}(\Omega)$
	and for $N \in \mathbb{N}$ let $(Y_n^{\hmax})_{0 \leq n \leq N}$ denote the approximations
	given by SRK method~\eqref{SRK-method} with $\Iihat_{(l,k),n} = \Ii_{(l,k),n}$ 
	on some grid $\IhN = \{t_0, \ldots, t_N\}$ with maximum step size $\hmax$.
	\begin{enumerate}[i)]
		\item
		If the coefficients of the SRK method~\eqref{SRK-method} fulfill the conditions
		\begin{equation} \label{Sec:Main-Result:Thm-Konv-SRK-allg-OrdCond-0.5}
			\begin{aligned}
				&\alpha^T e = 1 , &\quad \quad \quad &{\beta^{(1)}}^T e = 1 , &\quad \quad \quad 
				&{\beta^{(2)}}^T e =0 , \\
				&{\beta^{(2)}}^T c^{(1)} = 0 , &\quad \quad \quad &{\beta^{(2)}}^T A^{(1)} e =0 , & &
			\end{aligned}
		\end{equation}
		then there exist some $\hnull >0$ and some constant $C>0$ such that
		\begin{align*}
			\big( \Erw \big( \sup_{0 \leq n \leq N} \| X_{t_n} - Y_n^{\hmax} \|^p \big) 
			\big)^{\frac{1}{p}} \leq C \, \hmax^{\frac{1}{2}}
		\end{align*}
		on any discretization~$\IhN$ for all $0 < \hmax \leq \hnull$.
		\item
		If the coefficients of the SRK method~\eqref{SRK-method} fulfill in addition 
		to~\eqref{Sec:Main-Result:Thm-Konv-SRK-allg-OrdCond-0.5} the condition
		\begin{align}
			\label{Sec:Main-Result:Thm-Konv-SRK-allg-OrdCond-1.0}
			&{\beta^{(2)}}^T B^{(1)} e =1 ,
		\end{align}
		then there exist some $\hnull >0$ and some constant $C>0$ such that
		\begin{align*}
			\big( \Erw \big( \sup_{0 \leq n \leq N} \| X_{t_n} - Y_n^{\hmax} \|^p \big) 
			\big)^{\frac{1}{p}} \leq C \, \hmax
		\end{align*}
		on any discretization~$\IhN$ for all $0 < \hmax \leq \hnull$.
	\end{enumerate}
\end{thm}
Since the proof for Theorem~\ref{Sec:Main-Result:Thm-Konv-SRK-allg} is quite
lengthy, it is postponed to Section~\ref{Sec:proofs}. In 
Section~\ref{Sub:Sec:Proof-Moment-Bound} we prove a
uniform bound for the moments of the approximations and in
Section~\ref{Sub:Sec:Proof-Convergence-SRK-method} we give a detailed proof
for the order of convergence stated in Theorem~\ref{Sec:Main-Result:Thm-Konv-SRK-allg}.

In general, the maximum possible order of convergence for SRK 
method~\eqref{SRK-method} is bounded by~$\sfrac{1}{2}$ if the iterated stochastic 
integrals $\Ii_{(l,k),n}$ are not utilized by the scheme, see~\cite{ClCa80}. 
This is the case if, e.~g., $B_{i,j}^{(1)}=0$ for all $1 \leq i,j \leq s$. However,
the order of convergence is~$1.0$ if conditions~\eqref{Sec:Main-Result:Thm-Konv-SRK-allg-OrdCond-0.5}--\eqref{Sec:Main-Result:Thm-Konv-SRK-allg-OrdCond-1.0}
are fulfilled and the information of the iterated stochastic integrals~$\Ii_{(l,k),n}$ is
taken into account. For an higher order of convergence one needs to 
implement more information about
the driving Wiener processes usually given by higher order iterated stochastic integrals
compared to twice iterated stochastic integrals as in~\eqref{Sec:SRK-method:RVs}.
However, in general higher order iterated stochastic integrals can neither be 
simulated exactly nor easily approximated as yet.

From the convergence theory for Runge-Kutta methods for ODEs, it is a usual 
simplifying assumption to choose $c^{(0)} = A^{(0)} e$ which applies
for SRK method~\eqref{SRK-method} as well. Analogously, one can apply
the simplifying assumption $c^{(1)} = A^{(1)} e$ which is in accordance with
the order conditions in Theorem~\ref{Sec:Main-Result:Thm-Konv-SRK-allg}.
Especially, the SRK method has order of convergence~$\sfrac{1}{2}$ if, e.~g.,
we choose $\beta_i^{(2)} = 0$ for $i=1, \ldots, s$. In this case, there is no 
restriction for choosing $c^{(1)}$ and $A^{(1)}$ anymore due to the 
order~$\sfrac{1}{2}$ 
conditions in~\eqref{Sec:Main-Result:Thm-Konv-SRK-allg-OrdCond-0.5}.
Note that it is recommended to choose $c_i^{(0)} , c_i^{(1)} \in [0,1]$ which assures
that $a$ and $b^k$ need to be evaluated only at time points belonging to $[t_n, t_{n+1}]
\subseteq \mathcal{I}$ each step, respectively.
\begin{table}[htbp]
	{
		\renewcommand{\arraystretch}{1.5}
		\begin{align*} %
			\begin{array}{r|cc|cc|cc}
				0 & \makebox[0.5cm][c]{ } & \makebox[0.5cm][c]{ }  & \makebox[0.5cm][c]{ } 
				& \makebox[0.5cm][c]{ } & \makebox[0.5cm][c]{ } & \makebox[0.5cm][c]{ } \\
				0 & 0 & & & & & \\
				\cline{1-5}
				0 & & & & & & \\
				0 & 0 & & 1 & & & \\
				\hline
				& 1 & 0  & 1 & 0  &  -1 & 1
			\end{array}
			\quad \quad \quad \quad
			%
			\begin{array}{r|cc|cc|cc} 
				0 &\makebox[0.5cm][c]{ } & \makebox[0.5cm][c]{ } 
				& \makebox[0.5cm][c]{ } & \makebox[0.5cm][c]{ } 
				& \makebox[0.5cm][c]{ } & \makebox[0.5cm][c]{ } \\
				1 & 1 &  &  & & & \\
				\cline{1-5}
				0 & & & & & & \\
				0 & 0 & & 1 & & & \\
				\hline
				& \frac{1}{2} & \frac{1}{2} &  1 & 0  &  -1 & 1
			\end{array}
		\end{align*}
	}
	\caption{SRK scheme $\SRIzweiSeins$ of order~$(1,1)$ on the left hand side and SRK scheme 
		$\SRIzweiSzwei$ of order~$(2,1)$ on the right hand side.}
	\label{Table1}
\end{table}

As an example, we calculate some coefficients fulfilling the order conditions of
Theorem~\ref{Sec:Main-Result:Thm-Konv-SRK-allg}. First of all, the
Euler-Maruyama scheme~($\EM$) of order~$\sfrac{1}{2}$ is a covered by 
choosing~$s=1$ stage with coefficients $\alpha_1= \beta_1^{(1)}=1$, $\beta_1^{(2)}=0$,
$c_1^{(0)} = c_1^{(1)} = 0$ and $A_{1,1}^{(0)} = A_{1,1}^{(1)} = 0$. To be precise, the 
$\EM$ scheme has order of convergence~$(1,0.5)$. We note that the well known
split-step stochastic backward Euler ($\SSBE$) method considered in, e.~g., \cite{HiMaoStu02}
is also a special case of SRK method~\eqref{SRK-method} that is drift-implicit and has
order of convergence~$(1,0.5)$.

For an order~$1$ scheme, $s \geq 2$ stages are needed to fulfill the order 
conditions~\eqref{Sec:Main-Result:Thm-Konv-SRK-allg-OrdCond-0.5}--\eqref{Sec:Main-Result:Thm-Konv-SRK-allg-OrdCond-1.0}. This follows directly since in the case~$s=1$ we always
need $B_{1,1}^{(1)} = 0$ to be explicit in the diffusion term and thus
condition ${\beta^{(2)}}^T B^{(1)} e =1$ can not be fulfilled. However, for
$s=2$ stages there already exist some coefficients for explicit SRK schemes of
order~$1$. A very simple SRK scheme of order~$(1,1)$ denoted as $\SRIzweiSeins$
is presented on the left hand side of Table~\ref{Table1} and needs only~$1$ 
evaluation of the drift function $a$ and~$2$ evaluations of the diffusion functions 
$b^k$, $k=1, \ldots, m$. 
If we use the Heun scheme for the drift approximation we get
the SRK scheme $\SRIzweiSzwei$ of order~$(2,1)$ presented on the right hand side
of Table~\ref{Table1}. Note that further SRK schemes that make use of well known 
coefficients for $\alpha$, $c^{(0)}$ and $A^{(0)}$ of any explicit or implicit 
Runge-Kutta scheme for ODEs of any order $p_D \geq 1$, see, e.~g., \cite{HNW93}, can 
be obtained easily.
%

If SRK method~\eqref{SRK-method} with the definition $\Iihat_{(l,k),n} = \Ji_{(l,k),n}$
is applied to a Stratonovich SDE then we get the following convergence
result.
\begin{thm}
	\label{Sec:Main-Result:Thm-Konv-SRK-allg-Strato}
	Let $p \geq 2$, let $a \in C^{1,2}(\mathcal{I} \times \mathbb{R}^d, \mathbb{R}^d)$,
	$b^k \in C^{2,3}(\mathcal{I} \times \mathbb{R}^d, \mathbb{R}^d)$ for $k = 1, \ldots, m$
	and 
	let~\eqref{Assumption-a-bk:lin-growth}--\eqref{Assumption-LinGrow-derivative-tx-bk-bk-Strato}
	be fulfilled. 
	Further, let $(X_t)_{t \in \mathcal{I}}$ be the solution of the Stratonovich
	SDE~\eqref{SDE-Integral-form} with $X_{t_0} \in L^{4p}(\Omega)$
	and for $N \in \mathbb{N}$ let $(Y_n^{\hmax})_{0 \leq n \leq N}$ denote the
	approximations given by SRK method~\eqref{SRK-method} with 
	$\Iihat_{(l,k),n} = \Ji_{(l,k),n}$ on some grid $\IhN = \{t_0, \ldots, t_N\}$
	with maximum step size $\hmax$.
	If the coefficients of the SRK method~\eqref{SRK-method} fulfill the
	conditions~\eqref{Sec:Main-Result:Thm-Konv-SRK-allg-OrdCond-0.5} and~\eqref{Sec:Main-Result:Thm-Konv-SRK-allg-OrdCond-1.0}
	then there exist some $\hnull >0$ and some constant $C>0$ such that
	\begin{align*}
		\big( \Erw \big( \sup_{0 \leq n \leq N} \| X_{t_n} - Y_n^{\hmax} \|^p \big) 
		\big)^{\frac{1}{p}} \leq C \, \hmax
	\end{align*}
	on any discretization~$\IhN$ for all $0 < \hmax \leq \hnull$.
\end{thm}
For a proof of Theorem~\ref{Sec:Main-Result:Thm-Konv-SRK-allg-Strato} we refer to
Section~\ref{Sub:Sec:Proof-Convergence-SRK-method-Strato}.
We note that there exists no subset of the order conditions such that convergence
order~$\sfrac{1}{2}$ is attained but not convergence with order~$1$ in case of Stratonovich
SDEs. This is different to the case of It{\^o} SDEs and also due to the design of the scheme.

Since the order~$1$ conditions for the coefficients of the SRK method~\eqref{SRK-method}
with~\eqref{SRK-method-Ito-Strato-Iihat}
for Stratonovich SDEs are the same as the ones for It{\^o} SDEs, we can choose the same
coefficients for both SRK methods. Therefore, we denote by $\SRSzweiSeins$ the 
SRK scheme that we get for the coefficients on the left hand side of Table~\ref{Table1}
and by $\SRSzweiSzwei$ the SRK scheme we get for the coefficients on the right hand
side of Table~\ref{Table1} in case of SRK method~\eqref{SRK-method}
with~\eqref{SRK-method-Ito-Strato-Iihat} for Stratonovich SDEs.
%

For the implementation of SRK method~\eqref{SRK-method}
with~\eqref{SRK-method-Ito-Strato-Iihat} increments and 
the corresponding iterated stochastic 
integrals~\eqref{Sec:SRK-method:RVs}--\eqref{Sec:SRK-method:RVs-Strato}
of the Wiener processes are needed each time step. 
Different efficient algorithms for the approximation of iterated stochastic integrals 
are available, see, e.g., \cite{MroRoe22,Wik01} and the software toolbox~\cite{KaRoe23}. 
However, if the iterated stochastic integrals $\Ii_{(i,j),n}$ and $\Ji_{(i,j),n}$ are 
replaced by some approximations $\Iitilde_{(i,j),n}$ for $i,j=1, \ldots, m$, 
respectively, an additional
source of errors comes into the play and has to be taken into account. The following 
proposition gives sufficient conditions such that the convergence results from
Theorem~\ref{Sec:Main-Result:Thm-Konv-SRK-allg} and
Theorem~\ref{Sec:Main-Result:Thm-Konv-SRK-allg-Strato} remain valid if 
approximate iterated stochastic integrals $\Iitilde_{(i,j),n}$ are used for the SRK method,
i.~e., if we choose $\Iihat_{(i,j),n} = \Iitilde_{(i,j),n}$ for $i,j=1, \ldots, m$.
\begin{prop} \label{Sec:Main-Result:Prop-Konv-Approx-IterIntegrals}
	Let $p \geq 2$. For $h>0$ and $t, t+h \in \mathcal{I}$ assume that
	$\Iitilde_{(i,j),t,t+h} \in L^p(\Omega)$ is $\mathcal{F}_{t+h}$-measurable and
	independent of~$\mathcal{F}_t$ for $1 \leq i,j \leq m$. 
	If for some $\cIierror = \cIierror(p) >0$ it holds
	\begin{align} \label{Sec:Main-Result:Prop-Konv-Approx-IterIntegrals-eqn}
		\Erw( \Iitilde_{(i,j),t,t+h} ) = 0 \quad \text{ and } \quad
		\big( \Erw \big( | \Iihat_{(i,j),t,t+h} - \Iitilde_{(i,j),t,t+h} |^p \big) \big)^{\frac{1}{p}}
		\leq \cIierror \, h^{\frac{3}{2}}
	\end{align}
	for all $1 \leq i,j \leq m$, any $h>0$ and $t, t+h \in \mathcal{I}$
	then the convergence results in
	Theorem~\ref{Sec:Main-Result:Thm-Konv-SRK-allg}	
	and Theorem~\ref{Sec:Main-Result:Thm-Konv-SRK-allg-Strato}
	in $L^p(\Omega)$-norm
	remain valid if $\Iihat_{(i,j),t,t+h}$ is replaced by $\Iitilde_{(i,j),t,t+h}$, i.~e., 
	if $\Iihat_{(i,j),t,t+h} = \Iitilde_{(i,j),t,t+h}$ is used  for the SRK 
	method~\eqref{SRK-method}.
\end{prop}
For the proof of Proposition~\ref{Sec:Main-Result:Prop-Konv-Approx-IterIntegrals}
we refer to Section~\ref{Sub:Sec:Proof-Prop-Konv-Approx-IterIntegrals}.
Note that the approximation algorithms for iterated stochastic integrals in, e.~g.,
\cite{Wik01,KP10,Mil95,MilTret21,KaRoe23}, fulfill the assumptions in 
Proposition~\ref{Sec:Main-Result:Prop-Konv-Approx-IterIntegrals} for $p=2$, whereas
the recently proposed algorithm in~\cite{MroRoe22} fulfills the conditions for any 
$p \geq 2$. 
\subsection{The SRK Method for SDEs with Commutative Noise}
\label{Sec:SRK-Method-CommNoise}
In some cases, the considered SDE system~\eqref{SDE-Integral-form} possesses
so-called commutative noise, i.~e., it fulfills some commutativity condition~\cite{KP10}. 
In this case it is possible to preserve convergence with order $1$ although iterated 
stochastic integrals are not employed in the scheme. Clearly, this simplifies the 
application of such schemes considerably as the approximation of the iterated stochastic 
integrals in~\eqref{SRK-method-Ito-Strato-Iihat} usually requires additional computations.

In this section, we assume that the SDE system~\eqref{SDE-Integral-form} 
fulfills the commutativity condition
\begin{align} \label{Sec:SRK-Method-CommNoise-Cond}
	\sum_{i=1}^d b^{i,j_1}(t,x) \, \frac{\partial}{\partial x^i} b^{k,j_2}(t,x)
	= \sum_{i=1}^d b^{i,j_2}(t,x) \, \frac{\partial}{\partial x^i} b^{k,j_1}(t,x)
\end{align}
for all $j_1, j_2 \in \{1, \ldots, m\}$, $k \in \{1, \ldots, d\}$ and all
$(t,x) \in \mathbb{R}_{+} \times \mathbb{R}^d$. The commutativity
condition~\eqref{Sec:SRK-Method-CommNoise-Cond} is fulfilled, e.~g., in case of linear
noise where $b^{i,j}(t,x) = g^{i,j}(t) \, x^i$ for some functions
$g^{i,j} \colon \mathbb{R}_+ \to \mathbb{R}$, and for diagonal noise~\cite{KP10}.
Now, we define the~$s$ stages SRK method for SDE~\eqref{SDE-Integral-form}
with commutative noise by SRK method~\eqref{SRK-method}
with the random variables $\Iihat_{(l,k),n}$ that are defined in case of an 
It{\^o} SDE~\eqref{SDE-Integral-form} by 
\begin{align} \label{Sec:SRK-method-CommNoise:RVs}
		\Iihat_{(l,k),n} = \IiC_{(l,k),n} &= \begin{cases}
		\frac{1}{2} \, \Ii_{(l),n} \, \Ii_{(k),n} & \text{ if } k \neq l , \\
		\frac{1}{2} \, (\Ii_{(l),n} \, \Ii_{(k),n} - h_n) & \text{ if } k = l ,
	\end{cases}
\end{align}
and in case of a Stratonovich SDE~\eqref{SDE-Integral-form} by
\begin{align} \label{Sec:SRK-method-CommNoise:RVs-Strato}
	\Iihat_{(l,k),n} = \JiC_{(l,k),n} &= \frac{1}{2} \, \Ii_{(l),n} \, \Ii_{(k),n}
\end{align}
noting that $\Ji_{(k),n} = \Ii_{(k),n}$ for $n=0,1, \ldots, N-1$ and 
$k,l=1, \ldots, m$.

We point out that the random variables $\IiC_{(l,k),n}$ and $\JiC_{(l,k),n}$ for
$k,l=1, \ldots, m$ defined in~\eqref{Sec:SRK-method-CommNoise:RVs} 
and~\eqref{Sec:SRK-method-CommNoise:RVs-Strato} can be easily simulated and
that $\IiC_{(k,k),n} = \Ii_{(k,k),n}$ and $\JiC_{(k,k),n} = \Ji_{(k,k),n}$. 
The SRK method~\eqref{SRK-method} with~\eqref{Sec:SRK-method-CommNoise:RVs} 
or~\eqref{Sec:SRK-method-CommNoise:RVs-Strato} is a variant of the 
general SRK method~\eqref{SRK-method} and the corresponding Butcher 
tableau is of the form given in~\eqref{Butcher-Tableau-SRK}.
Especially, in case of scalar noise where $m=1$, the commutativity
condition~\eqref{Sec:SRK-Method-CommNoise-Cond} is obviously always 
fulfilled and the SRK methods~\eqref{SRK-method} with $\Iihat_{(1,1),n} = \Ii_{(1,1),n}$
in the It{\^o} case or $\Iihat_{(1,1),n} = \Ji_{(1,1),n}$ in the Stratonovich case and
the ones with~\eqref{Sec:SRK-method-CommNoise:RVs} 
or~\eqref{Sec:SRK-method-CommNoise:RVs-Strato}, respectively, coincide.
\subsubsection{Convergence of the SRK Method for SDEs with Commutative Noise}
\label{SubSec:Convergence-CommNoise}
For SRK method~\eqref{SRK-method} with $\Iihat_{(l,k),n} = \IiC_{(l,k),n}$ as 
in~\eqref{Sec:SRK-method-CommNoise:RVs} or with $\Iihat_{(l,k),n} = \JiC_{(l,k),n}$
given in~\eqref{Sec:SRK-method-CommNoise:RVs-Strato} we analyze its order 
of convergence and give order conditions for the coefficients of the method. 
%
%
Then, in case of the 
approximation of solutions to an It{\^o} SDE~\eqref{SDE-Integral-form} we get 
the following convergence result.
\begin{thm}
	\label{Sec:Main-Result:Thm-Konv-CommNoise}
	Let $p \geq 2$, let $a \in C^{1,2}(\mathcal{I} \times \mathbb{R}^d, \mathbb{R}^d)$,
	$b^k \in C^{2,2}(\mathcal{I} \times \mathbb{R}^d, \mathbb{R}^d)$ for $k = 1, \ldots, m$,
	let \eqref{Assumption-a-bk:lin-growth}--\eqref{Assumption-Bound-derivative-2tx-bk}
	and commutativity condition~\eqref{Sec:SRK-Method-CommNoise-Cond} 
	be fulfilled.
	Further, let $(X_t)_{t \in \mathcal{I}}$ be the solution of It{\^o}
	SDE~\eqref{SDE-Integral-form} with $X_{t_0} \in L^{2p}(\Omega)$
	and for $N \in \mathbb{N}$ let $(Y_n^{\hmax})_{0 \leq n \leq N}$ denote the approximations
	given by SRK method~\eqref{SRK-method} with $\Iihat_{(l,k),n} = \IiC_{(l,k),n}$
	on some grid $\IhN = \{t_0, \ldots, t_N\}$ with maximum step size $\hmax$.
	\begin{enumerate}[i)]
		\item
		If the coefficients of SRK method~\eqref{SRK-method} fulfill the
		conditions~\eqref{Sec:Main-Result:Thm-Konv-SRK-allg-OrdCond-0.5}
		%
		then there exist some $\hnull >0$ and some constant $C>0$ such that
		\begin{align*}
			\big( \Erw \big( \sup_{0 \leq n \leq N} \| X_{t_n} - Y_n^{\hmax} \|^p \big) 
			\big)^{\frac{1}{p}} \leq C \, \hmax^{\frac{1}{2}}
		\end{align*}
		on any discretization~$\IhN$ for all $0 < \hmax \leq \hnull$.
		\item
		If the coefficients of SRK method~\eqref{SRK-method} fulfill the
		conditions~\eqref{Sec:Main-Result:Thm-Konv-SRK-allg-OrdCond-0.5} and~\eqref{Sec:Main-Result:Thm-Konv-SRK-allg-OrdCond-1.0}
		then there exist some $\hnull >0$ and some constant $C>0$ such that
		\begin{align*}
			\big( \Erw \big( \sup_{0 \leq n \leq N} \| X_{t_n} - Y_n^{\hmax} \|^p \big) 
			\big)^{\frac{1}{p}} \leq C \, \hmax
		\end{align*}
		on any discretization~$\IhN$ for all $0 < \hmax \leq \hnull$.
	\end{enumerate}
\end{thm}

For a proof of Theorem~\ref{Sec:Main-Result:Thm-Konv-CommNoise} we refer to 
Section~\ref{Sub:Sec:Proof-Convergence-CommNoise}. Since the order conditions for
SRK method~\eqref{SRK-method} with~\eqref{Sec:SRK-method-CommNoise:RVs} 
are the same as the ones given
in Theorem~\ref{Sec:Main-Result:Thm-Konv-SRK-allg} and following the reasoning in
Section~\ref{Sec:Convergence-Result} it follows that $s \geq 2$ stages are necessary for 
convergence with order~$1$. Moreover, we can apply the coefficients
presented in Table~\ref{Table1}. Therefore, we denote the SRK
scheme~\eqref{SRK-method} with~\eqref{Sec:SRK-method-CommNoise:RVs} 
using the coefficients on the left and right hand 
side of Table~\ref{Table1} as $\SRICzweiSeins$ and $\SRICzweiSzwei$, respectively.
The schemes $\SRICzweiSeins$ and $\SRICzweiSzwei$ attain orders $(1,1)$ and $(2,1)$, 
respectively, like the corresponding schemes for the general non-commutative case.

Considering Stratonovich SDE~\eqref{SDE-Integral-form} fulfilling the commutative
noise condition in~\eqref{Sec:SRK-Method-CommNoise-Cond}, we can also derive 
conditions for the coefficients of SRK method~\eqref{SRK-method}
with~\eqref{Sec:SRK-method-CommNoise:RVs-Strato} such that convergence
is order~$1$ is assured. The following theorem specifies the conditions for
convergence in case of Stratonovich SDEs.
\begin{thm}
	\label{Sec:Main-Result:Thm-Konv-CommNoise-Strato}
	Let $p \geq 2$, let $a \in C^{1,2}(\mathcal{I} \times \mathbb{R}^d, \mathbb{R}^d)$,
	$b^k \in C^{2,3}(\mathcal{I} \times \mathbb{R}^d, \mathbb{R}^d)$ for $k = 1, \ldots, m$,
	let~\eqref{Assumption-a-bk:lin-growth}--\eqref{Assumption-LinGrow-derivative-tx-bk-bk-Strato}
	and commutativity condition~\eqref{Sec:SRK-Method-CommNoise-Cond} 
	be fulfilled.
	Further, let $(X_t)_{t \in \mathcal{I}}$ be the solution of Stratonovich
	SDE~\eqref{SDE-Integral-form} with $X_{t_0} \in L^{4p}(\Omega)$
	and for $N \in \mathbb{N}$ let $(Y_n^{\hmax})_{0 \leq n \leq N}$ denote the
	approximations given by SRK method~\eqref{SRK-method} 
	with~\eqref{Sec:SRK-method-CommNoise:RVs-Strato}
	on some grid $\IhN = \{t_0, \ldots, t_N\}$ with maximum step size $\hmax$.
	If the coefficients of SRK method~\eqref{SRK-method} fulfill the
	conditions~\eqref{Sec:Main-Result:Thm-Konv-SRK-allg-OrdCond-0.5}
	and~\eqref{Sec:Main-Result:Thm-Konv-SRK-allg-OrdCond-1.0}
	then there exist some $\hnull >0$ and some constant $C>0$ such that
	\begin{align*}
		\big( \Erw \big( \sup_{0 \leq n \leq N} \| X_{t_n} - Y_n^{\hmax} \|^p \big) 
		\big)^{\frac{1}{p}} \leq C \, \hmax
	\end{align*}
	on any discretization~$\IhN$ for all $0 < \hmax \leq \hnull$.
\end{thm}
For the proof of Theorem~\ref{Sec:Main-Result:Thm-Konv-CommNoise-Strato} we
refer to Section~\ref{Sub:Sec:Proof-Convergence-CommNoise-Srato}. As in the 
non-commutative case it turns out that there are no relaxed conditions for 
convergence with order~$\sfrac{1}{2}$ for the SRK method in case  of Stratonovich
SDEs.
Since the order~$1$ conditions in
Theorem~\ref{Sec:Main-Result:Thm-Konv-CommNoise-Strato}
are the same as the ones in Theorem~\ref{Sec:Main-Result:Thm-Konv-SRK-allg}, we
need at least $s \geq 2$ stages. Further, the coefficients proposed in Table~\ref{Table1}
can be applied for the Stratonovich case with commutative noise as well. Therefore,
we denote the SRK scheme~\eqref{SRK-method} 
with~\eqref{Sec:SRK-method-CommNoise:RVs-Strato} based on the coefficients 
on the left hand side in Table~\ref{Table1} as $\SRSCzweiSeins$ converging with
order~$(1,1)$ and the scheme based on the coefficients on the right hand side in 
Table~\ref{Table1} as $\SRSCzweiSzwei$ converging with order~$(2,1)$.
\subsection{The SRK Method for SDEs with Additive Noise}
\label{Sec:SRK-Method-AdditiveNoise}
Stochastic differential equations with so-called additive noise appear in many applications
and there exist simplified numerical schemes for such equations, see, e.~g., 
\cite{KP10,Mil95,MilTret21,Roe10}. If SRK method~\eqref{SRK-method} is applied
to SDEs with additive noise, we get a significantly simplified SRK method with considerably
less order conditions. Therefore, assume that $b^k(t,x) \equiv b^k(t)$, i.~e.\ 
$b^k$ may depend on time $t$ but does not depend on the state variable $x$. Then
SDE~\eqref{SDE-Integral-form} reduces to a so-called SDE with additive noise 
given as
\begin{align} \label{SDE-Integral-form-AdditiveNoise}
	X_t &= X_{t_0} + \int_{t_0}^t a(s,X_s) \, \mathrm{d}s
	+ \sum_{k=1}^m \int_{t_0}^t b^k(s) \, \mathrm{d}W_s^k
\end{align}
for $t \in \mathcal{I}=[t_0,T]$ with initial value $X_{t_0} \in L^p(\Omega)$ for some $p \geq 2$ 
and measurable functions $a \colon \mathcal{I} \times \mathbb{R}^d \to \mathbb{R}^d$ and 
$b^k \colon \mathcal{I} \to \mathbb{R}^d$. Note that in the case of additive noise there is no
difference between It{\^o} and Stratonovich SDEs~\eqref{SDE-Integral-form-AdditiveNoise} 
and thus only one numerical method is needed. It is always assumed that 
SDE~\eqref{SDE-Integral-form-AdditiveNoise} has a unique strong solution.

We consider the following $s$ stages SRK method for SDEs with additive noise of 
type~\eqref{SDE-Integral-form-AdditiveNoise} that is defined by $Y_0=X_{t_0}$ and
\begin{subequations} \label{SRK-method-AdditiveNoise}
\begin{align} \label{SRK-method-AdditiveNoise-eqn}
	Y_{n+1} &= Y_n + \sum_{i=1}^s \alpha_i \, a(t_n + c_i^{(0)} h_n, H_{i,n}^{(0)}) \, h_n
	+ \sum_{i=1}^s \sum_{k=1}^m \beta_i^{(1)} \, \Ii_{(k),n} \,
	b^k(t_n + c_i^{(1)} h_n)
\end{align}
for $n=0,1, \ldots, N-1$ with stages
\begin{align}
	H_{i,n}^{(0)} &= Y_n + \sum_{j=1}^s A_{i,j}^{(0)} \, a(t_n + c_j^{(0)} h_n, H_{j,n}^{(0)}) 
	\, h_n
	\label{SRK-method-stage-H0-AdditiveNoise}
\end{align}
\end{subequations}
for $i=1, \ldots, s$ with Gaussian random variables $\Ii_{(k),n}$ defined 
in~\eqref{Sec:SRK-method:RVs}. Then, the coefficients of SRK
method~\eqref{SRK-method-AdditiveNoise} can be represented
by a reduced Butcher tableau given as
{
\renewcommand{\arraystretch}{1.6}
\begin{align} \label{Butcher-Tableau-SRK-AdditiveNoise}
	\begin{array}{r|c|c}
		c^{(0)} & A^{(0)} & \multicolumn{1}{l}{ c^{(1)} } \\
		\hline
		& \alpha^T & {\beta^{(1)}}^T
	\end{array}
\end{align}
}
that defines an SRK scheme of type~\eqref{SRK-method-AdditiveNoise}.
\subsubsection{Convergence of the SRK Method for SDEs with Additive Noise}
\label{SubSec:Convergence-AdditiveNoise}
For the analysis of convergence for SRK method~\eqref{SRK-method-AdditiveNoise},
we note that this method turns out to be a special case of SRK
method~\eqref{SRK-method}. However, in case of additive noise no stage 
values for the functions $b^k$ are needed which results in simplified 
order conditions. 
Therefore, we get  the following convergence result.
\begin{thm} 
	\label{Sec:Main-Result:Thm-Konv-AdditiveNoise}
	Let $p \geq 2$, let $X_{t_0} \in L^{2p}(\Omega)$ and let the assumptions in 
	Section~\ref{Sec:Setting} be fulfilled. Further, let $(X_t)_{t \in \mathcal{I}}$ 
	denote the solution of SDE~\eqref{SDE-Integral-form-AdditiveNoise} and
	for $N \in \mathbb{N}$ let $(Y_n^{\hmax})_{0 \leq n \leq N}$
	denote the approximations calculated by SRK 
	method~\eqref{SRK-method-AdditiveNoise} on some grid 
	$\IhN = \{t_0, \ldots, t_N\}$ with maximum step size $\hmax$.
	If the coefficients of SRK method~\eqref{SRK-method-AdditiveNoise} fulfill
	the conditions
	\begin{align} \label{Sec:Main-Result:Thm-Konv-SRK-AddNoise-OrdCond-1.0}
		&\alpha^T e = 1, 
		&\quad \quad &{\beta^{(1)}}^T e = 1 , 
	\end{align}
	then there exist some $\hnull >0$ and some constant $C>0$ such that
	\begin{align}
		\big( \Erw \big( \sup_{0 \leq n \leq N} \| X_{t_n} - Y_n^{\hmax} \|^p \big) 
		\big)^{\frac{1}{p}} \leq C \, \hmax
		\label{Sec:Main-Result:Thm-Konv-AdditiveNoise-Eqn}
	\end{align}
	on any discretization~$\IhN$ for all $0 < \hmax \leq \hnull$.
\end{thm}
%
%

\begin{proof}
	In case of additive noise, i.~e., if $b^k(t,x) = b^k(t)$, the SRK method~\eqref{SRK-method}
	simplifies to the SRK method
	\begin{align} \label{SRK-method-AdditiveNoise-full-eqn}
		Y_{n+1} &= Y_n + \sum_{i=1}^s \alpha_i \, 
		a(t_n + c_i^{(0)} h_n, H_{i,n}^{(0)}) \, h_n
		+ \sum_{i=1}^s \sum_{k=1}^m 
		\big( \beta_i^{(1)} \, \Ii_{(k),n} + \beta_i^{(2)} \big) \, b^k(t_n + c_i^{(1)} h_n)
	\end{align}
	with stages values $H_{i,n}^{(0)}$ given in~\eqref{SRK-method-stage-H0-AdditiveNoise}.
	Now, Theorem~\ref{Sec:Main-Result:Thm-Konv-SRK-allg} can be applied for
	SRK method~\eqref{SRK-method-AdditiveNoise-full-eqn} and 
	conditions~\eqref{Sec:Main-Result:Thm-Konv-SRK-allg-OrdCond-0.5}--\eqref{Sec:Main-Result:Thm-Konv-SRK-allg-OrdCond-1.0} need to be fulfilled for
	convergence with order $1$. However, since $b^k$ does not depend on $x$ it follows 
	from the proof of Theorem~\ref{Sec:Main-Result:Thm-Konv-SRK-allg} that 
	condition~\eqref{Sec:Main-Result:Thm-Konv-SRK-allg-OrdCond-1.0} is not relevant anymore,
	see, e.~g., term~\eqref{Proof:MainThm:Teil-A2-1} that vanishes in the proof.
	In the absence of condition~\eqref{Sec:Main-Result:Thm-Konv-SRK-allg-OrdCond-1.0}
	we can choose $\beta_i^{(2)} =0$ for $i=1, \ldots, s$. Then, the order conditions 
	given in~\eqref{Sec:Main-Result:Thm-Konv-SRK-allg-OrdCond-0.5} reduce 
	to~\eqref{Sec:Main-Result:Thm-Konv-SRK-AddNoise-OrdCond-1.0} and 
	\eqref{SRK-method-AdditiveNoise-full-eqn} coincides with SRK 
	method~\eqref{SRK-method-AdditiveNoise}.
\end{proof}
\begin{table}[htbp]
	{
		\renewcommand{\arraystretch}{1.5}
		\begin{align*} %
			\begin{array}{r|c|c}
				0 & \makebox[0.5cm][c]{ } & \makebox[0.5cm][c]{0} \\
				\hline
				& 1 & 1
			\end{array}
			\quad \quad \quad \quad \quad \quad
			%
			\begin{array}{r|cc|cc} 
				0 &\makebox[0.5cm][c]{ } & \makebox[0.5cm][c]{ } &  
				\makebox[0.5cm][c]{0} & \makebox[0.5cm][c]{ } \\
				1 & 1 &  &  0 &  \\
				\hline
				& \frac{1}{2} & \frac{1}{2} & 1 & 0
			\end{array}
		\end{align*}
	}
	\caption{SRK scheme $\SRIAzweiSeins$ with $s=1$ stage of order $(1,1)$ on 
		the left hand side and SRK scheme 
		$\SRIAzweiSzwei$ with $s=2$ stages of order~$(2,1)$ on the right
		hand side.}
	\label{Table2}
\end{table}

Considering the order 
conditions~\eqref{Sec:Main-Result:Thm-Konv-SRK-AddNoise-OrdCond-1.0}
it directly follows that solutions are already available for $s=1$ stage. 
As an example, one can consider the $\SRIAzweiSeins$ scheme with 
$s=1$ stage given on the left hand side of Table~\ref{Table2}. This 
scheme coincides with the Euler-Maruyama scheme
if it is applied to SDE~\eqref{SDE-Integral-form-AdditiveNoise} and has 
order~$(1,1)$. For $\alpha$ and $A^{(0)}$ we can use coefficients of 
any deterministic explicit or implicit Runge-Kutta scheme for ODEs, 
see, e.~g., \cite{HNW93}. As an example, we once more choose the 
Heun scheme for $\alpha$ and $A^{(0)}$ and consider the 
$\SRIAzweiSzwei$ scheme with $s=2$ stages given on the right hand 
side of Table~\ref{Table2}. 
Further, the well known drift-implicit split-step stochastic backward Euler 
($\SSBE$) method for SDEs with additive noise considered in, e.~g., 
\cite{MaStuHig02} is also a special case of SRK 
method~\eqref{SRK-method-AdditiveNoise} that has order of 
convergence~$(1,0.5)$.
We note that it may be beneficial to choose 
the coefficients $c^{(1)}$ in a suitable  way, e.~g., either to use some 
quadrature formula with high accuracy or to the minimize the 
error constant.
%
%
%
%
%
\subsection{Pathwise Convergence of the SRK Method}
\label{SubSec:Pathwise-Convergence}
Convergence of numerical approximations to solutions of SDEs in
$L^p(\Omega)$-norm may provide also pathwise convergence. So,
if the approximations by the SRK method~\eqref{SRK-method} are 
converging to the solutions of It{\^o} SDE~\eqref{SDE-Integral-form}
in $L^p(\Omega)$-norm for any $p \geq 2$, i.~e., if
Theorem~\ref{Sec:Main-Result:Thm-Konv-SRK-allg} or
Theorem~\ref{Sec:Main-Result:Thm-Konv-CommNoise} in case of 
commutative noise hold for any $p \geq 2$, then pathwise convergence 
with nearly the same order of convergence applies. Especially, this 
remains also valid if the iterated stochastic integrals are approximated
such that Proposition~\ref{Sec:Main-Result:Prop-Konv-Approx-IterIntegrals}
holds for any $p \geq 2$. The following corollary is an immediate 
consequence of~\cite[Lemma~2.1]{KN07} which goes back to an idea
in~\cite{Gy98}.
\begin{cor} \label{Cor:Pathwise-conv-SRK-method}
	Let $(X_t)_{t \in \mathcal{I}}$ be the solution of It{\^o}
	SDE~\eqref{SDE-Integral-form},
	let Theorem~\ref{Sec:Main-Result:Thm-Konv-SRK-allg} and 
	Proposition~\ref{Sec:Main-Result:Prop-Konv-Approx-IterIntegrals} or
	Theorem~\ref{Sec:Main-Result:Thm-Konv-CommNoise} in case of 
	commutative noise for all $p \geq 2$ be valid. Then, it holds:
	\begin{enumerate}[i)]
		\item \label{Cor:Pathwise-conv-SRK-method-i}
		If the coefficients of SRK method~\eqref{SRK-method}
		fulfill conditions~\eqref{Sec:Main-Result:Thm-Konv-SRK-allg-OrdCond-0.5}, 
		then there exists some $\hnull > 0$ and
		for any $\varepsilon>0$ there exists some random variable $\Theta_{\varepsilon}
		\in L^p(\Omega)$ for all $p \geq 2$ such that
		\begin{align*}
			\sup_{0 \leq n \leq N} \| X_{t_n} - Y_n^{\hmax} \| 
			\leq \Theta_{\varepsilon} \, \hmax^{\frac{1}{2} -\varepsilon}
			\quad \quad \Prob\text{-a.s.}
		\end{align*}
		on any discretization~$\IhN$ for all $0 < \hmax \leq \hnull$.
		\item \label{Cor:Pathwise-conv-SRK-method-ii}
		If the coefficients of SRK method~\eqref{SRK-method} fulfill
		conditions~\eqref{Sec:Main-Result:Thm-Konv-SRK-allg-OrdCond-0.5}
		and~\eqref{Sec:Main-Result:Thm-Konv-SRK-allg-OrdCond-1.0}, 
		then there exists some $\hnull > 0$ and
		for any $\varepsilon>0$ there exists some random variable $\Theta_{\varepsilon}
		\in L^p(\Omega)$ for all $p \geq 2$ such that
		\begin{align*}
			\sup_{0 \leq n \leq N} \| X_{t_n} - Y_n^{\hmax} \| 
			\leq \Theta_{\varepsilon} \, \hmax^{1 -\varepsilon}
			\quad \quad \Prob\text{-a.s.}
		\end{align*}
		on any discretization~$\IhN$ for all $0 < \hmax \leq \hnull$.
	\end{enumerate}
\end{cor}

For the approximation of solutions to Stratonovich SDEs using SRK
method~\eqref{SRK-method} with~\eqref{SRK-method-Ito-Strato-Iihat} 
or~\eqref{Sec:SRK-method-CommNoise:RVs-Strato} in case of commutative
noise, we get analogous results for pathwise convergence
of the approximations if convergence
in $L^p(\Omega)$-norm is guaranteed for any $p \geq 2$.
%
%
\begin{cor} \label{Cor:Pathwise-conv-SRK-method-Comm}
	Let $(X_t)_{t \in \mathcal{I}}$ be the solution of
	Stratonovich SDE~\eqref{SDE-Integral-form},
	let Theorem~\ref{Sec:Main-Result:Thm-Konv-SRK-allg-Strato} and 
	Proposition~\ref{Sec:Main-Result:Prop-Konv-Approx-IterIntegrals} or
	Theorem~\ref{Sec:Main-Result:Thm-Konv-CommNoise-Strato} in case of 
	commutative noise for all $p \geq 2$ be valid.
	If the coefficients of SRK method~\eqref{SRK-method} fulfill 
	conditions~\eqref{Sec:Main-Result:Thm-Konv-SRK-allg-OrdCond-0.5}
	and~\eqref{Sec:Main-Result:Thm-Konv-SRK-allg-OrdCond-1.0}, 
	then there exists some $\hnull > 0$ and
	for any $\varepsilon>0$ there exists some random variable $\Theta_{\varepsilon}
	\in L^p(\Omega)$ for all $p \geq 2$ such that
	\begin{align*}
		\sup_{0 \leq n \leq N} \| X_{t_n} - Y_n^{\hmax} \| 
		\leq \Theta_{\varepsilon} \, \hmax^{1 -\varepsilon}
		\quad \quad \Prob\text{-a.s.}
	\end{align*}
	on any discretization~$\IhN$ for all $0 < \hmax \leq \hnull$.
\end{cor}

We point out that the pathwise convergence results in 
Theorem~\ref{Cor:Pathwise-conv-SRK-method}
and Theorem~\ref{Cor:Pathwise-conv-SRK-method-Comm} remain valid
if corresponding approximate iterated stochastic integrals
$\Iihat_{(l,k),n} = \Iitilde_{(l,k),n}$ are used for SRK 
method~\eqref{SRK-method} such that
Proposition~\ref{Sec:Main-Result:Prop-Konv-Approx-IterIntegrals} 
holds for all $p \geq 2$.

Finally, we also get a corresponding pathwise convergence result for the approximations
calculated by SRK method~\eqref{SRK-method-AdditiveNoise} in case of additive noise.
\begin{cor} \label{Cor:Pathwise-conv-SRK-method-AddNoise}
	Let $(X_t)_{t \in \mathcal{I}}$ be the solution of 
	SDE~\eqref{SDE-Integral-form-AdditiveNoise}
	and let Theorem~\ref{Sec:Main-Result:Thm-Konv-AdditiveNoise}
	for all $p \geq 2$ be valid.
	If the coefficients of SRK method~\eqref{SRK-method-AdditiveNoise}
	fulfill conditions~\eqref{Sec:Main-Result:Thm-Konv-SRK-AddNoise-OrdCond-1.0}, 
	then there exists some $\hnull > 0$ and
	for any $\varepsilon>0$ there exists some random variable $\Theta_{\varepsilon}
	\in L^p(\Omega)$ for all $p \geq 2$ such that
	\begin{align*}
		\sup_{0 \leq n \leq N} \| X_{t_n} - Y_n^{\hmax} \| 
		\leq \Theta_{\varepsilon} \, \hmax^{1 -\varepsilon}
		\quad \quad \Prob\text{-a.s.}
	\end{align*}
	on any discretization~$\IhN$ for all $0 < \hmax \leq \hnull$.
\end{cor}
%
%
%
%
%
%
\section{Computational Cost and Numerical Examples}
\label{Sec:CompCost}
For an analysis of the performance of the proposed SRK schemes, we compare
these schemes with some well known schemes from the literature. Therefore,
we first need to measure the computational complexity of a one step approximation
scheme for SDEs. Then, some numerical experiments are executed for a comparison
of performance and to verify the order of convergence for the numerical schemes
under consideration.
\subsection{A Computational Cost Model}
\label{Sec:CompCost:Cost-Model}
For a comparison of the performance of numerical schemes we need to measure 
the computational cost or effort for each scheme and relate it to the corresponding 
approximation error. Here, we follow the approach considered in, e.~g., \cite{Roe10}
and measure the computational complexity or cost~$\ccosts$ as the number of
necessary evaluations of real valued functions like the drift functions~$a^i$ and 
the diffusion functions~$b^{i,k}$ for $1 \leq i \leq d$ and $1 \leq k \leq m$, 
or any derivative of one of these functions and so forth. 
We presume that each evaluation of any scalar valued function is at
cost of~$1$ unit. Elementary arithmetic operations are usually much less expensive
and therefore they are neglected to keep the cost model simple.
Additionally, the generation of each realization of some standard Gaussian
random variable is taken into account with cost of~$1$ unit as well. 
This is reasonable as higher order approximation schemes in general 
need realizations of increments of the Wiener process and of iterated stochastic 
integrals which can be computed by
realizations of independent Gaussian random variables if, e.~g., state-of-the-art 
Fourier based algorithms are used~\cite{Wik01,MroRoe22}. Since this also
takes relevant computing time that may actually depend on the step size it is 
reasonable to factor their computational effort into the entire 
computation cost. Clearly, if there will be more efficient algorithms available
in future, this may be adjusted.

\begin{table}[htbp]
	\begin{center}
	{\renewcommand{\arraystretch}{1.45}
	\begin{tabular}{|c|c|c|c|c|c|c|}
		\cline{3-7}
		\multicolumn{2}{c|}{} & \multicolumn{3}{c|}{Number of evaluations} 
		& \multicolumn{2}{c|}{random variables} \\
		\hline
		\makebox[1.6cm][c]{Scheme} & \makebox[1.6cm][c]{Order} 
		& \makebox[1.35cm][c]{$a^i$} 
		& \makebox[1.35cm][c]{$b^{i,k}$} 
		& \makebox[1.35cm][c]{$\frac{\partial b^{i,k}}{\partial x^j}$}
		& \makebox[1.35cm][c]{$\Ii_{(k)}$} & \makebox[1.35cm][c]{$\Ii_{(k,l)}$} \\
		\hline
		\hline
		$\EM$ & $(1,0.5)$ & $d$ & $d \, m$ & $-$ & $+$ & $-$ \\
		%
		$\MIL$ & $(1,1)$ & $d$ & $d \, m$ & $d^2 \, m$ & $+$ & $+$ \\
		%
		$\SPLI$ & $(1,1)$ & $d$ & $d \, (m^2+m)$ & $-$ & $+$ & $+$ \\
		\hline
		$\SRIeins$ & $(1,1)$ & $d$ & $3 \, d \, m$ & $-$ & $+$ & $+$ \\
		%
		$\SRICeins$ & $(1,1)$ & $d$ & $3 \, d \, m$ & $-$ & $+$ & $-$ \\
		\hline
		$\SRIzweiSeins$ & $(1,1)$ & $d$ & $2 \, d \, m$ & $-$ & $+$ & $+$ \\
		%
		$\SRIzweiSzwei$ & $(2,1)$ & $2 \, d$ & $2 \, d \, m$ & $-$ & $+$ & $+$ \\
		%
		$\SRICzweiSeins$ & $(1,1)$ & $d$ & $2 \, d \, m$ & $-$ & $+$ & $-$ \\
		%
		$\SRICzweiSzwei$ & $(2,1)$ & $2 \, d$ & $2 \, d \, m$ & $-$ & $+$ & $-$ \\
		%
		$\SRSzweiSeins$ & $(1,1)$ & $d$ & $2 \, d \, m$ & $-$ & $+$ & $+$ \\
		%
		$\SRSzweiSzwei$ & $(2,1)$ & $2 \, d$ & $2 \, d \, m$ & $-$ & $+$ & $+$ \\
		%
		$\SRSCzweiSeins$ & $(1,1)$ & $d$ & $2 \, d \, m$ & $-$ & $+$ & $-$ \\
		%
		$\SRSCzweiSzwei$ & $(2,1)$ & $2 \, d$ & $2 \, d \, m$ & $-$ & $+$ & $-$ \\
		%
		$\SRIAzweiSeins$ & $(1,1)$ & $d$ & $d \, m$ & $-$ & $+$ & $-$ \\
		%
		$\SRIAzweiSzwei$ & $(2,1)$ & $2 \, d$ & $d \, m$ & $-$ & $+$ & $-$ \\
		\hline
	\end{tabular}
	}
	\end{center}
	\caption{Computational complexity of the schemes under consideration each time step for 
	$d$-dimensional SDEs with an $m$-dimensional Wiener process together with their
	order of convergence $(p_D, p_S)$ and the type of random variables needed by
	the schemes.}
	\label{Table3}
\end{table}

In the following, we compare the performance of the Euler-Maruyama ($\EM$)
scheme of order~$\sfrac{1}{2}$, the Milstein ($\MIL$) scheme of order~$1$
and 
the SRK scheme $\SPLI$ of order~$1$ due to Kloeden and Platen~\cite[(11.1.7)]{KP10}
with the newly proposed SRK schemes $\SRIzweiSeins$, $\SRSzweiSeins$,
$\SRICzweiSeins$ and $\SRSCzweiSeins$
of order~$1$ for SDEs with general and commutative noise, respectively. 
Since the SRK schemes $\SRIzweiSeins$ and $\SRICzweiSeins$ as well as 
the scheme $\SRSzweiSeins$ and $\SRSCzweiSeins$ coincide and thus
give exactly the same numerical results if they are applied to SDEs with an $m=1$ 
dimensional Wiener process, we only need to consider the results of one of the two 
schemes, respectively.
Further, in case of additive noise we consider the $\EM$, the $\MIL$, the $\SPLI$,
the $\SRIAzweiSeins$ and the $\SRIAzweiSzwei$ scheme of order~$1$. Since all these schemes
except of $\SRIAzweiSzwei$ coincide in case of SDEs with additive noise, we get exactly 
the same numerical results for each of these schemes whereas the $\SRIAzweiSzwei$
scheme provides different results.
The computational complexity of these schemes for each time step is summarized
in Table~\ref{Table3}.

For the simulation of
iterated stochastic integrals $\Ii_{(k,l),n}$, we apply the recently proposed 
algorithm due to Mrongowius and R{\"o}{\ss}ler~\cite{MroRoe22} that is based 
on the seminal work by Wiktorsson~\cite{Wik01}. For this algorithm there exist
also error estimates in $L^p(\Omega)$-norm for $p \geq 2$, see~\cite{MroRoe22}.
To be precise, we use the software
toolbox~\cite{KaRoe23} for the simulation of iterated stochastic integrals in all
presented simulations.
In order to fulfill conditions~\eqref{Sec:Main-Result:Prop-Konv-Approx-IterIntegrals-eqn} 
of Proposition~\ref{Sec:Main-Result:Prop-Konv-Approx-IterIntegrals}, i.~e., to 
approximate the iterated stochastic integrals $\Ii_{(k,l),n}$ for $1 \leq k,l \leq m$, $k \neq l$,
with accuracy $\Oo(h^{\sfrac{3}{2}})$, we need to simulate realizations of
$\cCost=\cCost(m,h) = \lceil \frac{m^{\sfrac{3}{2}}}{\sqrt{3 \, h} \, \pi}  + \frac{m^2+m}{2} \rceil$ 
standard Gaussian random variables each step with step size $h$ provided realizations of
$\Ii_{(k),n}$ for $1 \leq k \leq m$ are given, see~\cite[Table~2]{KaRoe23}. 

Thus, if we consider a $d$-dimensional SDE system with an $m$-dimensional Wiener
process then the computational cost $\ccosts$ of calculating one step with step size~$h$ 
for, e.~g., the 
$\MIL$ scheme is $\ccosts_{\MIL}(d,m,h) = d + d \, m + d^2 \, m + m + \cCost(m,h)$ 
and for the $\SPLI$ scheme it is $\ccosts_{\SPLI}(d,m,h) = d+d \, (m^2+m) + m + \cCost(m,h)$ 
whereas for the newly proposed $\SRIzweiSeins$ scheme it is only 
$\ccosts_{\SRIzweiSeins}(d,m,h) = d + 2 \, d \, m + m + \cCost(m,h)$ 
each step. Thus, 
compared to the Milstein scheme or compared to the SRK scheme due to Kloeden and
Platen, the proposed SRK method has much lower computational 
complexity that becomes more and more significant for higher dimensional problems.
We point out that in contrast to the SRK schemes $\SRIeins$ and $\SRICeins$ 
given in~\cite{Roe10}, the new SRK schemes $\SRIzweiSeins$ and $\SRICzweiSeins$ are now
also competitive in case of $d=m=1$. Moreover note that the improved efficiency of the SRK 
scheme~$\SRIzweiSeins$ is independent of $\cCost$, i.~e., it does not depend on the
method chosen to approximate or simulate the iterated stochastic integrals.
\subsection{Numerical Examples}
\label{Sec:NumEx}
For a performance analysis, we consider several test equations for 
the numerical experiments where
we use equidistant time discretizations $\IhN$ of $\mathcal{I}=[0,T]$
with different step sizes. For each step size~$h$, we 
compare the $L^2(\Omega)$-errors $\err(h)$ estimated based on~$M$ simulated 
independent trajectories
%
%
for each numerical scheme under consideration at time $T=1$.
Then, the $L^2(\Omega)$-errors versus step sizes as 
well as the $L^2(\Omega)$-errors versus computational effort are plotted in 
$\log$-$\log$-diagrams with base~$2$. Since the computational 
effort depends non-linearly on the step size~$h$ whenever iterated stochastic integrals 
need to be approximated at cost $\cCost(m,h) = \Oo(h^{-\sfrac{1}{2}})$, it is reasonable 
to compare errors versus computational effort rather than errors versus step sizes. Therefore,
following~\cite{Roe10} we denote by~$\peff$ the effective order of convergence defined 
by
\begin{align}
	\err(h) = \Oo \big(\ccosts(d,m,h)^{-\peff} \big) \, .
\end{align}
Clearly, if $\ccosts(d,m,h)$ does not depend on the step size~$h$, then the effective 
order~$\peff$ coincides with the usual order of convergence~$\gamma$ if 
$\err(h) = \Oo(h^{\gamma})$ is guaranteed. In general, it holds $\peff \leq \gamma$.
In our numerical experiments we determine the effective order of convergence by 
the relation
\begin{align}
	\peff = \lim_{h \to 0} \bigg| \frac{\log( \err(h)) - \log( \err(h/2))}{\log(\ccosts(d,m,h)) 
		- \log(\ccosts(d,m,h/2))} \bigg| \, .
\end{align}
That is, we get~$\peff$ as the slope of the resulting line in $L^2(\Omega)$-errors versus
computational effort $\log$-$\log$-diagrams whereas we get the usual order
of convergence~$\gamma$ as the slope of the resulting line in $L^2(\Omega)$-errors versus
step sizes $\log$-$\log$-diagrams as $h \to 0$.

The effective order of convergence is convincing if order~$1$ schemes are compared
with order~$\sfrac{1}{2}$ schemes likes the $\EM$ scheme in case of 
non-commutative noise. The question if it is worth
applying an order~$1$ scheme can be affirmed whenever for the effective order of 
convergence~$\peff$ of the order~$1$ scheme holds that~$\peff \geq \sfrac{1}{2}$.
For example, for the $\EM$ scheme we have~$\peff = \sfrac{1}{2}$ and for
all order~$1$ schemes under consideration like the proposed SRK schemes we calculate
$\peff = \sfrac{2}{3}$ in case of general SDEs with a multidimensional driving Wiener 
process. This justifies the efficiency and superiority of order~$1$ schemes
compared to order~$\sfrac{1}{2}$ schemes like the $\EM$ scheme. Note that the 
reduction of the effective order of convergence results solely from the computational
costs due to the used algorithm for the approximation or simulation 
of realizations of iterated stochastic integrals, see~\cite{MroRoe22}.
Although the algorithm proposed in~\cite{MroRoe22} is already very efficient
the development of more efficient
algorithms for the approximation of iterated stochastic integrals is an open problem
and subject to ongoing research. In case of scalar noise, commutative noise or additive
noise there is no need for the approximation of iterated stochastic integrals and for the
corresponding order~$1$ schemes it thus follows that~$\peff = 1$.
%
%
%
%

For the analysis of the order of convergence we consider several test SDEs with
different properties. For better comparability, we apply a set of test equations
proposed in~\cite{Roe10} that can also be found in~\cite{KP10}. All test SDEs are
given as It{\^o} SDEs and the Stratonovich schemes $\SRSzweiSeins$ and 
$\SRSCzweiSeins$ are always applied to the corresponding Stratonovich SDE with
the same solution process as the presented It{\^o} SDE, see 
also~\cite[Chap.~4.9]{KP10}.
\subsubsection{Test Equation~1}
\begin{figure}[tbp]
		\includegraphics[width=8.2cm]{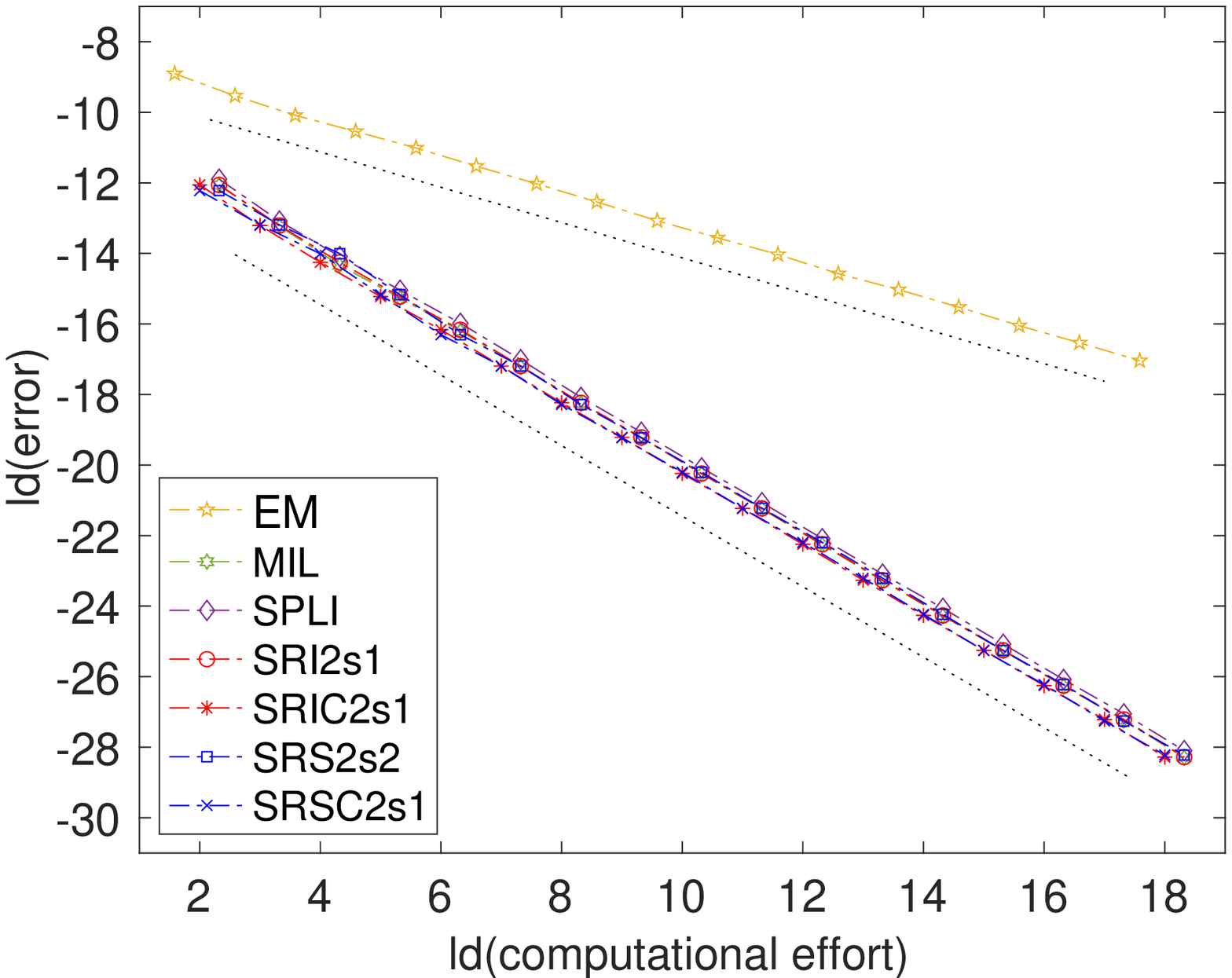}
		\includegraphics[width=8.2cm]{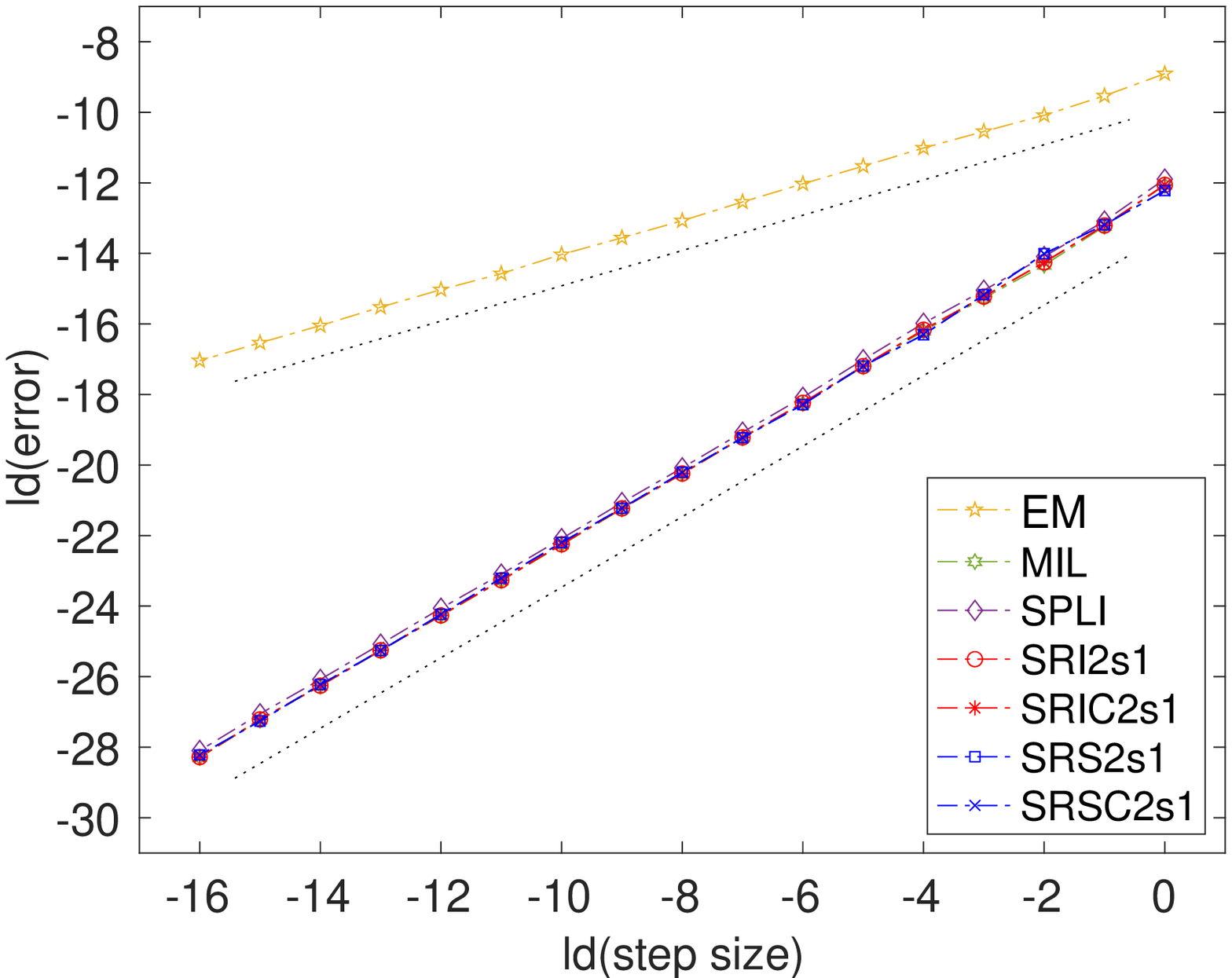}
		%
	\caption{Errors vs.\ computational effort in the left and errors vs.\ step sizes
		in the right figure for SDE~\eqref{Test-SDE-1} with dotted order~$\sfrac{1}{2}$ and 
		order~$1$ lines.}
	\label{Fig-Plot-SDE-4-4-27}
\end{figure}
The first test equation under consideration is the scalar nonlinear 
SDE~\cite[(4.4.27)]{KP10} with $d=m=1$ of the form
\begin{align} \label{Test-SDE-1}
	\mathrm{d}X_t = - \frac{1}{10^2} \, \sin(X_t) \, \cos^3(X_t) \, \mathrm{d}t
	+ \frac{1}{10} \, \cos^2(X_t) \, \mathrm{d}W_t, \quad \quad X_0 = 1,
\end{align}
with solution $X_t = \arctan( \frac{1}{10} \, W_t + \tan(X_0))$. For the numerical
approximations we consider step sizes $h= 2^0, 2^{-1}, \ldots, 2^{-16}$ and 
we simulate $M=2000$ trajectories for each step size. The numerical results are
presented in Figure~\ref{Fig-Plot-SDE-4-4-27}. 
The plots confirm order $\peff = \sfrac{1}{2}$ for the $\EM$ scheme and order 
$\peff = 1$ for the strong order~$1$ schemes under consideration since
SDE~\eqref{Test-SDE-1} involves only a scalar Wiener process. We note that
the $\SRICzweiSeins$ scheme and the $\SRSCzweiSeins$ scheme provide the
best performance compared to the other schemes for this test equation.
\subsubsection{Test Equation~2}
\begin{figure}[tbp]
		\includegraphics[width=8.2cm]{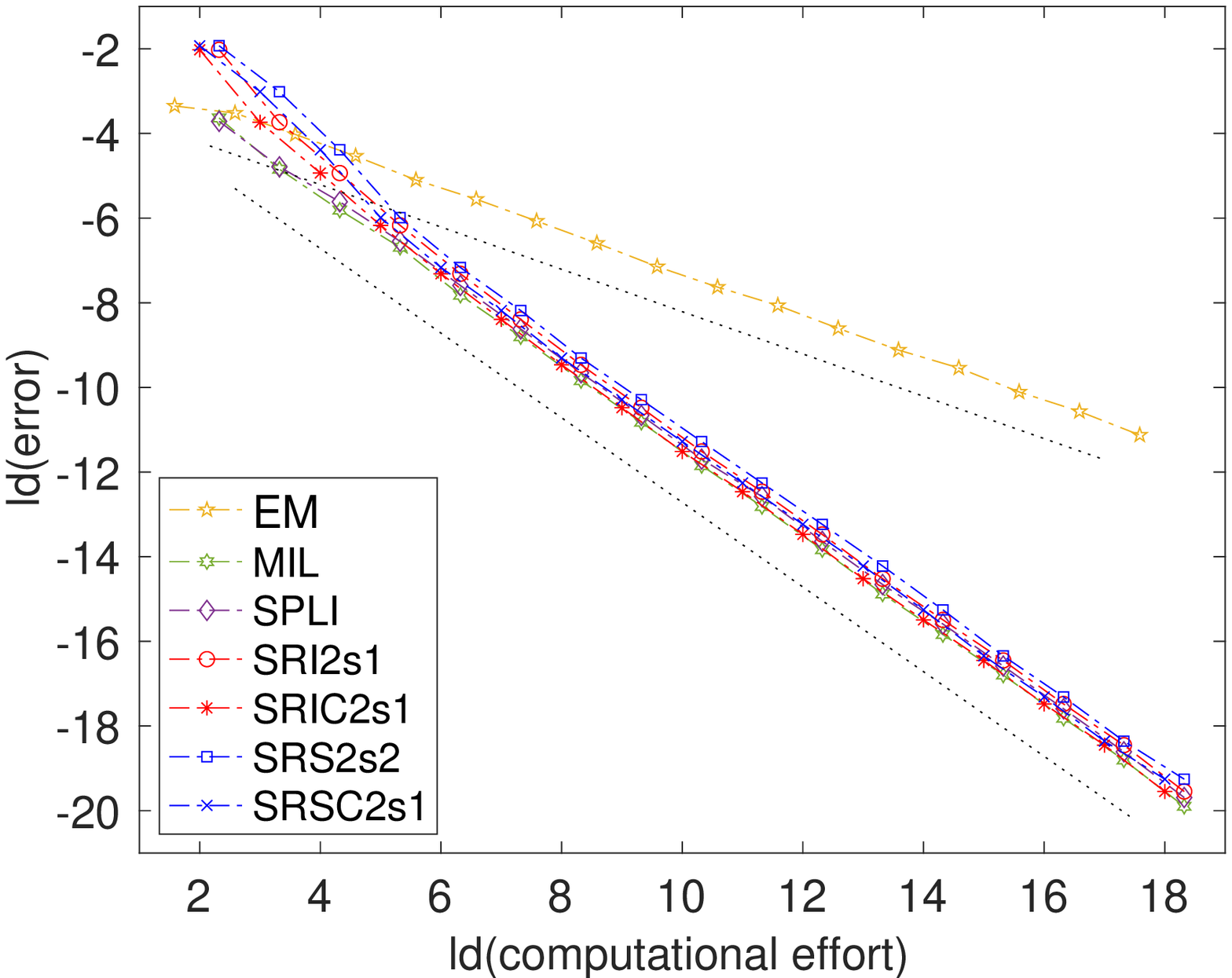}
		\includegraphics[width=8.2cm]{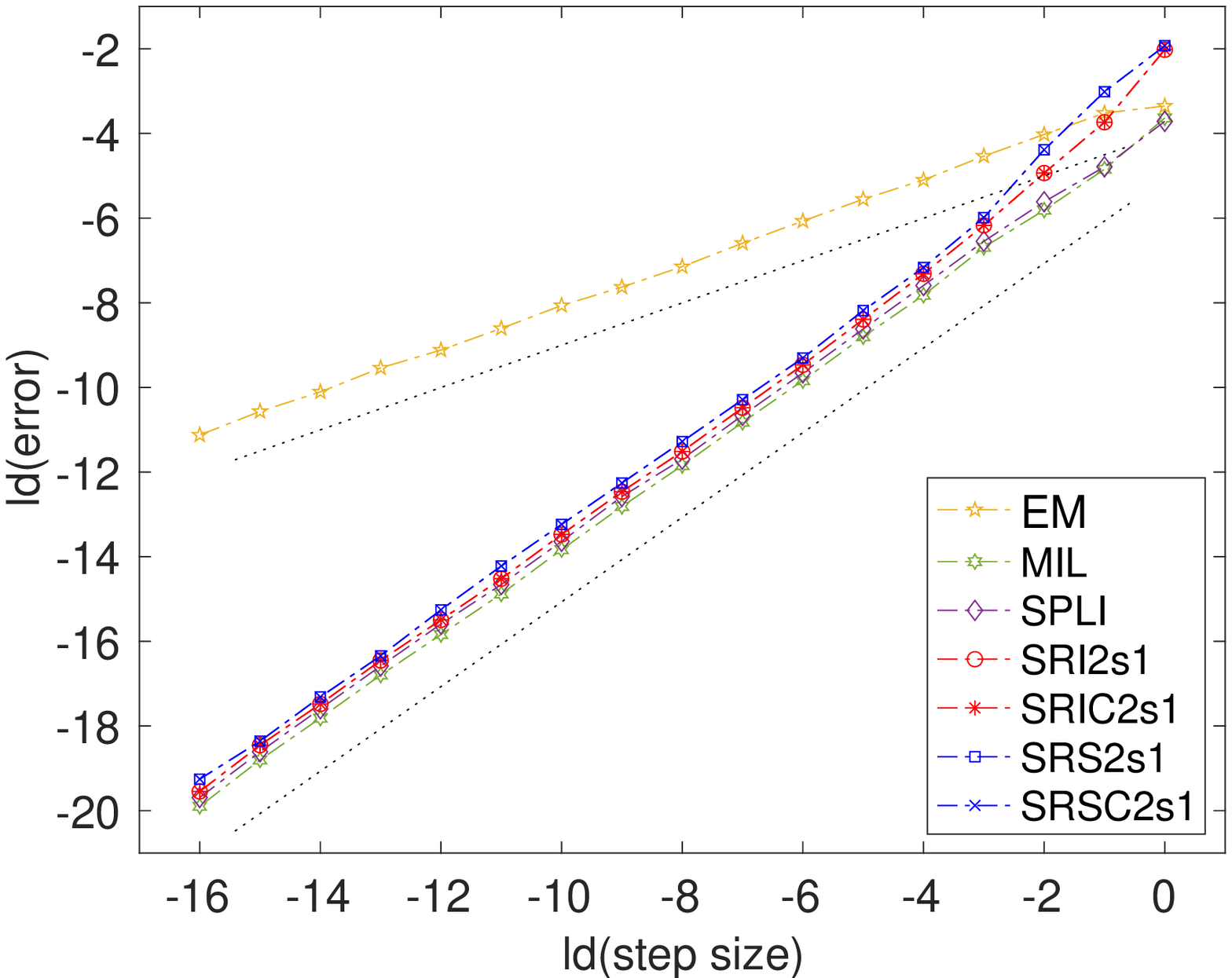}
		%
	\caption{Errors vs.\ computational effort in the left and errors vs.\ step sizes in
		the right figure for SDE~\eqref{Test-SDE-2} with dotted order~$\sfrac{1}{2}$ and
		order~$1$ lines.}
	\label{Fig-Plot-SDE-4-4-40}
\end{figure}
As another nonlinear test equation we consider the following autonomous 
SDE~\cite[(4.4.40)]{KP10}  with $d=m=1$ given by
\begin{align} \label{Test-SDE-2}
	\mathrm{d}X_t = -\frac{1}{4} X_t (1-X_t^2) \, \mathrm{d}t + \frac{1}{2} (1-X_t^2) \,
	\mathrm{d}W_t, \quad \quad X_0 = \frac{1}{2} \, .
\end{align}
The solution to SDE~\eqref{Test-SDE-2} can be explicitly calculated as
$X_t = \frac{(1+X_0) \, \exp(W_t) + X_0 -1}{(1+X_0) \, \exp(W_t) -X_0 +1}$. 
The numerical results are given in Figure~\ref{Fig-Plot-SDE-4-4-40} for step sizes 
$h = 2^0, 2^{-1}, \ldots , 2^{-16}$ with $M=2000$ simulated trajectories. Here,
the order plots reveal $\peff = \sfrac{1}{2}$ for the $\EM$ scheme and due to
$m=1$ we can see $\peff = 1$ for the strong order $1$ schemes. We observe that 
the $\MIL$ scheme and the $\SRICzweiSeins$ scheme yield similar results and 
seem to perform slightly better than the other schemes in the errors versus
computational effort plot.
It has to be noted that SDE~\eqref{Test-SDE-2} does not fulfill the
assumptions in Section~\ref{Sub-Sec:Assumptions}. Nevertheless, the numerical
approximations of all schemes under consideration show convergence without any
reduction of their theoretical orders of convergence.
\subsubsection{Test Equation~3}
\begin{figure}[tbp]
		\includegraphics[width=8.2cm]{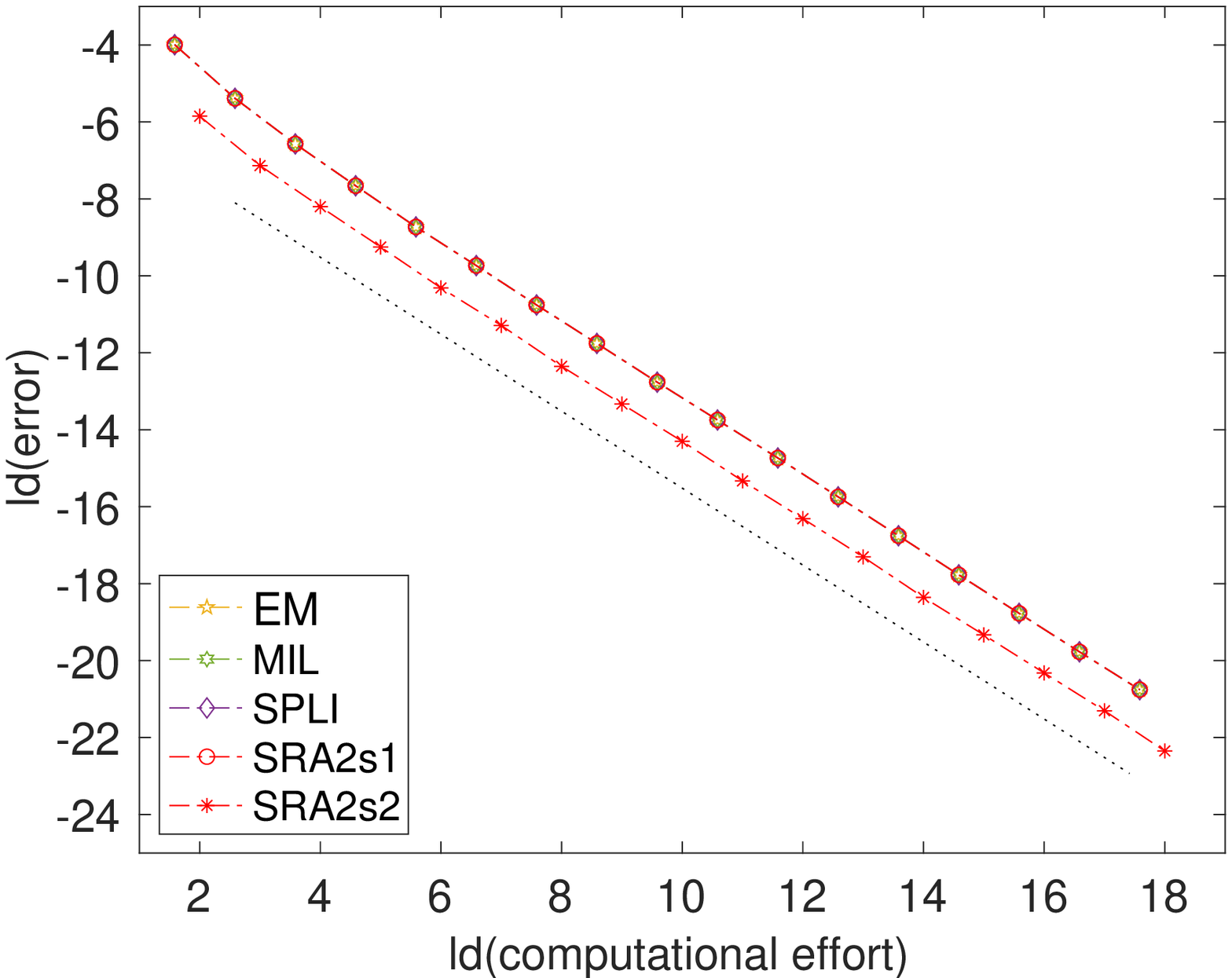}
		\includegraphics[width=8.2cm]{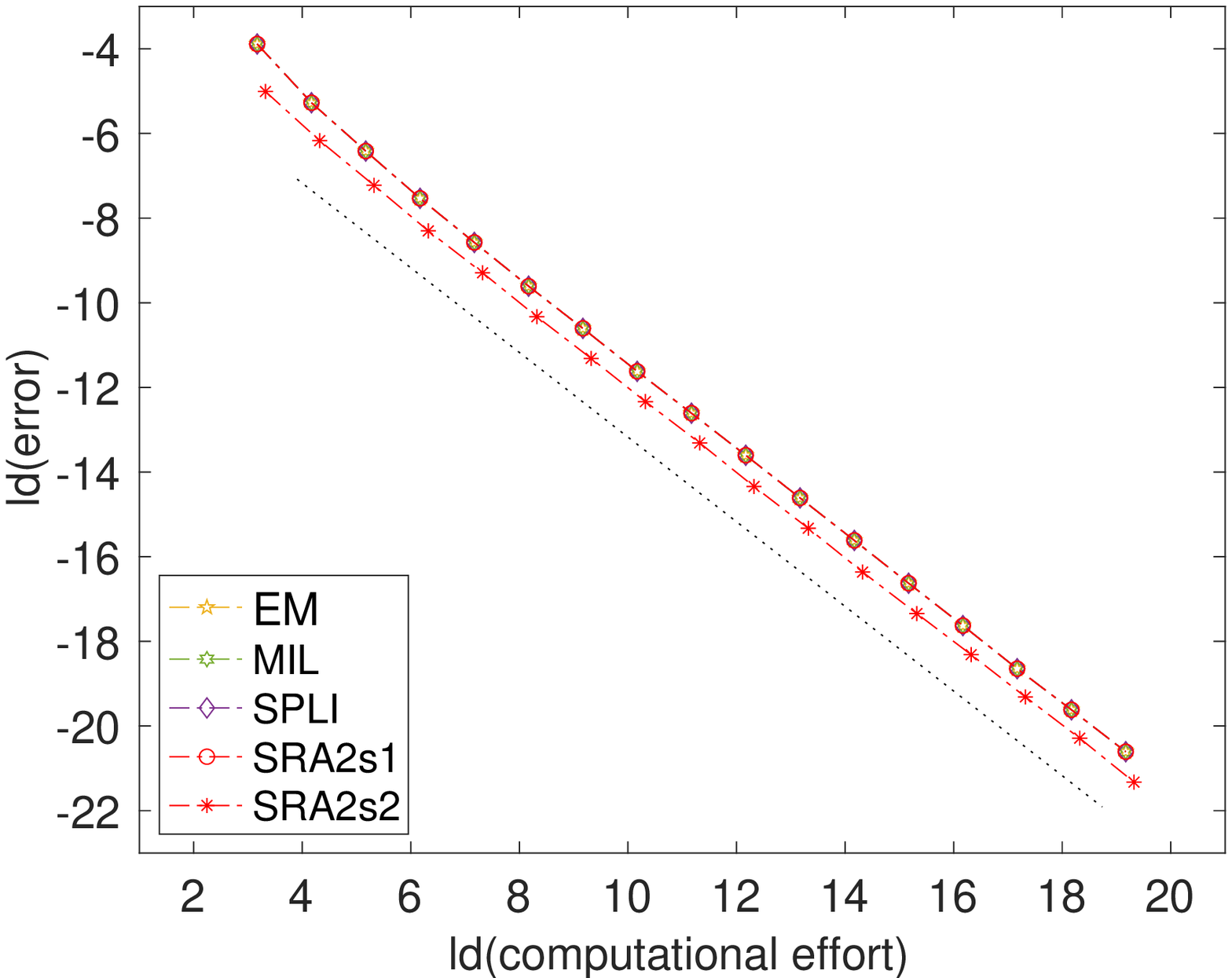}
		\includegraphics[width=8.2cm]{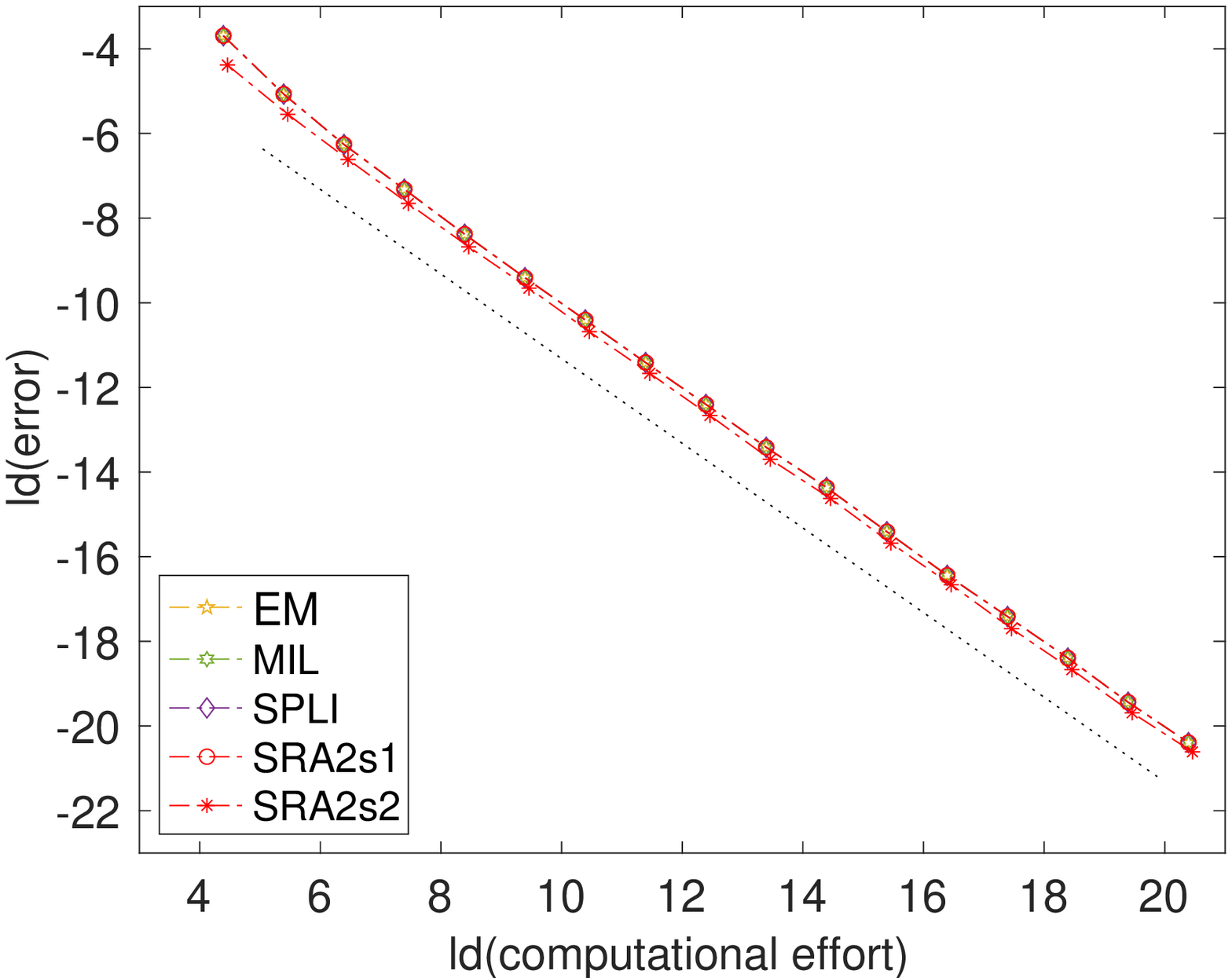}
		\includegraphics[width=8.25cm]{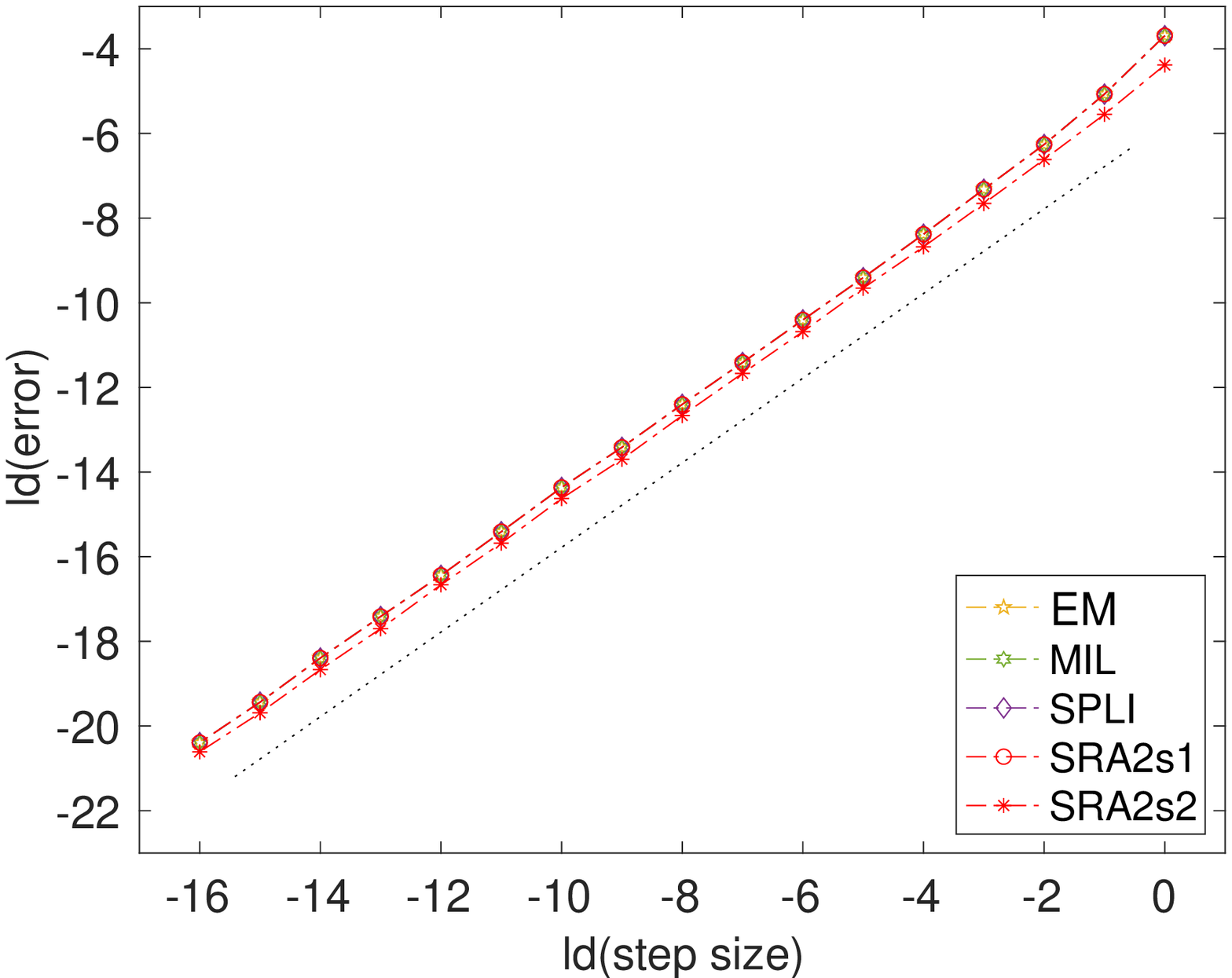}
		%
	\caption{Errors vs.\ computational effort with $m=1$ (top left), $m=4$ (top right) and
		$m=10$ (bottom left) for SDE~\eqref{Test-SDE-3}. 
		Errors vs.\ step sizes with $m=10$ for SDE~\eqref{Test-SDE-3} (bottom right). All
		figures contain dotted order~$1$ lines.}
	\label{Fig-Plot-SDE-4-4-4}
\end{figure}
As an example for a scalar SDE with some multi-dimensional additive noise that possesses
an explicitly known solution, we consider the SDE~\cite[(4.4.4)]{KP10} given as
\begin{align} \label{Test-SDE-3}
	\mathrm{d}X_t = \bigg( \frac{\beta}{\sqrt{1+t}} - \frac{1}{2(1+t)} \, X_t \bigg) \, \mathrm{d}t
	+ \sum_{k=1}^m \alpha_k \, \frac{\beta}{\sqrt{1+t}} \, \mathrm{d}W_t^k, \quad \quad X_0=1,
\end{align}
with solution $X_t = \frac{1}{\sqrt{1+t}} \, X_0 + \frac{\beta}{\sqrt{1+t}} 
(t + \sum_{k=1}^m \alpha_k \, W_t^k)$. Here, we choose parameters $\alpha_k = \frac{1}{10}$
for $k=1, \ldots, m$ and $\beta=\frac{1}{2}$ in SDE~\eqref{Test-SDE-3}. Numerical approximations
are calculated for $d=1$ and $m=1$, $m=4$ and $m=10$. For each of these parameters
we consider step sizes $h=2^0, 2^{-1}, \ldots, 2^{-16}$ and the numerical results based
on $M=2000$ simulated trajectories are
presented in Fig.~\ref{Fig-Plot-SDE-4-4-4}. In case of additive noise, the scheme $\EM$,
$\MIL$, $\SPLI$ and $\SRIAzweiSeins$ coincide and thus give the same numerical results.
The scheme $\SRIAzweiSzwei$ has order $(2,1)$ which results in a better performance
for the considered test equation as can be noticed in Fig.~\ref{Fig-Plot-SDE-4-4-4}.
\subsubsection{Test Equation~4}
\begin{figure}[tbp]
		\includegraphics[width=8.2cm]{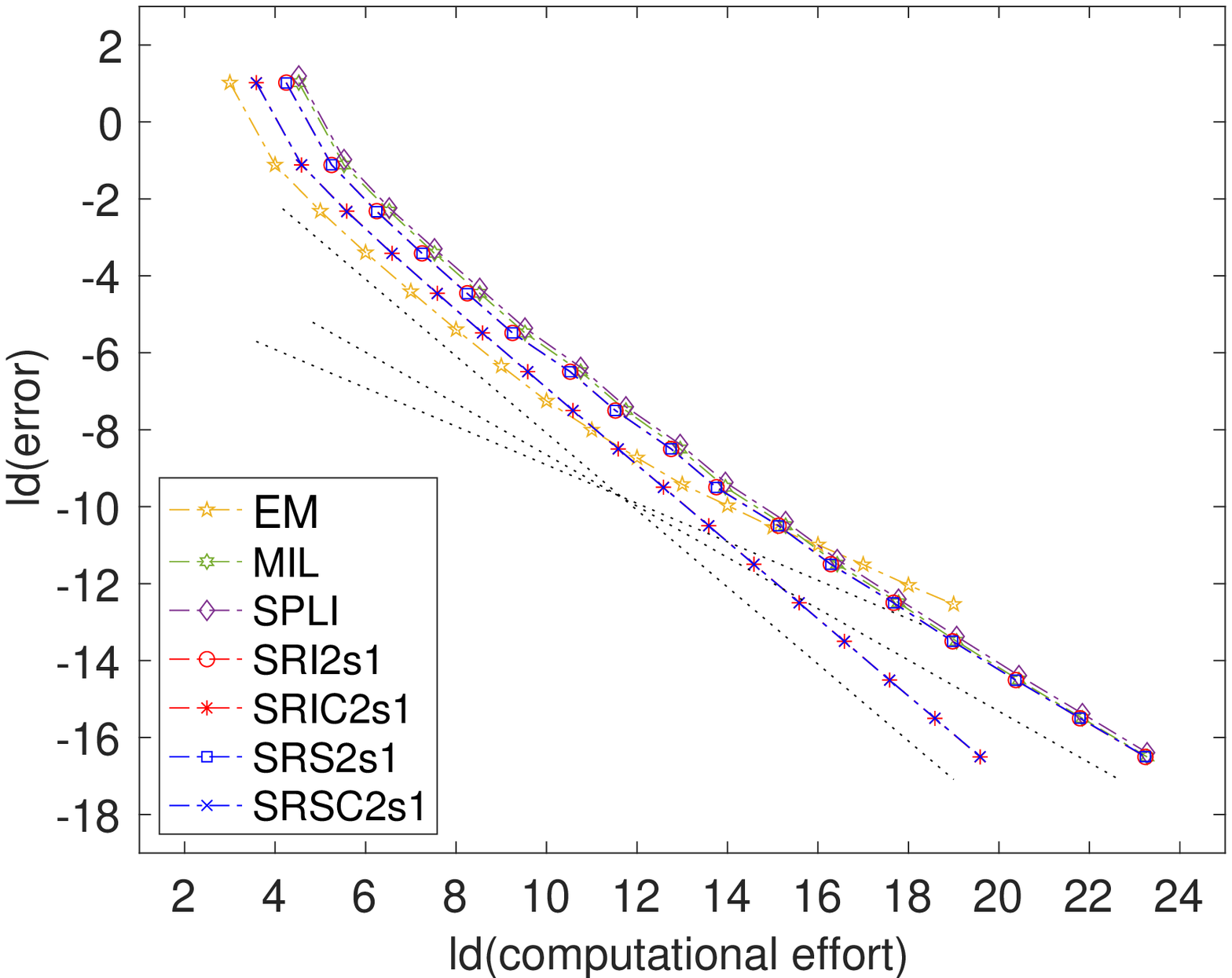}
		\includegraphics[width=8.2cm]{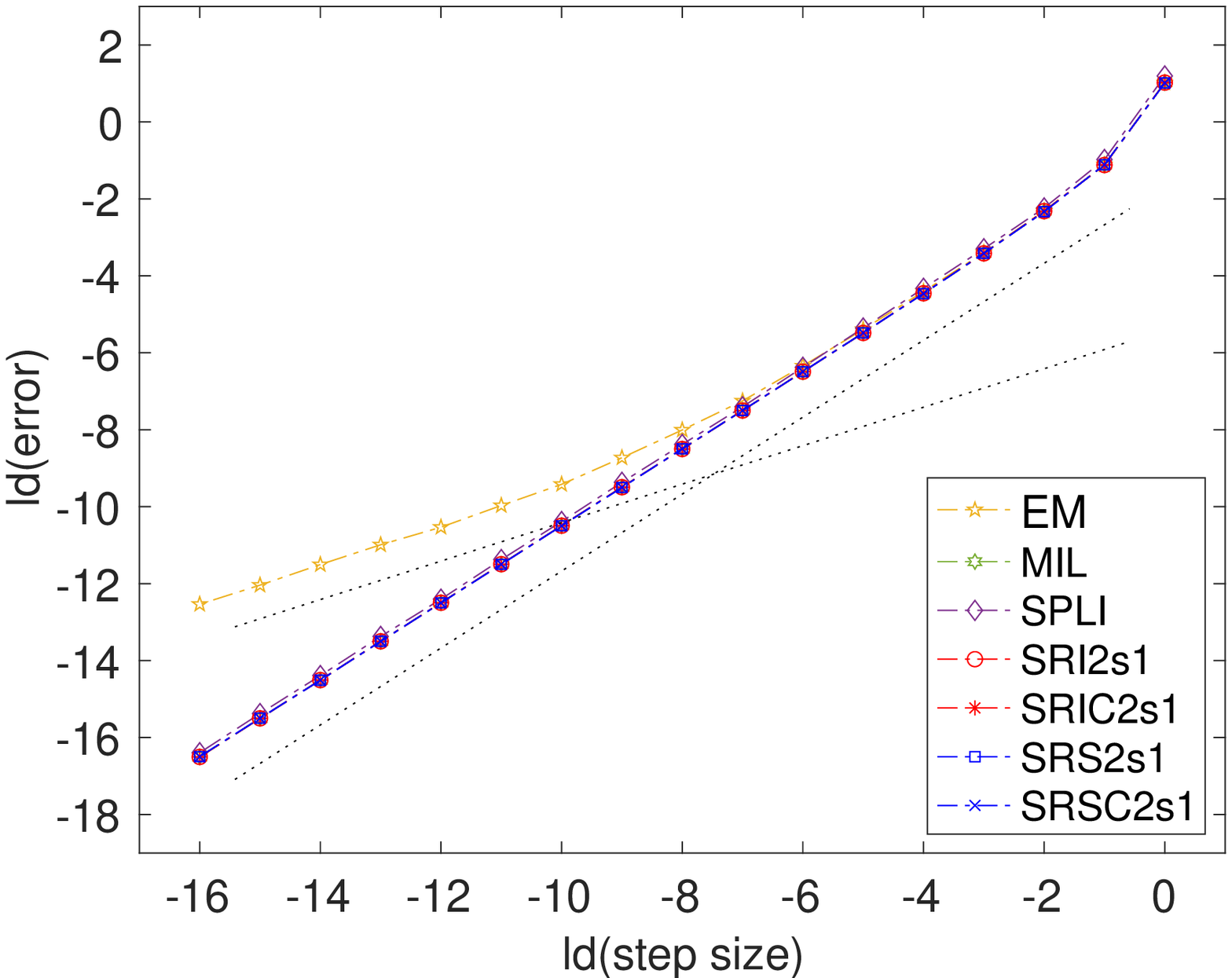}
		\includegraphics[width=8.2cm]{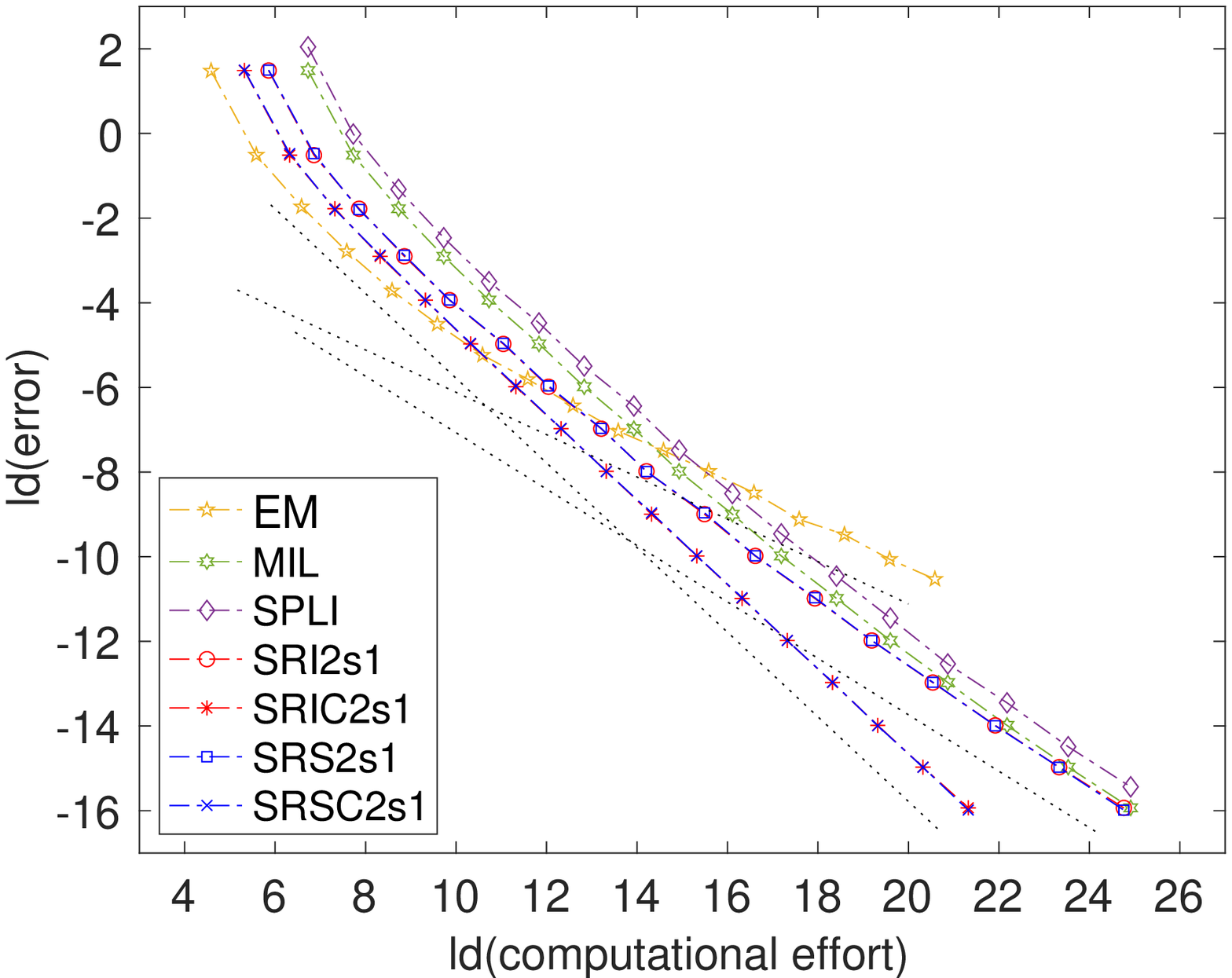}
		\includegraphics[width=8.2cm]{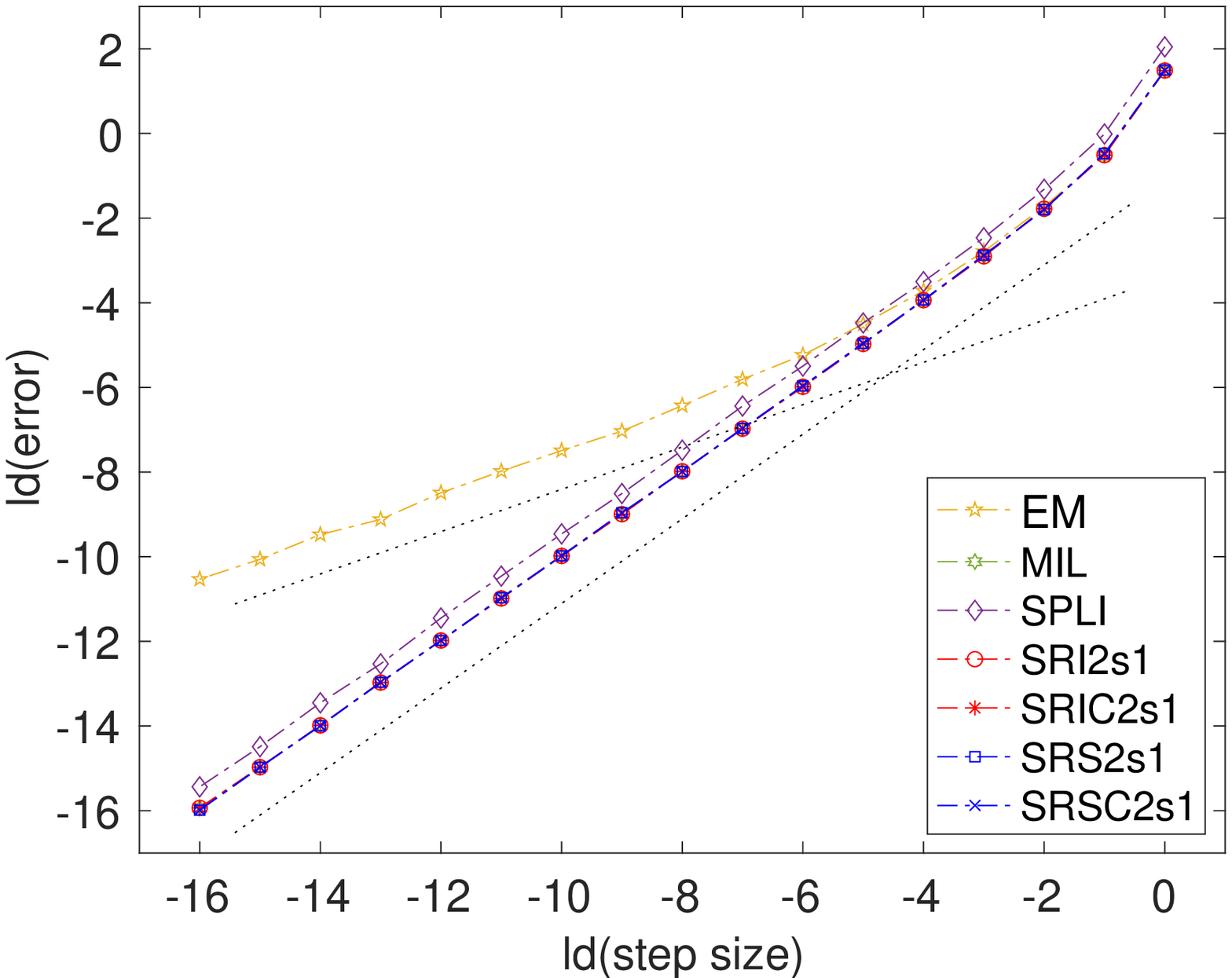}
		\includegraphics[width=8.2cm]{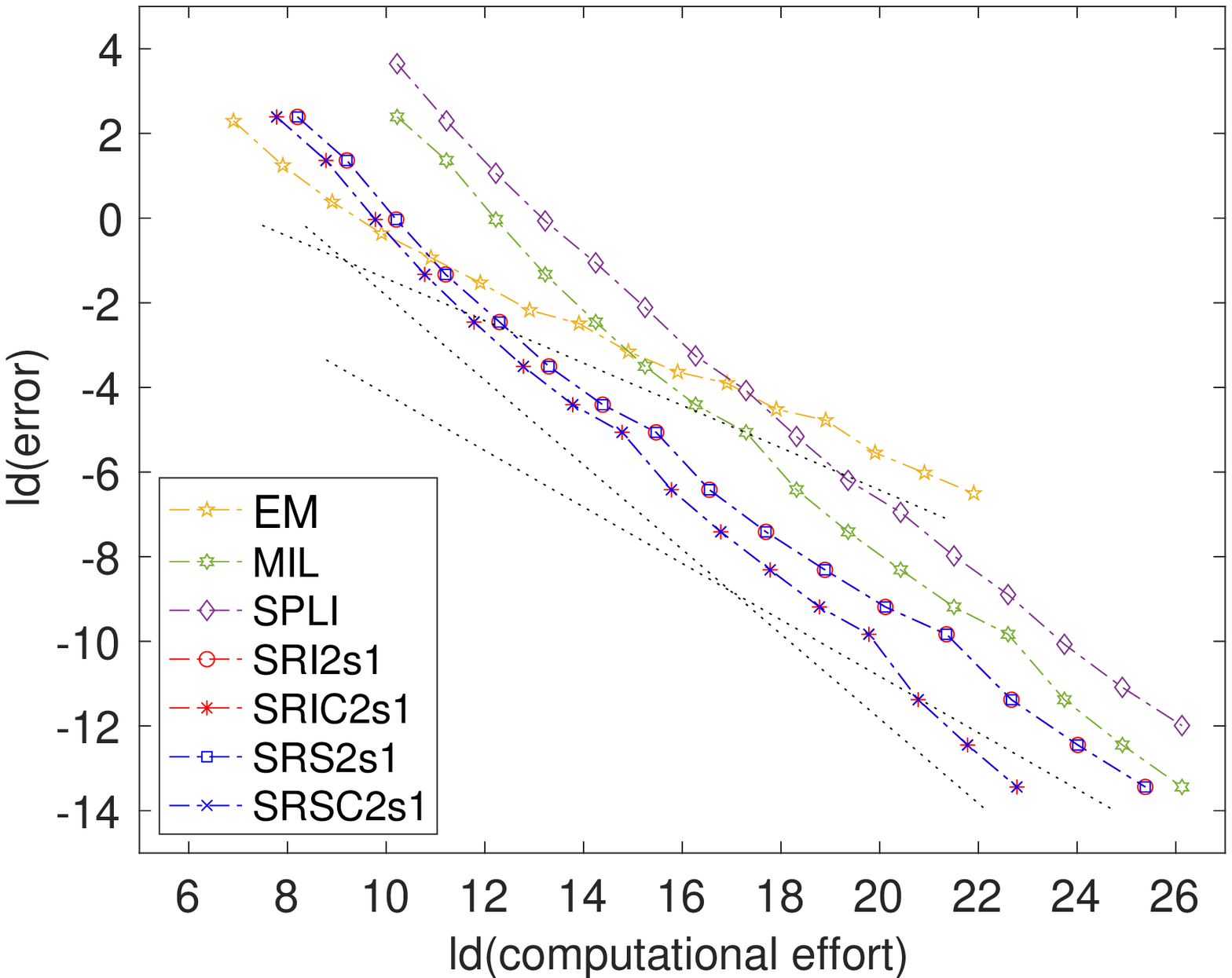}
		\includegraphics[width=8.2cm]{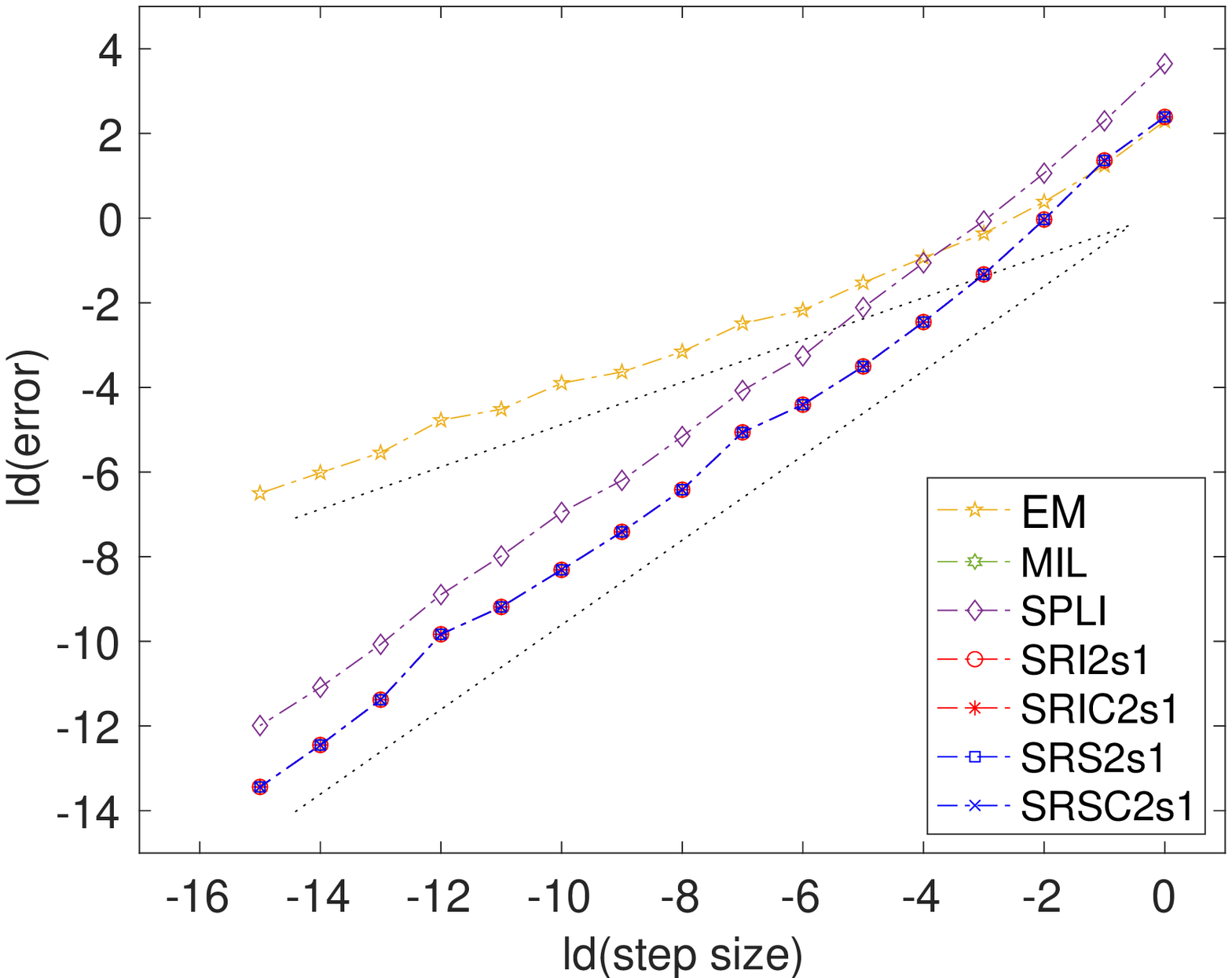}
		%
	\caption{Errors vs.\ computational effort in the left and errors vs.\ step sizes in
		the right figures for SDE~\eqref{Test-SDE-4} with dotted order~$\sfrac{1}{2}$ lines
		in all, order~$\sfrac{2}{3}$ lines in the left and order~$1$ lines in the right figures
		for $d=m=2$, $d=m=4$ and $d=m=10$ from top to bottom.}
	\label{Fig-Plot-SDE-4-4-6-d10m10}
\end{figure}
In order to consider some higher dimensional SDE systems with multiplicative noise, 
the following $d$-dimensional system of linear SDEs~\cite[(4.8.4)]{KP10} with an 
$m$-dimensional Wiener process given by
\begin{align} \label{Test-SDE-4}
	\mathrm{d}X_t = A \, X_t \, \mathrm{d}t + \sum_{k=1}^m B^k \, X_t \, \mathrm{d}W_t^k,
	\quad \quad X_0 = (2, \ldots, 2)^T \in \mathbb{R}^d,
\end{align}
is applied. Here, we choose $d=m$ and $A \in \mathbb{R}^{d \times d}$ 
with $A_{i,j} = \frac{1}{20}$ if $i \neq j$ and $A_{i,i} = -\frac{3}{2}$,
$B^k \in \mathbb{R}^{d \times d}$ with $B_{i,j}^k = \frac{1}{100}$ for $i \neq j$
and $B_{i,i}^k = \frac{1}{5}$ for $1 \leq i,j \leq d$ and $k=1, \ldots, m$. 
The SDE system~\eqref{Test-SDE-4} has the
solution $X_t = \exp((A- \frac{1}{2} \sum_{k=1}^m (B^k)^2) \, t 
+ \sum_{k=1}^m B^k \, W_t^k) \, X_0$ and possesses commutative noise. 
This test equation is of special interest as its an SDE system with a 
multi-dimensional driving Wiener process that allows for an explicitly
known solution.
We consider the cases $d=m=2$ and $d=m=4$ where we apply 
step sizes $h=2^0, 2^{-1}, \ldots, 2^{-16}$, and we consider the case
$d=m=10$ with step sizes $h=2^0, 2^{-1}, \ldots, 2^{-15}$. Further, we compute 
$M=2000$ simulations for each step size.

The numerical results are
presented in Fig.~\ref{Fig-Plot-SDE-4-4-6-d10m10}. Here, we can see that the $\EM$
scheme has order~$\sfrac{1}{2}$ and all other schemes have order~$1$ if errors versus
step sizes are considered. However, if errors versus computational effort are considered,
then the $\EM$ scheme still exhibits order $\peff = \sfrac{1}{2}$ whereas the newly proposed 
SRK schemes $\SRICzweiSeins$ and $\SRSCzweiSeins$ attain order $\peff=1$ since
SDE~\eqref{Test-SDE-4} has commutative noise. On the other hand, any scheme that
works for non-commutative noise may also be applied to SDEs with commutative noise 
although this might not be an optimal choice. 
In Fig.~\ref{Fig-Plot-SDE-4-4-6-d10m10} we can also see that all order~$1$
schemes that incorporate iterated stochastic integrals attain the expected effective 
order $\peff = \sfrac{2}{3}$. Thus, all order~$1$ schemes clearly outperform the $\EM$ 
scheme. Note that the $\MIL$, $\SPLI$, $\SRIzweiSeins$ and $\SRSzweiSeins$ schemes
first show effective order~$1$ convergence for large step sizes until the computational
cost for the approximation of iterated stochastic integrals dominates the overall 
computational cost for smaller step sizes pushing
the effective order of convergence down to $\peff = \sfrac{2}{3}$.
It has to be pointed out that the SRK schemes $\SRIzweiSeins$ and $\SRSzweiSeins$ 
preform significantly better than the $\MIL$ scheme and the $\SPLI$ scheme which 
becomes more and more significant as the dimension of the SDE system increases.
This is due to the reduced computational complexity of the newly proposed SRK method.
The SRK schemes $\SRICzweiSeins$ and $\SRSCzweiSeins$ show the best performance
for this example with commutative noise.
\subsubsection{Test Equation~5}
\begin{figure}[tbp]
		\includegraphics[width=8.2cm]{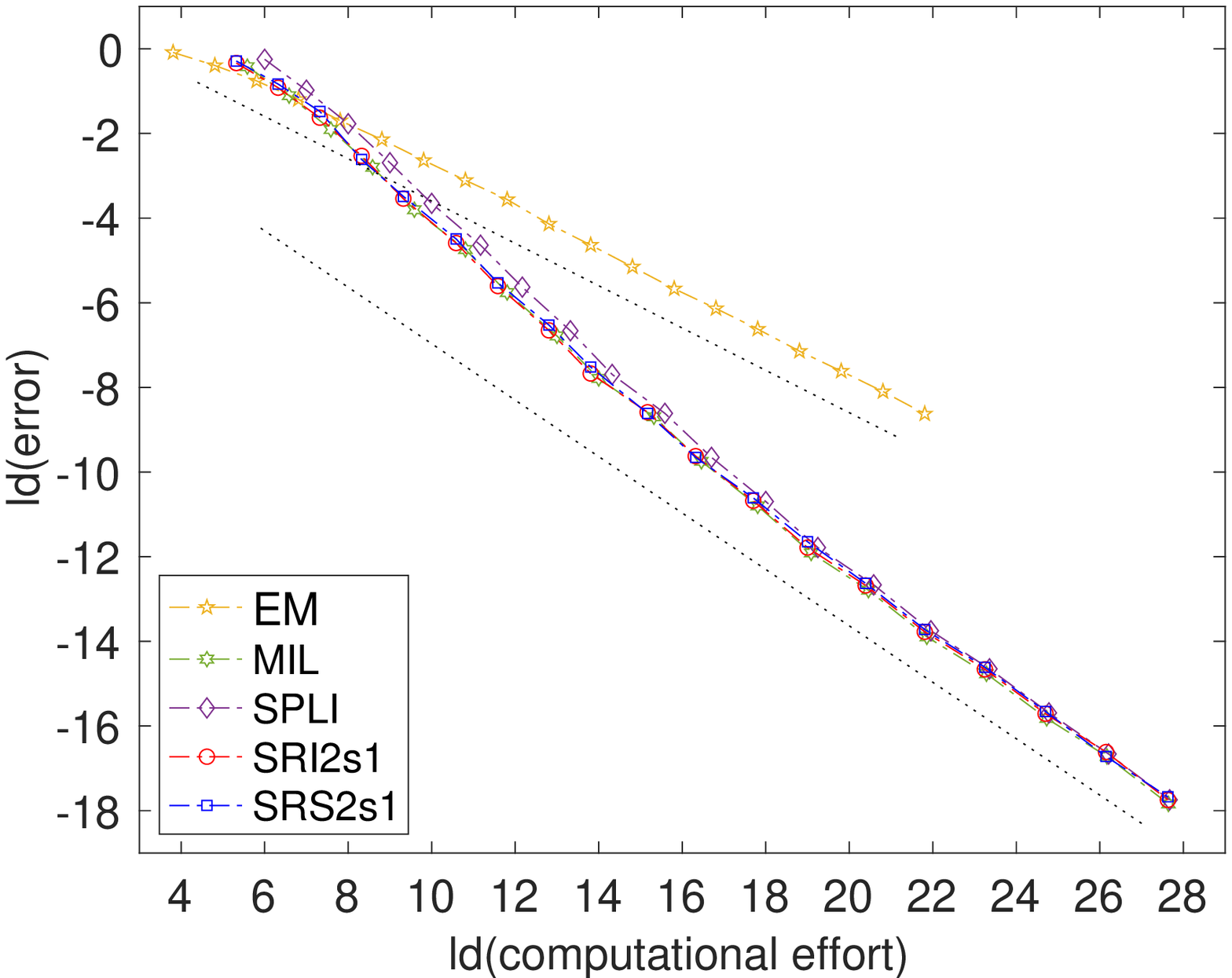}
		\includegraphics[width=8.2cm]{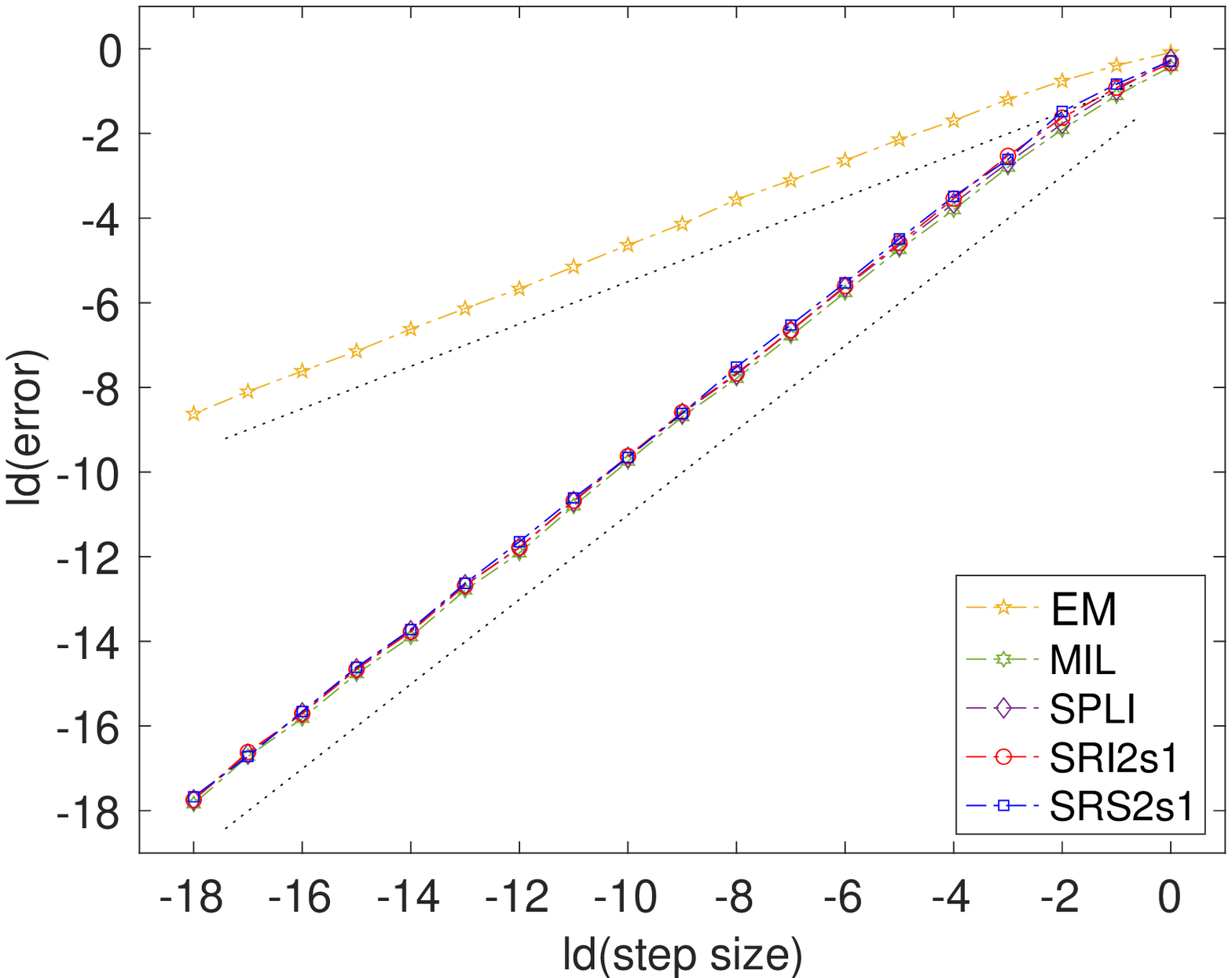}
		%
	\caption{Errors vs.\ computational effort in the left and errors vs.\ step sizes in
		the right figure for SDE~\eqref{Test-SDE-5} with dotted order~$\sfrac{1}{2}$ lines
		in both, order~$\sfrac{2}{3}$ line in the left and order~$1$ line in the right figure.}
	\label{Fig-Plot-SDE-13-1-5}
\end{figure}
A nonlinear $2$-dimensional SDE system with a $4$-dimensional driving Wiener process
where the solution process describes a stochastic flow on a torus~\cite[(13.1.5)]{KP10}
is given by
\begin{equation} \label{Test-SDE-5}
\begin{split}
	\mathrm{d} \begin{pmatrix} X_t^1 \\ X_t^2 \end{pmatrix} 
	&= \begin{pmatrix} \cos(\alpha) \\ \sin(\alpha) \end{pmatrix} \, \sin(X_t^1) \, \mathrm{d}W_t^1
	+ \begin{pmatrix} \cos(\alpha) \\ \sin(\alpha) \end{pmatrix} \, \cos(X_t^1) \, \mathrm{d}W_t^2 \\
	&\quad + \begin{pmatrix} -\sin(\alpha) \\ \cos(\alpha) \end{pmatrix} \, \sin(X_t^2) \, \mathrm{d}W_t^3
	+ \begin{pmatrix} -\sin(\alpha) \\ \cos(\alpha) \end{pmatrix} \, \cos(X_t^2) \, \mathrm{d}W_t^4
\end{split}
\end{equation}
where we choose the initial value $X_0 = (2,2)^T$ and parameter $\alpha=\frac{1}{2}$.
The SDE system~\eqref{Test-SDE-5} has non-commutative noise and there exists
no explicitly known solution. Therefore, a reference solution needs to be computed
where we use the $\MIL$ scheme with step size $h=2^{-20}$. Then, each scheme
under consideration is applied with step sizes $h=2^0, 2^{-1}, \ldots, 2^{-18}$ using
the same realizations of the Wiener processes as for the reference solution, respectively.
The numerical results based on $M=2000$ simulated realizations
are presented in Fig.~\ref{Fig-Plot-SDE-13-1-5}. Here, the $\EM$ scheme
has order $\peff = \sfrac{1}{2}$ whereas the order~$1$ schemes under consideration 
attain asymptotically order $\peff = \sfrac{2}{3}$ when the computational effort for the
approximation of iterated stochastic integrals dominates the overall computational costs.
The $\MIL$, the $\SPLI$, the $\SRIzweiSeins$ and the $\SRSzweiSeins$ schemes clearly
outperform the $\EM$ scheme. We note that the $\MIL$, $\SRIzweiSeins$ and $\SRSzweiSeins$
schemes yield similar results due to rather low dimension $d=2$ and 
perform slightly better than the $\SPLI$ scheme. The reduction of the computational 
complexity for the newly proposed 
SRK method illustrated in Table~\ref{Table3} is always present even though the 
computational cost each step for the approximation of iterated stochastic integrals 
increase with decreasing step sizes for all order~$1$ schemes.
\subsubsection{Test Equation~6}
\begin{figure}[tbp]
		\includegraphics[width=8.2cm]{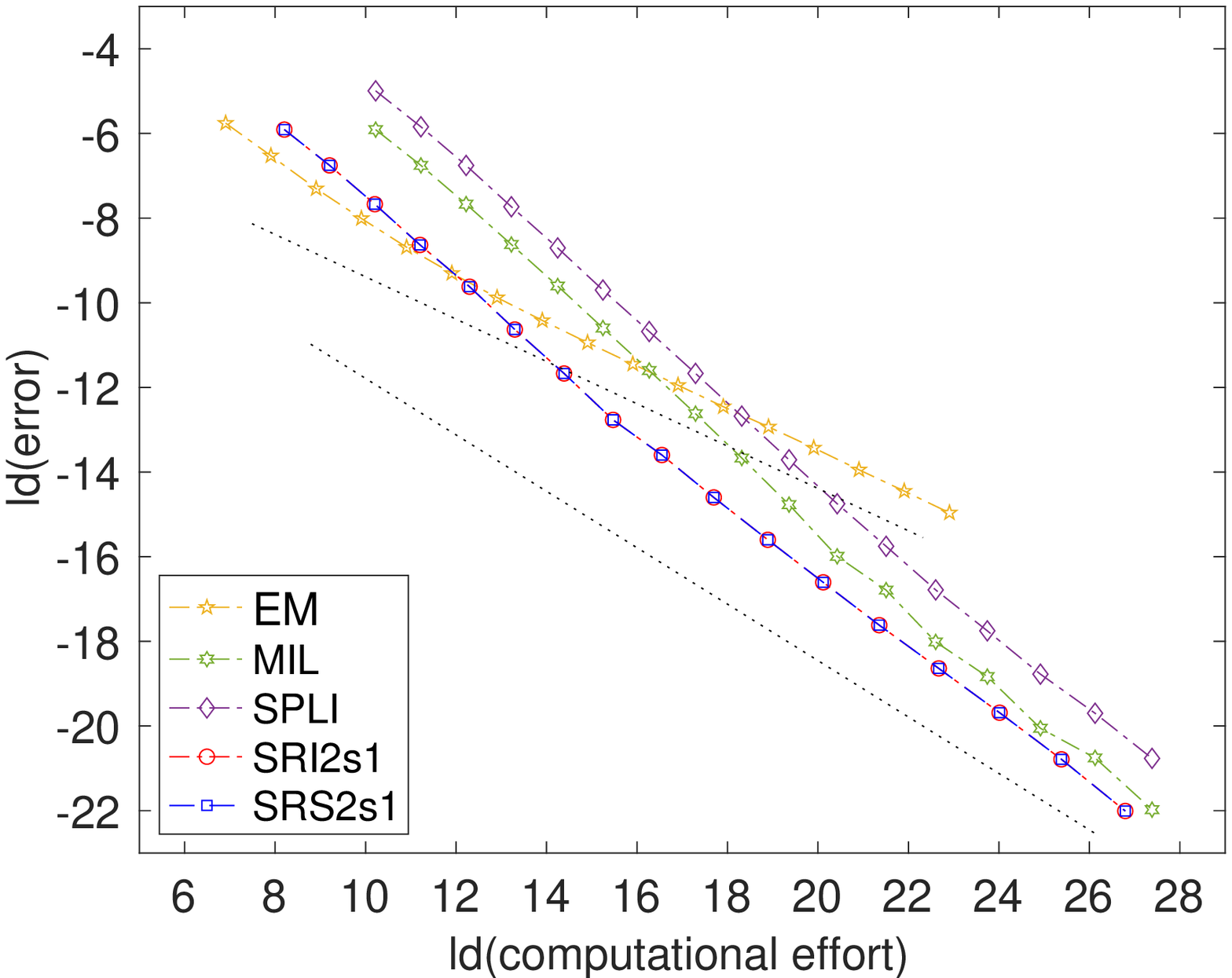}
		\includegraphics[width=8.2cm]{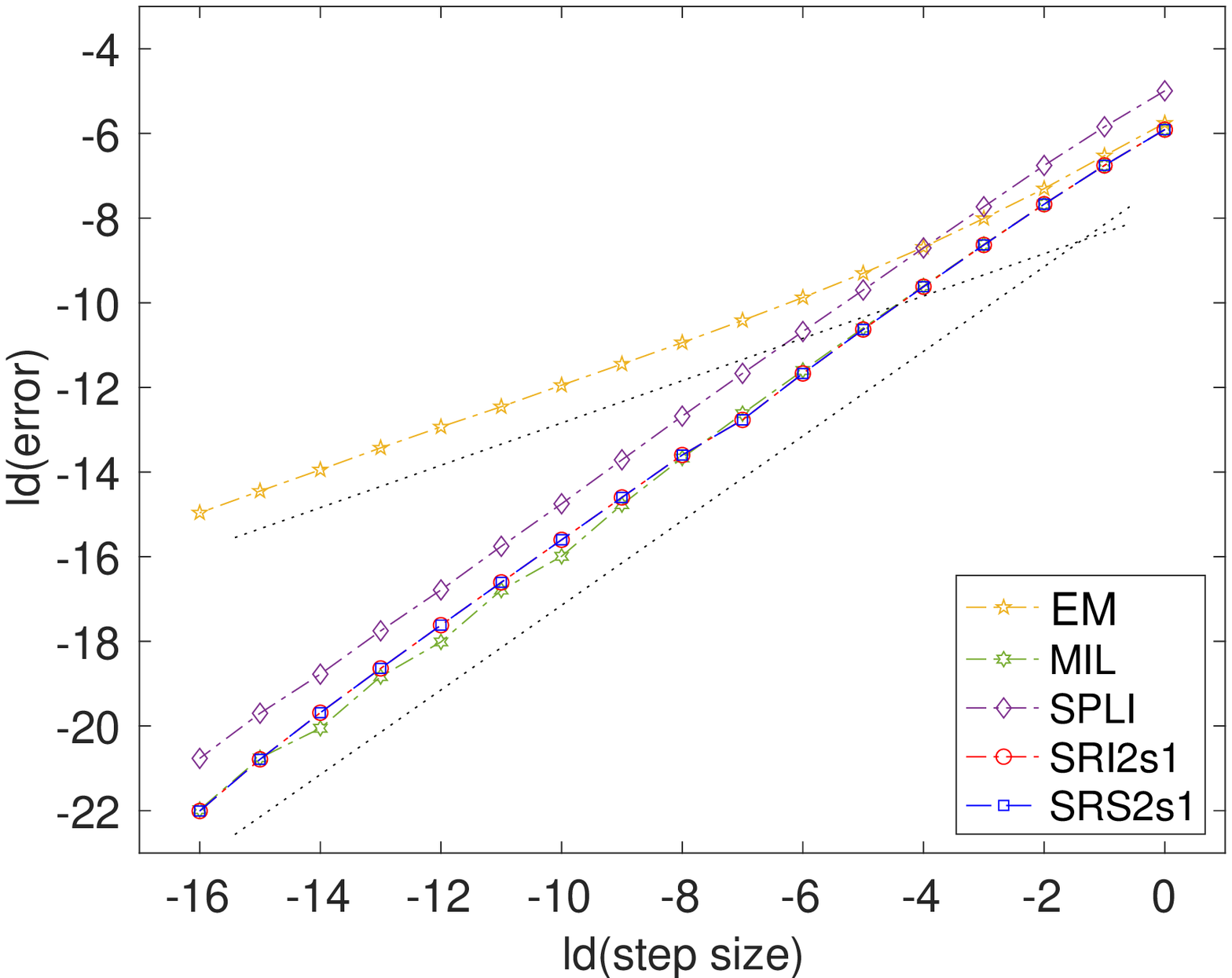}
		\includegraphics[width=8.2cm]{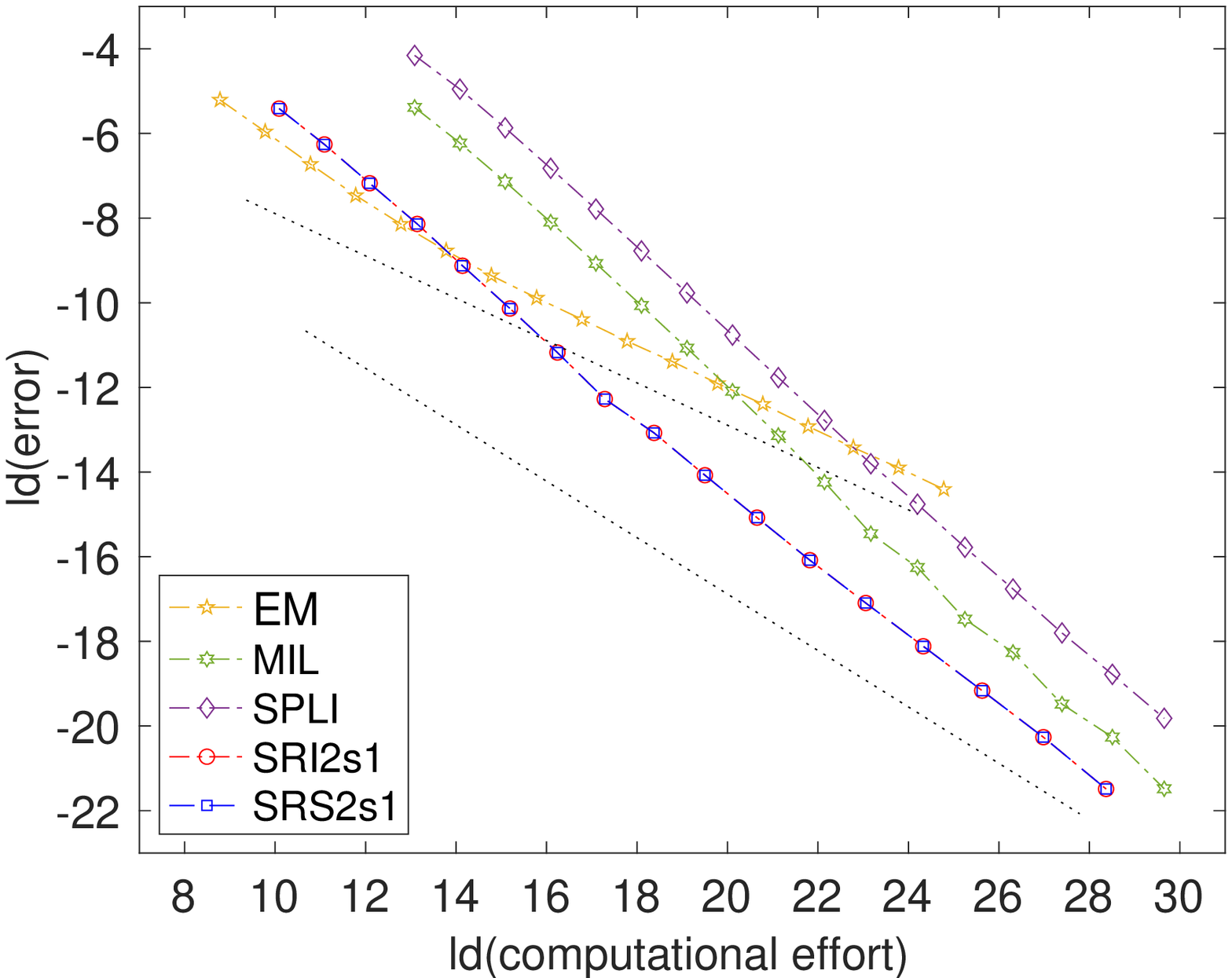}
		\includegraphics[width=8.2cm]{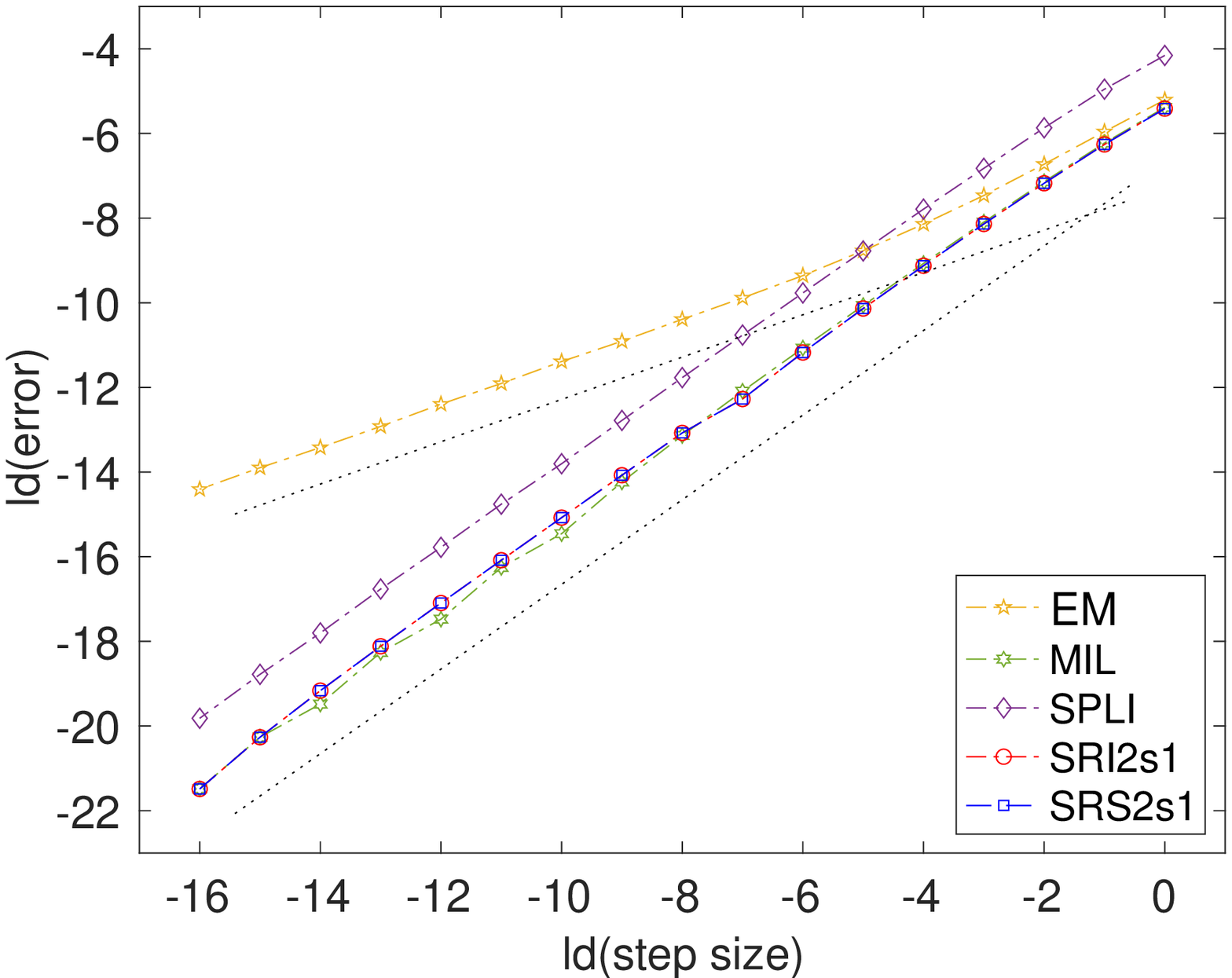}
		%
	\caption{Errors vs.\ computational effort in the left and errors vs.\ step sizes in
		the right figure for SDE~\eqref{Test-SDE-6} with dimensions $d=m=10$ on the top and
		$d=m=20$ on the bottom and with dotted order~$\sfrac{1}{2}$ lines 
		in all, order~$\sfrac{2}{3}$ lines in the left and order~$1$ lines in the right figures.}
	\label{Fig-Plot-SDE-7-1-6}
\end{figure}
A further example for a nonlinear high dimensional SDE system with non-commutative
noise arises from a stochastic Lotka-Volterra system that is used as a multispecies model 
with interaction in population dynamics~\cite[(7.1.6)]{KP10}. The $i$-th component of
the SDE system is given by
\begin{align} \label{Test-SDE-6}
	\mathrm{d}X_t^i = \bigg( \alpha^i \, X_t^i + \sum_{j=1}^d \beta^{i,j} \, X_t^j \, X_t^i \bigg)
	\, \mathrm{d}t + \sum_{k=1}^m b^{i,k}(X_t) \, \mathrm{d}W_t^k
\end{align}
for $i=1, \ldots, d$. We always choose $d=m$ and initial value $X_0 = \frac{1}{10} 
(1, \ldots, 1)^T \in \mathbb{R}^d$. Further, we apply the coefficients
$\alpha^i = \frac{1}{5}$, $\beta^{i,j} = \frac{1}{10} \, d^{-1}$ for $i \neq j$, 
$\beta^{i,i} = \frac{3}{2}$, $b^{i,k}(X_t) = \frac{1}{5} \, X_t^k$ for $k-1 \leq i \leq k+1$, 
and $b^{i,k}(X_t) = 0$ for $i \notin \{ k-1, k , k+1 \}$.
We note that SDE~\eqref{Test-SDE-6} does not fulfill all assumptions in 
Section~\ref{Sub-Sec:Assumptions}. However, it is of interest to analyze the 
convergence of numerical methods anyway.

Since an explicit solution is not known for SDE~\eqref{Test-SDE-6}, we compute a 
reference solution using the $\MIL$ scheme with step size $h=2^{-18}$. Then, all
considered numerical schemes are applied with step sizes $h=2^0, 2^{-1}, \ldots, 2^{-16}$
based on the same realizations of the driving Wiener processes as used for the
reference solution. We calculate $M=1000$ independent realizations for each scheme
in order to estimate the $L^2(\Omega)$-errors. The numerical results for $d=m=10$
and $d=m=20$ are presented in Fig.~\ref{Fig-Plot-SDE-7-1-6} on the top and on
the bottom, respectively. Although SDE~\eqref{Test-SDE-6} does not fulfill all 
assumptions in Section~\ref{Sub-Sec:Assumptions},
all schemes under consideration confirm their theoretical orders of convergence.
Considering errors versus step sizes, the $\EM$ scheme attains order~$\sfrac{1}{2}$ 
whereas the $\MIL$, the $\SPLI$, the $\SRIzweiSeins$ and the $\SRSzweiSeins$ 
schemes attain order~$1$. Comparing their errors versus computational effort 
gives $\peff = \sfrac{1}{2}$ for the $\EM$ scheme that is again outperformed
by the $\MIL$, the $\SPLI$, the $\SRIzweiSeins$ and the $\SRSzweiSeins$
schemes confirming $\peff = \sfrac{2}{3}$.

In both cases, for $d=m=10$ and  for $d=m=20$ the $\SRIzweiSeins$ scheme as well as 
the $\SRSzweiSeins$ scheme clearly outperform all other schemes under consideration.
It can be seen that the reduction of the computational cost for the SRK schemes
$\SRIzweiSeins$ and $\SRSzweiSeins$ increases from $d=m=10$ to the case of
$d=m=20$ what is in accordance with our theoretical results summarized in 
Table~\ref{Table3}.
\section{Discussion of the Design of the SRK Method}
\label{Sec:Derivation-SRK-method}
At this point, we like to give some explanations for the design of the proposed SRK 
method~\eqref{SRK-method}.
In order to construct SRK methods of order~$1$ in $L^p(\Omega)$-norm for general
SDEs~\eqref{SDE-Integral-form} with an $m$-dimensional Wiener process, a usual 
way is to start with an Ansatz for the design of the method and the stages. Then, 
a comparison of the Taylor expansions of the solution process with that of the
numerical approximation process gives the order conditions for the coefficients of
the SRK method. It is well known that a naive design of the SRK method can result in
worse computational complexity due to $\Oo(d \, m^2)$ or $\Oo(d^2 \, m)$ 
necessary function 
evaluations, see also discussion in Section~\ref{Sec:CompCost}. Here, we consider
the very general class of SRK methods proposed in~\cite[(2.1)]{Roe10} and~\cite{Roe06}.
The key idea is firstly to place the random variables $\Ii_{(l,k),n}$ in the stage values
and not to use the typical scaling $\frac{\Ii_{(l,k),n}}{\sqrt{h}}$. Secondly, we make use
of weights $\beta_i^{(2)}$ without any factor tending to $0$ as $h \to 0$. This idea
is something artificial but does a great job in order to reduce the computational 
complexity and is inspired by the approach in~\cite{HaRoe23} in the context of SPDEs.
As a result of this, the design of the SRK method~\eqref{SRK-method} allows for 
a reduced
computational complexity of order $\Oo(d \, m)$, which is the same as that for 
the very efficient order~$1$ SRK method proposed in~\cite{Roe10} and 
also the same as for the Euler-Maruyama scheme. It is much more efficient 
than the approach for standard SRK schemes as in, e.~g., \cite[(11.1.7)]{KP10} 
that suffer from a computational complexity of $\Oo(d \, m^2)$.
Roughly speaking, with this approach we approximate the 
so-called Milstein term in the scalar case by
\begin{align*}
	\frac{\partial b^k}{\partial x}(t,x) \, b^l(t,x) \, \Ii_{(l,k),n} \approx 
	\sum_{i=1}^s \beta_i^{(2)} \, b^k \Big(t, x
	+ \sum_{j=1}^{i-1} \sum_{l=1}^m B_{i,j}^{(1)} \, b^l(t, x) \, \Ii_{(l,k),n} \Big) \, .
\end{align*}
Then, we get from the matching of Taylor expansions that
\begin{align*}
	\frac{\partial b^k}{\partial x}(t,x) \, b^l(t,x) \, \Ii_{(l,k),n} &= {\beta^{(2)}}^T B^{(1)} e
	\Ii_{(l,k),n} \, \frac{\partial b^k}{\partial x}(t,x) \, b^l(t,x)
\end{align*}
needs to be fulfilled, i.~e., that ${\beta^{(2)}}^T B^{(1)} e = 1$ is necessary. 
For the corresponding
remainder terms of the Taylor expansion for the SRK method we further have
\begin{align*}
	&\Big\| {\beta^{(2)}}^T (B^{(1)} e)^2 \frac{\partial^2 b^k}{\partial x^2}(t,x) \, b^l(t,x) \, b^j(t,x) \,
	\Ii_{(l,k),n} \, \Ii_{(j,k),n} \Big\|_{L^p(\Omega)} = \Oo(h^2) \, , \\
	&\Big\| {\beta^{(2)}}^T (B^{(1)} (B^{(1)} e)) \frac{\partial b^k}{\partial x}(t,x) \, 
	\frac{\partial b^l}{\partial x} (t,x) \, b^j(t,x) \,
	\Ii_{(l,k),n} \, \Ii_{(j,l),n} \Big\|_{L^p(\Omega)} = \Oo(h^2) \, .
\end{align*}
Thus, these remainder terms are of sufficiently high order to be negligible for an
order~$1$ method. This would not be the case if 
the scaling $\frac{\Ii_{(l,k),n}}{\sqrt{h}}$ in the stages together with weights 
$\beta_i^{(2)} \, \sqrt{h}$ are used, as it is the case for the SRK method proposed
in~\cite{Roe10}. In contrast to the SRK method in~\cite{Roe10} we do not need 
the order conditions ${\beta^{(2)}}^T (B^{(1)} e)^2 = 0$ and 
${\beta^{(2)}}^T (B^{(1)} (B^{(1)} e)) = 0$. Finally, this is the main reason why we can 
find coefficients for SRK method~\eqref{SRK-method} with only $s=2$ stages, 
that is a reduction of the number of stages compared to the SRK method 
developed in~\cite{Roe10}. Probably the newly proposed approach with constant 
weights relies on the smoothness of the functions $b^k$ and is a good choice
especially when the evaluation of each function $b^k$ is rather costly, as it is the
case in high dimensional settings like in~\cite{HaRoe23}.
However, we also note that it is an open question to figure out 
what are the qualities and characteristics as well as the differences
between the two different approaches for efficient SRK methods leading to 
two and three stages, respectively. This may be subject to future research.
%
%
%
%
%
%
\section{Appendix. Proofs of Convergence} 
\label{Sec:proofs}
In the following, it is often more convenient to apply It{\^o} calculus and
thus we always transform the Stratonovich SDE~\eqref{SDE-Integral-form} to an 
It{\^o} SDE with the same solution. 
Therefore, in case of an It{\^o} SDE~\eqref{SDE-Integral-form} define 
$\underline{a}(t,x) = a(t,x)$ and in case of a Stratonovich SDE~\eqref{SDE-Integral-form}
define 
%
\begin{align} \label{SDE-Strato-to-Ito-KorrekturDrift}
	\underline{a}(t,x) = a(t,x) + \frac{1}{2} \sum_{k=1}^d \sum_{j=1}^m b^{k,j}(t,x)
	\, \frac{\partial}{\partial x_k} b^{j}(t,x)
\end{align}
for $(t,x) \in [t_0,T] \times \mathbb{R}^d$. Then, we can always rewrite 
SDE~\eqref{SDE-Integral-form} as the It{\^o} SDE
\begin{align} \label{SDE-Strato-to-Ito}
	X_t &= X_{t_0} + \int_{t_0}^t \underline{a}(s,X_s) \, \mathrm{d}s
	+ \sum_{j=1}^m \int_{t_0}^t b^j(s,X_s) \, \mathrm{d}W_s^j \, ,
\end{align}
with the same solution, see~\cite[Sec.~4.9]{KP10}. We note that
from~\eqref{Assumption-a-bk:lin-growth} and~\eqref{Assumption-a-bk:Bound-derivative-1}
it follows that $\underline{a}$ fulfills a linear growth condition in case 
of~\eqref{SDE-Strato-to-Ito-KorrekturDrift}.

For the proofs of convergence, we need some auxiliary result and estimates
that we give first. In the following, let $\| \cdot \|$ denote the euclidean or the 
maximum norm on $\mathbb{R}^d$. W.~l.~o.~g., we
consider for the proofs an equidistant discretization of the time interval
$[t_0,T]$. Therefore, let some sufficient large $N \in \mathbb{N}$ be given,
let $h = h(N) = \frac{T-t_0}{N}$ and let $\IhN = \{ t_0, t_1, \ldots, t_N\}$
with $t_n = t_0 + n \, h$ for $n=0, 1, \ldots, N$ in the following.
We note that assumptions~\eqref{Assumption-a-bk:lin-growth} 
and~\eqref{Assumption-a-bk:Bound-derivative-1} already imply that
$\underline{a}$ fulfills a linear growth condition.
The following lemma are valid for both, 
the solution of the It{\^o} and the Stratonovich SDE~\eqref{SDE-Integral-form}.
%
\begin{lem} \label{Lem:Lp-bound-SDE-sol}
	Assume that $X_{t_0} \in L^p(\Omega)$ for some $p \geq 2$ and that the linear
	growth conditions in~\eqref{Assumption-a-bk:lin-growth} 
	are fulfilled. 
	In case of Stratonovich calculus, additionally assume that $\underline{a}$
	defined in~\eqref{SDE-Strato-to-Ito-KorrekturDrift} fulfills a linear growth condition.
	Then, there exists some constant $\ccp = \ccp(p,T-t_0) >0$ 
	such that for the solution 
	$(X_t)_{t \in [t_0,T]}$ of SDE~\eqref{SDE-Integral-form} it holds
	\begin{align*}
		\Erw \big( \sup_{t_0 \leq t \leq T} \| X_t \|^p \big) 
		&\leq \ccp (1+ \Erw ( \|X_{t_0} \|^p )) \, .
	\end{align*}
\end{lem}
\begin{proof}
	For a proof, we refer to \cite[Thm~4.4, p.~61]{Mao07}.
\end{proof}

For the smoothness of the solution process of SDE~\eqref{SDE-Integral-form}
we need the following lemma that is used in, e.~g., the proof for 
Theorem~\ref{Sec:Main-Result:Thm-Konv-SRK-allg}.
%
%
\begin{lem} \label{Lem:Xs-Xtl-estimate}
	Assume that the linear
	growth conditions~\eqref{Assumption-a-bk:lin-growth} are fulfilled.
	In case of Stratonovich calculus, additionally assume
	that $\underline{a}$ fulfills a linear growth condition.
	For $p \geq 2$ with $X_{t_0} \in L^p(\Omega)$, for $s,t \in [t_0,T]$ 
	with $s \leq t$ and for the solution process
	$(X_t)_{t \in [t_0,T]}$ of
	SDE~\eqref{SDE-Integral-form} there exists 
	some constant $\cMXinc =\cMXinc(p,\cc,m) > 0$ such that
	\begin{align}
		\Erw \big( \| X_t - X_s \|^{p} \big) 
		\leq \cMXinc (1+ \Erw( \|X_{t_0} \|^p)) \, (t-s)^{\frac{p}{2}} \, .
	\end{align}
\end{lem}
\begin{proof}
	We consider $X$ as the solution of SDE~\eqref{SDE-Strato-to-Ito}.
	With H\"older inequality, a Burkholder inequality, see, 
	e.g., \cite{Burk88} or \cite[Prop.~2.2]{PlRoe21}, the 
	linear growth conditions~\eqref{Assumption-a-bk:lin-growth},
	the linear growth of $\underline{a}$ in case of Stratonovich SDEs
	and with Lemma~\ref{Lem:Lp-bound-SDE-sol}
	it holds
	%
	\begin{align*}
		\Erw \big( \| X_t - X_s \|^{p} \big) 
		&= \Erw \bigg( \bigg\| \int_s^t \underline{a}(u,X_u) \, \mathrm{d}u
		+ \sum_{j=1}^m \int_s^t b^j(u,X_u) \, \mathrm{d}W_u^j \bigg\|^p 
		\bigg) \\
		&\leq 2^{p-1} \Erw \bigg( \bigg\| \int_s^t \underline{a}(u,X_u) 
		\, \mathrm{d}u \bigg\|^p \bigg)
		+ 2^{p-1} \Erw \bigg( \bigg\| \sum_{j=1}^m \int_s^t b^j(u,X_u) 
		\, \mathrm{d}W_u^j \bigg\|^p \bigg) \\
		%
		&\leq 2^{p-1} \Erw \bigg( \int_s^t \| \underline{a}(u,X_u) \|^p
		\, \mathrm{d}u \bigg) \, (t-s)^{p-1} \\
		&\quad + 2^{p-1} \, (p-1)^{\frac{p}{2}} \bigg( \int_s^t \bigg[ 
		\Erw \bigg( \bigg( \sum_{j=1}^m \| b^j(u,X_u) \|^2 \bigg)^{\frac{p}{2}}
		\bigg) \bigg]^{\frac{2}{p}} \, \mathrm{d}u \bigg)^{\frac{p}{2}} \\
		&\leq 2^{p-1} \, (t-s)^{p-1} \Erw \bigg( \int_s^t \cc^p (1+ \| X_u \| )^p
		\, \mathrm{d}u \bigg) \\
		&\quad + 2^{p-1} \, (p-1)^{\frac{p}{2}} \bigg( \int_s^t \bigg[ 
		\Erw \bigg( m^{\frac{p}{2}} \, \cc^p (1+ \| X_u\|)^p
		\bigg) \bigg]^{\frac{2}{p}} \, \mathrm{d}u \bigg)^{\frac{p}{2}} \\
		%
		&\leq 2^{p-1} \, (t-s)^{p-1} \, \cc^p 
		\int_s^t \Erw \big( (1+ \| X_u \| )^p  \big) \, \mathrm{d}u \\
		&\quad + 2^{p-1} \, (p-1)^{\frac{p}{2}} \, m^{\frac{p}{2}}  \, \cc^p 
		\bigg( \int_s^t \big[ 2^{p-1} + 2^{p-1} \Erw \big(  \| X_u\|^p
		\big) \big]^{\frac{2}{p}} \, \mathrm{d}u \bigg)^{\frac{p}{2}} \\
		&\leq 2^{p-1} \, (t-s)^{p-1} \, \cc^p \, 2^{p-1}
		\int_s^t 1 + \Erw \big( \sup_{s \leq q \leq t} \| X_q \|^p  \big) \, \mathrm{d}u \\
		&\quad + 2^{p-1} \, (p-1)^{\frac{p}{2}} \, m^{\frac{p}{2}}  \, \cc^p 
		\bigg( \int_s^t \big[ 2^{p-1} + 2^{p-1} \Erw \big(  \sup_{s \leq q \leq t} 
		\| X_q\|^p \big) \big]^{\frac{2}{p}} \, \mathrm{d}u \bigg)^{\frac{p}{2}} \\
		&\leq 4^{p-1} \, (t-s)^p \, \cc^p \big(1 + 
		\Erw \big( \sup_{s \leq q \leq t} \| X_q \|^p  \big) \big) \\
		&\quad + 4^{p-1} \, (p-1)^{\frac{p}{2}} \, m^{\frac{p}{2}}  \, \cc^p  
		\, (t-s)^{\frac{p}{2}}
		\big( 1 + \Erw \big(  \sup_{s \leq q \leq t} \| X_q\|^p \big) \big) \\
		&\leq 4^{p-1} \, (t-s)^p \, \cc^p \big(1 +  \ccp (1+ \Erw( \|X_{t_0} \|^p)) \big) \\
		&\quad + 4^{p-1} \, (p-1)^{\frac{p}{2}} \, m^{\frac{p}{2}}  \, \cc^p  
		\, (t-s)^{\frac{p}{2}}
		\big( 1 + \ccp (1+ \Erw( \|X_{t_0} \|^p)) \big) \\
		&\leq \cMXinc (1+ \Erw( \|X_{t_0} \|^p)) \, (t-s)^{\frac{p}{2}} \, .
	\end{align*}
\end{proof}
%
%
Next, we give some estimates for the stages of the SRK method~\eqref{SRK-method} that
will be consistently used in subsequent proofs. 
In the following, for some variable $\varGamma$ we make use of the notation
\begin{align}
	H_{i,l}^{(0),\varGamma} &= \varGamma + \sum_{j=1}^s A_{i,j}^{(0)} 
	a(t_l+c_j^{(0)} h, H_{j,l}^{(0), \varGamma}) \, h
	\label{Proof:MainThm:Stages-H-0-Gamma} \\
	H_{i,l}^{(k),\varGamma} &= \varGamma + \sum_{j=1}^s A_{i,j}^{(1)} 
	a(t_l+c_j^{(0)} h, H_{j,l}^{(0), \varGamma}) \, h 
	+ \sum_{r=1}^m \sum_{j=1}^{i-1} B_{i,j}^{(1)} 
	b^r(t_l+c_j^{(1)} h, H_{j,l}^{(r), \varGamma}) \Iihat_{(r,k),l}
	\label{Proof:MainThm:Stages-H-k-Gamma}
\end{align}
for $l=0, \ldots, N-1$, $k \in \{1, \ldots, m\}$ and $i \in \{1, \ldots, s\}$ for the stage
values~\eqref{SRK-method-stage-H0}, \eqref{SRK-method-stage-Hk}
and~\eqref{SRK-method-stage-H0-AdditiveNoise} of the considered SRK method, 
respectively. 

We define the constant
\begin{align}
	\czwei &:= \max \{ \| A^{(0)} \|_{M}, \| A^{(1)} \|_M, \| B^{(1)} \|_M, 
				\| c^{(0)} \|_{\infty}, \| c^{(1)} \|_{\infty} \}
\end{align}
where $\| A \|_M = \max_{1 \leq i,j \leq s} | A_{i,j} |$ for $A \in \mathbb{R}^{s \times s}$
denotes the maximum norm. Further let $\Ii_n^{(2)} = ( \Iihat_{(i,j),n} )_{1 \leq i,j \leq m}$
denote a $m \times m$-random matrix containing. Then, let 
\begin{align*}
	\| \Ii_n^{(2)} \|_1 &= \sum_{i,j=1}^m | \Iihat_{(i,j),n} |
\end{align*}
denote the $1$-norm applied to the matrix $\Ii_n^{(2)}$. 
%

Firstly, we derive estimates for the stages values $H_{i,n}^{(0),\Zz}$ for some 
variable $\Zz$ in~\eqref{SRK-method-stage-H0} or~\eqref{SRK-method-stage-H0-AdditiveNoise}.
With~\eqref{Assumption-a-bk:lin-growth}
it holds
%
\begin{align*}
	\| H_{i,n}^{(0),\Zz} \| 
	&\leq \| \Zz \| + \sum_{j=1}^s |A_{i,j}^{(0)} | \, 
	\| a(t_n + c_j^{(0)} h, H_{j,n}^{(0),\Zz}) \| \, h \\
	&\leq \| \Zz \| + \sum_{j=1}^s \czwei \, \cc \, (1+ \| H_{j,n}^{(0),\Zz} \| ) \, h
\end{align*}
Then, for $0 < h < \frac{1}{s \, \cc \, \czwei}$ it follows
\begin{align}
	\max_{1 \leq i \leq s} \| H_{i,n}^{(0),\Zz} \|
	&\leq \| \Zz \| + \sum_{j=1}^s \czwei \, \cc \, 
	(1+ \max_{ 1 \leq j \leq s} \| H_{j,n}^{(0),\Zz} \| ) \, h \nonumber \\
	&\leq \| \Zz \| + s \, \czwei \, \cc \, 
	(1+ \max_{1 \leq j \leq s} \| H_{j,n}^{(0),\Zz} \| ) \, h \nonumber \\
	\Rightarrow \quad \max_{1 \leq i \leq s} \| H_{i,n}^{(0),\Zz} \|
	&\leq \frac{\| \Zz \| + s \, \czwei \, \cc \, h}{1-s \, \czwei \, \cc \, h} \, .
	\label{Estimate-H0-01}
\end{align}
%
%
%

Next, we calculate estimates for stages~$H_{i,n}^{(k),\Zz}$ in~\eqref{SRK-method-stage-Hk}.
To start, we consider the first stage~$H_{1,n}^{(k),\Zz}$.
For $0 < h < \frac{1}{s \, \cc \, \czwei}$ we get with~\eqref{Assumption-a-bk:lin-growth}
that
\begin{align*}
	\max_{1 \leq k \leq m} \| H_{1,n}^{(k),\Zz} \| 
	&\leq \| \Zz \| + \sum_{j=1}^s | A_{1,j}^{(1)} | \, 
	\| a(t_n + c_j^{(0)} h, H_{j,n}^{(0),\Zz}) \| \, h \\
	&\leq \| \Zz \| + \sum_{j=1}^s \czwei \, \cc \, (1+ \| H_{j,n}^{(0),\Zz} \| ) \, h \\
	&\leq \| \Zz \| + s \, \czwei \, \cc \, \big( 1+ 
	\max_{ 1 \leq j \leq s} \| H_{j,n}^{(0),\Zz} \| \big) \, h
\end{align*}
and thus with~\eqref{Estimate-H0-01} it follows that
\begin{align}
	\max_{1 \leq k \leq m} \| H_{1,n}^{(k),\Zz} \| 
	&\leq \| \Zz \| + s \, \czwei \, \cc \, \Big(1+ 
	\frac{\| \Zz \| + s \, \czwei \, \cc \, h}{1-s \, \czwei \, \cc \, h} 
	\Big) \, h \nonumber \\
	&= \frac{\| \Zz \| + s \, \czwei \, \cc \, h}{1-s \, \czwei \, \cc \, h} \, .
	\label{Estimate-Hk1-02}
\end{align}
%
%
Then, with~\eqref{Assumption-a-bk:lin-growth} 
we calculate
\begin{align*}
	\| H_{i,n}^{(k),\Zz} \| 
	&\leq \| \Zz \| + \sum_{j=1}^s |A_{i,j}^{(1)}| \, \cc \, (1+ \| H_{j,n}^{(0),\Zz} \| ) \, h
	+ \sum_{l=1}^m \sum_{j=1}^{i-1} |B_{i,j}^{(1)}| \, \cc \, (1+ \| H_{j,n}^{(l),\Zz} \| ) \, 
	| \Iihat_{(l,k),n} | \\
	&\leq \| \Zz \| + \sum_{j=1}^s \czwei \, \cc \, (1+ \| H_{j,n}^{(0),\Zz} \| ) \, h
	+ \sum_{l=1}^m \sum_{j=1}^{i-1} \czwei \, \cc \, (1+ \| H_{j,n}^{(l),\Zz} \| ) \, | \Iihat_{(l,k),n} | \\
	&\leq \| \Zz \| + s \, \czwei \, \cc \, (1+ \max_{1 \leq j \leq s} \| H_{j,n}^{(0),\Zz} \| ) \, h
	+ \sum_{j=1}^{i-1} \czwei \, \cc \, (1+ \max_{1 \leq l \leq m} \| H_{j,n}^{(l),\Zz} \| )
	\sum_{l=1}^m | \Iihat_{(l,k),n} | \\
	&\leq \| \Zz \| + s \, \czwei \, \cc \, (1+ \max_{1 \leq j \leq s} \| H_{j,n}^{(0),\Zz} \| ) \, h
	+ (i-1) \, \czwei \, \cc \, \sum_{k,l=1}^m | \Iihat_{(l,k),n} | \\
	&\quad + \czwei \, \cc \, \sum_{k,l=1}^m | \Iihat_{(l,k),n} | \sum_{j=1}^{i-1} 
	\max_{1 \leq l \leq m} \| H_{j,n}^{(l),\Zz} \| \\
	&\leq \| \Zz \| + s \, \czwei \, \cc \, (1+ \max_{1 \leq j \leq s} \| H_{j,n}^{(0),\Zz} \| ) \, h
	+ s \, \czwei \, \cc \, \| \Ii_n^{(2)} \|_1
	+ \czwei \, \cc \, \| \Ii_n^{(2)} \|_1 \sum_{j=1}^{i-1} \max_{1 \leq l \leq m} \| H_{j,n}^{(l),\Zz} \|
\end{align*}
and we thus obtain the estimate
\begin{align} \label{Estimate-Hk-01b}
	\max_{1 \leq k \leq m} \| H_{i,n}^{(k),\Zz} \| 
	&\leq \| \Zz \| + s \, \czwei \, \cc \, (1+ \max_{1 \leq j \leq s} \| H_{j,n}^{(0),\Zz} \| ) \, h
	+ s \, \czwei \, \cc \, \| \Ii_n^{(2)} \|_1 \nonumber \\
	&\quad 
	+ \czwei \, \cc \, \| \Ii_n^{(2)} \|_1 \sum_{j=1}^{i-1} \max_{1 \leq l \leq m} \| H_{j,n}^{(l),\Zz} \|
\end{align}
for $i=1, \ldots, s$. Now, we define
\begin{align*}
	\Ceins &= \Ceins(\Zz) = \| \Zz \| + s \, \czwei \, \cc \, 
	\big( 1+ \max_{1 \leq j \leq s} \| H_{j,n}^{(0),\Zz} \| \big) \, h + s \, \czwei \, \cc \, \| \Ii_n^{(2)} \|_1 \, .
\end{align*}
Then, with~\eqref{Estimate-H0-01} it follows that
\begin{align} \label{Estimate-C1-Yn}
	\Ceins(\Zz) &\leq \frac{\| \Zz \| + s \, \czwei \, \cc \, h}{1-s \, \czwei \, \cc \, h}
	+ s \, \czwei \, \cc \, \| \Ii_n^{(2)} \|_1 \, .
\end{align}
Using the definition of $\Ceins$, we obtain the estimate
\begin{align*}
	\max_{1 \leq k \leq m} \| H_{i,n}^{(k),\Zz} \| 
	&\leq \Ceins + \czwei \, \cc \| \Ii_n^{(2)} \|_1 \sum_{j=1}^{i-1} 
	\max_{1 \leq l \leq m} \| H_{j,n}^{(l),\Zz} \| \\
	&\leq \Ceins + \czwei \, \cc \| \Ii_n^{(2)} \|_1 \sum_{j=1}^{i-1} 
	\Big( \Ceins + \czwei \, \cc \| \Ii_n^{(2)} \|_1 \sum_{j_2=1}^{j-1} 
	\max_{1 \leq l \leq m} \| H_{j_2,n}^{(l),\Zz} \| \Big) \\
	&\leq \Ceins + \czwei \, \cc \| \Ii_n^{(2)} \|_1 s \, \Ceins
	+ \czwei^2 \, \cc^2 \| \Ii_n^{(2)} \|_1^2 \, s \sum_{j=1}^{i-2} 
	\max_{1 \leq l \leq m} \| H_{j,n}^{(l),\Zz} \| \\
	&\leq \Ceins + \czwei \, \cc \, \| \Ii_n^{(2)} \|_1 \, s \, \Ceins
	+ \czwei^2 \, \cc^2 \, \| \Ii_n^{(2)} \|_1^2 \, s \\
	&\quad \times \sum_{j=1}^{i-2} 
	\Big( \Ceins + \czwei \, \cc \| \Ii_n^{(2)} \|_1 \sum_{j_2=1}^{j-1} 
	\max_{1 \leq l \leq m} \| H_{j_2,n}^{(l),\Zz} \| \Big) \\
	&\leq \Ceins + \czwei \, \cc \, \| \Ii_n^{(2)} \|_1 s \, \Ceins
	+ \czwei^2 \, \cc^2 \, \| \Ii_n^{(2)} \|_1^2 s^2 \, \Ceins
	+ \czwei^3 \, \cc^3 \, \| \Ii_n^{(2)} \|_1^3 \, s^2  \\
	&\quad \times \sum_{j=1}^{i-3} 
	\max_{1 \leq l \leq m} \| H_{j,n}^{(l),\Zz} \| \\
	&\leq \ldots \\
	&\leq \Ceins \sum_{j=0}^{i-2} \czwei^j \, \cc^j \| \Ii_n^{(2)} \|_1^j s^j
	+ \czwei^{i-1} \, \cc^{i-1} \| \Ii_n^{(2)} \|_1^{i-1} s^{i-2} \, s \,
	\max_{1 \leq l \leq m} \| H_{1,n}^{(l),\Zz} \|
\end{align*}
and thus with~\eqref{Estimate-Hk1-02} we finally obtain the estimate
\begin{align}
	\max_{1 \leq k \leq m} \| H_{i,n}^{(k),\Zz} \| 
	&\leq \Ceins \sum_{j=0}^{i-2} \czwei^j \, \cc^j \, \| \Ii_n^{(2)} \|_1^j \, s^j 
	+ \czwei^{i-1} \, \cc^{i-1} \, \| \Ii_n^{(2)} \|_1^{i-1} \, s^{i-1} \,
	\frac{\| \Zz \| + s \, \czwei \, \cc \, h}{1-s \, \czwei \, \cc \, h} \nonumber \\
	&\leq \Ceins \sum_{j=0}^{i-1} \czwei^j \, \cc^j \, s^j \, \| \Ii_n^{(2)} \|_1^j
	\nonumber \\
	&\leq \bigg( \frac{\| \Zz \| + s \, \czwei \, \cc \, h}{1-s \, \czwei \, \cc \, h}
	+ s \, \czwei \, \cc \, \| \Ii_n^{(2)} \|_1 \bigg) 
	\sum_{j=0}^{i-1} \czwei^j \, \cc^j \, s^j \, \| \Ii_n^{(2)} \|_1^j
	\label{Estimate-Hk-03}
\end{align}
for $i=1, \ldots, s$ provided that $0 < h < \frac{1}{s \, \cc \, \czwei}$ is fulfilled. 

%
%
Next, for some variable $\Zz$ we consider the following difference that can be estimated
with~\eqref{Assumption-a-bk:lin-growth}
such that
%
\begin{align*}
	\| H_{i,n}^{(0),\Zz} - \Zz \| 
	&= \| \Zz + \sum_{j=1}^s A_{i,j}^{(0)} \, a(t_n + c_j^{(0)} h, H_{j,n}^{(0), \Zz}) \, h - \Zz \| \\
	&= \| \sum_{j=1}^s A_{i,j}^{(0)} \, a(t_n + c_j^{(0)} h, H_{j,n}^{(0),\Zz}) \, h \| \\
	&\leq \sum_{j=1}^s |A_{i,j}^{(0)} | \, \| a(t_n + c_j^{(0)} h, H_{j,n}^{(0),\Zz}) \| \, h \\
	&\leq \sum_{j=1}^s \czwei \, \cc \, (1+ \| H_{j,n}^{(0),\Zz} \| ) \, h \, .
\end{align*}
Then, for $0 < h < \frac{1}{s \, \cc \, \czwei}$ it follows with~\eqref{Estimate-H0-01} that
\begin{align}
	\max_{1 \leq i \leq s} \| H_{i,n}^{(0),\Zz} - \Zz \|
	&\leq \sum_{j=1}^s \czwei \, \cc \, (1+ \max_{ 1 \leq j \leq s} \| H_{j,n}^{(0),\Zz} \| ) \, h \nonumber \\
	&\leq s \, \czwei \, \cc \, \big( 1+ (1-s \, \czwei \, \cc \, h)^{-1} (\| \Zz\| + s \, \czwei \, \cc \, h) \big) \, h \nonumber \\
	\Rightarrow \quad \max_{1 \leq i \leq s} \| H_{i,n}^{(0),\Zz} - \Zz\|
	&\leq s \, \czwei \, \cc \, \frac{1 + \| \Zz\|}{1-s \, \czwei \, \cc \, h} \, h \, .
	\label{Estimate-H0-Xn-01}
\end{align}
%

%
%
As the next step, making use of~\eqref{Assumption-a-bk:lin-growth}
we firstly estimate the following difference and calculate that
%
%
\begin{align}
	\max_{1 \leq k \leq m} \| H_{1,n}^{(k), \Zz} - \Zz\| 
	&= \| \Zz + \sum_{j=1}^s A_{1,j}^{(1)} \, a(t_n + c_j^{(0)} h, H_{j,n}^{(0),\Zz}) \, h - \Zz \| 
	\nonumber \\
	&\leq \sum_{j=1}^s | A_{1,j}^{(1)} | \, \| a(t_n + c_j^{(0)} h, H_{j,n}^{(0),\Zz}) \| \, h 
	\nonumber \\
	&\leq \sum_{j=1}^s \czwei \, \cc \, (1+ \| H_{j,n}^{(0),\Zz} \| ) \, h 
	\nonumber \\
	&\leq s \, \czwei \, \cc \, \big( 1+ \max_{ 1 \leq j \leq s} \| H_{j,n}^{(0),\Zz} \| \big) \, h 
	\nonumber \\
	&\leq s \, \czwei \, \cc \, \big( 1+ (1-s \, \czwei \, \cc \, h)^{-1} (\| \Zz\| + s \, \czwei \, \cc \, h) \big) \, h 
	\nonumber \\
	&= s \, \czwei \, \cc \, \frac{1 + \| \Zz\|}{1-s \, \czwei \, \cc \, h} \, h \, .
	\label{Estimate-Hk1-Xn-01}
\end{align}
Further, it follows with analogous arguments as for~\eqref{Estimate-Hk-01b} that
\begin{align*}
	\max_{1 \leq k \leq m} \| H_{i,n}^{(k),\Zz} - \Zz \| 
	&\leq s \, \czwei \, \cc \, (1+ \max_{1 \leq j \leq s} \| H_{j,n}^{(0),\Zz} \| ) \, h
	+ s \, \czwei \, \cc \, \| \Ii_n^{(2)} \|_1 \\
	&\quad + \czwei \, \cc \, \| \Ii_n^{(2)} \|_1 \sum_{j=1}^{i-1} 
	\max_{1 \leq l \leq m} \| H_{j,n}^{(l),\Zz} \| \, .
\end{align*}
Thus, we calculate with~\eqref{Estimate-H0-01} and the definition
\begin{align}
	\Czwei = \Czwei(\Zz) &= s \, \czwei \, \cc \, \big( 1+ \max_{1 \leq j \leq s} \| H_{j,n}^{(0),\Zz} \| \big) \, h
	+ s \, \czwei \, \cc \, \| \Ii_n^{(2)} \|_1
	\nonumber \\
	&\leq s \, \czwei \, \cc \, \frac{1 + \| \Zz \|}{1-s \, \czwei \, \cc \, h} \, h
	+ s \, \czwei \, \cc \, \| \Ii_n^{(2)} \|_1
	\label{Estimate-C2-Xn}
\end{align}
together with~\eqref{Estimate-Hk-03} that
\begin{align*}
	\max_{1 \leq k \leq m} \| H_{i,n}^{(k),\Zz} -\Zz \| 
	&\leq \Czwei + \czwei \, \cc \| \Ii_n^{(2)} \|_1 \sum_{j=1}^{i-1} 
	\max_{1 \leq l \leq m} \| H_{j,n}^{(l),\Zz} \| \\
	&\leq \Czwei + \czwei \, \cc \| \Ii_n^{(2)} \|_1 \sum_{j=1}^{i-1} 
	\Big( 
	\Ceins(\Zz) \sum_{j_2=0}^{j-1} \czwei^{j_2} \, \cc^{j_2} \, s^{j_2} \, \| \Ii_n^{(2)} \|_1^{j_2}
	\Big) \\
	&= \Czwei + \Ceins(\Zz) \sum_{j=1}^{i-1} \sum_{j_2=0}^{j-1} \czwei^{j_2+1} \, \cc^{j_2+1} \, s^{j_2} 
	\, \| \Ii_n^{(2)} \|_1^{j_2+1} \\
	&\leq \Czwei + \Ceins(\Zz) \sum_{j=0}^{i-2} \czwei^{j+1} \, \cc^{j+1} \, s^{j+1} 
	\, \| \Ii_n^{(2)} \|_1^{j+1} \\
	&= C_2 + C_1(\Zz) \sum_{j=1}^{i-1} \czwei^{j} \, \cc^{j} \, s^{j} 
	\, \| \Ii_n^{(2)} \|_1^{j} \, .
\end{align*}
Thus, with~\eqref{Estimate-H0-01}, \eqref{Estimate-Hk1-Xn-01} and~\eqref{Estimate-C2-Xn}
we finally get the estimate
\begin{align}
	\max_{1 \leq k \leq m} \| H_{i,n}^{(k),\Zz} -\Zz \| 
	&\leq s \, \czwei \, \cc \, \frac{1 + \| \Zz \|}{1-s \, \czwei \, \cc \, h} \, h
	+ s \, \czwei \, \cc \, \| \Ii_n^{(2)} \|_1 \nonumber \\
	&\quad + 
	%
	\Big( \frac{\| \Zz \| + s \, \czwei \, \cc \, h}{1-s \, \czwei \, \cc \, h}
	+ s \, \czwei \, \cc \, \| \Ii_n^{(2)} \|_1 \Big)
	\sum_{j=1}^{i-1} \czwei^{j} \, \cc^{j} \, s^{j} \, 
	\| \Ii_n^{(2)} \|_1^{j} \label{Estimate-Hk-X_n-03}
\end{align}
for $i=1, \ldots, s$ provided that $0 < h < \frac{1}{s \, \cc \, \czwei}$ is fulfilled. 
%
%

For the subsequent proofs of convergence we repeatedly apply the following
auxiliary lemmas that are also of interest on its own.
%
\begin{lem} \label{Lem:Ij-Iij-Moment-estimate}
	Let $h>0$ and $i,j,k \in \{1, \ldots, m\}$. Then, for $p \geq 1$ it holds
	\begin{equation} \label{Lem:Ij-Iij-Moment-estimate-eq1}
		\begin{split}
		\| \Ii_{(i),t,t+h} \|_{L^p(\Omega)}	&\leq (\max\{2,p\}-1) \, \sqrt{h} \, , \\
		\| \Ii_{(i,j),t,t+h} \|_{L^p(\Omega)} &\leq \frac{\max\{2,p\}-1}{\sqrt{2}} \, h \, , 
		\\
		\| \Ji_{(i,j),t,t+h} \|_{L^p(\Omega)} &\leq \frac{\max\{2,p\}-1 + \frac{1}{\sqrt{2}}}{\sqrt{2}} \, h , 
		\end{split}
	\end{equation}
	and for $p_1, p_2 \geq 1$ it holds
	\begin{equation} \label{Lem:Ij-Iij-Moment-estimate-eq2}
		\begin{split}
		\Erw \big( | \Ii_{(k),t,t+h} |^{p_1} | \Ii_{(i,j),t,t+h} |^{p_2} \big)
		&\leq \frac{(2p_1 -1)^{p_1} (2p_2 -1)^{p_2}}{2^{p_2/2}} \, h^{\frac{p_1}{2} + p_2} \, , \\
		\Erw \big( | \Ii_{(k),t,t+h} |^{p_1} | \Ji_{(i,j),t,t+h} |^{p_2} \big)
		&\leq \frac{(2p_1 -1)^{p_1} (2p_2 -1 + \frac{1}{\sqrt{2}})^{p_2}}{2^{p_2/2}} \, 
		h^{\frac{p_1}{2} + p_2} \, .
		\end{split}
	\end{equation}
\end{lem}
\begin{proof}
	This follows from the Burkholder inequalities, see, e.g., 
	\cite{Burk88} or \cite[Prop.~2.2]{PlRoe21}, and 
	Cauchy-Schwarz inequality. The case $p \in {[1,2[}$ follows from the case 
	$p=2$ and Jensen's inequality.
\end{proof}
%
%
\begin{lem} \label{Lem:Ij-Iij-H-Z-estimate}
	Let $0< \hnull<\frac{1}{s \cc \czwei}$ and $j_1, j_2, j_3 \in \{1, \ldots, m\}$. Then, for 
	$p_1, p_2, p_3 \in {[1, \infty[} \cup \{0\}$, for $\Iihat_{(j_2,j_3),t,t+h} = \Ii_{(j_2,j_3),t,t+h}$
	or $\Iihat_{(j_2,j_3),t,t+h} = \Ji_{(j_2,j_3),t,t+h}$ as in~\eqref{SRK-method-Ito-Strato-Iihat}
	there
	exists some constant $\cMH =\cMH(p_1,p_2,p_3,s,\czwei,\cc,m,\hnull) > 0$ such that
	\begin{align}
		\Erw \big( | \Ii_{(j_1),t,t+h} |^{p_1} | \Iihat_{(j_2,j_3),t,t+h} |^{p_2} 
		\max_{\substack{1 \leq k \leq m \\ 1 \leq i \leq s}} \| H_{i,l}^{(k),Z} - Z\|^{p_3} \big) 
		\leq \cMH (1+ \Erw( \| Z \|^{p_3}) ) \, h^{\frac{p_1}{2} + p_2 + p_3}
	\end{align}
	for $0< h< \hnull$, $t_l=t$, $t_{l+1}=t+h$ and any $\mathcal{F}_{t}$-measurable 
	$Z \in L^{p_3}(\Omega)$.
\end{lem}
\begin{proof}
	We prove the case $p_1, p_2, p_3 \geq 1$, the other cases follow analogously.
	Then, it holds with \eqref{Estimate-Hk-X_n-03},
	Lemma~\ref{Lem:Ij-Iij-Moment-estimate}, with $2p-1 < 2p-1+\frac{1}{\sqrt{2}} < 2p$
	and H\"older inequality
	\begin{align*}
		&\Erw \big( | \Ii_{(j_1),t,t+h} |^{p_1} | \Iihat_{(j_2,j_3),t,t+h} |^{p_2} 
		\max_{\substack{1 \leq k \leq m \\ 1 \leq i \leq s}} \| H_{i,l}^{(k),Z} - Z\|^{p_3} \big)
		\nonumber \\
		&\leq \Erw \bigg( | \Ii_{(j_1),t,t+h} |^{p_1} | \Iihat_{(j_2,j_3),t,t+h} |^{p_2} 
		\bigg[ s \czwei \cc \, \frac{1 + \| Z \|}{1-s \czwei \cc \, h} \, h
		+ s \czwei \cc \, \| \Ii_n^{(2)} \|_1 \nonumber \\
		&\quad + \bigg( \frac{\| Z \| + s \czwei \cc \, h}{1-s \czwei \cc \, h}
		+ s \czwei \cc \, \| \Ii_n^{(2)} \|_1 \bigg)
		\sum_{j=1}^{s-1} \czwei^{j} \cc^{j} s^{j} \, 
		\| \Ii_n^{(2)} \|_1^{j} \bigg]^{p_3} \bigg) \\
		&\leq 4^{p_3-1} \bigg[ 
		\| | \Ii_{(j_1),t,t+h}|^{p_1} \|_{L^2(\Omega)}
		\| | \Iihat_{(j_2,j_3),t,t+h} |^{p_2} \|_{L^2(\Omega)} (s \czwei \cc)^{p_3}
		\Erw \bigg( \bigg| \frac{ 1 + \|Z\|}{1-s \czwei \cc \, h} \bigg|^{p_3} \bigg) \, h^{p_3} \\
		&\quad + (s \czwei \cc)^{p_3} m^{2p_3-2} \sum_{j_4,j_5=1}^m 
		\| | \Iihat_{(j_4,j_5),t,t+h} |^{p_3} \|_{L^2(\Omega)} \,
		\big\| | \Ii_{(j_1),t,t+h} |^{p_1} | \Iihat_{(j_2,j_3),t,t+h} |^{p_2} \big\|_{L^2(\Omega)} \\
		&\quad + \Erw \bigg( \bigg| \frac{\| Z \| + s \czwei \cc \, h}{1-s \czwei \cc \, h}
		\bigg|^{p_3} \bigg) \,
		\Erw \bigg( | \Ii_{(j_1),t,t+h}|^{p_1} | \Iihat_{(j_2,j_3),t,t+h}|^{p_2} \\
		&\quad \times 
		\bigg| \sum_{j=1}^{s-1} (s \czwei \cc)^j \bigg( \sum_{j_4,j_5=1}^m 
		| \Iihat_{(j_4,j_5),t,t+h}| \bigg)^j \bigg|^{p_3} \bigg) \\
		&\quad + \| | \Iihat_{(j_2,j_3),t,t+h} |^{p_2} \|_{L^2(\Omega)} 
		\bigg\| | \Ii_{(j_1),t,t+h}|^{p_1} \bigg| \sum_{j=1}^{s-1} (s \czwei \cc)^{j+1}
		\bigg( \sum_{j_4, j_5=1}^m | \Iihat_{(j_4,j_5),t,t+h} | \bigg)^{j+1} \bigg|^{p_3} \bigg\|_{L^2(\Omega)}
		\bigg] \\
		&\leq 4^{p_3-1} \bigg[ 
		(2 p_1 -1)^{p_1} h^{\frac{p_1}{2}} \bigg( \frac{4p_2}{\sqrt{2}} \bigg)^{p_2} h^{p_2}
		\bigg( \frac{s \czwei \cc}{1-s \czwei \cc \, h} \bigg)^{p_3} 2^{p_3-1} 
		\big(1+ \Erw \big( \|Z\|^{p_3} \big) \big) h^{p_3} \\
		&\quad + (s \czwei \cc m^2)^{p_3} \bigg( \frac{2p_3}{\sqrt{2}} \bigg)^{p_3} h^{p_3}
		(4p_1-1)^{p_1} h^{\frac{p_1}{2}} \bigg( \frac{4p_2}{\sqrt{2}} \bigg)^{p_2} h^{p_2} \\
		&\quad + 2^{p_3-1} \frac{ \Erw \big( \|Z\|^{p_3} \big) + (s \czwei \cc \, h)^{p_3}}{(1-s \czwei \cc \, h)^{p_3}}
		s^{p_3-1} \sum_{j=1}^{s-1} (s \czwei \cc)^{j p_3} m^{2j p_3 -2} \\
		&\quad \times \sum_{j_4,j_5=1}^m \| | \Iihat_{(j_2,j_3),t,t+h}|^{p_2} \|_{L^2(\Omega)} 
		\| | \Ii_{(j_1),t,t+h}|^{p_1} | \Iihat_{(j_4,j_5),t,t+h}|^{jp_3} \|_{L^2(\Omega)} \\
		&\quad + \bigg( \frac{2p_2}{\sqrt{2}} \bigg)^{p_2} h^{p_2} \bigg( s^{2p_3-1} 
		\sum_{j=1}^{s-1} (s \czwei \cc)^{2(j+1)p_3} m^{4(j+1)p_3-2} \sum_{j_4,j_5=1}^m (4p_1-1)^{2p_1} 
		h^{p_1} \\
		&\quad \times \bigg( \frac{4(j+1)p_3}{\sqrt{2}} \bigg)^{2(j+1)p_3} h^{2(j+1)p_3} \bigg)^{\frac{1}{2}}
		\bigg] \\
		&\leq 8^{p_3-1} (2p_1-1)^{p_1} \bigg( \frac{4p_2}{\sqrt{2}} \bigg)^{p_2} 
		\bigg( \frac{s \czwei \cc}{1-s \czwei \cc \, h} \bigg)^{p_3} h^{\frac{p_1}{2}+p_2+p_3} 
		( 1 + \Erw(\|Z\|^{p_3}) ) \\
		&\quad + 4^{p_3-1} (s \czwei \cc m^2)^{p_3} \bigg( \frac{2p_3}{\sqrt{2}} \bigg)^{p_3}
		(4p_1-1)^{p_1} \bigg( \frac{4p_2}{\sqrt{2}} \bigg)^{p_2} h^{\frac{p_1}{2}+p_2+p_3} \\
		&\quad + 8^{p_3-1} (2s)^{p_3-1}	(1+s \czwei \cc)^{sp_3} m^{2sp_3} s 
		\bigg( \frac{2p_2}{\sqrt{2}} \bigg)^{p_2} (4p_1-1)^{p_1} \bigg(\frac{4sp_3}{\sqrt{2}} \bigg)^{sp_3} \\
		&\quad \times \frac{(s \czwei \cc)^{p_3} + \Erw( \|Z\|^{p_3})}{(1-s \czwei \cc \,
		h)^{p_3}} 
		h^{\frac{p_1}{2}+p_2+p_3} \\
		&\quad + 4^{p_3-1} \bigg( \frac{2p_2}{\sqrt{2}} \bigg)^{p_2} s^{p_3} (1+s \czwei \cc)^{p_3 s} 
		m^{2s p_3} (4p_1-1)^{p_1} \bigg( \frac{4 p_3 s}{\sqrt{2}} \bigg)^{p_3 s} \\
		&\quad \times (1 + \hnull)^{(s-1) p_3} \, h^{\frac{p_1}{2}+p_2+p_3} \\
		&\leq \cMH(p_1,p_2,p_3,s,\czwei,\cc,m,\hnull) \, (1+\Erw(\|Z\|^{p_3})) \, h^{\frac{p_1}{2}+p_2+p_3} .
	\end{align*}
\end{proof}
%
%
%
\begin{lem} \label{Lem:Ij-Iij-Hk-estimate}
	Let $0< \hnull<\frac{1}{s \cc \czwei}$ and $j_1, j_2, j_3 \in \{1, \ldots, m\}$. 
	Then, for $p_1, p_2, p_3 \in {[1, \infty[} \cup \{0\}$ and
	for $\Iihat_{(j_2,j_3),t,t+h} = \Ii_{(j_2,j_3),t,t+h}$
	or $\Iihat_{(j_2,j_3),t,t+h} = \Ji_{(j_2,j_3),t,t+h}$ as in~\eqref{SRK-method-Ito-Strato-Iihat}
	there
	exists some constant $\cMHk =\cMHk(p_1,p_2,p_3,s,\czwei,\cc,m,\hnull) > 0$ such that
	\begin{align}
		\Erw \big( | \Ii_{(j_1),t,t+h} |^{p_1} | \Iihat_{(j_2,j_3),t,t+h} |^{p_2} 
		\max_{\substack{1 \leq k \leq m \\ 1 \leq i \leq s}} \| H_{i,l}^{(k),Z} \|^{p_3} \big) 
		\leq \cMHk (1+ \Erw( \| Z \|^{p_3}) ) \, h^{\frac{p_1}{2} + p_2}
	\end{align}
	for $0 < h < \hnull$, $t_l=t$, $t_{l+1}=t+h$ and any $\mathcal{F}_{t}$-measurable 
	$Z \in L^{p_3}(\Omega)$.
\end{lem}
\begin{proof}
	We prove the case $p_1, p_2, p_3 \geq 1$, the other cases follow analogously.
	Then, it holds with \eqref{Estimate-Hk-03}, Lemma~\ref{Lem:Ij-Iij-Moment-estimate},
	with $2p-1 < 2p-1+\frac{1}{\sqrt{2}} < 2p$ and H\"older inequality
	\begin{align*}
		&\Erw \big( | \Ii_{(j_1),t,t+h} |^{p_1} | \Iihat_{(j_2,j_3),t,t+h} |^{p_2} 
		\max_{\substack{1 \leq k \leq m \\ 1 \leq i \leq s}} \| H_{i,l}^{(k),Z} \|^{p_3} \big)
		\nonumber \\
		&\leq \Erw \bigg( | \Ii_{(j_1),t,t+h} |^{p_1} | \Iihat_{(j_2,j_3),t,t+h} |^{p_2} 
		\bigg[ \bigg( \frac{\| Z \| + s \czwei \cc \, h}{1-s \czwei \cc \, h}
		+ s \czwei \cc \, \| \Ii_n^{(2)} \|_1 \bigg)
		\sum_{j=0}^{s-1} \czwei^{j} \cc^{j} s^{j} \, \| \Ii_n^{(2)} \|_1^{j} \bigg]^{p_3} 
		\bigg) \\
		&\leq 2^{p_3-1} \bigg[ 
		\Erw \bigg( \bigg| \frac{\|Z\| + s \czwei \cc \, h}{1-s \czwei \cc \, h} 
		\bigg|^{p_3} \bigg) \,
		\Erw \bigg( | \Ii_{(j_1),t,t+h}|^{p_1} | \Iihat_{(j_2,j_3),t,t+h}|^{p_2} \\
		&\quad \times 
		\bigg| \sum_{j=0}^{s-1} (s \czwei \cc)^j \bigg( \sum_{j_4,j_5=1}^m 
		| \Iihat_{(j_4,j_5),t,t+h}| \bigg)^j \bigg|^{p_3} \bigg) \\
		&\quad + \| | \Iihat_{(j_2,j_3),t,t+h} |^{p_2} \|_{L^2(\Omega)} 
		\bigg\| | \Ii_{(j_1),t,t+h}|^{p_1} \bigg| \sum_{j=0}^{s-1} (s \czwei \cc)^{j+1}
		\bigg( \sum_{j_4, j_5=1}^m | \Iihat_{(j_4,j_5),t,t+h} | \bigg)^{j+1} 
		\bigg|^{p_3} \bigg\|_{L^2(\Omega)}
		\bigg] \\
		&\leq 2^{p_3-1} \bigg[ 
		2^{p_3-1} \frac{ \Erw \big( \|Z\|^{p_3} \big) 
		+ (s \czwei \cc \, h)^{p_3}}{(1-s \czwei \cc \, h)^{p_3}}
		s^{p_3-1} \bigg( \sum_{j=1}^{s-1} (s \czwei \cc)^{j p_3} m^{2j p_3 -2} \\
		&\quad \times \sum_{j_4,j_5=1}^m \| | \Iihat_{(j_2,j_3),t,t+h}|^{p_2} \|_{L^2(\Omega)} 
		\| | \Ii_{(j_1),t,t+h}|^{p_1} | \Iihat_{(j_4,j_5),t,t+h}|^{jp_3} \|_{L^2(\Omega)} \\
		&\quad + \Erw \big( | \Ii_{(j_1),t,t+h}|^{p_1} | \Iihat_{(j_2,j_3),t,t+h}|^{p_2} \big) 
		\bigg) \\
		&\quad + \bigg( \frac{2p_2}{\sqrt{2}} \bigg)^{p_2} h^{p_2} \bigg( s^{2p_3-1} 
		\sum_{j=0}^{s-1} (s \czwei \cc)^{2(j+1)p_3} m^{4(j+1)p_3-2} 
		\sum_{j_4,j_5=1}^m (4p_1-1)^{2p_1} 
		h^{p_1} \\
		&\quad \bigg( \frac{4(j+1)p_3}{\sqrt{2}} \bigg)^{2(j+1)p_3} 
		h^{2(j+1)p_3} \bigg)^{\frac{1}{2}}
		\bigg] \\
		&\leq 4^{p_3-1} \, s^{p_3-1} \bigg( (1+s \czwei \cc)^{s p_3} m^{2s p_3} 
		\bigg( \frac{2p_2}{\sqrt{2}} \bigg)^{p_2} (4p_1-1)^{p_1}
		\bigg(\frac{4sp_3}{\sqrt{2}} \bigg)^{sp_3} 
		\sum_{j=1}^{s-1} h^{j p_3} \\
		&\quad + (2p_1 -1)^{p_1} \Big(\frac{ 2p_2}{\sqrt{2}} \Big)^{p_2} \bigg) \,
		\frac{(s \czwei \cc \, h)^{p_3} + \Erw( \|Z\|^{p_3})}{(1-s \czwei \cc \, h)^{p_3}} 
		h^{\frac{p_1}{2}+p_2} \\
		&\quad + 2^{p_3-1} \bigg( \frac{2p_2}{\sqrt{2}} \bigg)^{p_2} 
		s^{p_3-\frac{1}{2}} (1+s \czwei \cc)^{p_3 s} 
		m^{2s p_3} (4p_1-1)^{p_1} \bigg( \frac{4 p_3 s}{\sqrt{2}} \bigg)^{p_3 s} \\
		&\quad \times 
		\bigg( \sum_{j=0}^{s-1} \hnull^{2 j p_3} \bigg)^{\frac{1}{2}}
		h^{\frac{p_1}{2}+p_2+p_3} \\
		&\leq \cMHk(p_1,p_2,p_3,s,\czwei,\cc,m,\hnull) \, (1+\Erw(\|Z\|^{p_3})) \, h^{\frac{p_1}{2}+p_2}  \, .
	\end{align*}
\end{proof}
%
%
%
\begin{lem} \label{Lem:H0-Xtl-estimate}
	For $p \geq 2$ and some $0< \hnull<\frac{1}{s \cc \czwei}$ 
	there exists some $\cMHdetInc =\cMHdetInc(p,s,\cc,\czwei,\hnull) > 0$ such that
	\begin{align}
		\Erw \big( \max_{1 \leq i \leq s} \| H_{i,l}^{(0),Z} - Z \|^p \big) 
		\leq \cMHdetInc (1+ \Erw( \|Z \|^p)) \, h^p
	\end{align}
	for $0<h \leq \hnull$ and any $\mathcal{F}_{t_l}$-measurable $Z \in L^{p}(\Omega)$.
\end{lem}
\begin{proof} 
	With \eqref{Estimate-H0-Xn-01} it holds for $0<h \leq \hnull <\frac{1}{s \cc \czwei}$
	that
	\begin{align*}
		\Erw \big( \max_{1 \leq i \leq s} \| H_{i,l}^{(0),Z} - Z \|^p \big) 
		&\leq (s \, \czwei \, \cc)^p \, \frac{\Erw \big( (1 + \| Z\|)^p 
			\big)}{(1-s \, \czwei \, \cc \, h)^p} \, h^p \\
		&\leq \Big( \frac{s \, \czwei \, \cc}{1-s \, \czwei \, \cc \, \hnull} \Big)^p \, 
		2^{p-1} \big( 1+ \Erw(\| Z \|^p) \big) \, h^p
	\end{align*}
\end{proof}
%
%
\begin{lem} \label{Lem:Ij-Iij-H0-estimate}
	Let $0< \hnull<\frac{1}{s \cc \czwei}$ and $j_1, j_2, j_3 \in \{1, \ldots, m\}$. 
	Then, for $p_1, p_2, p_3 \in {[1, \infty[} \cup \{0\}$ and
	for $\Iihat_{(j_2,j_3),t,t+h} = \Ii_{(j_2,j_3),t,t+h}$
	or $\Iihat_{(j_2,j_3),t,t+h} = \Ji_{(j_2,j_3),t,t+h}$ as in~\eqref{SRK-method-Ito-Strato-Iihat}
	there exists some constant $\cMHnull =\cMHnull(p_1,p_2,p_3,s,\czwei,\cc,m, \hnull) > 0$ 
	such that
	\begin{align}
		\Erw \big( | \Ii_{(j_1),t,t+h} |^{p_1} | \Iihat_{(j_2,j_3),t,t+h} |^{p_2} 
		\max_{1 \leq i \leq s} \| H_{i,l}^{(0),Z} \|^{p_3} \big) 
		\leq \cMHnull (1+ \Erw( \| Z \|^{p_3}) ) \, h^{\frac{p_1}{2} + p_2}
	\end{align}
	for $0 < h < \hnull$, $t_l=t$, $t_{l+1}=t+h$ and any $\mathcal{F}_{t}$-measurable 
	$Z \in L^{p_3}(\Omega)$.
\end{lem}
\begin{proof}
	We prove the case $p_1, p_2, p_3 \geq 1$, the other cases follow analogously.
	Then, for $0<h \leq \hnull <\frac{1}{s \cc \czwei}$
	it holds with \eqref{Estimate-H0-01}, Lemma~\ref{Lem:Ij-Iij-Moment-estimate}
	and with $2p-1 < 2p-1+\frac{1}{\sqrt{2}} < 2p$ that 
	\begin{align*}
		&\Erw \big( | \Ii_{(j_1),t,t+h} |^{p_1} | \Iihat_{(j_2,j_3),t,t+h} |^{p_2} 
		\max_{1 \leq i \leq s} \| H_{i,l}^{(0),Z} \|^{p_3} \big)
		\nonumber \\
		&\leq \Erw \Big( | \Ii_{(j_1),t,t+h} |^{p_1} | \Iihat_{(j_2,j_3),t,t+h} |^{p_2} 
		\Big( \frac{\| Z \| + s \, \czwei \, \cc \, h}{1-s \, \czwei \, \cc \, h} \Big)^{p_3} \Big)
		\nonumber \\
		&\leq \Erw \big( | \Ii_{(j_1),t,t+h} |^{p_1} | \Iihat_{(j_2,j_3),t,t+h} |^{p_2} \big)
		\, 2^{p_3-1} \, \frac{ \Erw ( \| Z \|^{p_3}) + (s \, \czwei \, \cc \, h)^{p_3}}{(1-s \, 
			\czwei \, \cc \, h)^{p_3}}
		\nonumber \\
		&\leq \frac{(2p_1 -1)^{p_1} (2p_2)^{p_2}}{2^{p_2/2}} \, h^{\frac{p_1}{2} + p_2}
		\, 2^{p_3-1} \, \frac{ \Erw ( \| Z \|^{p_3}) + (s \, \czwei \, \cc \, \hnull)^{p_3}}{(1-s \, 
			\czwei \, \cc \, \hnull)^{p_3}}
		\nonumber \\
		&\leq \cMHnull (1+ \Erw( \| Z \|^{p_3}) ) \, h^{\frac{p_1}{2} + p_2} \, .
		\nonumber
	\end{align*}
\end{proof}
\subsection{Proof of an Uniform Bound for Approximations} 
\label{Sub:Sec:Proof-Moment-Bound}
%
%
%
The following uniform estimate for the moments of the approximation by the SRK 
method~\eqref{SRK-method} are valid for both choices 
in~\eqref{SRK-method-Ito-Strato-Iihat}, i.~e., for It{\^o} and Stratonovich calculus.
%
%
\begin{prop} \label{Prop:Lp-bound-Approximation}
	Assume that~\eqref{Assumption-a-bk:lin-growth}, \eqref{Assumption-a-bk:Bound-derivative-1},
	\eqref{Assumption-Bound-derivative-1t-and-2t} and ${\beta^{(2)}}^T e =0$
	are fulfilled.
	For $p \geq 2$ there exists some constant 
	$\cYMB = \cYMB(T,p,m,d,\cc,s,\czwei)>0$
	such that for any $N \in \mathbb{N}$ and the approximation process 
	$(Y_n)_{0 \leq n \leq N}$ defined 
	by the SRK method~\eqref{SRK-method} with~\eqref{SRK-method-Ito-Strato-Iihat} 
	it holds
	\begin{align} \label{Prop:Lp-bound-Approximation:eqn}
		\Erw \big( \max_{0 \leq n \leq N} \| Y_n \|^p \big) 
		\leq \cYMB (1+ \Erw ( \| Y_0 \|^p )) \, .
	\end{align}
\end{prop}
%
\begin{proof}
	Let $p \geq 2$ and $N \in \mathbb{N}$ be given. Then, 
	\begin{align}
		\Erw \big( \max_{0 \leq n \leq N} \| Y_n \|^p \big)
		&= \Erw \bigg( \max_{0 \leq n \leq N} \bigg\| Y_0 + \sum_{l=0}^{n-1} 
		\bigg( \sum_{i=1}^s \alpha_i \, 
		a(t_l + c_i^{(0)} h, H_{i,l}^{(0)}) \, h \nonumber \\
		&\quad + \sum_{k=1}^m \sum_{i=1}^s
		\big( \beta_i^{(1)} \Ii_{(k),l} + \beta_i^{(2)} \big) \,
		b^{k}(t_l + c_i^{(1)} h, H_{i,l}^{(k)}) \bigg) \bigg\|^p \bigg) \nonumber \\
		&\leq 4^{p-1} \Erw ( \| Y_0 \|^p )
		+ 4^{p-1} \Erw \bigg( \max_{0 \leq n \leq N} \bigg\| \sum_{l=0}^{n-1} 
		\sum_{i=1}^s \alpha_i \, a(t_l + c_i^{(0)} h, H_{i,l}^{(0)}) \, h \bigg\|^p \bigg) \nonumber \\
		&\quad + 4^{p-1} \Erw \bigg( \max_{0 \leq n \leq N} \bigg\| \sum_{l=0}^{n-1} 
		\sum_{k=1}^m \sum_{i=1}^s \beta_i^{(1)} \Ii_{(k),l} \,
		b^{k}(t_l + c_i^{(1)} h, H_{i,l}^{(k)}) \bigg\|^p \bigg) \nonumber \\
		&\quad + 4^{p-1} \Erw \bigg( \max_{0 \leq n \leq N} \bigg\| \sum_{l=0}^{n-1} 
		\sum_{k=1}^m \sum_{i=1}^s \beta_i^{(2)} \,
		b^{k}(t_l + c_i^{(1)} h, H_{i,l}^{(k)}) \bigg\|^p \bigg) \, .
		\label{Proof:Lp-MB:eqn-1}
	\end{align}
	Now, we give an estimate for each summand. With the linear growth 
	condition~\eqref{Assumption-a-bk:lin-growth}
	it holds
	\begin{align}
		&\Erw \bigg( \max_{0 \leq n \leq N} \bigg\| \sum_{l=0}^{n-1} 
		\sum_{i=1}^s \alpha_i \, a(t_l + c_i^{(0)} h, H_{i,l}^{(0)}) \, h \bigg\|^p \bigg) \nonumber \\
		&\leq \Erw \bigg( \bigg( \sum_{l=0}^{N-1} \sum_{i=1}^s \czwei \, \cc \, h 
		(1 + \| H_{i,l}^{(0)} \| ) \bigg)^p \bigg) \nonumber \\
		&\leq N^{p-1} \sum_{l=0}^{N-1} h^p s^{p-1} \sum_{i=1}^s (\cc \, \czwei)^p 
		\, 2^{p-1} \, ( 1 + \Erw ( \| H_{i,l}^{(0)} \|^p ) ) \nonumber \\
		&\leq (T-t_0)^{p-1} \sum_{l=0}^{N-1} h \, s^{p-1} \sum_{i=1}^s 
		(\cc \, \czwei)^p \, 2^{p-1} 
		\bigg( 1 + \Erw \bigg( \bigg( \frac{\| Y_l \| + s \, \czwei \, \cc \, h}{1-s \, 
		\czwei \, \cc \, h} \bigg)^p \bigg) \bigg) \nonumber \\
		&\leq (T-t_0)^{p-1} \sum_{l=0}^{N-1} h \, s^{p-1} \sum_{i=1}^s (\cc \, \czwei)^p 
		\, 2^{p-1} \bigg( 1 + 2^{p-1} \frac{ \Erw ( \| Y_l \|^p ) + (s \, \czwei \, \cc \, h)^p}{
		(1-s \, \czwei \, \cc \, h)^p } \bigg) \nonumber \\
		&\leq (T-t_0)^p (s \czwei \cc)^p 2^{p-1} \bigg( 1+ 2^{p-1} \bigg( 
		\frac{s \, \czwei \, \cc \, h}{1-s \, \czwei \, \cc \, h} \bigg)^p \bigg) \nonumber \\
		&\quad + \frac{(T-t_0)^{p-1} (s \, \czwei \, \cc)^p 2^{2p-2}}{(1-s \, \czwei \, 
		\cc \, h)^p}
		\sum_{l=0}^{N-1} h \, \Erw \big( \max_{0 \leq k \leq l} \| Y_k \|^p \big)
		\label{Proof:Lp-MB:eqn-Term1}
	\end{align}
	for $0<h< \tfrac{1}{s \cc \czwei}$. Further, with Taylor expansion, Burkholder's inequality 
	(see, e.g., \cite{Burk88} or \cite[Prop.~2.1, Prop.~2.2]{PlRoe21}),
	linear growth~\eqref{Assumption-a-bk:lin-growth}
	and Lemma~\ref{Lem:Ij-Iij-H-Z-estimate} it follows
	%
	\begin{align}
		&\Erw \bigg( \max_{0 \leq n \leq N} \bigg\| \sum_{l=0}^{n-1} 
		\sum_{k=1}^m \sum_{i=1}^s \beta_i^{(1)} \Ii_{(k),l} \,
		b^{k}(t_l + c_i^{(1)} h, H_{i,l}^{(k)}) \bigg\|^p \bigg) \nonumber \\
		&= \Erw \bigg( \max_{0 \leq n \leq N} \bigg\| \sum_{l=0}^{n-1} 
		\sum_{k=1}^m \sum_{i=1}^s \beta_i^{(1)} \Ii_{(k),l} \, \bigg[
		b^k(t_l+c_i^{(1)}h,Y_l) \nonumber \\
		&\quad + \int_0^1 \sum_{r=1}^d \frac{\partial}{\partial x_r}
		b^k(t_l+c_i^{(1)}h,Y_l + u(H_{i,l}^{(k)} - Y_l)) \, \big( H_{i,l}^{(k)^r} - Y_l^r \big) \,
		\mathrm{d}u \bigg] \bigg\|^p \bigg) \nonumber \\
		&\leq 2^{p-1} \Erw \bigg( \max_{0 \leq n \leq N} \bigg\| \sum_{l=0}^{n-1}
		\sum_{k=1}^m \int_{t_l}^{t_{l+1}} \sum_{i=1}^s \beta_i^{(1)} 
		b^k(t_l+c_i^{(1)}h,Y_l) \, \mathrm{d}W_t^k \bigg\|^p \bigg) \nonumber \\
		&\quad + 2^{p-1} \Erw \bigg( \max_{0 \leq n \leq N} \bigg\| \sum_{l=0}^{n-1}
		\sum_{k=1}^m \sum_{i=1}^s \beta_i^{(1)} \Ii_{(k),l} \nonumber \\
		&\quad \times \int_0^1 \sum_{r=1}^d \frac{\partial}{\partial x_r}
		b^k(t_l+c_i^{(1)}h,Y_l + u(H_{i,l}^{(k)} - Y_l)) \, 
		\big( H_{i,l}^{(k)^r} - Y_l^r \big) \, \mathrm{d}u \bigg\|^p \bigg) \nonumber \\
		&\leq 2^{p-1} \Big( \frac{p}{\sqrt{p-1}} \Big)^p \bigg( \sum_{l=0}^{N-1}
		\int_{t_l}^{t_{l+1}} \Big[ \Erw \Big( \Big| \sum_{k=1}^m \Big\| \sum_{i=1}^s
		\beta_i^{(1)} b^k(t_l+c_i^{(1)} h,Y_l) \Big\|^2 \Big|^{\frac{p}{2}} \Big) 
		\Big]^{\frac{2}{p}} \, \mathrm{d}t \bigg)^{\frac{p}{2}} \nonumber \\
		&\quad + 2^{p-1} \Erw \bigg( \max_{0 \leq n \leq N} \bigg( \sum_{l=0}^{n-1}
		\sum_{k=1}^m \sum_{i=1}^s \big| \beta_i^{(1)} \big| \, \big| \Ii_{(k),l} \big|
		\nonumber \\
		&\quad \times \int_0^1 \sum_{r=1}^d \Big\| \frac{\partial}{\partial x_r}
		b^k(t_l+c_i^{(1)}h,Y_l + u(H_{i,l}^{(k)} - Y_l)) \Big\| \, 
		\big| H_{i,l}^{(k)^r} - Y_l^r \big| \, \mathrm{d}u \bigg)^p \bigg) \nonumber \\
		&\leq 2^{p-1} \Big( \frac{p}{\sqrt{p-1}} \Big)^p N^{\frac{p}{2}-1} 
		\sum_{l=0}^{N-1} h^{\frac{p}{2}} \, \Erw \Big( m^{\frac{p}{2}-1} \sum_{k=1}^m
		\Big\| \sum_{i=1}^s \beta_i^{(1)} b^k(t_l+c_i^{(1)} h,Y_l) \Big\|^p \Big) \nonumber \\
		&\quad + 2^{p-1} \Erw \Big( \Big( \sum_{l=0}^{N-1} \sum_{k=1}^m \sum_{i=1}^s
		\czwei \cc \, | \Ii_{(k),l}| \, \big\| H_{i,l}^{(k)} - Y_l \big\| \Big)^p \Big) 
		\nonumber \\
		&\leq 2^{p-1} \Big( \frac{p}{\sqrt{p-1}} \Big)^p (T-t_0)^{\frac{p}{2}-1}
		\sum_{l=0}^{N-1} h \, m^{\frac{p}{2}-1} \sum_{k=1}^m s^{p-1} \sum_{i=1}^s
		\big| \beta_i^{(1)} \big|^p \, \Erw \big( \| b^k(t_l+c_i^{(1)} h,Y_l) \|^p \big) \nonumber \\
		&\quad + 2^{p-1} N^{p-1} \sum_{l=0}^{N-1} m^{p-1} \sum_{k=1}^m s^{p-1}
		\sum_{i=1}^s (\czwei \, \cc)^p \, \Erw \big( \big| \Ii_{(k),l} \big|^p 
		\big\| H_{i,l}^{(k)}-Y_l \big\|^p \big) \nonumber \\
		&\leq 2^{p-1} \Big( \frac{p}{\sqrt{p-1}} \Big)^p (T-t_0)^{\frac{p}{2}-1} 
		m^{\frac{p}{2}}	(s \, \czwei \, \cc)^p \sum_{l=0}^{N-1} h \, 2^{p-1} 
		\big( 1+ \Erw \big( \| Y_l \|^p \big) \big) \nonumber \\
		&\quad + 2^{p-1} N^{p-1} \sum_{l=0}^{N-1} m^{p-1} \sum_{k=1}^m 
		s^{p-1} \sum_{i=1}^s (\czwei \, \cc)^p \big( \cMH \, h^{\frac{p}{2}+p}
		+ \cMH \, h^{\frac{p}{2}+p} \, \Erw \big( \| Y_l \|^p \big) \big) \nonumber \\
		&\leq 4^{p-1} \Big( \frac{p}{\sqrt{p-1}} \Big)^p (T-t_0)^{\frac{p}{2}}
		m^{\frac{p}{2}} (s \, \czwei \, \cc)^p \nonumber \\
		&\quad + 4^{p-1} \Big( \frac{p}{\sqrt{p-1}} \Big)^p (T-t_0)^{\frac{p}{2}-1}
		m^{\frac{p}{2}} (s \, \czwei \, \cc)^p \sum_{l=0}^{N-1} h \, 
		\Erw \big( \max_{0 \leq k \leq l} \| Y_k \|^p \big) \nonumber \\
		&\quad + 2^{p-1} (T-t_0)^p (m \, s \, \czwei \, \cc)^p \cMH \, h^{\frac{p}{2}} \nonumber \\
		&\quad + 2^{p-1} (T-t_0)^{p-1} (m \, s \, \czwei \, \cc)^p \cMH \, h^{\frac{p}{2}}
		\sum_{l=0}^{N-1} h \, \Erw \big( \max_{0 \leq k \leq l} \| Y_k \|^p \big) 
		\label{Proof:Lp-MB:eqn-Term2}
		\, .
	\end{align}
	For the last summand, we get with $\beta^{(2)^T} e = 0$, 
	\eqref{Assumption-a-bk:Bound-derivative-1}, \eqref{Assumption-Bound-derivative-1t-and-2t}
	and Lemma~\ref{Lem:Ij-Iij-H-Z-estimate} that
	%
	\begin{align}
		&\Erw \bigg( \max_{0 \leq n \leq N} \bigg\| \sum_{l=0}^{n-1} 
		\sum_{k=1}^m \sum_{i=1}^s \beta_i^{(2)} \,
		b^{k}(t_l + c_i^{(1)} h, H_{i,l}^{(k)}) \bigg\|^p \bigg) \nonumber \\
		&= \Erw \bigg( \max_{0 \leq n \leq N} \bigg\| \sum_{l=0}^{n-1} 
		\sum_{k=1}^m \sum_{i=1}^s \beta_i^{(2)} \bigg[ b^{k}(t_l + c_i^{(1)} h,Y_l) 
		\nonumber \\
		&\quad + \int_0^1 \sum_{r=1}^d \frac{\partial}{\partial x_r} 
		b^k(t_l+c_i^{(1)} h, Y_l+u(H_{i,l}^{(k)}-Y_l)) \, (H_{i,l}^{(k)^r}-Y_l^r)
		\, \mathrm{d}u \bigg] \bigg\|^p \bigg) \nonumber \\
		&= \Erw \bigg( \max_{0 \leq n \leq N} \bigg\| \sum_{l=0}^{n-1} 
		\sum_{k=1}^m \sum_{i=1}^s \beta_i^{(2)} \bigg[ b^{k}(t_l,Y_l)
		+ \int_0^1 \frac{\partial}{\partial t} b^k(t_l+u \, c_i^{(1)}h,Y_l) 
		\, c_i^{(1)} h \, \mathrm{d}u \nonumber \\
		&\quad + \int_0^1 \sum_{r=1}^d \frac{\partial}{\partial x_r} 
		b^k(t_l+c_i^{(1)} h, Y_l+u(H_{i,l}^{(k)}-Y_l)) \, (H_{i,l}^{(k)^r}-Y_l^r)
		\, \mathrm{d}u \bigg] \bigg\|^p \bigg) \nonumber \\
		&\leq 3^{p-1} \Erw \bigg( \max_{0 \leq n \leq N} \bigg\| \sum_{l=0}^{n-1} 
		\sum_{k=1}^m \sum_{i=1}^s \beta_i^{(2)} b^{k}(t_l,Y_l) \bigg\|^p \bigg) \nonumber \\
		&\quad + 3^{p-1} \Erw \bigg( \max_{0 \leq n \leq N} \bigg\| \sum_{l=0}^{n-1} 
		\sum_{k=1}^m \sum_{i=1}^s \beta_i^{(2)} 
		\int_0^1 \frac{\partial}{\partial t} b^k(t_l+u \, c_i^{(1)}h,Y_l) 
		\, c_i^{(1)} h \, \mathrm{d}u \bigg\|^p \bigg) \nonumber \\
		&\quad + 3^{p-1} \Erw \bigg( \max_{0 \leq n \leq N} \bigg\| \sum_{l=0}^{n-1} 
		\sum_{k=1}^m \sum_{i=1}^s \beta_i^{(2)}
		\int_0^1 \sum_{r=1}^d \frac{\partial}{\partial x_r} 
		b^k(t_l+c_i^{(1)} h, Y_l+u(H_{i,l}^{(k)}-Y_l)) \nonumber \\
		&\quad \times (H_{i,l}^{(k)^r}-Y_l^r) \, \mathrm{d}u \bigg\|^p \bigg) 
		\nonumber \\
		&\leq 3^{p-1} \Erw \bigg( \max_{0 \leq n \leq N} \bigg( \sum_{l=0}^{n-1} 
		\sum_{k=1}^m \sum_{i=1}^s \big| \beta_i^{(2)} \big|
		\int_0^1 \Big\| \frac{\partial}{\partial t} b^k(t_l+u \, c_i^{(1)}h,Y_l) \Big\| 
		\, \big| c_i^{(1)} \big| \, h \, \mathrm{d}u \bigg)^p \bigg) \nonumber \\
		&\quad + 3^{p-1} \Erw \bigg( \max_{0 \leq n \leq N} \bigg( \sum_{l=0}^{n-1} 
		\sum_{k=1}^m \sum_{i=1}^s \big| \beta_i^{(2)} \big|
		\int_0^1 \sum_{r=1}^d \Big\| \frac{\partial}{\partial x_r} 
		b^k(t_l+c_i^{(1)} h, Y_l+u(H_{i,l}^{(k)}-Y_l)) \Big\| \nonumber \\
		&\quad \times \big| H_{i,l}^{(k)^r}-Y_l^r \big| \, \mathrm{d}u
		\bigg)^p \bigg) \nonumber \\
		&\leq 3^{p-1} N^{p-1} \sum_{l=0}^{N-1} m^{p-1} \sum_{k=1}^m s^{p-1} \sum_{i=1}^s
		\czwei^{2p} \, h^p \, \Erw \big( \cc^p \big(1+ \| Y_l \| \big)^p \big) \nonumber \\
		&\quad + 3^{p-1} N^{p-1} \sum_{l=0}^{N-1} m^{p-1} \sum_{k=1}^m s^{p-1} 
		\sum_{i=1}^s \czwei^{p} \, d^{p-1} \sum_{r=1}^d \cc^p \,
		\Erw \big( \big\| H_{i,l}^{(k)}-Y_l \big\|^p \big) \nonumber \\
		&\leq 3^{p-1} (T-t_0)^p (m \, s \, \czwei^2 \, \cc)^p 2^{p-1}
		+ 3^{p-1} (T-t_0)^{p-1} (m \, s \, \czwei^2 \, \cc)^p 2^{p-1} \sum_{l=0}^{N-1} h \,
		\Erw \big( \max_{0 \leq k \leq l} \| Y_k \|^p \big) \nonumber \\
		&\quad + 3^{p-1} N^{p-1} \sum_{l=0}^{N-1} m^{p-1} \sum_{k=1}^m
		s^{p-1} \sum_{i=1}^s (\czwei \, d \, \cc)^p \big[ \cMH \, h^p + \cMH \, h^p
		\, \Erw \big( \| Y_l \|^p \big) \big] \nonumber \\
		&\leq 3^{p-1} (T-t_0)^p (m \, s \, \czwei^2 \, \cc)^p 2^{p-1}
		+ 3^{p-1} (T-t_0)^{p-1} (m \, s \, \czwei^2 \, \cc)^p 2^{p-1} \sum_{l=0}^{N-1} h \,
		\Erw \big( \max_{0 \leq k \leq l} \| Y_k \|^p \big) \nonumber \\
		&\quad + 3^{p-1} (T-t_0)^p (m \, s \, \czwei^2 \, d \, \cc)^p \, \cMH
		+ 3^{p-1} (T-t_0)^{p-1} (m \, s \, \czwei^2 \, d \, \cc)^p \, \cMH 
		\sum_{l=0}^{N-1} h \, \Erw \big( \max_{0 \leq k \leq l} \| Y_k \|^p \big) \, .
		\label{Proof:Lp-MB:eqn-Term3}
	\end{align}
	Thus, from \eqref{Proof:Lp-MB:eqn-1} and 
	\eqref{Proof:Lp-MB:eqn-Term1}--\eqref{Proof:Lp-MB:eqn-Term3} we get with
	some constant $\cvierM = \cvierM(T,p,m,d,\cc,s,\czwei)>0$ that
	\begin{align*}
		\Erw \big( \max_{0 \leq n \leq N} \| Y_n \|^p \big)
		&\leq \cvierM \bigg( 1 + \Erw ( \| Y_0 \|^p )
		+ \sum_{l=0}^{N-1} h \, \Erw \big( \max_{0 \leq k \leq l} \| Y_k \|^p \big)
		\bigg) \, .
	\end{align*}
	Finally, the assumption follows with Gr\"onwall's lemma.
\end{proof}
\subsection{Proof of Convergence for the SRK Method for It{\^o} SDEs}
\label{Sub:Sec:Proof-Convergence-SRK-method}
In this section, we give a proof for Theorem~\ref{Sec:Main-Result:Thm-Konv-SRK-allg}.
Therefore, in this Section~\ref{Sub:Sec:Proof-Convergence-SRK-method} we always 
consider SDE~\eqref{SDE-Integral-form} as an It{\^o} SDE and thus assume that 
$\Iihat_{(i,j),n} = \Ii_{(i,j),n}$ for SRK method~\eqref{SRK-method}.
%
%
%
\begin{proof}[Proof of Theorem~\ref{Sec:Main-Result:Thm-Konv-SRK-allg}]
Let $N_0 \in \mathbb{N}$ such that $0 < \hnull = \frac{T-t_0}{N_0} 
< \frac{1}{s \, \cc \, \czwei}$,
let $N \in \mathbb{N}$ with $N \geq N_0$ and $p \geq 2$. 
For $n \in \{0, 1, \ldots, N\}$ and $h = \tfrac{T-t_0}{N} \leq \hnull$ 
it holds
\begin{align*}
	X_{t_n} &= X_{t_0} + \sum_{l=0}^{n-1} \int_{t_l}^{t_{l+1}} a(s,X_s) \, \mathrm{d}s
		+ \sum_{k=1}^m \sum_{l=0}^{n-1} \int_{t_l}^{t_{l+1}} b^k(s,X_s) 
		\, \mathrm{d}W_s^k
\end{align*}
and we define the processes
\begin{align}
	Z_n &= X_{t_0} + \sum_{l=0}^{n-1} \int_{t_l}^{t_{l+1}} \sum_{i=1}^s
	\alpha_i \, a(t_l+c_i^{(0)} h, H_{i,l}^{(0),X_{t_l}} ) \, \mathrm{d}s 
	\nonumber \\
	&\quad + \sum_{k=1}^m \sum_{l=0}^{n-1} \sum_{i=1}^s \big( \beta_i^{(1)}
	\Ii_{(k),l} + \beta_i^{(2)} \big) \, b^k(t_l+c_i^{(1)} h, H_{i,l}^{(k),X_{t_l}}) \, , 
	\label{Proof:MainThm:Process-Z} \\
	Y_n &= Y_0 + \sum_{l=0}^{n-1} \int_{t_l}^{t_{l+1}} \sum_{i=1}^s \alpha_i \,
	a(t_l+c_i^{(0)} h, H_{i,l}^{(0),Y_l}) \, \mathrm{d}s
	\nonumber \\
	&\quad + \sum_{k=1}^m \sum_{l=0}^{n-1} \sum_{i=1}^s \big( \beta_i^{(1)}
	\Ii_{(k),l} + \beta_i^{(2)} \big) \, b^k(t_l+c_i^{(1)} h, H_{i,l}^{(k),Y_l}) \, .
	\label{Proof:MainThm:Process-Y}
\end{align}
%
%
Further, we make use of the notation introduced in~\eqref{Proof:MainThm:Stages-H-0-Gamma}
and~\eqref{Proof:MainThm:Stages-H-k-Gamma} for the stage values.
For ease of notation, we sometimes neglect to explicitly indicate the variable $\varGamma$
if it is clear from the context which variable is considered for $\varGamma$.
Here, we have $\Iihat_{(r,k),l} = \Ii_{(r,k),l}$ due to considered It{\^o} SDEs.
%

In order to prove the assumption, we firstly split the error term into two
parts
%
\begin{align} \label{Proof:MainThm:Teil-A-B}
	\Erw \big( \sup_{0 \leq n \leq N} \| X_{t_n}-Y_n \|^p \big)
	&\leq 
	2^{p-1} \Erw \big( \sup_{0 \leq n \leq N} \| X_{t_n}-Z_n \|^p \big)
	+ 2^{p-1} \Erw \big( \sup_{0 \leq n \leq N} \| Z_n-Y_n \|^p \big)
\end{align}
which are estimated separately in the following.

%
%
We start estimating the first summand on the right hand side of
\eqref{Proof:MainThm:Teil-A-B} and split this term into two further terms
%
\begin{align}
	&\Erw \big( \sup_{0 \leq n \leq N} \| X_{t_n}-Z_n \|^p \big) 
	\nonumber \\
	&= \Erw \bigg( \sup_{0 \leq n \leq N} \bigg\| \sum_{l=0}^{n-1} \int_{t_l}^{t_{l+1}}
	a(s,X_s) - \sum_{i=1}^s \alpha_i a(t_l+c_i^{(0)} h, H_{i,l}^{(0),X_{t_l}})
	\, \mathrm{d}s \nonumber \\
	&\quad + \sum_{k=1}^m \sum_{l=0}^{n-1} 
	\bigg( \int_{t_l}^{t_{l+1}} b^k(s,X_s) \, \mathrm{d}W_s^k
	- \sum_{i=1}^s \big( \beta_i^{(1)} \Ii_{(k),l} + \beta_i^{(2)} \big)
	b^k(t_l+c_i^{(1)} h, H_{i,l}^{(k),X_{t_l}} ) \bigg\|^p \bigg) \nonumber \\
	&\leq 2^{p-1} \Erw \bigg( \sup_{0 \leq n \leq N} \bigg\| \sum_{l=0}^{n-1}
	\int_{t_l}^{t_{l+1}} a(s,X_s) - \sum_{i=1}^s \alpha_i 
	a(t_l+c_i^{(0)} h, H_{i,l}^{(0),X_{t_l}} ) \, \mathrm{d}s \bigg\|^p \bigg) 
	\label{Proof:MainThm:Teil-A1} \\
	&\quad + 2^{p-1} \Erw \bigg( \sup_{0 \leq n \leq N} \bigg\| \sum_{k=1}^m
	\sum_{l=0}^{n-1} \bigg( \int_{t_l}^{t_{l+1}} b^k(s,X_s) \, \mathrm{d}W_s^k 
	\nonumber \\
	&\quad 
	- \sum_{i=1}^s \big( \beta_i^{(1)} \Ii_{(k),l} + \beta_i^{(2)} \big)
	b^k(t_l+c_i^{(1)} h, H_{i,l}^{(k),X_{t_l}}) \bigg) \bigg\|^p \bigg) \, .
	\label{Proof:MainThm:Teil-A2}
\end{align}
%
%
Next, we consider the term \eqref{Proof:MainThm:Teil-A1} where we get
with Taylor expansion
%
\begin{align}
	&\Erw \bigg( \sup_{0 \leq n \leq N} \bigg\| \sum_{l=0}^{n-1}
	\int_{t_l}^{t_{l+1}} a(s,X_s) - \sum_{i=1}^s \alpha_i 
	a(t_l+c_i^{(0)} h, H_{i,l}^{(0),X_{t_l}} ) \, \mathrm{d}s \bigg\|^p \bigg) \nonumber \\
	&= \Erw \bigg( \sup_{0 \leq n \leq N} \bigg\| \sum_{l=0}^{n-1}
	\int_{t_l}^{t_{l+1}} a(t_l,X_{t_l}) + \sum_{r=1}^d \frac{\partial}{\partial x_r}
	a(t_l,X_{t_l}) \, (X_s^r-X_{t_l}^r) \nonumber \\
	&\quad + \int_0^1 \sum_{q,r=1}^d \frac{\partial^2}{\partial x_q \partial x_r}
	a(t_l,X_l + u(X_s-X_{t_l}) ) \, (X_s^q-X_{t_l}^q) \, (X_s^r-X_{t_l}^r)
	(1-u) \, \mathrm{d}u \nonumber \\
	&\quad + \int_0^1 \frac{\partial}{\partial t} a(t_l+u (s-t_l), X_s) \, (s-t_l) 
	\, \mathrm{d}u
	- \sum_{i=1}^s \alpha_i a(t_l,X_{t_l}) \nonumber \\
	&\quad -\sum_{i=1}^s \alpha_i \int_0^1 \sum_{r=1}^d 
	\frac{\partial}{\partial x_r} a(t_l, X_{t_l} + u(H_{i,l}^{(0),X_{t_l}} - X_{t_l})) 
	\, (H_{i,l}^{(0),X_{t_l}^r} - X_{t_l}^r) \, \mathrm{d}u \nonumber \\
	&\quad -\sum_{i=1}^s \alpha_i \int_0^1 \frac{\partial}{\partial t}
	a(t_l+u c_i^{(0)} h, H_{i,l}^{(0),X_{t_l}} ) \, c_i^{(0)} h \, \mathrm{d}u
	\, \mathrm{d}s \bigg\|^p \bigg) \nonumber \\
%
%
	&\leq 6^{p-1} \bigg[ \Erw \bigg( \sup_{0 \leq n \leq N} \bigg\| \sum_{l=0}^{n-1}
	\int_{t_l}^{t_{l+1}} \Big( 1 -\sum_{i=1}^s \alpha_i \Big) \, a(t_l,X_{t_l}) 
	\, \mathrm{d}s \bigg\|^p \bigg) 
	\label{Proof:MainThm:Teil-A1-1} \\
	&\quad + \Erw \bigg( \sup_{0 \leq n \leq N} \bigg\| \sum_{l=0}^{n-1}
	\int_{t_l}^{t_{l+1}} \sum_{r=1}^d \frac{\partial}{\partial x_r} a(t_l,X_{t_l}) 
	\, (X_s^r-X_{t_l}^r) \, \mathrm{d}s \bigg\|^p \bigg) 
	\label{Proof:MainThm:Teil-A1-2} \\
	&\quad +\Erw \bigg( \sup_{0 \leq n \leq N} \bigg\| \sum_{l=0}^{n-1}
	\int_{t_l}^{t_{l+1}} \int_0^1 \sum_{q,r=1}^d \frac{\partial^2}{\partial x_q \partial x_r}
	a(t_l, X_{t_l}+u(X_s-X_{t_l})) \nonumber \\ 
	&\quad \times (X_s^q-X_{t_l}^q) \, (X_s^r-X_{t_l}^r) (1-u)
	\, \mathrm{d}u \, \mathrm{d}s \bigg\|^p \bigg) 
	\label{Proof:MainThm:Teil-A1-3} \\
	&\quad + \Erw \bigg( \sup_{0 \leq n \leq N} \bigg\| \sum_{l=0}^{n-1}
	\int_{t_l}^{t_{l+1}} \int_0^1 \frac{\partial}{\partial t} a(t_l+u(s-t_l), X_s) \, (s-t_l)
	\, \mathrm{d}u \, \mathrm{d}s \bigg\|^p \bigg) 
	\label{Proof:MainThm:Teil-A1-4} \\
	&\quad + \Erw \bigg( \sup_{0 \leq n \leq N} \bigg\| \sum_{l=0}^{n-1}
	\sum_{i=1}^s \alpha_i h \int_0^1 \sum_{r=1}^d \frac{\partial}{\partial x_r}
	a(t_l, X_{t_l}+u(H_{i,l}^{(0),X_{t_l}}-X_{t_l}) ) \, 
	( {H_{i,l}^{(0),X_{t_l}}}^r -X_{t_l}^r ) \, \mathrm{d}u \bigg\|^p \bigg) 
	\label{Proof:MainThm:Teil-A1-5} \\
	&\quad + \Erw \bigg( \sup_{0 \leq n \leq N} \bigg\| \sum_{l=0}^{n-1}
	\sum_{i=1}^s \alpha_i h \int_0^1 \frac{\partial}{\partial t} 
	a(t_l+u c_i^{(0)} h, H_{i,l}^{(0),X_{t_l}}) \, c_i^{(0)} h \, \mathrm{d}u \bigg\|^p 
	\bigg) \bigg] \, .
	\label{Proof:MainThm:Teil-A1-6}
\end{align}
%
%
Here, it follows that \eqref{Proof:MainThm:Teil-A1-1} vanishes if the condition
$\sum_{i=1}^s \alpha_i =1$ is fulfilled. 
%
Considering \eqref{Proof:MainThm:Teil-A1-2} we get
%
\begin{align}
	&\Erw \bigg( \sup_{0 \leq n \leq N} \bigg\| \sum_{l=0}^{n-1}
	\int_{t_l}^{t_{l+1}} \sum_{r=1}^d \frac{\partial}{\partial x_r} a(t_l,X_{t_l}) 
	\, (X_s^r-X_{t_l}^r) \, \mathrm{d}s \bigg\|^p \bigg) 
	\nonumber \\
	&= \Erw \bigg( \sup_{0 \leq n \leq N} \bigg\| \sum_{l=0}^{n-1}
	\int_{t_l}^{t_{l+1}} \sum_{r=1}^d \frac{\partial}{\partial x_r} a(t_l,X_{t_l}) 
	\, \bigg( \int_{t_l}^s a^r(u,X_u) \, \mathrm{d}u + \sum_{j=1}^m \int_{t_l}^s
	b^{r,j}(u,X_u) \, \mathrm{d}W_u^j \bigg) \mathrm{d}s \bigg\|^p \bigg)
	\nonumber \\
	&\leq 2^{p-1} \Erw \bigg( \sup_{0 \leq n \leq N} \bigg\| \sum_{l=0}^{n-1}
	\int_{t_l}^{t_{l+1}} \sum_{r=1}^d \frac{\partial}{\partial x_r} a(t_l,X_{t_l})
	\int_{t_l}^s a^r(u,X_u) \, \mathrm{d}u \, \mathrm{d}s \bigg\|^p \bigg)
	\label{Proof:MainThm:Teil-A1-2a} \\
	&\quad + 2^{p-1} \Erw \bigg( \sup_{0 \leq n \leq N} \bigg\| \sum_{l=0}^{n-1}
	\int_{t_l}^{t_{l+1}} \sum_{r=1}^d \frac{\partial}{\partial x_r} a(t_l,X_{t_l})
	\sum_{j=1}^m \int_{t_l}^s b^{r,j}(u,X_u) \, \mathrm{d}W_u^j \, \mathrm{d}s
	\bigg\|^p \bigg) \, .
	\label{Proof:MainThm:Teil-A1-2b}
\end{align}
%
%
For $t_l \leq s < t_{l+1}$ let $\lfloor s \rfloor = t_l$, $l=0, \ldots, N-1$
and $\lfloor s \rfloor = t_N$ if $s \geq t_N$. With H\"older's inequality, $p \geq 2$,
\eqref{Assumption-a-bk:lin-growth}, \eqref{Assumption-a-bk:Bound-derivative-1}
and
Lemma~\ref{Lem:Lp-bound-SDE-sol} we get for~\eqref{Proof:MainThm:Teil-A1-2a}
%
\begin{align*}
	&\Erw \bigg( \sup_{0 \leq n \leq N} \bigg\| \sum_{l=0}^{n-1}
	\int_{t_l}^{t_{l+1}} \sum_{r=1}^d \frac{\partial}{\partial x_r} a(t_l,X_{t_l})
	\int_{t_l}^s a^r(u,X_u) \, \mathrm{d}u \, \mathrm{d}s \bigg\|^p \bigg) 
	\nonumber \\
	&= \Erw \bigg( \sup_{0 \leq n \leq N} \bigg\|
	\int_{t_0}^{t_n} \sum_{r=1}^d \frac{\partial}{\partial x_r} 
	a( \lfloor s \rfloor,X_{\lfloor s \rfloor})
	\int_{\lfloor s \rfloor}^s a^r(u,X_u) \, \mathrm{d}u \, \mathrm{d}s \bigg\|^p \bigg) 
	\nonumber \\
	&\leq \Erw \bigg( \sup_{0 \leq n \leq N} \bigg(
	\int_{t_0}^{t_n} \bigg\| \sum_{r=1}^d \frac{\partial}{\partial x_r} 
	a( \lfloor s \rfloor,X_{\lfloor s \rfloor})
	\int_{\lfloor s \rfloor}^s a^r(u,X_u) \, \mathrm{d}u  \bigg\| \, \mathrm{d}s 
	\bigg)^p \bigg) 
	\nonumber \\
	&\leq \Erw \bigg( \int_{t_0}^{t_N} \bigg\| \sum_{r=1}^d \frac{\partial}{\partial x_r} 
	a( \lfloor s \rfloor,X_{\lfloor s \rfloor}) \int_{\lfloor s \rfloor}^s a^r(u,X_u) 
	\, \mathrm{d}u  \bigg\|^p \, \mathrm{d}s \, (t_N-t_0)^{p-1} \bigg) 
	\nonumber \\
%
	&\leq (T-t_0)^{p-1} \Erw \bigg( \int_{t_0}^{t_N} \bigg( \sum_{r=1}^d 
	\bigg\| \frac{\partial}{\partial x_r} 
	a( \lfloor s \rfloor,X_{\lfloor s \rfloor}) \bigg\| \int_{\lfloor s \rfloor}^s 
	| a^r(u,X_u) | \, \mathrm{d}u  \bigg)^p \, \mathrm{d}s \bigg) 
	\nonumber \\
	&\leq (T-t_0)^{p-1} \int_{t_0}^{t_N} \cc^p d^{p-1} \sum_{r=1}^d \Erw \bigg(
	\bigg( \int_{\lfloor s \rfloor}^s | a^r(u,X_u) | \, \mathrm{d}u  \bigg)^p \bigg) \, \mathrm{d}s 
	\nonumber \\
	&\leq (T-t_0)^{p-1} \int_{t_0}^{t_N} \cc^p d^p \Erw \big(
	\cc^p \big(1 + \sup_{t_0 \leq t \leq T} \| X_t \| \big)^p \big) 
	\bigg( \int_{\lfloor s \rfloor}^s \, \mathrm{d}u  \bigg)^p \, \mathrm{d}s 
	\nonumber \\
	&\leq (T-t_0)^p \, (\cc^2 d)^p \big( 2^{p-1} + 2^{p-1} 
	\Erw \big( \sup_{t_0 \leq t \leq T} \| X_t \|^p \big) \big) \, h^p 
	\nonumber \\
	&\leq (T-t_0)^p \, (\cc^2 d)^p \big( 2^{p-1} + 2^{p-1} \ccp \big(1+
	\Erw \big( \| X_{t_0} \|^p \big) \big) \, h^p \, .
\end{align*}
%
Now, consider \eqref{Proof:MainThm:Teil-A1-2b}. 
Note that $(M_n)_{n \in \{0,\ldots,N\}}$ with $M_0=0$ and
%
\begin{align*}
	M_n := \sum_{l=0}^{n-1}
	\int_{t_l}^{t_{l+1}} \sum_{r=1}^d \frac{\partial}{\partial x_r} a(t_l,X_{t_l})
	\sum_{j=1}^m \int_{t_l}^s b^{r,j}(u,X_u) \, \mathrm{d}W_u^j \, \mathrm{d}s
\end{align*}
for $n=1, \ldots, N$ is a discrete time martingale w.r.t.\ the filtration 
$(\mathcal{F}_{t_n})_{n \in \{0, \ldots, N\}}$.
Then, for $p \geq 2$ it follows with Burkholder's inequality (see, e.g.,
\cite{Burk88} or \cite[Prop.~2.1, Prop.~2.2]{PlRoe21}), H\"older's
inequality, It{\^o} isometry and Lemma~\ref{Lem:Lp-bound-SDE-sol} that
%
\begin{align*}
	&\Erw \bigg( \sup_{0 \leq n \leq N} \bigg\| \sum_{l=0}^{n-1}
	\int_{t_l}^{t_{l+1}} \sum_{r=1}^d \frac{\partial}{\partial x_r} a(t_l,X_{t_l})
	\sum_{j=1}^m \int_{t_l}^s b^{r,j}(u,X_u) \, \mathrm{d}W_u^j \, \mathrm{d}s
	\bigg\|^p \bigg) \\
	&\leq \Big( \frac{p^2}{p-1} \Big)^{\frac{p}{2}} \bigg( \sum_{l=0}^{N-1} \bigg(
	\Erw \bigg( \bigg\| \int_{t_l}^{t_{l+1}} \sum_{r=1}^d \frac{\partial}{\partial x_r}
	a(t_l,X_{t_l}) \sum_{j=1}^m \int_{t_l}^s b^{r,j}(u,X_u) \, \mathrm{d}W_u^j
	\, \mathrm{d}s \bigg\|^p \bigg) \bigg)^{\frac{2}{p}} \bigg)^{\frac{p}{2}} \\
%
	&\leq \Big( \frac{p^2}{p-1} \Big)^{\frac{p}{2}} \bigg( \sum_{l=0}^{N-1} \bigg(
	\Erw \bigg( \int_{t_l}^{t_{l+1}} \bigg\| \sum_{r=1}^d \frac{\partial}{\partial x_r}
	a(t_l,X_{t_l}) \sum_{j=1}^m \int_{t_l}^s b^{r,j}(u,X_u) \, \mathrm{d}W_u^j
	\bigg\|^p \mathrm{d}s \, h^{p-1} \bigg) \bigg)^{\frac{2}{p}} \bigg)^{\frac{p}{2}} \\
	%
	%
	&\leq \Big( \frac{p^2}{p-1} \Big)^{\frac{p}{2}} \bigg( \sum_{l=0}^{N-1} 
	h^{2-\frac{2}{p}} \bigg( \int_{t_l}^{t_{l+1}} d^{p-1} \sum_{r=1}^d \cc^p
	\Erw \bigg( \bigg| \sum_{j=1}^m \int_{t_l}^s b^{r,j}(u,X_u) \, \mathrm{d}W_u^j 
	\bigg|^p \bigg) \, \mathrm{d}s \bigg)^{\frac{2}{p}} \bigg)^{\frac{p}{2}} \\
	&\leq \Big( \frac{p^2}{p-1} \Big)^{\frac{p}{2}} \bigg( \sum_{l=0}^{N-1} 
	h^{2-\frac{2}{p}} \bigg( \int_{t_l}^{t_{l+1}} d^{p-1} \sum_{r=1}^d \cc^p
	\Erw \bigg( \sup_{t_l \leq s \leq t_{l+1}} 
	\bigg| \sum_{j=1}^m \int_{t_l}^s b^{r,j}(u,X_u) \, \mathrm{d}W_u^j 
	\bigg|^p \bigg) \, \mathrm{d}s \bigg)^{\frac{2}{p}} \bigg)^{\frac{p}{2}} \\
%
	&\leq \Big( \frac{p^2}{p-1} \Big)^{\frac{p}{2}} \bigg( \sum_{l=0}^{N-1} 
	h^{2-\frac{2}{p}} \cc^2 \bigg( \int_{t_l}^{t_{l+1}} d^{p-1} \sum_{r=1}^d
	\Big( \frac{p^2}{p-1} \Big)^{\frac{p}{2}} \\
	&\quad \times 
	\bigg( \int_{t_l}^{t_{l+1}} \bigg(
	\Erw \bigg( \bigg( \sum_{j=1}^m | b^{r,j}(u,X_u) |^2 \bigg)^{\frac{p}{2}} \bigg)
	\bigg)^{\frac{2}{p}} \mathrm{d}u \bigg)^{\frac{p}{2}} \bigg)^{\frac{2}{p}}
	\bigg)^{\frac{p}{2}} \\
	&\leq \Big( \frac{p^2}{p-1} \Big)^{\frac{p}{2}} \bigg( \sum_{l=0}^{N-1} 
	h^{2-\frac{2}{p}} \cc^2 d^2 h^{\frac{2}{p}} \frac{p^2}{p-1} \bigg( \bigg( 
	\int_{t_l}^{t_{l+1}} \big( \Erw \big( m^{\frac{p}{2}} \cc (1+ \| X_u \|^p ) \big) 
	\big)^{\frac{2}{p}} \, \mathrm{d}u \bigg)^{\frac{p}{2}} \bigg)^{\frac{2}{p}}
	\bigg)^{\frac{p}{2}} \\
	&\leq \Big( \frac{p^2}{p-1} \Big)^{\frac{p}{2}} \bigg( \sum_{l=0}^{N-1} 
	h^2 \cc^2 d^2 \frac{p^2}{p-1} \big( m^{\frac{p}{2}} \cc (1+ 
	\ccp (1 + \Erw(\|X_{t_0}\|^p)) )
	\big)^{\frac{2}{p}} \int_{t_l}^{t_{l+1}} \, \mathrm{d}u \bigg)^{\frac{p}{2}} \\
	&= \Big( \frac{p^2}{p-1} \Big)^p (\cc d)^p \big( m^{\frac{p}{2}} 
	\cc (1+ \ccp (1 + \Erw(\|X_{t_0}\|^p)) ) \, (T-t_0)^{\frac{p}{2}} \, h^p \, .
\end{align*}
%
Next, we consider term \eqref{Proof:MainThm:Teil-A1-3}. With H\"older's
inequality, \eqref{Assumption-a-bk:Bound-derivative2-x}
and Lemma~\ref{Lem:Xs-Xtl-estimate} follows
%
\begin{align*}
	&\Erw \bigg( \sup_{0 \leq n \leq N} \bigg\| \sum_{l=0}^{n-1}
	\int_{t_l}^{t_{l+1}} \int_0^1 \sum_{q,r=1}^d \frac{\partial^2}{\partial x_q 
	\partial x_r} a(t_l, X_{t_l}+u(X_s-X_{t_l})) \nonumber \\ 
	&\quad \times (X_s^q-X_{t_l}^q) \, (X_s^r-X_{t_l}^r) \, (1-u) \, \mathrm{d}u 
	\, \mathrm{d}s \bigg\|^p \bigg) \\
	%
	&\leq \Erw \bigg( \sup_{0 \leq n \leq N} \bigg( \sum_{l=0}^{n-1} 
	\int_{t_l}^{t_{l+1}} \int_0^1 \sum_{q,r=1}^d \bigg\| \frac{\partial^2}{\partial x_q 
	\partial x_r} a(t_l, X_{t_l}+u(X_s-X_{t_l})) \bigg\| \\
	&\quad \times | (X_s^q-X_{t_l}^q) | \, | (X_s^r-X_{t_l}^r) | \, (1-u) \, \mathrm{d}u 
	\, \mathrm{d}s \bigg)^p \bigg) \\
	%
%
	&\leq \Erw \bigg( \sup_{0 \leq n \leq N} \bigg( \sum_{l=0}^{n-1} 
	\int_{t_l}^{t_{l+1}} \int_0^1 d^2 \cc \, \| X_s-X_{t_l} \|^2 \, (1-u)
	\, \mathrm{d}u \, \mathrm{d}s \bigg)^p \bigg) \\
	&= \Big( \frac{1}{2} d^2 \cc \Big)^p 
	\Erw \bigg( \bigg( \int_{t_0}^{t_{N}} 
	\| X_s-X_{\lfloor s \rfloor} \|^2 \, \mathrm{d}s \bigg)^p \bigg) \\
	&\leq \Big( \frac{1}{2} d^2 \cc \Big)^p 
	\int_{t_0}^{t_{N}} \Erw \big( \| X_s-X_{\lfloor s \rfloor} \|^{2p} \big) 
	\, \mathrm{d}s \, (t_N-t_0)^{p-1} \\
%
	&\leq \Big( \frac{1}{2} d^2 \cc \Big)^p (T-t_0)^{p}
	\cMXinc (1+ \Erw( \|X_{t_0} \|^{2p})) \, h^p \, .
\end{align*}
%
%
For term \eqref{Proof:MainThm:Teil-A1-4} we get 
with~\eqref{Assumption-Bound-derivative-1t-and-2t}
and Lemma~\ref{Lem:Lp-bound-SDE-sol}
%
\begin{align*}
	&\Erw \bigg( \sup_{0 \leq n \leq N} \bigg\| \sum_{l=0}^{n-1}
	\int_{t_l}^{t_{l+1}} \int_0^1 \frac{\partial}{\partial t} a(t_l+u(s-t_l), X_s) 
	\, (s-t_l) \, \mathrm{d}u \, \mathrm{d}s \bigg\|^p \bigg) \\
	&\leq \Erw \bigg( \sup_{0 \leq n \leq N} \bigg( \sum_{l=0}^{n-1}
	\int_{t_l}^{t_{l+1}} \int_0^1 \Big\| \frac{\partial}{\partial t} 
	a(t_l+u(s-t_l), X_s) \Big\|
	\, |s-t_l| \, \mathrm{d}u \, \mathrm{d}s \bigg)^p \bigg) \\
	&\leq \Erw \bigg( \bigg( \sum_{l=0}^{N-1}
	\int_{t_l}^{t_{l+1}} \int_0^1 \cc (1+\| X_s \|)
	\, |s-t_l| \, \mathrm{d}u \, \mathrm{d}s \bigg)^p \bigg) \\
	&\leq \cc^p \, \Erw \big( (1+ \sup_{t_0 \leq t \leq T} \| X_t \|)^p 
	 \big) \, \bigg( \sum_{l=0}^{N-1}	\int_{t_l}^{t_{l+1}} |s-t_l| 
	 \, \mathrm{d}s \bigg)^p \\
	&\leq \cc^p \, 2^{p-1} \, \big( 1 
	+ \Erw \big( \sup_{t_0 \leq t \leq T} \| X_t \|^p \big) \big)
	\, \bigg( \sum_{l=0}^{N-1} \frac{1}{2} h^2 \bigg)^p \\
	&\leq \frac{1}{2} \, (T-t_0)^p \, \cc^p \, 
	\big( 1 + \ccp (1+ \Erw( \|X_{t_0} \|^p)) \big) \, h^p \, .
\end{align*}
%
%
Considering term \eqref{Proof:MainThm:Teil-A1-5}, we get with
Lemma~\ref{Lem:H0-Xtl-estimate} and Lemma~\ref{Lem:Lp-bound-SDE-sol}
for $0 < h < \hnull$ that
%
\begin{align*}
	&\Erw \bigg( \sup_{0 \leq n \leq N} \bigg\| \sum_{l=0}^{n-1}
	\sum_{i=1}^s \alpha_i \, h \int_0^1 \sum_{r=1}^d \frac{\partial}{\partial x_r}
	a(t_l, X_{t_l}+u(H_{i,l}^{(0),X_{t_l}}-X_{t_l}) ) \, 
	( {H_{i,l}^{(0),X_{t_l}}}^r -X_{t_l}^r ) \, \mathrm{d}u \bigg\|^p \bigg) \\
	&\leq \Erw \bigg( \sup_{0 \leq n \leq N} \bigg( \sum_{l=0}^{n-1}
	\sum_{i=1}^s |\alpha_i| h \int_0^1 \sum_{r=1}^d \Big\| \frac{\partial}{\partial x_r}
	a(t_l, X_{t_l}+u(H_{i,l}^{(0),X_{t_l}}-X_{t_l}) ) \Big\|
	| {H_{i,l}^{(0),X_{t_l}}}^r -X_{t_l}^r | \mathrm{d}u \bigg)^p \bigg) \\
	&\leq \Erw \bigg( \bigg( \sum_{l=0}^{N-1} \sum_{i=1}^s \czwei \, h 
	\, \cc \, d \, \| {H_{i,l}^{(0),X_{t_l}}} -X_{t_l} \| \bigg)^p \bigg) \\
%
	&\leq N^{p-1} \sum_{l=0}^{N-1} h^p \, (\czwei \, \cc \, d)^p \, s^{p-1}
	\sum_{i=1}^s \Erw \big( \| {H_{i,l}^{(0),X_{t_l}}} -X_{t_l} \|^p \big) \\
	&\leq (T-t_0)^{p-1} \sum_{l=0}^{N-1} h \, (s \, \czwei \, \cc \, d)^p \,
	\cMHdetInc (1+ \Erw( \|X_{t_l} \|^p)) \, h^p \\
%
	&\leq (T-t_0)^{p} (s \, \czwei \, \cc \, d)^p \,
	\cMHdetInc \big(1+ \ccp (1+ \Erw( \|X_{t_0} \|^p)) \big) \, h^p \, .
\end{align*}
%
%
Next, we estimate \eqref{Proof:MainThm:Teil-A1-6}. With \eqref{Estimate-H0-01}
and Lemma~\ref{Lem:Lp-bound-SDE-sol} we have
%
\begin{align*}
	&\Erw \bigg( \sup_{0 \leq n \leq N} \bigg\| \sum_{l=0}^{n-1}
	\sum_{i=1}^s \alpha_i \, h \int_0^1 \frac{\partial}{\partial t} 
	a(t_l+u c_i^{(0)} h, H_{i,l}^{(0),X_{t_l}}) \, c_i^{(0)} \, h 
	\, \mathrm{d}u \bigg\|^p \bigg) \\
	&\leq \Erw \bigg( \sup_{0 \leq n \leq N} \bigg( \sum_{l=0}^{n-1}
	\sum_{i=1}^s |\alpha_i| \, h \int_0^1 \Big\| \frac{\partial}{\partial t} 
	a(t_l+u c_i^{(0)} h, H_{i,l}^{(0),X_{t_l}}) \Big\| \, |c_i^{(0)}| \, h 
	\, \mathrm{d}u \bigg)^p \bigg) \\
	&\leq \czwei^{2p} \, N^{p-1} \sum_{l=0}^{N-1} h^{2p} \, s^{p-1} \sum_{i=1}^s
	\Erw \big( \cc^p (1+ \| H_{i,l}^{(0),X_{t_l}} \| )^p \big) \\
	&\leq \czwei^{2p} \, \cc^p \, (T-t_0)^{p-1} \sum_{l=0}^{N-1} h^{p+1} \, s^{p-1} 
	\sum_{i=1}^s 2^{p-1} \Big( 1 + \Erw \Big( \Big(
	\frac{\| X_{t_l} \| + s \, \czwei \, \cc \, h}{1-s \, \czwei \, \cc \, h} 
	\Big)^p \Big) \Big) \\
	&\leq \czwei^{2p} \, \cc^p \, (T-t_0)^{p-1} \sum_{l=0}^{N-1} h^{p+1} \, s^p \,
	2^{p-1} \Big( 1 + 2^{p-1} \frac{\Erw ( \| X_{t_l} \|^p ) 
	+ (s \, \czwei \, \cc \, \hnull)^p}{(1-s \, \czwei \, \cc \, \hnull)^p} 
	\Big) \\
	&\leq \czwei^{2p} \, \cc^p \, (T-t_0)^{p} \, s^p \,
	2^{p-1} \Big( 1 + 2^{p-1} \frac{\Erw ( \| X_{t_l} \|^p ) 
	+ (s \, \czwei \, \cc \, \hnull)^p}{(1-s \, \czwei \, \cc \, \hnull)^p} 
	\Big) \, h^p \\
	&\leq \czwei^{2p} \, \cc^p \, (T-t_0)^{p} \, s^p \,
	2^{p-1} \Big( 1 + 2^{p-1} \frac{\ccp (1+ \Erw( \|X_{t_0} \|^p))
	+ (s \, \czwei \, \cc \, \hnull)^p}{(1-s \, \czwei \, \cc \, \hnull)^p} 
	\Big) \, h^p \, .
\end{align*}
%
%
As the next step, we consider \eqref{Proof:MainThm:Teil-A2}.
Here, we calculate with Taylor expansion
%
\begin{align}
	&\Erw \bigg( \sup_{0 \leq n \leq N} \bigg\| \sum_{k=1}^m
	\sum_{l=0}^{n-1} \bigg( \int_{t_l}^{t_{l+1}} b^k(s,X_s) \, \mathrm{d}W_s^k 
	- \sum_{i=1}^s \big( \beta_i^{(1)} \Ii_{(k),l} + \beta_i^{(2)} \big)
	b^k(t_l+c_i^{(1)} h, H_{i,l}^{(k),X_{t_l}}) \bigg\|^p \bigg) 
	\nonumber \\
	&= \Erw \bigg( \sup_{0 \leq n \leq N} \bigg\| \sum_{l=0}^{n-1} \bigg[
	\sum_{k=1}^m \int_{t_l}^{t_{l+1}} b^k(t_l,X_{t_l}) \, \mathrm{d}W_s^k
	+ \sum_{k=1}^m \int_{t_l}^{t_{l+1}} \sum_{r=1}^d \frac{\partial}{\partial x_r}
	b^k(t_l,X_{t_l}) \, (X_s^r-X_{t_l}^r) \, \mathrm{d}W_s^k 
	\nonumber \\
	&\quad + \sum_{k=1}^m \int_{t_l}^{t_{l+1}} \int_0^1 \sum_{q,r=1}^d 
	\frac{\partial^2}{\partial x_q \partial x_r} b^k(t_l, X_{t_l}+u(X_s-X_{t_l})) 
	\nonumber \\
	&\quad \times (X_s^q-X_{t_l}^q) \, (X_s^r-X_{t_l}^r) \, (1-u) 
	\, \mathrm{d}u \, \mathrm{d}W_s^k 
	\nonumber \\
	&\quad + \sum_{k=1}^m \int_{t_l}^{t_{l+1}} \int_0^1 \frac{\partial}{\partial t}
	b^k(t_l+u(s-t_l),X_s) \, (s-t_l) \, \mathrm{d}u \, \mathrm{d}W_s^k 
	\nonumber \\
	&\quad - \sum_{k=1}^m \sum_{i=1}^s ( \beta_i^{(1)} \, \Ii_{(k),l} + \beta_i^{(2)} )
	\, b^k(t_l,X_{t_l}) 
	\nonumber \\
	&\quad -\sum_{k=1}^m \sum_{i=1}^s ( \beta_i^{(1)} \, \Ii_{(k),l} + \beta_i^{(2)} )
	\sum_{r=1}^d \frac{\partial}{\partial x_r} b^k(t_l,X_{t_l}) 
	\, ( {H_{i,l}^{(k),X_{t_l}} }^r - X_{t_l}^r) 
	\nonumber \\
	&\quad -\sum_{k=1}^m \sum_{i=1}^s ( \beta_i^{(1)} \, \Ii_{(k),l} + \beta_i^{(2)} )
	\, \bigg( \frac{\partial}{\partial t} b^k(t_l,X_{t_l}) \, c_i^{(1)} h 
	\nonumber \\
	&\quad + \int_0^1 \frac{\partial^2}{\partial t^2} 
	b^k(t_l+u \, c_i^{(1)} \, h, H_{i,l}^{(k),X_{t_l}}) \, (c_i^{(1)} h)^2 \, (1-u) 
	\, \mathrm{d}u \bigg) 
	\nonumber \\
	&\quad -\sum_{k=1}^m \sum_{i=1}^s ( \beta_i^{(1)} \, \Ii_{(k),l} + \beta_i^{(2)} )
	\int_0^1 \sum_{r=1}^d \frac{\partial^2}{\partial t \partial x_r}
	b^k(t_l,X_{t_l} +u (H_{i,l}^{(k),X_{t_l}} -X_{t_l})) 
	\nonumber \\
	&\quad \times (c_i^{(1)} h) \, ( {H_{i,l}^{(k),X_{t_l}}}^r -X_{t_l}^r ) \, \mathrm{d}u 
	\nonumber \\
	&\quad -\sum_{k=1}^m \sum_{i=1}^s ( \beta_i^{(1)} \, \Ii_{(k),l} + \beta_i^{(2)} )
	\int_0^1 \sum_{q,r=1}^d \frac{\partial^2}{\partial x_q \partial x_r}
	b^k(t_l,X_{t_l} + u (H_{i,l}^{(k),X_{t_l}} - X_{t_l})) 
	\nonumber \\
	&\quad \times ({ H_{i,l}^{(k),X_{t_l}} }^q - X_{t_l}^q) 
	\, ( {H_{i,l}^{(k),X_{t_l}} }^r - X_{t_l}^r) \, (1-u) \, \mathrm{d}u 
	\bigg] \bigg\|^p \bigg) \, .
	\nonumber
\end{align}
%
Here, it follows that the conditions $\sum_{i=1}^s \beta_i^{(1)} =1$ and
$\sum_{i=1}^s \beta_i^{(2)} = 0$ need to be fulfilled. Further, we need that 
$\sum_{i=1}^s \beta_i^{(2)} c_i^{(1)} = 0$ for the local order $h$ term to vanish.
Then, 
it follows that \eqref{Proof:MainThm:Teil-A2} can be estimated as
%
\begin{align}
	&\Erw \bigg( \sup_{0 \leq n \leq N} \bigg\| \sum_{k=1}^m
	\sum_{l=0}^{n-1} \bigg( \int_{t_l}^{t_{l+1}} b^k(s,X_s) \, \mathrm{d}W_s^k 
	- \sum_{i=1}^s \big( \beta_i^{(1)} \Ii_{(k),l} + \beta_i^{(2)} \big)
	b^k(t_l+c_i^{(1)} h, H_{i,l}^{(k),X_{t_l}}) \bigg\|^p \bigg) 
	\nonumber \\
	&\leq 6^{p-1} \bigg[ \Erw \bigg( \sup_{0 \leq n \leq \mathbb{N}} \bigg\| 
	\sum_{l=0}^{n-1} \bigg( \sum_{k=1}^m \int_{t_l}^{t_{l+1}} \sum_{r=1}^d 
	\frac{\partial}{\partial x_r} b^k(t_l,X_{t_l}) \, (X_s^r-X_{t_l}^r) \, \mathrm{d}W_s^k 
	\nonumber \\
	&\quad -\sum_{k=1}^m \sum_{i=1}^s ( \beta_i^{(1)} \, \Ii_{(k),l} + \beta_i^{(2)} )
	\sum_{r=1}^d \frac{\partial}{\partial x_r} b^k(t_l,X_{t_l}) 
	\, ( {H_{i,l}^{(k),X_{t_l}} }^r - X_{t_l}^r) \bigg) \bigg\|^p \bigg)
	\label{Proof:MainThm:Teil-A2-1} \\ 
	&+ \Erw \bigg( \sup_{0 \leq n \leq N} \bigg\| \sum_{l=0}^{n-1} 
	\sum_{k=1}^m \int_{t_l}^{t_{l+1}} \int_0^1 \sum_{q,r=1}^d 
	\frac{\partial^2}{\partial x_q \partial x_r} b^k(t_l, X_{t_l}+u(X_s-X_{t_l})) 
	\nonumber \\
	&\quad \times (X_s^q-X_{t_l}^q) \, (X_s^r-X_{t_l}^r) \, (1-u) 
	\, \mathrm{d}u \, \mathrm{d}W_s^k \bigg\|^p \bigg) 
	\label{Proof:MainThm:Teil-A2-2} \\
	&+ \Erw \bigg( \sup_{0 \leq n \leq N} \bigg\| \sum_{l=0}^{n-1}
	\sum_{k=1}^m \int_{t_l}^{t_{l+1}} \int_0^1 \frac{\partial}{\partial t}
	b^k(t_l+u(s-t_l),X_s) \, (s-t_l) \, \mathrm{d}u \, \mathrm{d}W_s^k \bigg\|^p \bigg) 
	\label{Proof:MainThm:Teil-A2-3} \\
	&+ \Erw \bigg( \sup_{0 \leq n \leq N} \bigg\| \sum_{l=0}^{n-1}
	\sum_{k=1}^m \sum_{i=1}^s \bigg( \beta_i^{(1)} \, \Ii_{(k),l} \,
	\frac{\partial}{\partial t} b^k(t_l,X_{t_l}) \, c_i^{(1)} h 
	\nonumber \\
	&\quad + ( \beta_i^{(1)} \, \Ii_{(k),l} + \beta_i^{(2)} )
	\int_0^1 \frac{\partial^2}{\partial t^2} 
	b^k(t_l+u \, c_i^{(1)} \, h, H_{i,l}^{(k),X_{t_l}}) \, (c_i^{(1)} h)^2 \, (1-u) 
	\, \mathrm{d}u \bigg) \bigg\|^p \bigg) 
	\label{Proof:MainThm:Teil-A2-4} \\
	&+ \Erw \bigg( \sup_{0 \leq n \leq N} \bigg\| \sum_{l=0}^{n-1}
	\sum_{k=1}^m \sum_{i=1}^s ( \beta_i^{(1)} \, \Ii_{(k),l} + \beta_i^{(2)} )
	\int_0^1 \sum_{r=1}^d \frac{\partial^2}{\partial t \partial x_r}
	b^k(t_l,X_{t_l} +u (H_{i,l}^{(k),X_{t_l}} -X_{t_l})) 
	\nonumber \\
	&\quad \times (c_i^{(1)} h) \, ( {H_{i,l}^{(k),X_{t_l}}}^r -X_{t_l}^r ) 
	\, \mathrm{d}u \bigg\|^p \bigg) 
	\label{Proof:MainThm:Teil-A2-5} \\
	&+ \Erw \bigg( \sup_{0 \leq n \leq N} \bigg\| \sum_{l=0}^{n-1}
	\sum_{k=1}^m \sum_{i=1}^s ( \beta_i^{(1)} \, \Ii_{(k),l} + \beta_i^{(2)} )
	\int_0^1 \sum_{q,r=1}^d \frac{\partial^2}{\partial x_q \partial x_r}
	b^k(t_l,X_{t_l} + u (H_{i,l}^{(k),X_{t_l}} - X_{t_l})) 
	\nonumber \\
	&\quad \times ({ H_{i,l}^{(k),X_{t_l}} }^q - X_{t_l}^q) 
	\, ( {H_{i,l}^{(k),X_{t_l}} }^r - X_{t_l}^r) \, (1-u) \, \mathrm{d}u 
	\bigg\|^p \bigg) \bigg] \, .
	\label{Proof:MainThm:Teil-A2-6}
\end{align}
%
%
%
Firstly, consider term \eqref{Proof:MainThm:Teil-A2-1} and apply Taylor
expansion to get
%
\begin{align*}
	&\Erw \bigg( \sup_{0 \leq n \leq \mathbb{N}} \bigg\| 
	\sum_{l=0}^{n-1} \bigg( \sum_{k=1}^m \int_{t_l}^{t_{l+1}} \sum_{r=1}^d 
	\frac{\partial}{\partial x_r} b^k(t_l,X_{t_l}) \, (X_s^r-X_{t_l}^r) \, \mathrm{d}W_s^k 
	\nonumber \\
	&\quad -\sum_{k=1}^m \sum_{i=1}^s ( \beta_i^{(1)} \, \Ii_{(k),l} + \beta_i^{(2)} )
	\sum_{r=1}^d \frac{\partial}{\partial x_r} b^k(t_l,X_{t_l}) 
	\, ( {H_{i,l}^{(k),X_{t_l}} }^r - X_{t_l}^r) \bigg) \bigg\|^p \bigg) 
	\nonumber \\ 
	&= \Erw \bigg( \sup_{0 \leq n \leq \mathbb{N}} \bigg\| \sum_{l=0}^{n-1} 
	\sum_{k=1}^m \int_{t_l}^{t_{l+1}} \sum_{r=1}^d 
	\frac{\partial}{\partial x_r} b^k(t_l,X_{t_l}) \bigg[ \int_{t_l}^s a^r(u,X_u) 
	\, \mathrm{d}u
	+ \sum_{k_2=1}^m \int_{t_l}^s b^{r,k_2}(t_l,X_{t_l}) 
	\, \mathrm{d}W_u^{k_2}
	\nonumber \\
	&\quad + \sum_{k_2=1}^m \int_{t_l}^s \int_0^1 \sum_{q=1}^d
	\frac{\partial}{\partial x_q} b^{r,k_2}(t_l, X_{t_l}+v(X_u-X_{t_l})) \, (X_u^q-X_{t_l}^q)
	\, \mathrm{d}v \, \mathrm{d}W_u^{k_2}
	\nonumber \\
	&\quad + \sum_{k_2=1}^m \int_{t_l}^s \int_0^1
	\frac{\partial}{\partial t} b^{r,k_2}(t_l + v(s-t_l),X_u) \, (s-t_l) \, \mathrm{d}v
	\, \mathrm{d}W_u^{k_2} \bigg] \, \mathrm{d}W_s^k 
	\nonumber \\
	&\quad - \sum_{l=0}^{n-1} \sum_{k=1}^m \sum_{i=1}^s ( \beta_i^{(1)} \Ii_{(k),l}
	+ \beta_i^{(2)} ) \sum_{r=1}^d \frac{\partial}{\partial x_r} b^k(t_l,X_{t_l}) \bigg[
	\sum_{j=1}^s A_{i,j}^{(1)} \, a^r(t_l,X_{t_l}) \, h
	\nonumber \\
	&\quad + \sum_{j=1}^s A_{i,j}^{(1)} \, h \int_0^1 \sum_{q=1}^d 
	\frac{\partial}{\partial x_q} a^r(t_l,X_{t_l} +u(H_{j,l}^{(0),X_{t_l}}-X_{t_l})) 
	\, ( {H_{j,l}^{(0),X_{t_l}} }^q-X_{t_l}^q) \, \mathrm{d}u
	\nonumber \\
	&\quad + \sum_{j=1}^s A_{i,j}^{(1)} \, h \int_0^1 \frac{\partial}{\partial t} 
	a^r(t_l+u \, c_j^{(0)} \, h,H_{j,l}^{(0),X_{t_l}}) \, c_j^{(0)} \, h \, \mathrm{d}u
	\nonumber \\
	&\quad + \sum_{j=1}^{i-1} \sum_{k_2=1}^m B_{i,j}^{(1)} \, b^{r,k_2}(t_l,X_{t_l})
	\, \Ii_{(k_2,k),l}
	\nonumber \\
	&\quad + \sum_{j=1}^{i-1} \sum_{k_2=1}^m B_{i,j}^{(1)} \int_0^1 \sum_{q=1}^d
	\frac{\partial}{\partial x_q} b^{r,k_2}(t_l, X_{t_l}+u(H_{j,l}^{(k_2),X_{t_l}}-X_{t_l}))
	\, ( {H_{j,l}^{(k_2),X_{t_l}} }^q-X_{t_l}^q) \, \mathrm{d}u \, \Ii_{(k_2,k),l}
	\nonumber \\
	&\quad + \sum_{j=1}^{i-1} \sum_{k_2=1}^m B_{i,j}^{(1)} \int_0^1
	\frac{\partial}{\partial t} b^{r,k_2}(t_l+u \, c_j^{(1)} \, h, H_{j,l}^{(k_2),X_{t_l}})
	\, c_j^{(1)} \, h \, \mathrm{d}u \, \Ii_{(k_2,k),l}
	\bigg] \bigg\|^p \bigg) \, .
\end{align*}
%
As a result of this, we need that
$\sum_{i=1}^s \beta_i^{(2)} \sum_{j=1}^{i-1} B_{i,j}^{(1)} =1$ and
$\sum_{i=1}^s \beta_i^{(2)} \sum_{j=1}^s A_{i,j}^{(1)} =0$
have to be fulfilled. Taking into account these conditions, we get for
 \eqref{Proof:MainThm:Teil-A2-1} that
%
\begin{align}
	&\Erw \bigg( \sup_{0 \leq n \leq N} \bigg\| 
	\sum_{l=0}^{n-1} \bigg( \sum_{k=1}^m \int_{t_l}^{t_{l+1}} \sum_{r=1}^d 
	\frac{\partial}{\partial x_r} b^k(t_l,X_{t_l}) \, (X_s^r-X_{t_l}^r) \, \mathrm{d}W_s^k 
	\nonumber \\
	&\quad -\sum_{k=1}^m \sum_{i=1}^s ( \beta_i^{(1)} \, \Ii_{(k),l} + \beta_i^{(2)} )
	\sum_{r=1}^d \frac{\partial}{\partial x_r} b^k(t_l,X_{t_l}) 
	\, ( {H_{i,l}^{(k),X_{t_l}} }^r - X_{t_l}^r) \bigg) \bigg\|^p \bigg) 
	\nonumber \\ 
	&\leq 9^{p-1} \Erw \bigg( \sup_{0 \leq n \leq N} \bigg\| \sum_{l=0}^{n-1} 
	\sum_{k=1}^m \int_{t_l}^{t_{l+1}} \sum_{r=1}^d \frac{\partial}{\partial x_r}
	b^k(t_l,X_{t_l}) 
	\int_{t_l}^s a^r(u,X_u) \, \mathrm{d}u \, \mathrm{d}W_s^k \bigg\|^p \bigg)
	\label{Proof:MainThm:Teil-A2-1a} \\
	&\quad + 9^{p-1} \Erw \bigg( \sup_{0 \leq n \leq N} \bigg\| \sum_{l=0}^{n-1} 
	\sum_{k=1}^m \int_{t_l}^{t_{l+1}} \sum_{r=1}^d \frac{\partial}{\partial x_r}
	b^k(t_l,X_{t_l}) 
	\nonumber \\
	&\quad \times
	\sum_{k_2=1}^m \int_{t_l}^s \int_0^1 \sum_{q=1}^d \frac{\partial}{\partial x_q}
	b^{r,k_2}(t_l, X_{t_l} + v(X_u-X_{t_l})) \, (X_u^q-X_{t_l}^q) \, \mathrm{d}v
	\, \mathrm{d}W_u^{k_2} \, \mathrm{d}W_s^k \bigg\|^p \bigg)
	\label{Proof:MainThm:Teil-A2-1b} \\
	&\quad + 9^{p-1} \Erw \bigg( \sup_{0 \leq n \leq N} \bigg\| \sum_{l=0}^{n-1} 
	\sum_{k=1}^m \int_{t_l}^{t_{l+1}} \sum_{r=1}^d \frac{\partial}{\partial x_r}
	b^k(t_l,X_{t_l})
	\nonumber \\
	&\quad \times
	\sum_{k_2=1}^m \int_{t_l}^s \int_0^1 \frac{\partial}{\partial t} 
	b^{r,k_2}(t_l + v (s-t_l),X_u) \, (s-t_l) \, \mathrm{d}v \, \mathrm{d}W_u^{k_2}
	\, \mathrm{d}W_s^k \bigg\|^p \bigg)
	\label{Proof:MainThm:Teil-A2-1c} \\
	&\quad + 9^{p-1} \Erw \bigg( \sup_{0 \leq n \leq N} \bigg\| \sum_{l=0}^{n-1}
	\sum_{k=1}^m \sum_{i=1}^s \beta_i^{(1)} \, \Ii_{(k),l} \sum_{r=1}^d
	\frac{\partial}{\partial x_r} b^k(t_l,X_{t_l}) \sum_{j=1}^s A_{i,j}^{(1)} \,
	a^r(t_l,X_{t_l}) \, h \bigg\|^p \bigg)
	\label{Proof:MainThm:Teil-A2-1d} \\
	&\quad + 9^{p-1} \Erw \bigg( \sup_{0 \leq n \leq N} \bigg\| \sum_{l=0}^{n-1}
	\sum_{k=1}^m \sum_{i=1}^s (\beta_i^{(1)} \, \Ii_{(k),l} + \beta_i^{(2)} )
	\sum_{r=1}^d \frac{\partial}{\partial x_r} b^k(t_l,X_{t_l})
	\nonumber \\
	&\quad \times
	\sum_{j=1}^s A_{i,j}^{(1)} \, h \int_0^1 \sum_{q=1}^d \frac{\partial}{\partial x_q}
	a^r(t_l,X_{t_l} + u(H_{j,l}^{(0),X_{t_l}}-X_{t_l})) \, 
	( {H_{j,l}^{(0),X_{t_l}} }^q-X_{t_l}^q) \, \mathrm{d}u \bigg\|^p \bigg)
	\label{Proof:MainThm:Teil-A2-1e} \\
%
	&\quad + 9^{p-1} \Erw \bigg( \sup_{0 \leq n \leq N} \bigg\| \sum_{l=0}^{n-1}
	\sum_{k=1}^m \sum_{i=1}^s (\beta_i^{(1)} \, \Ii_{(k),l} + \beta_i^{(2)} )
	\sum_{r=1}^d \frac{\partial}{\partial x_r} b^k(t_l,X_{t_l})
	\nonumber \\
	&\quad \times \sum_{j=1}^s A_{i,j}^{(1)} \, h \int_0^1 \frac{\partial}{\partial t}
	a^r(t_l + u \, c_j^{(0)} \, h, H_{j,l}^{(0),X_{t_l}}) \, c_j^{(0)} \, h \, \mathrm{d}u
	\bigg\|^p \bigg)
	\label{Proof:MainThm:Teil-A2-1f} \\
	&\quad + 9^{p-1} \Erw \bigg( \sup_{0 \leq n \leq N} \bigg\| \sum_{l=0}^{n-1}
	\sum_{k=1}^m \sum_{i=1}^s \beta_i^{(1)} \, \Ii_{(k),l} \sum_{r=1}^d 
	\frac{\partial}{\partial x_r} b^k(t_l,X_{t_l})
	\nonumber \\
	&\quad \times 
	\sum_{j=1}^{i-1} \sum_{k_2=1}^m B_{i,j}^{(1)} \, b^{r,k_2}(t_l,X_{t_l}) 
	\, \Ii_{(k_2,k),l} \bigg\|^p \bigg)
	\label{Proof:MainThm:Teil-A2-1i} \\
	&\quad + 9^{p-1} \Erw \bigg( \sup_{0 \leq n \leq N} \bigg\| \sum_{l=0}^{n-1}
	\sum_{k=1}^m \sum_{i=1}^s (\beta_i^{(1)} \, \Ii_{(k),l} + \beta_i^{(2)} )
	\sum_{r=1}^d \frac{\partial}{\partial x_r} b^k(t_l,X_{t_l})
	\sum_{j=1}^{i-1} \sum_{k_2=1}^m B_{i,j}^{(1)} 
	\nonumber \\
	&\quad \times \int_0^1
	\sum_{q=1}^d \frac{\partial}{\partial x_q} 
	b^{r,k_2}(t_l,X_{t_l}+ u(H_{j,l}^{(k_2),X_{t_l}}-X_{t_l})) \,
	( {H_{j,l}^{(k_2),X_{t_l}} }^q-X_{t_l}^q) \, \mathrm{d}u \, \Ii_{(k_2,k),l}
	\bigg\|^p \bigg)
	\label{Proof:MainThm:Teil-A2-1g} \\
	&\quad + 9^{p-1} \Erw \bigg( \sup_{0 \leq n \leq N} \bigg\| \sum_{l=0}^{n-1}
	\sum_{k=1}^m \sum_{i=1}^s (\beta_i^{(1)} \, \Ii_{(k),l} + \beta_i^{(2)} )
	\sum_{r=1}^d \frac{\partial}{\partial x_r} b^k(t_l,X_{t_l})
	\nonumber \\
	&\quad \times \sum_{j=1}^{i-1} \sum_{k_2=1}^m B_{i,j}^{(1)} \int_0^1
	\frac{\partial}{\partial t} b^{r,k_2}(t_l +u \, c_j^{(1)} \, h, H_{j,l}^{(k_2),X_{t_l}})
	\, c_j^{(1)} \, h \, \mathrm{d}u \, \Ii_{(k_2,k),l} \bigg\|^p \bigg) 
	\label{Proof:MainThm:Teil-A2-1h} \, .
\end{align}
Next, each term in \eqref{Proof:MainThm:Teil-A2-1a}--\eqref{Proof:MainThm:Teil-A2-1h}
has to be estimated. 
%
%
%
For~\eqref{Proof:MainThm:Teil-A2-1a} we get with
Burkholder's inequality, see, e.g., \cite{Burk88} or
\cite[Prop.~2.1 \& 2.2]{PlRoe21}, with~\eqref{Assumption-a-bk:Bound-derivative-1},
\eqref{Assumption-a-bk:lin-growth} and Lemma~\ref{Lem:Lp-bound-SDE-sol} that
%
\begin{align*}
	&\Erw \bigg( \sup_{0 \leq n \leq N} \bigg\| \sum_{l=0}^{n-1} 
	\sum_{k=1}^m \int_{t_l}^{t_{l+1}} \sum_{r=1}^d \frac{\partial}{\partial x_r}
	b^k(t_l,X_{t_l}) 
	\int_{t_l}^s a^r(u,X_u) \, \mathrm{d}u \, \mathrm{d}W_s^k \bigg\|^p \bigg)
	\nonumber \\
	&\leq \Big( \frac{p}{\sqrt{p-1}} \Big)^p \bigg( \sum_{l=0}^{N-1} \bigg[
	\Erw \bigg( \bigg\| \sum_{k=1}^m \int_{t_l}^{t_{l+1}} \sum_{r=1}^d
	\frac{\partial}{\partial x_r} b^k(t_l,X_{t_l}) 
	\int_{t_l}^s a^r(u,X_u) \, \mathrm{d}u \, \mathrm{d}W_s^k \bigg\|^p \bigg)
	\bigg]^{\frac{2}{p}} \bigg)^{\frac{p}{2}}
	\nonumber \\
	&\leq \Big( \frac{p}{\sqrt{p-1}} \Big)^p \sqrt{p-1} \bigg( \sum_{l=0}^{N-1}
	\int_{t_l}^{t_{l+1}} \bigg[ \Erw \bigg( \bigg\| \sum_{k=1}^m \bigg\| \int_{t_l}^s
	\sum_{r=1}^d \frac{\partial}{\partial x_r} b^k(t_l,X_{t_l}) \, a^r(u,X_u) \,
	\mathrm{d}u \bigg\|^2 \bigg\|^{\frac{p}{2}} \bigg) \bigg]^{\frac{2}{p}}
	\mathrm{d}s \bigg)^{\frac{p}{2}}
	\nonumber \\
	&\leq \Big( \frac{p}{\sqrt{p-1}} \Big)^p \sqrt{p-1} \bigg( \sum_{l=0}^{N-1}
	\int_{t_l}^{t_{l+1}} \bigg[ \Erw \bigg( m^{\frac{p}{2}-1} \sum_{k=1}^m \bigg(
	\int_{t_l}^s \bigg\| \sum_{r=1}^d \frac{\partial}{\partial x_r} b^k(t_l,X_{t_l})
	\nonumber \\ 
	&\quad \times 
	a^r(u,X_u) \bigg\| \, \mathrm{d}u \bigg)^p \bigg) \bigg]^{\frac{2}{p}} \mathrm{d}s
	\bigg)^{\frac{p}{2}}
	\nonumber \\
%
	&\leq \Big( \frac{p}{\sqrt{p-1}} \Big)^p \sqrt{p-1} \, m^{\frac{p}{2}-1} \, (\cc \, d)^p
	\bigg( \sum_{l=0}^{N-1} \int_{t_l}^{t_{l+1}} \bigg[ \sum_{k=1}^m 
	\Erw \bigg( \bigg( \int_{t_l}^s \cc (1 + \| X_u \|) \, \mathrm{d}u \bigg)^p \bigg)
	\bigg]^{\frac{2}{p}} \mathrm{d}s \bigg)^{\frac{p}{2}}
	\nonumber \\
	&\leq \Big( \frac{p}{\sqrt{p-1}} \Big)^p \sqrt{p-1} \, m^{\frac{p}{2}} \, (\cc \, d)^p
	\bigg( \sum_{l=0}^{N-1} \int_{t_l}^{t_{l+1}} \big[ \cc^p \, 2^{p-1}
	\big( 1 + \Erw \big( \sup_{t_0 \leq t \leq T} \| X_t \|^p \big) \big) \, (s-t_l)^p
	\big]^{\frac{2}{p}} \mathrm{d}s \bigg)^{\frac{p}{2}}
	\nonumber \\
	&\leq \Big( \frac{p}{\sqrt{p-1}} \Big)^p \sqrt{p-1} \, m^{\frac{p}{2}} \, (\cc^2 \, d)^p
	\, 2^{p-1} \, \big( 1 + \Erw \big( \sup_{t_0 \leq t \leq T} \| X_t \|^p \big) \big)
	\, (T-t_0)^{\frac{p}{2}} \, h^p 
	\nonumber \\
	&\leq \Big( \frac{p}{\sqrt{p-1}} \Big)^p \sqrt{p-1} \, m^{\frac{p}{2}} \, (\cc^2 \, d)^p
	\, 2^{p-1} \, ( 1 + \ccp ( 1 + \Erw ( \| X_{t_0} \|^p ) ) ) \, (T-t_0)^{\frac{p}{2}} \, 
	h^p \, .
\end{align*}
%
%
%
Now, we consider~\eqref{Proof:MainThm:Teil-A2-1b}. With Burkholder's inequality,
see, e.g., \cite{Burk88} or \cite[Prop.~2.1 \& 2.2]{PlRoe21} and
Lemma~\ref{Lem:Xs-Xtl-estimate} it follows
%
\begin{align*}
	&\Erw \bigg( \sup_{0 \leq n \leq N} \bigg\| \sum_{l=0}^{n-1} 
	\sum_{k=1}^m \int_{t_l}^{t_{l+1}} \sum_{r=1}^d \frac{\partial}{\partial x_r}
	b^k(t_l,X_{t_l}) 
	\nonumber \\
	&\quad \times
	\sum_{k_2=1}^m \int_{t_l}^s \int_0^1 \sum_{q=1}^d \frac{\partial}{\partial x_q}
	b^{r,k_2}(t_l, X_{t_l} + v(X_u-X_{t_l})) \, (X_u^q-X_{t_l}^q) \, \mathrm{d}v
	\, \mathrm{d}W_u^{k_2} \, \mathrm{d}W_s^k \bigg\|^p \bigg)
	\nonumber \\
	&\leq \Big( \frac{p}{\sqrt{p-1}} \Big)^p \bigg( \sum_{l=0}^{N-1} \bigg[
	\Erw \bigg( \bigg\| \sum_{k=1}^m \int_{t_l}^{t_{l+1}} \sum_{r=1}^d b^k(t_l,X_{t_l})
	\nonumber \\
	&\quad \times
	\sum_{k_2=1}^m \int_{t_l}^s \int_0^1 \sum_{q=1}^d \frac{\partial}{\partial x_q}
	b^{r,k_2}(t_l,X_{t_l} + v(X_u-X_{t_l})) \, (X_u^q-X_{t_l}^q) \, \mathrm{d}v
	\, \mathrm{d}W_u^{k_2} \, \mathrm{d}W_s^k \bigg\|^p \bigg) \bigg]^{\frac{2}{p}}
	\bigg)^{\frac{p}{2}}
	\nonumber \\
	&\leq \Big( \frac{p}{\sqrt{p-1}} \Big)^p \bigg( \sum_{l=0}^{N-1} (p-1)
	\int_{t_l}^{t_{l+1}} \bigg[ \Erw \bigg( \bigg\| \sum_{k=1}^m \bigg\| \sum_{k_2=1}^m
	\int_{t_l}^s \sum_{r=1}^d \frac{\partial}{\partial x_r} b^k(t_l,X_{t_l})
	\nonumber \\
	&\quad \times
	\int_0^1 \sum_{q=1}^d \frac{\partial}{\partial x_q} 
	b^{r,k_2}(t_l,X_{t_l} + v(X_u-X_{t_l})) \, (X_u^q-X_{t_l}^q) \, \mathrm{d}v
	\, \mathrm{d}W_u^{k_2} \bigg\|^2 \bigg\|^{\frac{p}{2}} \bigg) \bigg]^{\frac{2}{p}}
	\, \mathrm{d}s \bigg)^{\frac{p}{2}}
	\nonumber \\
	&\leq \Big( \frac{p}{\sqrt{p-1}} \Big)^p \bigg( \sum_{l=0}^{N-1} (p-1)
	\int_{t_l}^{t_{l+1}} \bigg[ m^{\frac{p}{2}-1} \sum_{k=1}^m \Erw \bigg( \bigg\|
	\sum_{k_2=1}^m \int_{t_l}^s \sum_{r=1}^d \frac{\partial}{\partial x_r} b^k(t_l,X_{t_l})
	\nonumber \\
	&\quad \times
	\int_0^1 \sum_{q=1}^d \frac{\partial}{\partial x_q} 
	b^{r,k_2}(t_l,X_{t_l} +v (X_u-X_{t_l})) \, (X_u^q-X_{t_l}^q) \, \mathrm{d}v
	\, \mathrm{d}W_u^{k_2} \bigg\|^p \bigg) \bigg]^{\frac{2}{p}} \, \mathrm{d}s
	\bigg)^{\frac{p}{2}}
	\nonumber \\
%
	&\leq \Big( \frac{p}{\sqrt{p-1}} \Big)^p \bigg( \sum_{l=0}^{N-1} (p-1)
	\int_{t_l}^{t_{l+1}} m^{1-\frac{2}{p}} \bigg[ \sum_{k=1}^m (p-1)^{\frac{p}{2}}
	\bigg( \int_{t_l}^s \bigg[ \Erw \bigg( \bigg\| \sum_{k_2=1}^m \bigg\| \sum_{r=1}^d
	\frac{\partial}{\partial x_r} b^k(t_l,X_{t_l})
	\nonumber \\
	&\quad \times
	\int_0^1 \sum_{q=1}^d \frac{\partial}{\partial x_q} 
	b^{r,k_2}(t_l,X_{t_l}+v(X_u-X_{t_l})) \, (X_u^q-X_{t_l}^q) \, \mathrm{d}v \bigg\|^2
	\bigg\|^{\frac{p}{2}} \bigg) \bigg]^{\frac{2}{p}} \, \mathrm{d}u \bigg)^{\frac{p}{2}}
	\bigg]^{\frac{2}{p}} \, \mathrm{d}s \bigg)^{\frac{p}{2}}
	\nonumber \\
	&\leq \Big( \frac{p}{\sqrt{p-1}} \Big)^p \bigg( \sum_{l=0}^{N-1} (p-1)^2
	\int_{t_l}^{t_{l+1}} 
	m^{1-\frac{2}{p}} \bigg[ \sum_{k=1}^m
	\bigg( \int_{t_l}^s \bigg[ \Erw \bigg( m^{\frac{p}{2}-1} \sum_{k_2=1}^m 
	\bigg\| \int_0^1 \sum_{r=1}^d \frac{\partial}{\partial x_r} b^k(t_l,X_{t_l})
	\nonumber \\
	&\quad \times
	\sum_{q=1}^d \frac{\partial}{\partial x_q} b^{r,k_2}(t_l,X_{t_l}+v(X_u-X_{t_l}))
	\, (X_u^q-X_{t_l}^q) \, \mathrm{d}v \bigg\|^p \bigg) \bigg]^{\frac{2}{p}}
	\, \mathrm{d}u \bigg)^{\frac{p}{2}} \bigg]^{\frac{2}{p}} \mathrm{d}s 
	\bigg)^{\frac{p}{2}}
	\nonumber \\
	&\leq \Big( \frac{p}{\sqrt{p-1}} \Big)^p  (p-1)^p \big( m^{\frac{p}{2}-1} \big)^2
	\bigg( \sum_{l=0}^{N-1} \int_{t_l}^{t_{l+1}} \bigg[ \sum_{k=1}^m
	\bigg( \int_{t_l}^s \bigg[ \Erw \bigg( \sum_{k_2=1}^m \bigg( \int_0^1 
	\sum_{q,r=1}^d \Big\| \frac{\partial}{\partial x_r} b^k(t_l,X_{t_l}) \Big\|
	\nonumber \\
	&\quad \times
	\Big| \frac{\partial}{\partial x_q} b^{r,k_2}(t_l,X_{t_l}+v(X_u-X_{t_l})) \Big|
	\, | X_u^q-X_{t_l}^q | \, \mathrm{d}v \bigg)^p \bigg) \bigg]^{\frac{2}{p}}
	\, \mathrm{d}u \bigg)^{\frac{p}{2}} \bigg]^{\frac{2}{p}} \mathrm{d}s 
	\bigg)^{\frac{p}{2}}
	\nonumber \\
%
	&\leq \Big( \frac{p}{\sqrt{p-1}} \Big)^p  (p-1)^p \big( m^{\frac{p}{2}-1} \big)^2
	(\cc \, d)^{2p} \, m^2 \bigg( \sum_{l=0}^{N-1} \int_{t_l}^{t_{l+1}} \int_{t_l}^s
	\big[ \Erw \big( \| X_u-X_{t_l} \|^p \big)
	\big]^{\frac{2}{p}} \mathrm{d}u \, \mathrm{d}s \bigg)^{\frac{p}{2}}
	\nonumber \\
	&\leq \Big( \frac{p}{\sqrt{p-1}} \Big)^p  (p-1)^p \, m^p \,
	(\cc \, d)^{2p}
	\bigg( \sum_{l=0}^{N-1} \int_{t_l}^{t_{l+1}} \int_{t_l}^s
	\big[ \cMXinc (1+ \Erw( \|X_{t_0} \|^p)) \, (u-t_l)^{\frac{p}{2}} \big]^{\frac{2}{p}}
	\mathrm{d}u \, \mathrm{d}s \bigg)^{\frac{p}{2}}
	\nonumber \\
	&\leq \Big( \frac{p}{\sqrt{p-1}} \Big)^p  (p-1)^p \, m^p \,
	(\cc \, d)^{2p} \, \cMXinc (1+ \Erw( \|X_{t_0} \|^p))
	\bigg( \sum_{l=0}^{N-1} \int_{t_l}^{t_{l+1}} \int_{t_l}^s u-t_l
	\, \mathrm{d}u \, \mathrm{d}s \bigg)^{\frac{p}{2}}
	\nonumber \\
	&\leq \Big( \frac{p}{\sqrt{p-1}} \Big)^p  (p-1)^p \, m^p \,
	(\cc \, d)^{2p} \, \cMXinc (1+ \Erw( \|X_{t_0} \|^p)) \, (T-t_0)^{\frac{p}{2}}
	\, h^p \, .
\end{align*}
%
%
%
Considering~\eqref{Proof:MainThm:Teil-A2-1c}, we get with Burkholder's inequality,
see, e.g., \cite{Burk88} or \cite[Prop.~2.1 \& 2.2]{PlRoe21}, 
\eqref{Assumption-a-bk:Bound-derivative-1}, \eqref{Assumption-Bound-derivative-1t-and-2t}
and Lemma~\ref{Lem:Lp-bound-SDE-sol} that
%
\begin{align*}
	&\Erw \bigg( \sup_{0 \leq n \leq N} \bigg\| \sum_{l=0}^{n-1} 
	\sum_{k=1}^m \int_{t_l}^{t_{l+1}} \sum_{r=1}^d \frac{\partial}{\partial x_r}
	b^k(t_l,X_{t_l})
	\nonumber \\
	&\quad \times
	\sum_{k_2=1}^m \int_{t_l}^s \int_0^1 \frac{\partial}{\partial t} 
	b^{r,k_2}(t_l + v (s-t_l),X_u) \, (s-t_l) \, \mathrm{d}v \, \mathrm{d}W_u^{k_2}
	\, \mathrm{d}W_s^k \bigg\|^p \bigg)
	\nonumber \\
	&\leq \Big( \frac{p}{\sqrt{p-1}} \Big)^p \bigg( \sum_{l=0}^{N-1} \bigg[
	\Erw \bigg( \bigg\| \sum_{k=1}^m \int_{t_l}^{t_{l+1}} \sum_{r=1}^d 
	\frac{\partial}{\partial x_r} b^k(t_l,X_{t_l})
	\nonumber \\
	&\quad \times
	\sum_{k_2=1}^m \int_{t_l}^s \int_0^1 \frac{\partial}{\partial t} 
	b^{r,k_2}(t_l + v (s-t_l),X_u) \, (s-t_l) \, \mathrm{d}v \, \mathrm{d}W_u^{k_2}
	\, \mathrm{d}W_s^k \bigg\|^p \bigg) \bigg]^{\frac{2}{p}} \bigg)^{\frac{p}{2}}
	\nonumber \\
	&\leq \Big( \frac{p}{\sqrt{p-1}} \Big)^p \bigg( \sum_{l=0}^{N-1} (p-1)
	\int_{t_l}^{t_{l+1}} \bigg[ \Erw \bigg( \bigg\| \sum_{k=1}^m \bigg\| 
	\sum_{k_2=1}^m \int_{t_l}^s \int_0^1
	\sum_{r=1}^d \frac{\partial}{\partial x_r} b^k(t_l,X_{t_l}) 
	\nonumber \\
	&\quad \times
	\frac{\partial}{\partial t} b^{r,k_2}(t_l + v (s-t_l),X_u) \, (s-t_l) 
	\, \mathrm{d}v \, \mathrm{d}W_u^{k_2} \bigg\|^2 \bigg\|^{\frac{p}{2}} \bigg)
	\bigg]^{\frac{2}{p}} \mathrm{d}s \bigg)^{\frac{p}{2}}
	\nonumber \\
	&\leq \Big( \frac{p}{\sqrt{p-1}} \Big)^p (p-1)^{\frac{p}{2}} \bigg( 
	\sum_{l=0}^{N-1} \int_{t_l}^{t_{l+1}} \sum_{k=1}^m \bigg[ \Erw \bigg( \bigg\| 
	\sum_{k_2=1}^m \int_{t_l}^s \int_0^1
	\sum_{r=1}^d \frac{\partial}{\partial x_r} b^k(t_l,X_{t_l}) 
	\nonumber \\
	&\quad \times
	\frac{\partial}{\partial t} b^{r,k_2}(t_l + v (s-t_l),X_u) \, (s-t_l) 
	\, \mathrm{d}v \, \mathrm{d}W_u^{k_2} \bigg\|^p \bigg) \bigg]^{\frac{2}{p}}
	\mathrm{d}s \bigg)^{\frac{p}{2}}
	\nonumber \\
	&\leq \Big( \frac{p}{\sqrt{p-1}} \Big)^p (p-1)^{\frac{p}{2}} \bigg(
	\sum_{l=0}^{N-1} \int_{t_l}^{t_{l+1}} \sum_{k=1}^m (p-1) \int_{t_l}^s
	\bigg[ \Erw \bigg( \bigg\| \sum_{k_2=1}^m \bigg\| \int_0^1
	\sum_{r=1}^d \frac{\partial}{\partial x_r} b^k(t_l,X_{t_l}) 
	\nonumber \\
	&\quad \times
	\frac{\partial}{\partial t} b^{r,k_2}(t_l + v (s-t_l),X_u) \, (s-t_l) 
	\, \mathrm{d}v \bigg\|^2 \bigg\|^{\frac{p}{2}}
	\bigg) \bigg]^{\frac{2}{p}} \mathrm{d}u \, \mathrm{d}s \bigg)^{\frac{p}{2}}
	\nonumber \\
%
	&\leq \Big( \frac{p}{\sqrt{p-1}} \Big)^p (p-1)^p
	\bigg( \sum_{l=0}^{N-1} \int_{t_l}^{t_{l+1}} \sum_{k=1}^m \int_{t_l}^s
	\sum_{k_2=1}^m \bigg[ \Erw \bigg( \bigg( \int_0^1 \sum_{r=1}^d 
	\Big\| \frac{\partial}{\partial x_r} b^k(t_l,X_{t_l}) \Big\|
	\nonumber \\
	&\quad \times
	\Big| \frac{\partial}{\partial t} b^{r,k_2}(t_l + v (s-t_l),X_u) \Big| 
	\, |s-t_l| \, \mathrm{d}v \bigg)^p \bigg) \bigg]^{\frac{2}{p}} 
	\, \mathrm{d}u \, \mathrm{d}s \bigg)^{\frac{p}{2}}
	\nonumber \\
	&\leq \Big( \frac{p}{\sqrt{p-1}} \Big)^p (p-1)^p
	\bigg( \sum_{l=0}^{N-1} \sum_{k_2,k=1}^m \int_{t_l}^{t_{l+1}} \int_{t_l}^s
	[ \Erw ( ( d \, \cc^2 (1+\| X_u \|) \, h )^p ) ]^{\frac{2}{p}}
	\, \mathrm{d}u \, \mathrm{d}s \bigg)^{\frac{p}{2}}
	\nonumber \\
	&\leq \Big( \frac{p}{\sqrt{p-1}} \Big)^p (p-1)^p d^p \, \cc^{2p} \, m^p
	\bigg( \sum_{l=0}^{N-1} \int_{t_l}^{t_{l+1}} \int_{t_l}^s
	[ 2^{p-1} (1 + \Erw(\| X_u \|^p) ) ]^{\frac{2}{p}} 
	\, \mathrm{d}u \, \mathrm{d}s \bigg)^{\frac{p}{2}} \, h^p
	\nonumber \\
	&\leq \Big( \frac{p}{\sqrt{p-1}} \Big)^p (p-1)^p d^p \, \cc^{2p} 
	\, m^p \, 2^{p-1} \, (T-t_0)^{\frac{p}{2}} \, h^{\frac{p}{2}}
	\, (1 + \ccp (1+ \Erw( \|X_{t_0} \|^p)) ) \, h^p \, .
\end{align*}
%
%
%
Considering~\eqref{Proof:MainThm:Teil-A2-1d}, we get with Burkholder's inequality,
see, e.g., \cite{Burk88} or \cite[Prop.~2.1 \& 2.2]{PlRoe21}, 
\eqref{Assumption-a-bk:Bound-derivative-1}, \eqref{Assumption-Bound-derivative-1t-and-2t},
\eqref{Assumption-a-bk:lin-growth}
and Lemma~\ref{Lem:Lp-bound-SDE-sol} that
%
\begin{align*}
	&\Erw \bigg( \sup_{0 \leq n \leq N} \bigg\| \sum_{l=0}^{n-1}
	\sum_{k=1}^m \sum_{i=1}^s \beta_i^{(1)} \, \Ii_{(k),l} \sum_{r=1}^d
	\frac{\partial}{\partial x_r} b^k(t_l,X_{t_l}) \sum_{j=1}^s A_{i,j}^{(1)} \,
	a^r(t_l,X_{t_l}) \, h \bigg\|^p \bigg)
	\nonumber \\
	&\leq \Big( \frac{p}{\sqrt{p-1}} \Big)^p \, h^p \bigg( \sum_{l=0}^{N-1}
	\int_{t_l}^{t_{l+1}} 
	\nonumber \\
	&\quad \times
	\bigg[ \Erw \bigg( \bigg| \sum_{k=1}^m \bigg\|
	\sum_{i=1}^s \beta_i^{(1)} \sum_{j=1}^s A_{i,j}^{(1)} \, \sum_{r=1}^d
	\frac{\partial}{\partial x_r} b^k(t_l,X_{t_l}) \, a^r(t_l,X_{t_l})
	\bigg\|^2 \bigg|^{\frac{p}{2}} \bigg) \bigg]^{\frac{2}{p}} 
	\, \mathrm{d}t \bigg)^{\frac{p}{2}}
	\nonumber \\
	&\leq \Big( \frac{p}{\sqrt{p-1}} \Big)^p \, h^p \bigg( \sum_{l=0}^{N-1}
	\int_{t_l}^{t_{l+1}} \mathrm{d}t \, \sum_{k=1}^m \bigg[ 
	\bigg| \sum_{i=1}^s \beta_i^{(1)} \sum_{j=1}^s A_{i,j}^{(1)} \bigg|^p
	\nonumber \\
	&\quad \times
	\Erw \bigg( \bigg( \sum_{r=1}^d \Big\| \frac{\partial}{\partial x_r} 
	b^k(t_l,X_{t_l}) \Big\| \, |a^r(t_l,X_{t_l})| \bigg)^p \bigg) 
	\bigg]^{\frac{2}{p}} \bigg)^{\frac{p}{2}}
	\nonumber \\
%
	&\leq \Big( \frac{p}{\sqrt{p-1}} \Big)^p \, h^p 
	\, \bigg| \sum_{i=1}^s \beta_i^{(1)} \sum_{j=1}^s A_{i,j}^{(1)} \bigg|^p
	\bigg( \sum_{l=0}^{N-1}	h \, m \, [ d^p \, \cc^p \, 
	\Erw ( \cc^p 2^{p-1} (1+ \| X_{t_l} \|^p ) ) ]^{\frac{2}{p}} 
	\bigg)^{\frac{p}{2}}
	\nonumber \\
	&\leq \Big( \frac{p}{\sqrt{p-1}} \Big)^p (s^{2} \, \czwei^{2} \, d \, \cc^2)^p
	\, m^{\frac{p}{2}} \, (T-t_0)^{\frac{p}{2}} 2^{p-1} 
	(1+ \ccp (1+ \Erw( \|X_{t_0} \|^p)) ) \, h^p \, .
\end{align*}
%
%
%
Next, we consider the term \eqref{Proof:MainThm:Teil-A2-1e}.
%
%
With~\eqref{Assumption-a-bk:Bound-derivative-1},
Lemma~\ref{Lem:Ij-Iij-Moment-estimate}, Lemma~\ref{Lem:H0-Xtl-estimate} 
and Lemma~\ref{Lem:Lp-bound-SDE-sol} we get for sufficiently small 
$0< h <\frac{1}{s \cc \czwei}$ that
%
\begin{align}
	&\Erw \bigg( \sup_{0 \leq n \leq N} \bigg\| \sum_{l=0}^{n-1}
	\sum_{k=1}^m \sum_{i,j=1}^s (\beta_i^{(1)} \, \Ii_{(k),l} + \beta_i^{(2)} )
	A_{i,j}^{(1)} \, h \sum_{r=1}^d \frac{\partial}{\partial x_r} b^k(t_l,X_{t_l})
	\nonumber \\
	&\quad \times
	\int_0^1 \sum_{q=1}^d \frac{\partial}{\partial x_q}
	a^r(t_l,X_{t_l} + u(H_{j,l}^{(0),X_{t_l}}-X_{t_l})) \, 
	( {H_{j,l}^{(0),X_{t_l}} }^q-X_{t_l}^q) \, \mathrm{d}u \bigg\|^p \bigg)
	\nonumber \\
	&\leq \Erw \bigg( \sup_{0 \leq n \leq N} \bigg( \sum_{l=0}^{n-1}
	\sum_{k=1}^m \sum_{i,j=1}^s ( |\beta_i^{(1)}| \, |\Ii_{(k),l}| + |\beta_i^{(2)}| )
	\, |A_{i,j}^{(1)}| \, h
	\sum_{r=1}^d \Big\| \frac{\partial}{\partial x_r} b^k(t_l,X_{t_l}) \Big\|
	\nonumber \\
	&\quad \times
	\int_0^1 \sum_{q=1}^d \Big| \frac{\partial}{\partial x_q}
	a^r(t_l,X_{t_l} + u(H_{j,l}^{(0),X_{t_l}}-X_{t_l})) \Big| \, 
	\big| {H_{j,l}^{(0),X_{t_l}} }^q-X_{t_l}^q \big| \, \mathrm{d}u \bigg)^p \bigg)
	\nonumber \\
%
	&\leq \Erw \bigg( \bigg( \sum_{l=0}^{N-1} \sum_{k=1}^m s^2 \,
	(| \Ii_{(k),l} | + 1) \czwei^2 \, h \, \cc^2 \, d^2 \, 
	\max_{1 \leq j \leq s} \big\| H_{j,l}^{(0),X_{t_l}} -X_{t_l} \big\| \bigg)^p \bigg)
	\nonumber \\
	&\leq N^{p-1} \sum_{l=0}^{N-1}  h^p \, (s^2 \, \czwei^2 \, \cc^2 \, d^2)^p \, 
	m^{p-1} \sum_{k=1}^m \, 2^{p-1} ( \Erw ( | \Ii_{(k),l} |^p ) + 1) \, 
	\Erw \big( \max_{1 \leq j \leq s} \big\| H_{j,l}^{(0),X_{t_l}} -X_{t_l} \big\|^p \big)
	\nonumber \\
	&\leq N^{p-1} \sum_{l=0}^{N-1}  h^p  \, (s^2 \, \czwei^2 \, \cc^2 \, d^2 \, m)^p \,
	2^{p-1} \, ( (p-1)^p \, h^{\frac{p}{2}} + 1) \, 
	\cMHdetInc (1+ \Erw( \|X_{t_l} \|^p)) \, h^p
	\nonumber \\
%
	&\leq (T-t_0)^p \, (s^2 \, \czwei^2 \, \cc^2 \, d^2 \, m)^p \, 
	2^{p-1} \, ( (p-1)^p \, h^{\frac{p}{2}} + 1 ) \, 
	\cMHdetInc (1+ \ccp (1+ \Erw( \|X_{t_0} \|^p)) ) \, h^p \, .
	\nonumber
\end{align}
%
%
%
Now, we estimate term \eqref{Proof:MainThm:Teil-A2-1f}. Then, we get 
with~\eqref{Assumption-a-bk:Bound-derivative-1},
\eqref{Assumption-Bound-derivative-1t-and-2t},
\eqref{Estimate-H0-01}, Lemma~\ref{Lem:Ij-Iij-Moment-estimate} and
Lemma~\ref{Lem:Lp-bound-SDE-sol} that for $0< h <\frac{1}{s \cc \czwei}$
%
\begin{align}
	&\Erw \bigg( \sup_{0 \leq n \leq N} \bigg\| \sum_{l=0}^{n-1}
	\sum_{k=1}^m \sum_{i=1}^s (\beta_i^{(1)} \, \Ii_{(k),l} + \beta_i^{(2)} )
	\sum_{r=1}^d \frac{\partial}{\partial x_r} b^k(t_l,X_{t_l})
	\nonumber \\
	&\quad \times \sum_{j=1}^s A_{i,j}^{(1)} \, h \int_0^1 \frac{\partial}{\partial t}
	a^r(t_l + u \, c_j^{(0)} \, h, H_{j,l}^{(0),X_{t_l}}) \, c_j^{(0)} \, h \, \mathrm{d}u
	\bigg\|^p \bigg)
	\nonumber \\
	&\leq \Erw \bigg( \sup_{0 \leq n \leq N} \bigg( \sum_{l=0}^{n-1}
	\sum_{k=1}^m \sum_{i,j=1}^s ( |\beta_i^{(1)}| \, |\Ii_{(k),l}| + |\beta_i^{(2)}| ) \,
	|A_{i,j}^{(1)}| \, h^2  \, |c_j^{(0)}| \,
	\sum_{r=1}^d \Big\| \frac{\partial}{\partial x_r} b^k(t_l,X_{t_l}) \Big\|
	\nonumber \\
	&\quad \times
	\int_0^1 \Big| \frac{\partial}{\partial t} a^r(t_l + u \, c_j^{(0)} \, h, H_{j,l}^{(0),X_{t_l}})
	\Big| \, \mathrm{d}u \bigg)^p \bigg)
	\nonumber \\
	&\leq \Erw \bigg( \bigg( \sum_{l=0}^{N-1} h^2 \sum_{k=1}^m \sum_{i,j=1}^s 
	\czwei^3 (|\Ii_{(k),l}| + 1) \, d \, \cc^2 \, (1 + \| H_{j,l}^{(0),X_{t_l}} \| ) \bigg)^p \bigg)
	\nonumber \\
%
	&\leq N^{p-1} \sum_{l=0}^{N-1} h^{2p} \, \czwei^{3p} \, m^{p-1} \sum_{k=1}^m 
	s^p \, s^{p-1} \sum_{j=1}^s d^p \, \cc^{2p} \, \Erw \Big( (|\Ii_{(k),l}| + 1)^p
	\, \Big(1 + \frac{\| X_{t_l} \| + s \, \czwei \, \cc \, h}{1-s \, \czwei \, \cc \, h} \Big)^p \Big)
	\nonumber \\
	&\leq (T-t_0)^p \, (\czwei^3 \, s^2 \, d \, \cc^2)^p \, m^{p-1} \sum_{k=1}^m
	2^{p-1} \, ( \Erw( |\Ii_{(k),l}|^p ) + 1)
	\, 2^{p-1} \, \frac{ 1+ \Erw ( \| X_{t_l} \|^p ) }{(1-s \, \czwei \, \cc \, h)^p} \, h^p
	\nonumber \\
	&\leq (T-t_0)^p \, (\czwei^3 \, s^2 \, d \, \cc^2 \, m)^p \, 4^{p-1} \, 
	( (p-1)^{p} \, h^{\frac{p}{2}} + 1) \, 
	\frac{ 1+ \ccp (1+ \Erw( \|X_{t_0} \|^p)) }{(1-s \, \czwei \, \cc \, h)^p} \, h^p \, .
	\nonumber
\end{align}
%
%
%
We consider \eqref{Proof:MainThm:Teil-A2-1i}. With
$\Ii_{(k),l} \, \Ii_{(k_2,k),l} = \Ii_{(k,k_2,k),l} + 2 \Ii_{(k_2,k,k),l}
+ \Ii_{(k_2,0),l}$ we can estimate \eqref{Proof:MainThm:Teil-A2-1i} by the three
terms
%
\begin{align}
		&\Erw \bigg( \sup_{0 \leq n \leq N} \bigg\| \sum_{l=0}^{n-1}
	\sum_{k=1}^m \sum_{i=1}^s \beta_i^{(1)} \, \Ii_{(k),l} \sum_{r=1}^d 
	\frac{\partial}{\partial x_r} b^k(t_l,X_{t_l})
	\sum_{j=1}^{i-1} \sum_{k_2=1}^m B_{i,j}^{(1)} \, b^{r,k_2}(t_l,X_{t_l}) 
	\, \Ii_{(k_2,k),l} \bigg\|^p \bigg)
	\nonumber \\
	&\leq 3^{p-1} \Erw \bigg( \sup_{0 \leq n \leq N} \bigg\| \sum_{l=0}^{n-1}
	\sum_{k=1}^m \sum_{i=1}^s \beta_i^{(1)} \sum_{r=1}^d 
	\frac{\partial}{\partial x_r} b^k(t_l,X_{t_l})
	\sum_{j=1}^{i-1} \sum_{k_2=1}^m B_{i,j}^{(1)} \, b^{r,k_2}(t_l,X_{t_l}) 
	\, \Ii_{(k,k_2,k),l} \bigg\|^p \bigg)
	\label{Proof:MainThm:Teil-A2-1i-1} \\
	&\quad + 3^{p-1} \Erw \bigg( \sup_{0 \leq n \leq N} \bigg\| \sum_{l=0}^{n-1}
	\sum_{k=1}^m \sum_{i=1}^s \beta_i^{(1)} \sum_{r=1}^d 
	\frac{\partial}{\partial x_r} b^k(t_l,X_{t_l})
	\sum_{j=1}^{i-1} \sum_{k_2=1}^m B_{i,j}^{(1)} \, b^{r,k_2}(t_l,X_{t_l}) 
	\, 2 \Ii_{(k_2,k,k),l} \bigg\|^p \bigg)
	\label{Proof:MainThm:Teil-A2-1i-2} \\
	&\quad + 3^{p-1} \Erw \bigg( \sup_{0 \leq n \leq N} \bigg\| \sum_{l=0}^{n-1}
	\sum_{k=1}^m \sum_{i=1}^s \beta_i^{(1)} \sum_{r=1}^d 
	\frac{\partial}{\partial x_r} b^k(t_l,X_{t_l})
	\sum_{j=1}^{i-1} \sum_{k_2=1}^m B_{i,j}^{(1)} \, b^{r,k_2}(t_l,X_{t_l}) 
	\, \Ii_{(k_2,0),l} \bigg\|^p \bigg) \, .
	\label{Proof:MainThm:Teil-A2-1i-3}
\end{align}
%
%
%
Firstly, we estimate term \eqref{Proof:MainThm:Teil-A2-1i-1}. Here, we apply
Burkholder's inequality, see, e.g., \cite{Burk88} or 
\cite[Prop.~2.1 \& 2.2]{PlRoe21}, \eqref{Assumption-a-bk:lin-growth},
\eqref{Assumption-a-bk:Bound-derivative-1},
Lemma~\ref{Lem:Ij-Iij-Moment-estimate} and Lemma~\ref{Lem:Lp-bound-SDE-sol}
and we get
%
\begin{align}
	&\Erw \bigg( \sup_{0 \leq n \leq N} \bigg\| \sum_{l=0}^{n-1}
	\sum_{k=1}^m \sum_{i=1}^s \beta_i^{(1)} \sum_{r=1}^d 
	\frac{\partial}{\partial x_r} b^k(t_l,X_{t_l})
	\sum_{j=1}^{i-1} \sum_{k_2=1}^m B_{i,j}^{(1)} \, b^{r,k_2}(t_l,X_{t_l}) 
	\, \Ii_{(k,k_2,k),l} \bigg\|^p \bigg)
	\nonumber \\
%
	&\leq \Big( \frac{p}{\sqrt{p-1}} \Big)^p \bigg( \sum_{l=0}^{N-1} \bigg[ 
	\Erw \bigg( \bigg\| \sum_{k, k_2=1}^m
	\sum_{i=1}^s \beta_i^{(1)} \sum_{r=1}^d 
	\frac{\partial}{\partial x_r} b^k(t_l,X_{t_l}) 
	\nonumber \\
	&\quad \times
	\sum_{j=1}^{i-1} B_{i,j}^{(1)} \, b^{r,k_2}(t_l,X_{t_l}) \, \Ii_{(k,k_2,k),l} \bigg\|^p
	\bigg) \bigg]^{\frac{2}{p}} \bigg)^{\frac{p}{2}}
	\nonumber \\
	%
	&\leq \Big( \frac{p}{\sqrt{p-1}} \Big)^p \bigg( \sum_{l=0}^{N-1} (p-1)
	\int_{t_l}^{t_{l+1}} \bigg[ \Erw \bigg( \bigg\| \sum_{k=1}^m \bigg\| \sum_{k_2=1}^m
	\int_{t_l}^u \int_{t_l}^v \sum_{i=1}^s \beta_i^{(1)} \sum_{r=1}^d 
	\frac{\partial}{\partial x_r} b^k(t_l,X_{t_l}) 
	\nonumber \\
	&\quad \times
	\sum_{j=1}^{i-1} B_{i,j}^{(1)} \, b^{r,k_2}(t_l,X_{t_l}) \,
	\mathrm{d}W_w^k \, \mathrm{d}W_v^{k_2} \bigg\|^2 \bigg\|^{\frac{p}{2}}
	\bigg) \bigg]^{\frac{2}{p}} \mathrm{d}u \bigg)^{\frac{p}{2}}
	\nonumber \\
	%
	&\leq \Big( \frac{p}{\sqrt{p-1}} \Big)^p \, (p-1)^{\frac{p}{2}} \, m^{\frac{p}{2}-1} 
	\bigg( \sum_{l=0}^{N-1} \int_{t_l}^{t_{l+1}} \bigg[ \sum_{k=1}^m (p-1)^{\frac{p}{2}}
	\bigg( \int_{t_l}^u \bigg[ \Erw \bigg( \bigg\| \sum_{k_2=1}^m \bigg\| \int_{t_l}^v
	\sum_{i=1}^s \beta_i^{(1)}
	\nonumber \\
	&\quad \times
	\sum_{r=1}^d \frac{\partial}{\partial x_r} b^k(t_l,X_{t_l}) 
	\sum_{j=1}^{i-1} B_{i,j}^{(1)} \, b^{r,k_2}(t_l,X_{t_l}) \, \mathrm{d}W_w^k \bigg\|^2
	\bigg\|^{\frac{p}{2}} \bigg) \bigg]^{\frac{2}{p}} \, \mathrm{d}v \bigg)^{\frac{p}{2}}
	\bigg]^{\frac{2}{p}} \, \mathrm{d}u \bigg)^{\frac{p}{2}}
	\nonumber \\
	%
	%
	&\leq p^p \, m^{\frac{p}{2}-1} \, (p-1)^{\frac{p}{2}}
	\bigg( \sum_{l=0}^{N-1} \int_{t_l}^{t_{l+1}} \bigg[ \sum_{k=1}^m
	\bigg( \int_{t_l}^u \bigg[ m^{\frac{p}{2}-1} \sum_{k_2=1}^m 
	\Erw \bigg( \bigg\|
	\sum_{i=1}^s \beta_i^{(1)} \sum_{r=1}^d \frac{\partial}{\partial x_r} b^k(t_l,X_{t_l})
	\nonumber \\
	&\quad \times
	\sum_{j=1}^{i-1} B_{i,j}^{(1)} \, b^{r,k_2}(t_l,X_{t_l}) \bigg\|^p \bigg)
	\, \Erw \bigg( \bigg| \int_{t_l}^v \mathrm{d}W_w^k \bigg|^p \bigg)
	\bigg]^{\frac{2}{p}} \, \mathrm{d}v \bigg)^{\frac{p}{2}}
	\bigg]^{\frac{2}{p}} \, \mathrm{d}u \bigg)^{\frac{p}{2}}
	\nonumber \\
%
	&\leq p^p \, m^{p-2} \, (p-1)^{\frac{p}{2}}
	\bigg( \sum_{l=0}^{N-1} \int_{t_l}^{t_{l+1}} \bigg[ \sum_{k=1}^m
	\bigg( \int_{t_l}^u \bigg[ \sum_{k_2=1}^m \Erw \bigg( \bigg(
	\sum_{i=1}^s |\beta_i^{(1)}| \sum_{r=1}^d 
	\Big\| \frac{\partial}{\partial x_r} b^k(t_l,X_{t_l}) \Big\|
	\nonumber \\
	&\quad \times
	\sum_{j=1}^{i-1} |B_{i,j}^{(1)}| \, |b^{r,k_2}(t_l,X_{t_l})| \bigg)^p \bigg)
	\, (p-1)^p \, (v-t_l)^{\frac{p}{2}}
	\bigg]^{\frac{2}{p}} \, \mathrm{d}v \bigg)^{\frac{p}{2}}
	\bigg]^{\frac{2}{p}} \, \mathrm{d}u \bigg)^{\frac{p}{2}}
	\nonumber \\
	&\leq p^p \, (p-1)^{\frac{3p}{2}} \, (m \, \czwei^{2} \, \cc \, s^{2} \, d)^p
	\, \bigg( \sum_{l=0}^{N-1} \int_{t_l}^{t_{l+1}} 
	\int_{t_l}^u [ \Erw ( \cc^p ( 1 + \| X_{t_l} \| )^p ) ]^{\frac{2}{p}} 
	\, (v-t_l) \, \mathrm{d}v \, \mathrm{d}u \bigg)^{\frac{p}{2}}
	\nonumber \\
%
	&\leq (p-1)^{\frac{3p}{2}} \, (p \, m \, \czwei^{2} \, \cc^2 \, s^{2} \, d)^p
	\, (T-t_0)^{\frac{p}{2}} \, 2^{p-1} \, ( 1 + \ccp (1+ \Erw( \|X_{t_0} \|^p)) )
	\, h^p \, .
	\nonumber
\end{align}
%
%
%
Next, we estimate \eqref{Proof:MainThm:Teil-A2-1i-2} in a similar way
as term \eqref{Proof:MainThm:Teil-A2-1i-1} which results in
%
\begin{align}
	&\Erw \bigg( \sup_{0 \leq n \leq N} \bigg\| \sum_{l=0}^{n-1}
	\sum_{k=1}^m \sum_{i=1}^s \beta_i^{(1)} \sum_{r=1}^d 
	\frac{\partial}{\partial x_r} b^k(t_l,X_{t_l})
	\sum_{j=1}^{i-1} \sum_{k_2=1}^m B_{i,j}^{(1)} \, b^{r,k_2}(t_l,X_{t_l}) 
	\, 2 \Ii_{(k_2,k,k),l} \bigg\|^p \bigg)
	\nonumber \\
	%
	&\leq 2^p \, \Big( \frac{p}{\sqrt{p-1}} \Big)^p \bigg( \sum_{l=0}^{N-1} \bigg[
	\Erw \bigg( \bigg\| \sum_{k,k_2=1}^m \sum_{i=1}^s \beta_i^{(1)} \sum_{r=1}^d 
	\frac{\partial}{\partial x_r} b^k(t_l,X_{t_l})
	\nonumber \\
	&\quad \times
	\sum_{j=1}^{i-1} B_{i,j}^{(1)} \, b^{r,k_2}(t_l,X_{t_l}) 
	\, \Ii_{(k_2,k,k),l} \bigg\|^p \bigg) \bigg]^{\frac{2}{p}} \bigg)^{\frac{p}{2}}
	\nonumber \\
	%
	&\leq 2^p \, \Big( \frac{p}{\sqrt{p-1}} \Big)^p \bigg( \sum_{l=0}^{N-1} (p-1)
	\int_{t_l}^{t_{l+1}} \bigg[ \Erw \bigg( \bigg\| \sum_{k=1}^m \bigg\| \sum_{k_2=1}^m
	\int_{t_l}^u \int_{t_l}^v
	\sum_{i=1}^s \beta_i^{(1)} \sum_{r=1}^d \frac{\partial}{\partial x_r} b^k(t_l,X_{t_l})
	\nonumber \\
	&\quad \times
	\sum_{j=1}^{i-1} B_{i,j}^{(1)} \, b^{r,k_2}(t_l,X_{t_l}) \, \mathrm{d}W_w^{k_2}
	\, \mathrm{d}W_v^k \bigg\|^2 
	\bigg\|^{\frac{p}{2}} \bigg) \bigg]^{\frac{2}{p}} \mathrm{d}u \bigg)^{\frac{p}{2}}
	\nonumber \\
	%
	&\leq 2^p \, p^p \, m^{\frac{p}{2}-1}
	\bigg( \sum_{l=0}^{N-1} \int_{t_l}^{t_{l+1}} \bigg[ \sum_{k=1}^m (p-1)^{\frac{p}{2}}
	\bigg( \int_{t_l}^u \bigg[ \Erw \bigg( \bigg\| \sum_{k_2=1}^m \int_{t_l}^v
	\sum_{i=1}^s \beta_i^{(1)} \sum_{r=1}^d \frac{\partial}{\partial x_r} b^k(t_l,X_{t_l})
	\nonumber \\
	&\quad \times
	\sum_{j=1}^{i-1} B_{i,j}^{(1)} \, b^{r,k_2}(t_l,X_{t_l}) \, \mathrm{d}W_w^{k_2}
	\bigg\|^p \bigg) \bigg]^{\frac{2}{p}} \, \mathrm{d}v \bigg)^{\frac{p}{2}}
	\bigg]^{\frac{2}{p}} \, \mathrm{d}u \bigg)^{\frac{p}{2}}
	\nonumber \\
%
	%
%
	&\leq 2^p \, p^p \, m^{\frac{p}{2}-1} \, (p-1)^{\frac{p}{2}} \,
	\bigg( \sum_{l=0}^{N-1} \int_{t_l}^{t_{l+1}} \bigg[ \sum_{k=1}^m 
	\bigg( \int_{t_l}^u \bigg[ m^{p-1} \sum_{k_2=1}^m \Erw \bigg( \bigg( 
	\sum_{i=1}^s |\beta_i^{(1)}| \sum_{r=1}^d 
	\Big\| \frac{\partial}{\partial x_r} b^k(t_l,X_{t_l}) \Big\|
	\nonumber \\
	&\quad \times
	\sum_{j=1}^{i-1} |B_{i,j}^{(1)}| \, | b^{r,k_2}(t_l,X_{t_l}) | \bigg)^p \bigg) \, 
	\Erw \bigg( \bigg| \int_{t_l}^v \mathrm{d}W_w^{k_2} \bigg|^p \bigg)
	\bigg]^{\frac{2}{p}} \, \mathrm{d}v \bigg)^{\frac{p}{2}}
	\bigg]^{\frac{2}{p}} \, \mathrm{d}u \bigg)^{\frac{p}{2}}
	\nonumber \\
	&\leq 2^p \, p^p \, m^{\frac{3p}{2}-2} \, (p-1)^{\frac{p}{2}} \,
	\bigg( \sum_{l=0}^{N-1} \int_{t_l}^{t_{l+1}} \bigg[ \sum_{k=1}^m 
	\bigg( \int_{t_l}^u \bigg[ \sum_{k_2=1}^m \Erw ( s^{2p} \, \czwei^{2p} \,
	d^p \, \cc^{2p} (1 + \| X_{t_l} \| )^p )
	\nonumber \\
	&\quad \times
	(p-1)^p (v-t_l)^{\frac{p}{2}}
	\bigg]^{\frac{2}{p}} \, \mathrm{d}v \bigg)^{\frac{p}{2}}
	\bigg]^{\frac{2}{p}} \, \mathrm{d}u \bigg)^{\frac{p}{2}}
	\nonumber \\
%
	&\leq 2^p \, (p-1)^{\frac{3p}{2}} \, (p \, m^{\frac{3}{2}} \, s^2 \, \czwei^2 \, d 
	\, \cc^2 )^p \, 2^{p-1} \, (T-t_0)^{\frac{p}{2}}
	(1+ \ccp (1+ \Erw( \|X_{t_0} \|^p)) ) \, h^p \, .
	\nonumber
\end{align}
%
%
%
The third term \eqref{Proof:MainThm:Teil-A2-1i-3} has to be estimated 
slightly different. With Burkholder's inequality, see, e.g., \cite{Burk88} or 
\cite[Prop.~2.1 \& 2.2]{PlRoe21}, H\"older's inequality,
\eqref{Assumption-a-bk:Bound-derivative-1}, \eqref{Assumption-a-bk:lin-growth},
Lemma~\ref{Lem:Ij-Iij-Moment-estimate} and
Lemma~\ref{Lem:Lp-bound-SDE-sol} we get
%
\begin{align}
	&\Erw \bigg( \sup_{0 \leq n \leq N} \bigg\| \sum_{l=0}^{n-1}
	\sum_{k=1}^m \sum_{i=1}^s \beta_i^{(1)} \sum_{r=1}^d 
	\frac{\partial}{\partial x_r} b^k(t_l,X_{t_l})
	\sum_{j=1}^{i-1} \sum_{k_2=1}^m B_{i,j}^{(1)} \, b^{r,k_2}(t_l,X_{t_l}) 
	\, \Ii_{(k_2,0),l} \bigg\|^p \bigg)
	\nonumber \\
%
	&\leq \Big( \frac{p}{\sqrt{p-1}} \Big)^p \bigg( \sum_{l=0}^{N-1} 
	\bigg[ \Erw \bigg( \bigg\| \sum_{k,k_2=1}^m \sum_{i=1}^s \beta_i^{(1)} 
	\sum_{r=1}^d \frac{\partial}{\partial x_r} b^k(t_l,X_{t_l})
	\nonumber \\
	&\quad \times
	\sum_{j=1}^{i-1} B_{i,j}^{(1)} \, b^{r,k_2}(t_l,X_{t_l})
	\, \Ii_{(k_2,0),l} \bigg\|^p \bigg) \bigg]^{\frac{2}{p}} \bigg)^{\frac{p}{2}}
	\nonumber \\
%
	&\leq \Big( \frac{p}{\sqrt{p-1}} \Big)^p \bigg( \sum_{l=0}^{N-1}
	\bigg( \sum_{k,k_2=1}^m \bigg[ \Erw \bigg( \bigg( \sum_{i=1}^s |\beta_i^{(1)}| 
	\sum_{r=1}^d \Big\| \frac{\partial}{\partial x_r} b^k(t_l,X_{t_l}) \Big\|
	\sum_{j=1}^{i-1} |B_{i,j}^{(1)}| \, | b^{r,k_2}(t_l,X_{t_l}) | \bigg)^p \bigg)
	\nonumber \\
	&\quad \times
	\Erw ( |\Ii_{(k_2,0),l}|^p ) \bigg]^{\frac{1}{p}} \bigg)^2 \bigg)^{\frac{p}{2}}
	\nonumber \\
	&\leq \Big( \frac{p}{\sqrt{p-1}} \Big)^p \, \czwei^{2p} \, \cc^{2p} \, s^{2p} \, d^p \,
	m^{2p} \, \bigg( \sum_{l=0}^{N-1} \bigg[ \Erw ( (1+\| X_{t_l} \| )^p )
	\nonumber \\
	&\quad \times
	\Erw \bigg( \int_{t_l}^{t_{l+1}} \bigg| \int_{t_l}^u \mathrm{d}W_v^{k_2} \bigg|^p
	\, \mathrm{d}u \bigg) \, \bigg( \int_{t_l}^{t_{l+1}} \mathrm{d}u \bigg)^{p-1}
	\bigg]^{\frac{2}{p}} \bigg)^{\frac{p}{2}}
	\nonumber \\
	&\leq \Big( \frac{p}{\sqrt{p-1}} \Big)^p \, (\czwei^2 \, \cc^2 \, s^2 \, d \, m^2)^p \,
	2^{p-1} \, \bigg( \sum_{l=0}^{N-1} \bigg[  \int_{t_l}^{t_{l+1}} (p-1)^p 
	\, (u-t_l)^{\frac{p}{2}} \, \mathrm{d}u \bigg]^{\frac{2}{p}} \bigg)^{\frac{p}{2}}
	\nonumber \\
	&\quad \times
	(1+ \ccp (1+ \Erw( \|X_{t_0} \|^p)) ) \, h^{p-1}
	\nonumber \\
	&\leq (p \, \czwei^2 \, \cc^2 \, s^2 \, d \, m^2)^p \, (p-1)^{\frac{p}{2}} \, 2^{p-1} \,
	(T-t_0)^{\frac{p}{2}} \, (1+ \ccp (1+ \Erw( \|X_{t_0} \|^p)) ) \, h^p \, .
	\nonumber
\end{align}
%
%
%
Now, we proceed with term \eqref{Proof:MainThm:Teil-A2-1g}. Here, we calculate
with~\eqref{Assumption-a-bk:Bound-derivative-1},
Lemma~\ref{Lem:Ij-Iij-H-Z-estimate} and Lemma~\ref{Lem:Lp-bound-SDE-sol} that
%
\begin{align}
	&\Erw \bigg( \sup_{0 \leq n \leq N} \bigg\| \sum_{l=0}^{n-1}
	\sum_{k=1}^m \sum_{i=1}^s (\beta_i^{(1)} \, \Ii_{(k),l} + \beta_i^{(2)} )
	\sum_{r=1}^d \frac{\partial}{\partial x_r} b^k(t_l,X_{t_l})
	\sum_{j=1}^{i-1} \sum_{k_2=1}^m B_{i,j}^{(1)} 
	\nonumber \\
	&\quad \times \int_0^1
	\sum_{q=1}^d \frac{\partial}{\partial x_q} 
	b^{r,k_2}(t_l,X_{t_l}+ u(H_{j,l}^{(k_2),X_{t_l}}-X_{t_l})) \,
	( {H_{j,l}^{(k_2),X_{t_l}} }^q-X_{t_l}^q) \, \mathrm{d}u \, \Ii_{(k_2,k),l}
	\bigg\|^p \bigg)
	\nonumber \\
	&\leq \Erw \bigg( \sup_{0 \leq n \leq N} \bigg( \sum_{l=0}^{n-1}
	\sum_{k=1}^m \sum_{i=1}^s ( | \beta_i^{(1)} |  \, | \Ii_{(k),l} | + | \beta_i^{(2)} | )
	\sum_{r=1}^d \Big\| \frac{\partial}{\partial x_r} b^k(t_l,X_{t_l}) \Big\|
	\sum_{j=1}^{i-1} \sum_{k_2=1}^m | B_{i,j}^{(1)} |
	\nonumber \\
	&\quad \times \int_0^1
	\sum_{q=1}^d \Big| \frac{\partial}{\partial x_q} 
	b^{r,k_2}(t_l,X_{t_l}+ u(H_{j,l}^{(k_2),X_{t_l}}-X_{t_l})) \Big| \,
	| {H_{j,l}^{(k_2),X_{t_l}} }^q-X_{t_l}^q | \, \mathrm{d}u \, | \Ii_{(k_2,k),l} |
	\bigg)^p \bigg)
	\nonumber \\
	%
	&\leq \bigg( \sum_{l=0}^{N-1} \sum_{k,k_2=1}^m s^2 \, \czwei^2 \, d^2 \, \cc^2 \,
	\Big[ \Erw \Big( (| \Ii_{(k),l} | + 1)^p \, 
	\max_{1 \leq j \leq s} \| H_{j,l}^{(k_2),X_{t_l}} -X_{t_l} \|^p \, | \Ii_{(k_2,k),l} |^p 
	\Big) \Big]^{\frac{1}{p}} \bigg)^p
	\nonumber \\
	&\leq (s^2 \, \czwei^2 \, d^2 \, \cc^2)^p \,
	\bigg( \sum_{l=0}^{N-1} \sum_{k,k_2=1}^m \Big[ 2^{p-1} 
	\Erw \Big( | \Ii_{(k),l} |^p \, | \Ii_{(k_2,k),l} |^p \, 
	\max_{1 \leq j \leq s} \| H_{j,l}^{(k_2),X_{t_l}} -X_{t_l} \|^p \Big) 
	\nonumber \\
	&\quad + 2^{p-1} \Erw \Big( 
	\max_{1 \leq j \leq s} \| H_{j,l}^{(k_2),X_{t_l}} -X_{t_l} \|^p \, | \Ii_{(k_2,k),l} |^p
	\Big) \Big]^{\frac{1}{p}} \bigg)^p
	\nonumber \\
%
	&\leq (s^2 \, \czwei^2 \, d^2 \, \cc^2 \, m^2)^p \, 2^{p-1} \,
	\bigg( \sum_{l=0}^{N-1} \Big[
	\cMH (1+ \Erw( \| X_{t_l} \|^p) ) \, h^{\frac{5p}{2}} + 
	\cMH (1+ \Erw( \| X_{t_l} \|^p) ) \, h^{2p}
	\Big]^{\frac{1}{p}} \bigg)^p
	\nonumber \\
	&\leq (s^2 \, \czwei^2 \, d^2 \, \cc^2 \, m^2)^p \, 2^{p-1} \, (T-t_0)^p \,
	(h^{\frac{p}{2}} +1) \, \cMH (1+ \ccp (1+ \Erw( \|X_{t_0} \|^p)) ) \, h^p \, .
	\nonumber
\end{align}
%
%
%
Consider now term \eqref{Proof:MainThm:Teil-A2-1h}. Here, we get with
Lemma~\ref{Lem:Ij-Iij-Moment-estimate} and Lemma~\ref{Lem:Ij-Iij-Hk-estimate} 
that
%
\begin{align}
	&\Erw \bigg( \sup_{0 \leq n \leq N} \bigg\| \sum_{l=0}^{n-1}
	\sum_{k=1}^m \sum_{i=1}^s (\beta_i^{(1)} \, \Ii_{(k),l} + \beta_i^{(2)} )
	\sum_{r=1}^d \frac{\partial}{\partial x_r} b^k(t_l,X_{t_l})
	\nonumber \\
	&\quad \times \sum_{j=1}^{i-1} \sum_{k_2=1}^m B_{i,j}^{(1)} \int_0^1
	\frac{\partial}{\partial t} b^{r,k_2}(t_l +u \, c_j^{(1)} \, h, H_{j,l}^{(k_2),X_{t_l}})
	\, c_j^{(1)} \, h \, \mathrm{d}u \, \Ii_{(k_2,k),l} \bigg\|^p \bigg) 
	\nonumber \\
	&\leq \Erw \bigg( \sup_{0 \leq n \leq N} \bigg( \sum_{l=0}^{n-1}
	\sum_{k=1}^m \sum_{i=1}^s ( |\beta_i^{(1)} | \, | \Ii_{(k),l} | + | \beta_i^{(2)} | )
	\sum_{r=1}^d \Big\| \frac{\partial}{\partial x_r} b^k(t_l,X_{t_l}) \Big\|
	\nonumber \\
	&\quad \times \sum_{j=1}^{i-1} \sum_{k_2=1}^m | B_{i,j}^{(1)} | \int_0^1
	\Big| \frac{\partial}{\partial t} b^{r,k_2}(t_l +u \, c_j^{(1)} \, h, H_{j,l}^{(k_2),X_{t_l}})
	\Big| \, | c_j^{(1)} | \, h \, \mathrm{d}u \, | \Ii_{(k_2,k),l} | \bigg)^p \bigg) 
	\nonumber \\
	&\leq \Erw \bigg( \bigg( \sum_{l=0}^{N-1}
	\sum_{k,k_2=1}^m \sum_{i=1}^s \czwei \, ( | \Ii_{(k),l} | + 1 )
	\, d \, \cc \, \sum_{j=1}^{i-1} \czwei \,
	\cc ( 1 + \| H_{j,l}^{(k_2),X_{t_l}} \| )
	\, \czwei \, h \, | \Ii_{(k_2,k),l} | \bigg)^p \bigg) 
	\nonumber \\
%
	&\leq (\czwei^3 \, d \, \cc^2 \, s^2)^p \, N^{p-1} \sum_{l=0}^{N-1} h^p \, 
	m^{2p-2} \sum_{k,k_2=1}^m \Erw \big( ( | \Ii_{(k),l} | + 1 )^p \, 
	( 1 + \max_{1 \leq j \leq s} \| H_{j,l}^{(k_2),X_{t_l}} \| )^p \, | \Ii_{(k_2,k),l} |^p \big)
	\nonumber \\
	&\leq (\czwei^3 \, d \, \cc^2 \, s^2)^p \, N^{p-1} \sum_{l=0}^{N-1} h^p \,
	m^{2p-2} \sum_{k,k_2=1}^m (2^{p-1})^2 \big(
	\Erw ( | \Ii_{(k),l} |^p \, \max_{1 \leq j \leq s} \| H_{j,l}^{(k_2),X_{t_l}} \|^p \, 
	| \Ii_{(k_2,k),l} |^p )
	\nonumber \\
	&\quad + 
	\Erw ( | \Ii_{(k),l} |^p \, | \Ii_{(k_2,k),l} |^p )
	+ \Erw ( | \Ii_{(k_2,k),l} |^p )
	+ \Erw ( \max_{1 \leq j \leq s} \| H_{j,l}^{(k_2),X_{t_l}} \|^p \, | \Ii_{(k_2,k),l} |^p )
	\big) 
	\nonumber \\
%
	&\leq (\czwei^3 \, d \, \cc^2 \, s^2 \, m^2)^p \, (T-t_0)^p \, 4^{p-1} \,
	\Big( \cMHk (1+ \Erw( \| X_{t_0} \|^p) ) \, ( h^{\frac{p}{2}} + 1 )
	+ \Big( \frac{(2p -1)^2}{\sqrt{2}} \Big)^p \, h^{\frac{p}{2} } \Big) \, h^p \, .
	\nonumber
\end{align}
%
%
%
Now, we consider term \eqref{Proof:MainThm:Teil-A2-2}. With Burkholder's 
inequality, see, e.g., \cite{Burk88} or \cite[Prop.~2.1 \& 2.2]{PlRoe21},
\eqref{Assumption-a-bk:Bound-derivative2-x},
Lemma~\ref{Lem:Xs-Xtl-estimate} and under
the assumption $X_{t_0} \in L^{2p}(\Omega)$ we get
%
\begin{align}
	&\Erw \bigg( \sup_{0 \leq n \leq N} \bigg\| \sum_{l=0}^{n-1} 
	\sum_{k=1}^m \int_{t_l}^{t_{l+1}} \int_0^1 \sum_{q,r=1}^d 
	\frac{\partial^2}{\partial x_q \partial x_r} b^k(t_l, X_{t_l}+u(X_s-X_{t_l})) 
	\nonumber \\
	&\quad \times (X_s^q-X_{t_l}^q) \, (X_s^r-X_{t_l}^r) \, (1-u) 
	\, \mathrm{d}u \, \mathrm{d}W_s^k \bigg\|^p \bigg) 
	\nonumber \\
	%
	&\leq \Big( \frac{p}{\sqrt{p-1}} \Big)^p \bigg( \sum_{l=0}^{N-1} \bigg[ 
	\Erw \bigg( \bigg\| \sum_{k=1}^m \int_{t_l}^{t_{l+1}} \int_0^1 \sum_{q,r=1}^d 
	\frac{\partial^2}{\partial x_q \partial x_r} b^k(t_l, X_{t_l}+u(X_s-X_{t_l})) 
	\nonumber \\
	&\quad \times (X_s^q-X_{t_l}^q) \, (X_s^r-X_{t_l}^r) \, (1-u) 
	\, \mathrm{d}u \, \mathrm{d}W_s^k \bigg\|^p \bigg) \bigg]^{\frac{2}{p}} 
	\bigg)^{\frac{p}{2}}
	\nonumber \\
	%
	&\leq \Big( \frac{p}{\sqrt{p-1}} \Big)^p \bigg( \sum_{l=0}^{N-1} (p-1)
	\int_{t_l}^{t_{l+1}} \bigg[ \Erw \bigg( \bigg\| \sum_{k=1}^m \bigg\|
	\int_0^1 \sum_{q,r=1}^d 
	\frac{\partial^2}{\partial x_q \partial x_r} b^k(t_l, X_{t_l}+u(X_s-X_{t_l})) 
	\nonumber \\
	&\quad \times (X_s^q-X_{t_l}^q) \, (X_s^r-X_{t_l}^r) \, (1-u) 
	\, \mathrm{d}u \bigg\|^2 \bigg\|^{\frac{p}{2}} \bigg) \bigg]^{\frac{2}{p}} 
	\, \mathrm{d}s \bigg)^{\frac{p}{2}}
	\nonumber \\
	&\leq \Big( \frac{p}{\sqrt{p-1}} \Big)^p \bigg( \sum_{l=0}^{N-1} (p-1)
	\int_{t_l}^{t_{l+1}} \bigg[ \Erw \bigg( \bigg\| \sum_{k=1}^m \bigg(
	\int_0^1 \sum_{q,r=1}^d 
	\Big\| \frac{\partial^2}{\partial x_q \partial x_r} b^k(t_l, X_{t_l}+u(X_s-X_{t_l})) \Big\|
	\nonumber \\
	&\quad \times |X_s^q-X_{t_l}^q| \, |X_s^r-X_{t_l}^r| \, (1-u) 
	\, \mathrm{d}u \bigg)^2 \bigg\|^{\frac{p}{2}} \bigg) \bigg]^{\frac{2}{p}} 
	\, \mathrm{d}s \bigg)^{\frac{p}{2}}
	\nonumber \\
	&\leq \Big( \frac{p}{\sqrt{p-1}} \Big)^p \bigg( \sum_{l=0}^{N-1} (p-1)
	\int_{t_l}^{t_{l+1}} \bigg[ \Erw \bigg( \bigg\| m \bigg(
	\int_0^1 d^2 \, \cc \, \| X_s-X_{t_l} \|^2 \, (1-u) 
	\, \mathrm{d}u \bigg)^2 \bigg\|^{\frac{p}{2}} \bigg) \bigg]^{\frac{2}{p}} 
	\, \mathrm{d}s \bigg)^{\frac{p}{2}}
	\nonumber \\
	&\leq \Big( \frac{p}{\sqrt{p-1}} \Big)^p \bigg( \sum_{l=0}^{N-1} (p-1)
	\, m \, d^4 \, \cc^2 \int_{t_l}^{t_{l+1}} \big[ \Erw ( \| X_s-X_{t_l} \|^{2p} )
	\big]^{\frac{2}{p}} \, \mathrm{d}s \bigg)^{\frac{p}{2}}
	\nonumber \\
%
	&\leq (p \, \sqrt{m} \, d^2 \, \cc)^p \bigg( \sum_{l=0}^{N-1} \int_{t_l}^{t_{l+1}}
	\big[ \cMXinc (1+ \Erw( \|X_{t_0} \|^{2p})) \, (s-t_l)^p \big]^{\frac{2}{p}}
	\, \mathrm{d}s \bigg)^{\frac{p}{2}}
	\nonumber \\
	&\leq (p \, \sqrt{m} \, d^2 \, \cc)^p \, (T-t_0)^{\frac{p}{2}} \,
	\cMXinc (1+ \Erw( \|X_{t_0} \|^{2p})) \, h^p \, .
	\nonumber
\end{align}
%
%
%
Next, we estimate \eqref{Proof:MainThm:Teil-A2-3}. With Burkholder's 
inequality, see, e.g., \cite{Burk88} or \cite[Prop.~2.1 \& 2.2]{PlRoe21},
\eqref{Assumption-Bound-derivative-1t-and-2t}
and Lemma~\ref{Lem:Lp-bound-SDE-sol} we obtain
%
\begin{align}
	&\Erw \bigg( \sup_{0 \leq n \leq N} \bigg\| \sum_{l=0}^{n-1}
	\sum_{k=1}^m \int_{t_l}^{t_{l+1}} \int_0^1 \frac{\partial}{\partial t}
	b^k(t_l+u(s-t_l),X_s) \, (s-t_l) \, \mathrm{d}u \, \mathrm{d}W_s^k \bigg\|^p \bigg) 
	\nonumber \\
	%
	&\leq \Big( \frac{p}{\sqrt{p-1}} \Big)^p \bigg( \sum_{l=0}^{N-1} \bigg[ 
	\Erw \bigg( \bigg\| \sum_{k=1}^m \int_{t_l}^{t_{l+1}} \int_0^1 \frac{\partial}{\partial t}
	b^k(t_l+u(s-t_l),X_s) \, (s-t_l) \, \mathrm{d}u \, \mathrm{d}W_s^k \bigg\|^p \bigg)
	\bigg]^{\frac{2}{p}} \bigg)^{\frac{p}{2}}
	\nonumber \\
	%
	&\leq \Big( \frac{p}{\sqrt{p-1}} \Big)^p \bigg( \sum_{l=0}^{N-1} (p-1)
	\int_{t_l}^{t_{l+1}} \bigg[ \Erw \bigg( \bigg| \sum_{k=1}^m \bigg\| \int_0^1 
	\frac{\partial}{\partial t} b^k(t_l+u(s-t_l),X_s)
	\nonumber \\
	&\quad \times
	(s-t_l) \, \mathrm{d}u \bigg\|^2 \bigg|^{\frac{p}{2}} \bigg) \bigg]^{\frac{2}{p}}
	\, \mathrm{d}s \bigg)^{\frac{p}{2}}
	\nonumber \\
	&\leq p^p \bigg( \sum_{l=0}^{N-1} 
	\int_{t_l}^{t_{l+1}} \bigg[ \Erw \bigg( \bigg| \sum_{k=1}^m \bigg(
	\int_0^1 \Big\| \frac{\partial}{\partial t} b^k(t_l+u(s-t_l),X_s) \Big\|
	|s-t_l| \, \mathrm{d}u \bigg)^2 \bigg|^{\frac{p}{2}} \bigg) \bigg]^{\frac{2}{p}}
	\, \mathrm{d}s \bigg)^{\frac{p}{2}}
	\nonumber \\
	&\leq p^p \bigg( \sum_{l=0}^{N-1} 
	\int_{t_l}^{t_{l+1}} \big[ \Erw \big( m^{\frac{p}{2}} \,
	\cc^p \, (1+ \| X_s \|)^p \big)
	\big]^{\frac{2}{p}}  (s-t_l)^2 \, \mathrm{d}s \bigg)^{\frac{p}{2}}
	\nonumber \\
%
	&\leq (p \, \sqrt{m} \, \cc)^p \bigg( \sum_{l=0}^{N-1} h^3 \, \big[ 2^{p-1} \,
	(1+ \ccp (1+ \Erw( \|X_{t_0} \|^p)) ) \big]^{\frac{2}{p}} \bigg)^{\frac{p}{2}}
	\nonumber \\
	&\leq (p \, \sqrt{m} \, \cc)^p \, (T-t_0)^{\frac{p}{2}} \, 2^{p-1} \,
	(1+ \ccp (1+ \Erw( \|X_{t_0} \|^p)) ) \, h^p \, .
	\nonumber
\end{align}
%
%
%
For term \eqref{Proof:MainThm:Teil-A2-4} we calculate with Burkholder's 
inequality, see, e.g., \cite{Burk88} or \cite[Prop.~2.1 \& 2.2]{PlRoe21},
\eqref{Assumption-Bound-derivative-1t-and-2t},
Lemma~\ref{Lem:Ij-Iij-Moment-estimate},
Lemma~\ref{Lem:Ij-Iij-Hk-estimate} and Lemma~\ref{Lem:Lp-bound-SDE-sol} that
%
\begin{align}
	&\Erw \bigg( \sup_{0 \leq n \leq N} \bigg\| \sum_{l=0}^{n-1}
	\sum_{k=1}^m \sum_{i=1}^s \bigg( \beta_i^{(1)} \, \Ii_{(k),l} \,
	\frac{\partial}{\partial t} b^k(t_l,X_{t_l}) \, c_i^{(1)} h 
	\nonumber \\
	&\quad + ( \beta_i^{(1)} \, \Ii_{(k),l} + \beta_i^{(2)} )
	\int_0^1 \frac{\partial^2}{\partial t^2} 
	b^k(t_l+u \, c_i^{(1)} \, h, H_{i,l}^{(k),X_{t_l}}) \, (c_i^{(1)} h)^2 \, (1-u) 
	\, \mathrm{d}u \bigg) \bigg\|^p \bigg) 
	\nonumber \\
	%
	&\leq \Erw \bigg( \sup_{0 \leq n \leq N} 2^{p-1} \bigg( \bigg\| \sum_{l=0}^{n-1}
	\sum_{k=1}^m \int_{t_l}^{t_{l+1}} \sum_{i=1}^s \beta_i^{(1)} \,
	\frac{\partial}{\partial t} b^k(t_l,X_{t_l}) \, c_i^{(1)} h \, \mathrm{d}W_u^k \bigg\|^p
	\nonumber \\
	&\quad
	+ \bigg\| \sum_{l=0}^{n-1} \sum_{k=1}^m \sum_{i=1}^s 
	( \beta_i^{(1)} \, \Ii_{(k),l} + \beta_i^{(2)} )
	\int_0^1 \frac{\partial^2}{\partial t^2} 
	b^k(t_l+u \, c_i^{(1)} \, h, H_{i,l}^{(k),X_{t_l}}) \, (c_i^{(1)} h)^2 \, (1-u) 
	\, \mathrm{d}u \bigg\|^p \bigg) \bigg)
	\nonumber \\
%
	&\leq 2^{p-1} \, \Big( \frac{p}{\sqrt{p-1}} \Big)^p \bigg( \sum_{l=0}^{N-1} 
	\bigg[ \Erw \bigg( \bigg\| \sum_{k=1}^m \int_{t_l}^{t_{l+1}} \sum_{i=1}^s 
	\beta_i^{(1)} \, \frac{\partial}{\partial t} b^k(t_l,X_{t_l}) \, c_i^{(1)} h \,
	\mathrm{d}W_u^k \bigg\|^p \bigg) \bigg]^{\frac{2}{p}} \bigg)^{\frac{p}{2}}
	\nonumber \\
	&\quad + 2^{p-1} \, \Erw \bigg( \sup_{0 \leq n \leq N} \bigg( \sum_{l=0}^{n-1}
	\sum_{k=1}^m \sum_{i=1}^s ( |\beta_i^{(1)}| \, |\Ii_{(k),l}| + |\beta_i^{(2)}| )
	\nonumber \\
	&\quad \times
	\int_0^1 \Big\| \frac{\partial^2}{\partial t^2} 
	b^k(t_l+u \, c_i^{(1)} \, h, H_{i,l}^{(k),X_{t_l}}) \Big\| \, |c_i^{(1)}|^2 \, h^2 \, (1-u) 
	\, \mathrm{d}u \bigg)^p \bigg)
	\nonumber \\
	%
	&\leq 2^{p-1} \, \Big( \frac{p}{\sqrt{p-1}} \Big)^p \bigg( \sum_{l=0}^{N-1}
	(p-1) \int_{t_l}^{t_{l+1}} \bigg[ \Erw \bigg( \bigg\| \sum_{k=1}^m \bigg\|
	\sum_{i=1}^s \beta_i^{(1)} \, \frac{\partial}{\partial t} b^k(t_l,X_{t_l}) \, c_i^{(1)} h
	\bigg\|^2 \bigg\|^{\frac{p}{2}} \bigg) \bigg]^{\frac{2}{p}} \mathrm{d}u 
	\bigg)^{\frac{p}{2}}
	\nonumber \\
	&\quad + 2^{p-1}
	\Erw \bigg( \bigg( \sum_{l=0}^{N-1} \sum_{k=1}^m \sum_{i=1}^s 
	\czwei^3 \, ( |\Ii_{(k),l}| + 1 ) \int_0^1 \cc ( 1 + \| H_{i,l}^{(k),X_{t_l}} \| ) \, h^2 \,
	(1-u) \, \mathrm{d}u \bigg)^p \bigg)
	\nonumber \\
	%
	&\leq 2^{p-1} \, p^p \, \bigg( \sum_{l=0}^{N-1}
	\int_{t_l}^{t_{l+1}} \bigg[ \Erw \bigg( \bigg\| \sum_{k=1}^m \bigg(
	\sum_{i=1}^s |\beta_i^{(1)}| \, 
	\Big\| \frac{\partial}{\partial t} b^k(t_l,X_{t_l}) \Big\| \, |c_i^{(1)}| \, h
	\bigg)^2 \bigg\|^{\frac{p}{2}} \bigg) \bigg]^{\frac{2}{p}} \mathrm{d}u 
	\bigg)^{\frac{p}{2}}
	\nonumber \\
	&\quad + 2^{p-1} \bigg( \sum_{l=0}^{N-1} h^2
	\sum_{k=1}^m \sum_{i=1}^s \czwei^3 \, \cc \,
	\big[ \Erw \big( ( |\Ii_{(k),l}| + 1 )^p \, ( 1 + \| H_{i,l}^{(k),X_{t_l}} \| )^p
	\big) \big]^{\frac{1}{p}} \bigg)^p
	\nonumber \\
	&\leq 2^{p-1} \, p^p \, \bigg( \sum_{l=0}^{N-1}
	\int_{t_l}^{t_{l+1}} \bigg[ \Erw \bigg( \bigg\| \sum_{k=1}^m \big(
	s \, \czwei^2 \, \cc ( 1 + \| X_{t_l} \| ) \, h
	\big)^2 \bigg\|^{\frac{p}{2}} \bigg) \bigg]^{\frac{2}{p}} \mathrm{d}u 
	\bigg)^{\frac{p}{2}}
	\nonumber \\
	&\quad + 2^{p-1} \bigg( \sum_{l=0}^{N-1} h^2
	\sum_{k=1}^m \sum_{i=1}^s \czwei^3 \, \cc \, (2^{p-1})^{\frac{2}{p}} \, \big( 
	\big[ \Erw ( |\Ii_{(k),l}|^p ) \big]^{\frac{1}{p}} + \big[ \Erw ( |\Ii_{(k),l}|^p \, \| H_{i,l}^{(k),X_{t_l}} \|^p ) \big]^{\frac{1}{p}}
	\nonumber \\
	&\quad +
	1 + \big[ \Erw( \| H_{i,l}^{(k),X_{t_l}} \|^p ) \big]^{\frac{1}{p}} \big) \bigg)^p
	\nonumber \\
%
	&\leq 2^{p-1} \, (p \, \sqrt{m} \, s \, \czwei^2 \, \cc )^p \,
	\bigg( \sum_{l=0}^{N-1} h^2 \int_{t_l}^{t_{l+1}} 
	\big[ \Erw \big( ( 1 + \| X_{t_l} \| )^p \big) \big]^{\frac{2}{p}} \, \mathrm{d}u 
	\bigg)^{\frac{p}{2}}
	\nonumber \\
	&\quad +
	2^{p-1} \, (\czwei^3 \, \cc \, m \, s)^p \, 4^{p-1} \, \bigg( \sum_{l=0}^{N-1} h^2
	\big( (p-1) \, h^{\frac{1}{2}}
	+ (\cMHk (1+ \Erw( \| X_{t_l} \|^{p}) ) )^{\frac{1}{p}} \, h^{\frac{1}{2}}
	\nonumber \\
	&\quad +
	1 + (\cMHk (1+ \Erw( \| X_{t_l} \|^{p}) ) )^{\frac{1}{p}} \big) \bigg)^p
	\nonumber \\
	&\leq 2^{p-1} \, (p \, \sqrt{m} \, s \, \czwei^2 \, \cc )^p \,
	\bigg( \sum_{l=0}^{N-1} h^3 
	\big[ 2^{p-1} ( 1 + \ccp (1+ \Erw( \|X_{t_0} \|^p)) ) \big]^{\frac{2}{p}}
	\bigg)^{\frac{p}{2}}
	\nonumber \\
	&\quad +
	2^{p-1} \, (\czwei^3 \, \cc \, m \, s)^p \, 4^{p-1} \, (T-t_0)^p \,
	\big( (p-1) \, h^{\frac{1}{2}}
	+ (\cMHk \, (1+ \ccp (1+ \Erw( \|X_{t_0} \|^p)) ) )^{\frac{1}{p}} \, h^{\frac{1}{2}}
	\nonumber \\
	&\quad +
	1 + (\cMHk \, (1+ \ccp (1+ \Erw( \|X_{t_0} \|^p)) ) )^{\frac{1}{p}} \big)^p \, h^p
	\nonumber \\
%
	&\leq 2^{p-1} \, \big[ (p \, \sqrt{m} \, s \, \czwei^2 \, \cc )^p \, (T-t_0)^{\frac{p}{2}} \,
	2^{p-1} ( 1 + \ccp (1+ \Erw( \|X_{t_0} \|^p)) )
	\nonumber \\
	&\quad +
	(\czwei^3 \, \cc \, m \, s)^p \, 4^{p-1} \, (T-t_0)^p \, 4^{p-1} \,
	\big( (p-1)^p \, h^{\frac{p}{2}}
	+ \cMHk (1+ \ccp (1+ \Erw( \|X_{t_0} \|^p)) ) \, h^{\frac{p}{2}}
	\nonumber \\
	&\quad +
	1 + \cMHk (1+ \ccp (1+ \Erw( \|X_{t_0} \|^p)) ) \big) \big] \, h^p \, .
	\nonumber
\end{align}
%
%
%
Now we estimate \eqref{Proof:MainThm:Teil-A2-5}. Under the assumption that
$X_{t_0} \in L^{2p}(\Omega)$ and with \eqref{Assumption-Bound-derivative-2tx-bk}, 
H\"older's inequality,
Lemma~\ref{Lem:Ij-Iij-H-Z-estimate} and Lemma~\ref{Lem:Lp-bound-SDE-sol}
we get
%
\begin{align}
	&\Erw \bigg( \sup_{0 \leq n \leq N} \bigg\| \sum_{l=0}^{n-1}
	\sum_{k=1}^m \sum_{i=1}^s ( \beta_i^{(1)} \, \Ii_{(k),l} + \beta_i^{(2)} )
	\int_0^1 \sum_{r=1}^d \frac{\partial^2}{\partial t \partial x_r}
	b^k(t_l,X_{t_l} +u (H_{i,l}^{(k),X_{t_l}} -X_{t_l})) 
	\nonumber \\
	&\quad \times (c_i^{(1)} h) \, ( {H_{i,l}^{(k),X_{t_l}}}^r -X_{t_l}^r ) 
	\, \mathrm{d}u \bigg\|^p \bigg)
	\nonumber \\
	%
	&\leq \Erw \bigg( \sup_{0 \leq n \leq N} \bigg( \sum_{l=0}^{n-1}
	\sum_{k=1}^m \sum_{i=1}^s ( | \beta_i^{(1)} | \, | \Ii_{(k),l} | + | \beta_i^{(2)} | )
	\int_0^1 \sum_{r=1}^d \Big\| \frac{\partial^2}{\partial t \partial x_r}
	b^k(t_l,X_{t_l} +u (H_{i,l}^{(k),X_{t_l}} -X_{t_l})) \Big\|
	\nonumber \\
	&\quad \times
	|c_i^{(1)}| \, h \, | {H_{i,l}^{(k),X_{t_l}}}^r -X_{t_l}^r |
	\, \mathrm{d}u \bigg)^p \bigg)
	\nonumber \\
	%
	&\leq \Erw \bigg( \bigg( \sum_{l=0}^{N-1} \sum_{k=1}^m \sum_{i=1}^s
	\czwei^2 \, ( | \Ii_{(k),l} | + 1 ) 
	\nonumber \\
	&\quad \times
	\int_0^1 \sum_{r=1}^d 
	\cc (1+ \| X_{t_l} +u (H_{i,l}^{(k),X_{t_l}} -X_{t_l}) \| ) \, h \,
	\| H_{i,l}^{(k),X_{t_l}} -X_{t_l} \| \, \mathrm{d}u \bigg)^p \bigg)
	\nonumber \\
	%
	&\leq \bigg( \sum_{l=0}^{N-1} \sum_{k=1}^m \sum_{i=1}^s \czwei^2 \, d \, \cc \,
	h \, \bigg[ \Erw \bigg( ( | \Ii_{(k),l} | + 1 )^p
	\nonumber \\
	&\quad \times
	\bigg( \int_0^1 (1+ \| X_{t_l} \| +u  \, \| H_{i,l}^{(k),X_{t_l}} -X_{t_l} \| ) \,
	\| H_{i,l}^{(k),X_{t_l}} -X_{t_l} \| \, \mathrm{d}u \bigg)^p \bigg)
	\bigg]^{\frac{1}{p}} \bigg)^p
	\nonumber \\
	&\leq \bigg( \sum_{l=0}^{N-1} \sum_{k=1}^m \sum_{i=1}^s \czwei^2 \, d \, \cc \,
	h \, \big[ \Erw \big( \big( ( | \Ii_{(k),l} | + 1 )
	\nonumber \\
	&\quad \times
	( \| H_{i,l}^{(k),X_{t_l}} -X_{t_l} \|
	+ \| X_{t_l} \| \, \| H_{i,l}^{(k),X_{t_l}} -X_{t_l} \|
	+ \| H_{i,l}^{(k),X_{t_l}} -X_{t_l} \|^2 ) \big)^p \big) \big]^{\frac{1}{p}} \bigg)^p
	\nonumber \\
	%
	&\leq \bigg( \sum_{l=0}^{N-1} \sum_{k=1}^m \sum_{i=1}^s \czwei^2 \, 
	d \, \cc \, h \, \big( 
	\big[ \Erw \big( | \Ii_{(k),l} |^p \, \| H_{i,l}^{(k),X_{t_l}} -X_{t_l} \|^p \big)
	\big]^{\frac{1}{p}}
	\nonumber \\
	&\quad + 
	\big[ \Erw \big( \| X_{t_l} \|^{2p} \big) \big]^{\frac{1}{2p}}
	\big[ \Erw \big( | \Ii_{(k),l} |^{2p} \, \| H_{i,l}^{(k),X_{t_l}} -X_{t_l} \|^{2p}
	\big) \big]^{\frac{1}{2p}}
	+ \big[ \Erw \big( | \Ii_{(k),l} |^p \, \| H_{i,l}^{(k),X_{t_l}} -X_{t_l} \|^{2p}
	\big) \big]^{\frac{1}{p}}
	\nonumber \\
	&\quad +
	\big[ \Erw \big( \| H_{i,l}^{(k),X_{t_l}} -X_{t_l} \|^p \big) \big]^{\frac{1}{p}}
	+ \big[ \Erw \big( \| X_{t_l} \|^{2p} \big) \big]^{\frac{1}{2p}} \,
	\big[ \Erw \big( \| H_{i,l}^{(k),X_{t_l}} -X_{t_l} \|^{2p} \big) \big]^{\frac{1}{2p}}
	\nonumber \\
	&\quad +
	\big[ \Erw \big( \| H_{i,l}^{(k),X_{t_l}} -X_{t_l} \|^{2p}
	\big) \big]^{\frac{1}{p}}
	\big) \bigg)^p
	\nonumber \\
%
	&\leq \bigg( \sum_{l=0}^{N-1} h \sum_{k=1}^m \sum_{i=1}^s \czwei^2 \, 
	d \, \cc \, \big( 
	\big[ \cMH (1+ \Erw( \| X_{t_l} \|^p) ) \big]^{\frac{1}{p}} \, h^{\frac{3}{2}}
	\nonumber \\
	&\quad +
	\big[ \ccp (1+ \Erw( \|X_{t_0} \|^{2p})) \big]^{\frac{1}{2p}} \,
	\big[ \cMH (1+ \Erw( \| X_{t_l} \|^{2p}) ) \big]^{\frac{1}{2p}} \, h^{\frac{3}{2}}
	+ \big[ \cMH (1+ \Erw( \| X_{t_l} \|^{2p}) ) \big]^{\frac{1}{p}}  \, h^{\frac{5}{2}}
	\nonumber \\
	&\quad +
	\big[ \cMH (1+ \Erw( \| X_{t_l} \|^p) ) \big]^{\frac{1}{p}} \, h
	+ \big[ \ccp (1+ \Erw( \|X_{t_0} \|^{2p})) \big]^{\frac{1}{2p}} \,
	\big[ \cMH (1+ \Erw( \| X_{t_l} \|^{2p}) ) \big]^{\frac{1}{2p}} \, h
	\nonumber \\
	&\quad +
	\big[ \cMH (1+ \Erw( \| X_{t_l} \|^{2p}) ) \big]^{\frac{1}{p}} 
	\, h^2 \big) \bigg)^p
	\nonumber \\
	&\leq (T-t_0)^p \, (m \, s \, \czwei^2 \,  d \, \cc)^p \,  6^{p-1} \,
	\big[ \cMH (1+ \ccp (1+ \Erw( \|X_{t_0} \|^p)) ) \, h^{\frac{p}{2}}
	\nonumber \\
	&\quad +
	\big( \ccp (1+ \Erw( \|X_{t_0} \|^{2p})) \big)^{\frac{1}{2}} \,
	\big( \cMH (1+ \ccp (1+ \Erw( \|X_{t_0} \|^{2p})) ) \big)^{\frac{1}{2}} \, h^{\frac{p}{2}}
	\nonumber \\
	&\quad +
	\cMH (1+ \ccp (1+ \Erw( \|X_{t_0} \|^{2p})) ) \, h^{\frac{3p}{2}}
	+ \cMH (1+ \ccp (1+ \Erw( \|X_{t_0} \|^p)) )
	\nonumber \\
	&\quad + 
	\big( \ccp (1+ \Erw( \|X_{t_0} \|^{2p})) \big)^{\frac{1}{2}} \,
	\big( \cMH (1+ \ccp (1+ \Erw( \|X_{t_0} \|^{2p})) ) \big)^{\frac{1}{2}}
	\nonumber \\
	&\quad +
	\cMH (1+ \ccp (1+ \Erw( \|X_{t_0} \|^{2p})) ) \, h^{p} \big] \, h^p \, .
	\nonumber
\end{align}
%
%
%
Finally, we consider \eqref{Proof:MainThm:Teil-A2-6}. If we assume that $X_{t_0}
\in L^{2p}(\Omega)$, we calculate 
with~\eqref{Assumption-a-bk:Bound-derivative2-x},
Lemma~\ref{Lem:Ij-Iij-H-Z-estimate} and Lemma~\ref{Lem:Lp-bound-SDE-sol} that
%
\begin{align}
	&\Erw \bigg( \sup_{0 \leq n \leq N} \bigg\| \sum_{l=0}^{n-1}
	\sum_{k=1}^m \sum_{i=1}^s ( \beta_i^{(1)} \, \Ii_{(k),l} + \beta_i^{(2)} )
	\int_0^1 \sum_{q,r=1}^d \frac{\partial^2}{\partial x_q \partial x_r}
	b^k(t_l,X_{t_l} + u (H_{i,l}^{(k),X_{t_l}} - X_{t_l})) 
	\nonumber \\
	&\quad \times ({ H_{i,l}^{(k),X_{t_l}} }^q - X_{t_l}^q) 
	\, ( {H_{i,l}^{(k),X_{t_l}} }^r - X_{t_l}^r) \, (1-u) \, \mathrm{d}u 
	\bigg\|^p \bigg)
	\nonumber \\
	&\leq \Erw \bigg( \sup_{0 \leq n \leq N} \bigg( \sum_{l=0}^{n-1}
	\sum_{k=1}^m \sum_{i=1}^s ( | \beta_i^{(1)} | \, | \Ii_{(k),l} | + | \beta_i^{(2)} | )
	\nonumber \\
	&\quad \times
	\int_0^1 \sum_{q,r=1}^d \Big\| \frac{\partial^2}{\partial x_q \partial x_r}
	b^k(t_l,X_{t_l} + u (H_{i,l}^{(k),X_{t_l}} - X_{t_l})) \Big\|
	\, \| { H_{i,l}^{(k),X_{t_l}} } - X_{t_l} \|^2 \, (1-u) \, \mathrm{d}u 
	\bigg)^p \bigg)
	\nonumber \\
	&\leq \Erw \bigg( \bigg( \sum_{l=0}^{N-1}
	\sum_{k=1}^m \sum_{i=1}^s \czwei \, ( | \Ii_{(k),l} | + 1 ) \, d^2 \, \cc \,
	\| { H_{i,l}^{(k),X_{t_l}} } - X_{t_l} \|^2 \bigg)^p \bigg)
	\nonumber \\
	%
	&\leq \bigg( \sum_{l=0}^{N-1} \sum_{k=1}^m \sum_{i=1}^s \czwei \, d^2 \, \cc \,
	\big[ \Erw \big( \big( ( | \Ii_{(k),l} | + 1 ) \, 
	\| { H_{i,l}^{(k),X_{t_l}} } - X_{t_l} \|^2 \big)^p \big) \big]^{\frac{1}{p}} \bigg)^p
	\nonumber \\
	%
	&\leq \bigg( \sum_{l=0}^{N-1} \sum_{k=1}^m \sum_{i=1}^s \czwei \, d^2 \, \cc \,
	\big( \big[ \Erw \big( | \Ii_{(k),l} |^p \, \| { H_{i,l}^{(k),X_{t_l}} } - X_{t_l} \|^{2p} \big) \big]^{\frac{1}{p}} 
	+ \big[ \Erw \big( \| { H_{i,l}^{(k),X_{t_l}} } - X_{t_l} \|^{2p} \big) \big]^{\frac{1}{p}} 
	\big) \bigg)^p 
	\nonumber \\
	&\leq ( \czwei \, d^2 \, \cc)^p \,
	\bigg( \sum_{l=0}^{N-1} \sum_{k=1}^m \sum_{i=1}^s \big( 
	[ \cMH (1+ \Erw( \| X_{t_l} \|^{2p}) ) ]^{\frac{1}{p}} \, h^{\frac{5}{2}}
	+ [ \cMH (1+ \Erw( \| X_{t_l} \|^{2p}) ) ]^{\frac{1}{p}} \, h^2
	\big) \bigg)^p 
	\nonumber \\
	&\leq ( \czwei \, d^2 \, \cc \, m \, s)^p \, \cMH \bigg( \sum_{l=0}^{N-1} h^2 \,
	\big( [ 1+ \ccp (1+ \Erw( \|X_{t_0} \|^{2p})) ]^{\frac{1}{p}} h^{\frac{1}{2}}
	+ [ 1+ \ccp (1+ \Erw( \|X_{t_0} \|^{2p})) ]^{\frac{1}{p}}
	\big) \bigg)^p
	\nonumber \\
	&\leq ( \czwei \, d^2 \, \cc \, m \, s)^p \, \cMH \, (T-t_0)^p \, 2^{p-1} \,
	(1+ \ccp (1+ \Erw( \|X_{t_0} \|^{2p})) ) (h^{\frac{p}{2}} + 1) \, h^p \, .
	\nonumber
\end{align}
Thus, we have proved that for the first summand 
in~\eqref{Proof:MainThm:Teil-A-B} it holds
\begin{align}
	\Erw \big( \sup_{0 \leq n \leq N} \| X_{t_n}-Z_n \|^p \big)
	\leq \cXZ \, h^p \, .
	\label{Proof:MainThm:Teil-A-Estimate}
\end{align}
%
%
%
As the second step, we have to estimate the second summand in 
\eqref{Proof:MainThm:Teil-A-B}.
%
\begin{align}
	&\Erw \big( \sup_{0 \leq n \leq N} \| Z_n-Y_n \|^p \big)
	\nonumber \\
	&= \Erw \bigg( \sup_{0 \leq n \leq N} \bigg\| Y_0 + \sum_{l=0}^{n-1}
	\sum_{i=1}^s \alpha_i \, a(t_l + c_i^{(0)} h, H_{i,l}^{(0),X_{t_l}}) \, h
	\nonumber \\
	&\quad 
	+ \sum_{k=1}^m \sum_{l=0}^{n-1} \sum_{i=1}^s ( \beta_i^{(1)} \Ii_{(k),l} 
	+ \beta_i^{(2)} ) \, b^k(t_l + c_i^{(1)} h, H_{i,l}^{(k),X_{t_l}})
	\nonumber \\
	&\quad 
	- Y_0 - \sum_{l=0}^{n-1} \sum_{i=1}^s \alpha_i \, 
	a(t_l + c_i^{(0)} h, H_{i,l}^{(0),Y_l}) \, h
	\nonumber \\
	&\quad 
	- \sum_{k=1}^m \sum_{l=0}^{n-1} \sum_{i=1}^s ( \beta_i^{(1)} \Ii_{(k),l} 
	+ \beta_i^{(2)} ) \, b^k(t_l + c_i^{(1)} h, H_{i,l}^{(k),Y_l})
	\bigg\|^p \bigg)
	\nonumber \\
	&\leq 3^{p-1} \, \Erw \bigg( \sup_{0 \leq n \leq N} \bigg\| 
	\sum_{l=0}^{n-1} \sum_{i=1}^s \alpha_i \, \big( 
	a(t_l + c_i^{(0)} h, H_{i,l}^{(0),X_{t_l}}) - a(t_l + c_i^{(0)} h, H_{i,l}^{(0),Y_l}) \big)
	\, h \bigg\|^p \bigg)
	\label{Proof:MainThm:Teil-B1} \\
	&\quad
	+ 3^{p-1} \, \Erw \bigg( \sup_{0 \leq n \leq N} \bigg\| 
	\sum_{k=1}^m \sum_{l=0}^{n-1} \sum_{i=1}^s \beta_i^{(1)} \Ii_{(k),l}
	\big( b^k(t_l + c_i^{(1)} h, H_{i,l}^{(k),X_{t_l}})
	- b^k(t_l + c_i^{(1)} h, H_{i,l}^{(k),Y_l}) \big) \bigg\|^p \bigg)
	\label{Proof:MainThm:Teil-B2} \\
	&\quad
	+ 3^{p-1} \, \Erw \bigg( \sup_{0 \leq n \leq N} \bigg\| 
	\sum_{k=1}^m \sum_{l=0}^{n-1} \sum_{i=1}^s \beta_i^{(2)}
	\big( b^k(t_l + c_i^{(1)} h, H_{i,l}^{(k),X_{t_l}})
	- b^k(t_l + c_i^{(1)} h, H_{i,l}^{(k),Y_l}) \big) \bigg\|^p \bigg) \, .
	\label{Proof:MainThm:Teil-B3}
\end{align}
%
%
%
Firstly, we consider \eqref{Proof:MainThm:Teil-B1}. With the Lipschitz 
condition~\eqref{Assumption-a-bk:Lip},
the linear growth condition~\eqref{Assumption-a-bk:lin-growth},
Lemma~\ref{Lem:Ij-Iij-H0-estimate},
Lemma~\ref{Lem:Lp-bound-SDE-sol} and Proposition~\ref{Prop:Lp-bound-Approximation}
we get
%
\begin{align}
	&\Erw \bigg( \sup_{0 \leq n \leq N} \bigg\| 
	\sum_{l=0}^{n-1} \sum_{i=1}^s \alpha_i \, \big( 
	a(t_l + c_i^{(0)} h, H_{i,l}^{(0),X_{t_l}}) - a(t_l + c_i^{(0)} h, H_{i,l}^{(0),Y_l}) \big)
	\, h \bigg\|^p \bigg)
	\nonumber \\
	%
	&\leq \Erw \bigg( \sup_{0 \leq n \leq N} \bigg( \sum_{l=0}^{n-1} \sum_{i=1}^s 
	|\alpha_i | \, h \, \big\| a(t_l + c_i^{(0)} h, H_{i,l}^{(0),X_{t_l}}) 
	- a(t_l + c_i^{(0)} h, H_{i,l}^{(0),Y_l}) \big\| \bigg)^p \bigg)
	\nonumber \\
	%
	&\leq \Erw \bigg( \bigg( \sum_{l=0}^{N-1} \sum_{i=1}^s \czwei \, h \, \cc \,
	\| H_{i,l}^{(0),X_{t_l}} - H_{i,l}^{(0),Y_l} \| \bigg)^p \bigg)
	\nonumber \\
	&\leq (\czwei \, \cc)^p \, \Erw \bigg( \bigg( \sum_{l=0}^{N-1} h \sum_{i=1}^s
	\| X_{t_l} + \sum_{j=1}^s A_{i,j}^{(0)} \, a(t_l+c_j^{(0)} h, H_{j,l}^{(0),X_{t_l}}) h
	\nonumber \\
	&\quad 
	- Y_l - \sum_{j=1}^s A_{i,j}^{(0)} \, a(t_l+c_j^{(0)} h, H_{j,l}^{(0),Y_l}) h \| 
	\bigg)^p \bigg)
	\nonumber \\
	&\leq (\czwei \, \cc)^p \, 3^{p-1} \, \Erw \bigg( \bigg( \sum_{l=0}^{N-1} h \sum_{i=1}^s
	\| X_{t_l} - Y_l \| \bigg)^p \bigg) 
	\nonumber \\
	&\quad
	+ (\czwei \, \cc)^p \, 3^{p-1} \, \Erw \bigg( \bigg( \sum_{l=0}^{N-1} h^2 \sum_{i=1}^s
	\sum_{j=1}^s | A_{i,j}^{(0)} | \, \| a(t_l+c_j^{(0)} h, H_{j,l}^{(0),X_{t_l}}) \|
	\bigg)^p \bigg)
	\nonumber \\
	&\quad
	+ (\czwei \, \cc)^p \, 3^{p-1} \, \Erw \bigg( \bigg( \sum_{l=0}^{N-1} h^2 \sum_{i=1}^s
	\sum_{j=1}^s | A_{i,j}^{(0)} | \, \| a(t_l+c_j^{(0)} h, H_{j,l}^{(0),Y_l}) \|
	\bigg)^p \bigg)
	\nonumber \\
%
	&\leq (\czwei \, \cc \, s)^p \, 3^{p-1} \, \Erw \bigg( \bigg( \sum_{l=0}^{N-1} h
	\sup_{0 \leq n \leq l} \| X_{t_n} - Y_n \| \bigg)^p \bigg) 
	\nonumber \\
	&\quad
	+ (\czwei^2 \, \cc \, s)^p \, 3^{p-1} \, \Erw \bigg( \bigg( \sum_{l=0}^{N-1} h^2
	\sum_{j=1}^s \cc (1+ \| H_{j,l}^{(0),X_{t_l}} \| ) \bigg)^p \bigg)
	\nonumber \\
	&\quad
	+ (\czwei^2 \, \cc \, s)^p \, 3^{p-1} \, \Erw \bigg( \bigg( \sum_{l=0}^{N-1} h^2
	\sum_{j=1}^s \cc (1+ \| H_{j,l}^{(0),Y_l} \| ) \bigg)^p \bigg)
	\nonumber \\
	&\leq (\czwei \, \cc \, s)^p \, 3^{p-1} \, N^{p-1} \sum_{l=0}^{N-1} h^p \,
	\Erw \big( \sup_{0 \leq n \leq l} \| X_{t_n} - Y_n \|^p \big) 
	\nonumber \\
	&\quad
	+ (\czwei^2 \, \cc^2 \, s^2)^p \, 3^{p-1} \, N^{p-1} \sum_{l=0}^{N-1} h^{2p} \,
	2^{p-1} \, \big(1 + \Erw \big( \max_{1 \leq j \leq s} \| H_{j,l}^{(0),X_{t_l}} \|^p \big) \big)
	\nonumber \\
	&\quad
	+ (\czwei^2 \, \cc^2 \, s^2)^p \, 3^{p-1} \, N^{p-1} \sum_{l=0}^{N-1} h^{2p} \,
	2^{p-1} \, \big(1 + \Erw \big( \max_{1 \leq j \leq s} \| H_{j,l}^{(0),Y_l} \|^p \big) \big)
	\nonumber \\
	&\leq 3^{p-1} \, (\czwei \, \cc \, s)^p \, (T-t_0)^{p-1} \sum_{l=0}^{N-1} h \,
	\Erw \big( \sup_{0 \leq n \leq l} \| X_{t_n} - Y_n \|^p \big) 
	\nonumber \\
	&\quad
	+ 6^{p-1} \, (\czwei^2 \, \cc^2 \, s^2)^p \, (T-t_0)^{p-1} \sum_{l=0}^{N-1}
	h \, (1 + \cMHnull (1+ \Erw( \| X_{t_l} \|^p) ) ) \, h^p
	\nonumber \\
	&\quad
	+ 6^{p-1} \, (\czwei^2 \, \cc^2 \, s^2)^p \, (T-t_0)^{p-1} 
	\sum_{l=0}^{N-1} h \,
	(1 + \cMHnull (1+ \Erw( \| Y_l \|^p) ) ) \, h^p
	\nonumber \\
	&\leq 3^{p-1} \, (\czwei \, \cc \, s)^p \, (T-t_0)^{p-1} \sum_{l=0}^{N-1} h \,
	\Erw \big( \sup_{0 \leq n \leq l} \| X_{t_n} - Y_n \|^p \big) 
	\nonumber \\
	&\quad 
	+ 6^{p-1} \, (\czwei^2 \, \cc^2 \, s^2)^p \, (T-t_0)^p \, 
	(2 + 2 \cMHnull + \cMHnull \, \ccp (1+ \Erw( \|X_{t_0} \|^p))
	+ \cMHnull \, \cYMB (1+ \Erw ( \| Y_0 \|^p )) ) \, h^p \, .
	\nonumber 
	%
\end{align}
%
%
%
Next, we consider term \eqref{Proof:MainThm:Teil-B2}. We split term 
\eqref{Proof:MainThm:Teil-B2} into five parts which are estimated separately:
%
\begin{align}
	&\Erw \bigg( \sup_{0 \leq n \leq N} \bigg\| 
	\sum_{k=1}^m \sum_{l=0}^{n-1} \sum_{i=1}^s \beta_i^{(1)} \Ii_{(k),l}
	\big( b^k(t_l + c_i^{(1)} h, H_{i,l}^{(k),X_{t_l}})
	- b^k(t_l + c_i^{(1)} h, H_{i,l}^{(k),Y_l}) \big) \bigg\|^p \bigg)
	\nonumber \\
	&= \Erw \bigg( \sup_{0 \leq n \leq N} \bigg\| 
	\sum_{k=1}^m \sum_{l=0}^{n-1} \sum_{i=1}^s \beta_i^{(1)} \Ii_{(k),l}
	\bigg( b^k(t_l + c_i^{(1)} h, X_{t_l}) 
	\nonumber \\
	&\quad + \sum_{r=1}^d \frac{\partial}{\partial x_r}
	b^k(t_l + c_i^{(1)} h, X_{t_l}) \, ( {H_{i,l}^{(k),X_{t_l}} }^r - X_{t_l}^r )
	\nonumber \\
	&\quad + \int_0^1 \sum_{q,r=1}^d \frac{\partial^2}{\partial x_q \partial x_r}
	b^k(t_l + c_i^{(1)} h, X_{t_l} + u(H_{i,l}^{(k),X_{t_l}} - X_{t_l})) \,
	\nonumber \\
	&\quad \times ( {H_{i,l}^{(k),X_{t_l}} }^r - X_{t_l}^r ) \, 
	( {H_{i,l}^{(k),X_{t_l}} }^q - X_{t_l}^q ) \, (1-u) \, \mathrm{d}u
	- \bigg( b^k(t_l + c_i^{(1)} h, Y_l) 
	\nonumber \\
	&\quad + \sum_{r=1}^d \frac{\partial}{\partial x_r}
	b^k(t_l + c_i^{(1)} h, Y_l) \, ( {H_{i,l}^{(k),Y_l} }^r - Y_l^r)
	\nonumber \\
	&\quad + \int_0^1 \sum_{q,r=1}^d \frac{\partial^2}{\partial x_q \partial x_r}
	b^k(t_l + c_i^{(1)} h, Y_l +u( H_{i,l}^{(k),Y_l}-Y_{t_l})) \,
	\nonumber \\
	&\quad \times 
	( {H_{i,l}^{(k),Y_l} }^q-Y_l^q) \, ( {H_{i,l}^{(k),Y_l} }^r-Y_l^r) \, (1-u) \,
	\mathrm{d}u \bigg) \bigg) \bigg\|^p \bigg)
	\nonumber \\
	%
	&\leq 5^{p-1} \, \Erw \bigg( \sup_{0 \leq n \leq N} \bigg\| 
	\sum_{l=0}^{n-1} \sum_{k=1}^m \sum_{i=1}^s \beta_i^{(1)} \Ii_{(k),l}
	\big( b^k(t_l + c_i^{(1)} h, X_{t_l}) - b^k(t_l + c_i^{(1)} h, Y_l) \big) \bigg\|^p \bigg) 
	\label{Proof:MainThm:Teil-B2-1} \\
	&\quad + 5^{p-1} \, \Erw \bigg( \sup_{0 \leq n \leq N} \bigg\| 
	\sum_{l=0}^{n-1} \sum_{k=1}^m \sum_{i=1}^s \beta_i^{(1)} \Ii_{(k),l}
	\bigg( \sum_{r=1}^d \frac{\partial}{\partial x_r}
	b^k(t_l + c_i^{(1)} h, X_{t_l}) \, ( {H_{i,l}^{(k),X_{t_l}} }^r - X_{t_l}^r ) \bigg) 
	\bigg\|^p \bigg)
	\label{Proof:MainThm:Teil-B2-2} \\
	&\quad + 5^{p-1} \, \Erw \bigg( \sup_{0 \leq n \leq N} \bigg\| 
	\sum_{l=0}^{n-1} \sum_{k=1}^m \sum_{i=1}^s \beta_i^{(1)} \Ii_{(k),l}
	\bigg( \sum_{r=1}^d \frac{\partial}{\partial x_r}
	b^k(t_l + c_i^{(1)} h, Y_l) \, ( {H_{i,l}^{(k),Y_l} }^r - Y_l^r) \bigg) \bigg\|^p \bigg)
	\label{Proof:MainThm:Teil-B2-3} \\
	&\quad + 5^{p-1} \, \Erw \bigg( \sup_{0 \leq n \leq N} \bigg\| 
	\sum_{l=0}^{n-1} \sum_{k=1}^m \sum_{i=1}^s \beta_i^{(1)} \Ii_{(k),l}
	\nonumber \\
	&\quad \times \int_0^1 \sum_{q,r=1}^d \frac{\partial^2}{\partial x_q \partial x_r}
	b^k(t_l + c_i^{(1)} h, X_{t_l} + u(H_{i,l}^{(k),X_{t_l}} - X_{t_l})) \,
	\nonumber \\
	&\quad \times ( {H_{i,l}^{(k),X_{t_l}} }^r - X_{t_l}^r ) \, 
	( {H_{i,l}^{(k),X_{t_l}} }^q - X_{t_l}^q ) \, (1-u) \, \mathrm{d}u \bigg\|^p \bigg)
	\label{Proof:MainThm:Teil-B2-4} \\
	&\quad + 5^{p-1} \, \Erw \bigg( \sup_{0 \leq n \leq N} \bigg\| 
	\sum_{l=0}^{n-1} \sum_{k=1}^m \sum_{i=1}^s \beta_i^{(1)} \Ii_{(k),l}
	\nonumber \\
	&\quad \times \int_0^1 \sum_{q,r=1}^d \frac{\partial^2}{\partial x_q \partial x_r}
	b^k(t_l + c_i^{(1)} h, Y_l +u( H_{i,l}^{(k),Y_l}-Y_{t_l})) \,
	\nonumber \\
	&\quad \times 
	( {H_{i,l}^{(k),Y_l} }^q-Y_l^q) \, ( {H_{i,l}^{(k),Y_l} }^r-Y_l^r) \, (1-u) \,
	\mathrm{d}u \bigg\|^p \bigg) \, .
	\label{Proof:MainThm:Teil-B2-5} 
\end{align}
%
%
%
We start with term \eqref{Proof:MainThm:Teil-B2-1}. With Burkholder's 
inequality, see, e.g., \cite{Burk88} or \cite[Prop.~2.1 \& 2.2]{PlRoe21}
and the Lipschitz condition~\eqref{Assumption-a-bk:Lip}
we get
%
\begin{align}
	&\Erw \bigg( \sup_{0 \leq n \leq N} \bigg\| 
	\sum_{l=0}^{n-1} \sum_{k=1}^m \sum_{i=1}^s \beta_i^{(1)} \Ii_{(k),l}
	\big( b^k(t_l + c_i^{(1)} h, X_{t_l}) - b^k(t_l + c_i^{(1)} h, Y_l) \big) \bigg\|^p \bigg)
	\nonumber \\
	&= \Erw \bigg( \sup_{0 \leq n \leq N} \bigg\| 
	\sum_{l=0}^{n-1} \sum_{k=1}^m \int_{t_l}^{t_{l+1}} \sum_{i=1}^s \beta_i^{(1)}
	\big( b^k(t_l + c_i^{(1)} h, X_{t_l}) - b^k(t_l + c_i^{(1)} h, Y_l) \big) \,
	\mathrm{d}W_u^k \bigg\|^p \bigg)
	\nonumber \\
	%
	&\leq \Big( \frac{p}{\sqrt{p-1}} \Big)^p \bigg( \sum_{l=0}^{N-1} \int_{t_l}^{t_{l+1}}
	\bigg[ \Erw \bigg( \bigg| \sum_{k=1}^m \bigg\| \sum_{i=1}^s \beta_i^{(1)}
	\big( b^k(t_l + c_i^{(1)} h, X_{t_l}) 
	\nonumber \\
	&\quad - b^k(t_l + c_i^{(1)} h, Y_l) \big) \,
	\bigg\|^2 \bigg|^{\frac{p}{2}} \bigg) \bigg]^{\frac{2}{p}} \, \mathrm{d}u 
	\bigg)^{\frac{p}{2}}
	\nonumber \\
	&\leq \Big( \frac{p}{\sqrt{p-1}} \Big)^p \bigg( \sum_{l=0}^{N-1} \int_{t_l}^{t_{l+1}}
	\bigg[ \Erw \bigg( \bigg| \sum_{k=1}^m s \sum_{i=1}^s | \beta_i^{(1)} |^2 \,
	\big\| b^k(t_l + c_i^{(1)} h, X_{t_l}) 
	\nonumber \\
	&\quad - b^k(t_l + c_i^{(1)} h, Y_l) \big\|^2 \bigg|^{\frac{p}{2}} \bigg) 
	\bigg]^{\frac{2}{p}} \, \mathrm{d}u \bigg)^{\frac{p}{2}}
	\nonumber \\
	%
	&\leq \Big( \frac{p}{\sqrt{p-1}} \Big)^p \bigg( \sum_{l=0}^{N-1} \int_{t_l}^{t_{l+1}}
	\bigg[ \Erw \bigg( \bigg| \sum_{k=1}^m s^2 \, \czwei^2 \, \| X_{t_l} - Y_l \|^2
	\bigg|^{\frac{p}{2}} \bigg) \bigg]^{\frac{2}{p}} \, \mathrm{d}u \bigg)^{\frac{p}{2}}
	\nonumber \\
	&\leq \Big( \frac{p}{\sqrt{p-1}} \Big)^p \, s^p \, \czwei^p \, m^{\frac{p}{2}}
	\bigg( \sum_{l=0}^{N-1} h \, \big[ \Erw \big( \sup_{0 \leq k \leq l} \| X_{t_k} - Y_k \|^p
	\big) \big]^{\frac{2}{p}} \bigg)^{\frac{p}{2}}
	\nonumber \\
	%
	&\leq \Big( \frac{p}{\sqrt{p-1}} \Big)^p \, s^p \, \czwei^p \, m^{\frac{p}{2}} \,
	(T-t_0)^{\frac{p}{2}-1} \, \sum_{l=0}^{N-1} h \, \Erw \big( 
	\sup_{0 \leq k \leq l} \| X_{t_k} - Y_k \|^p \big) \, .
	\nonumber
\end{align}
%
%
%
The next term \eqref{Proof:MainThm:Teil-B2-2} can be estimated by the following
four terms
%
\begin{align}
	&\Erw \bigg( \sup_{0 \leq n \leq N} \bigg\| 
	\sum_{l=0}^{n-1} \sum_{k=1}^m \sum_{i=1}^s \beta_i^{(1)} \Ii_{(k),l}
	\bigg( \sum_{r=1}^d \frac{\partial}{\partial x_r}
	b^k(t_l + c_i^{(1)} h, X_{t_l}) \, ( {H_{i,l}^{(k),X_{t_l}} }^r - X_{t_l}^r ) \bigg) 
	\bigg\|^p \bigg)
	\nonumber \\
	&= \Erw \bigg( \sup_{0 \leq n \leq N} \bigg\| 
	\sum_{l=0}^{n-1} \sum_{k=1}^m \sum_{i=1}^s \beta_i^{(1)} \Ii_{(k),l}
	\bigg( \sum_{r=1}^d \frac{\partial}{\partial x_r} b^k(t_l + c_i^{(1)} h, X_{t_l})
	\nonumber \\
	&\quad \times
	\bigg[ \sum_{j=1}^s A_{i,j}^{(1)} \, a^r(t_l + c_j^{(0)} h, H_{j,l}^{(0),X_{t_l}}) \, h
	+ \sum_{j=1}^{i-1} \sum_{q=1}^m B_{i,j}^{(1)} \, 
	b^{r,q}(t_l + c_j^{(1)} h, H_{j,l}^{(q),X_{t_l}}) \, \Ii_{(q,k),l} \bigg] 
	\bigg) \bigg\|^p \bigg)
	\nonumber \\
	&= \Erw \bigg( \sup_{0 \leq n \leq N} \bigg\|
	\sum_{l=0}^{n-1} \sum_{k=1}^m \sum_{i=1}^s \beta_i^{(1)} \Ii_{(k),l}
	\bigg( \sum_{r=1}^d \frac{\partial}{\partial x_r} b^k(t_l + c_i^{(1)} h, X_{t_l})
	\nonumber \\
	&\quad \times
	\bigg[ \sum_{j=1}^s A_{i,j}^{(1)} \, a^r(t_l + c_j^{(0)} h, X_{t_l}) \, h
	\nonumber \\
	&\quad + \sum_{j=1}^s A_{i,j}^{(1)} \int_0^1 \sum_{v=1}^d \frac{\partial}{\partial x_v}
	a^r(t_l + c_j^{(0)} h, X_{t_l} + u(H_{j,l}^{(0),X_{t_l}}-X_{t_l})) \,
	( {H_{j,l}^{(0),X_{t_l}} }^v - X_{t_l}^v) \, \mathrm{d}u \, h
	\nonumber \\
	&\quad + \sum_{j=1}^{i-1} \sum_{q=1}^m B_{i,j}^{(1)} \, 
	b^{r,q}(t_l + c_j^{(1)} h, X_{t_l}) \, \Ii_{(q,k),l} 
	+ \sum_{j=1}^{i-1} \sum_{q=1}^m B_{i,j}^{(1)} 
	\nonumber \\
	&\quad \times \int_0^1 \sum_{v=1}^d
	\frac{\partial}{\partial x_v}
	b^{r,q}(t_l + c_j^{(1)} h, X_{t_l} + u(H_{j,l}^{(q),X_{t_l}}-X_{t_l})) \, 
	( {H_{j,l}^{(q),X_{t_l}} }^v -X_{t_l}^v) \, \mathrm{d}u \, \Ii_{(q,k),l}
	\bigg] \bigg) \bigg\|^p \bigg)
	\nonumber \\
	%
	&\leq 4^{p-1} \, \Erw \bigg( \sup_{0 \leq n \leq N} \bigg\|
	\sum_{l=0}^{n-1} \sum_{k=1}^m \sum_{i=1}^s \beta_i^{(1)} \Ii_{(k),l}
	\sum_{r=1}^d \frac{\partial}{\partial x_r} b^k(t_l + c_i^{(1)} h, X_{t_l})
	\nonumber \\
	&\quad \times 
	\sum_{j=1}^s A_{i,j}^{(1)} \, a^r(t_l + c_j^{(0)} h, X_{t_l}) \, h \bigg\|^p \bigg)
	\label{Proof:MainThm:Teil-B2-2a} \\
	&\quad + 4^{p-1} \, \Erw \bigg( \sup_{0 \leq n \leq N} \bigg\|
	\sum_{l=0}^{n-1} \sum_{k=1}^m \sum_{i=1}^s \beta_i^{(1)} \Ii_{(k),l}
	\sum_{r=1}^d \frac{\partial}{\partial x_r} b^k(t_l + c_i^{(1)} h, X_{t_l})
	\nonumber \\
	&\quad \times 
	\sum_{j=1}^s A_{i,j}^{(1)} \int_0^1 \sum_{v=1}^d \frac{\partial}{\partial x_v}
	a^r(t_l + c_j^{(0)} h, X_{t_l} + u(H_{j,l}^{(0),X_{t_l}}-X_{t_l})) \,
	( {H_{j,l}^{(0),X_{t_l}} }^v - X_{t_l}^v) \, \mathrm{d}u \, h \bigg\|^p \bigg)
	\label{Proof:MainThm:Teil-B2-2b} \\
	&\quad + 4^{p-1} \, \Erw \bigg( \sup_{0 \leq n \leq N} \bigg\|
	\sum_{l=0}^{n-1} \sum_{k=1}^m \sum_{i=1}^s \beta_i^{(1)} \Ii_{(k),l}
	\sum_{r=1}^d \frac{\partial}{\partial x_r} b^k(t_l + c_i^{(1)} h, X_{t_l})
	\nonumber \\
	&\quad \times 
	\sum_{j=1}^{i-1} \sum_{q=1}^m B_{i,j}^{(1)} \, 
	b^{r,q}(t_l + c_j^{(1)} h, X_{t_l}) \, \Ii_{(q,k),l} \bigg\|^p \bigg)
	\label{Proof:MainThm:Teil-B2-2c} \\
	&\quad + 4^{p-1} \, \Erw \bigg( \sup_{0 \leq n \leq N} \bigg\|
	\sum_{l=0}^{n-1} \sum_{k=1}^m \sum_{i=1}^s \beta_i^{(1)} \Ii_{(k),l}
	\sum_{r=1}^d \frac{\partial}{\partial x_r} b^k(t_l + c_i^{(1)} h, X_{t_l})
	\sum_{j=1}^{i-1} \sum_{q=1}^m B_{i,j}^{(1)} 
	\nonumber \\
	&\quad \times
	\int_0^1 \sum_{v=1}^d \frac{\partial}{\partial x_v}
	b^{r,q}(t_l + c_j^{(1)} h, X_{t_l} + u(H_{j,l}^{(q),X_{t_l}}-X_{t_l})) \, 
	( {H_{j,l}^{(q),X_{t_l}} }^v -X_{t_l}^v) \, \mathrm{d}u \, \Ii_{(q,k),l} \bigg\|^p 
	\bigg) \, .
	\label{Proof:MainThm:Teil-B2-2d}
\end{align}
%
%
%
We proceed with term \eqref{Proof:MainThm:Teil-B2-2a}. With Burkholder's 
inequality, see, e.g., \cite{Burk88} or \cite[Prop.~2.1 \& 2.2]{PlRoe21},
\eqref{Assumption-a-bk:Bound-derivative-1}, \eqref{Assumption-a-bk:lin-growth}
and Lemma~\ref{Lem:Lp-bound-SDE-sol} we calculate
%
\begin{align}
	&\Erw \bigg( \sup_{0 \leq n \leq N} \bigg\|
	\sum_{l=0}^{n-1} \sum_{k=1}^m \int_{t_l}^{t_{l+1}} \sum_{i,j=1}^s \beta_i^{(1)}
	\sum_{r=1}^d \frac{\partial}{\partial x_r} b^k(t_l + c_i^{(1)} h, X_{t_l}) \,
	A_{i,j}^{(1)} \, a^r(t_l + c_j^{(0)} h, X_{t_l}) \, h \, \mathrm{d}W_u^k
	\bigg\|^p \bigg)
	\nonumber \\
	%
	&\leq \Big( \frac{p}{\sqrt{p-1}} \Big)^p \bigg( \sum_{l=0}^{N-1} \int_{t_l}^{t_{l+1}}
	\bigg[ \Erw \bigg( \bigg| \sum_{k=1}^m \bigg\| \sum_{i=1}^s \beta_i^{(1)}
	\sum_{r=1}^d \frac{\partial}{\partial x_r} b^k(t_l + c_i^{(1)} h, X_{t_l}) 
	\nonumber \\
	&\quad \times \sum_{j=1}^s A_{i,j}^{(1)} \, a^r(t_l + c_j^{(0)} h, X_{t_l}) \, h 
	\bigg\|^2 \bigg|^{\frac{p}{2}}
	\bigg) \bigg]^{\frac{2}{p}} \, \mathrm{d}u \bigg)^{\frac{p}{2}}
	\nonumber \\
	%
	&\leq \Big( \frac{p}{\sqrt{p-1}} \Big)^p \bigg( \sum_{l=0}^{N-1} h \,
	\bigg[ \Erw \bigg( \bigg| \sum_{k=1}^m s \sum_{i=1}^s | \beta_i^{(1)} |^2 \, d
	\sum_{r=1}^d s \sum_{j=1}^s | A_{i,j}^{(1)} |^2 \, h^2 
	\nonumber \\ 
	&\quad \times 
	\Big\| \frac{\partial}{\partial x_r} b^k(t_l + c_i^{(1)} h, X_{t_l}) \Big\|^2
	\, | a^r(t_l + c_j^{(0)} h, X_{t_l}) |^2 \bigg|^{\frac{p}{2}} \bigg) 
	\bigg]^{\frac{2}{p}} \bigg)^{\frac{p}{2}}
	\nonumber \\
	%
	&\leq \Big( \frac{p}{\sqrt{p-1}} \Big)^p \bigg( \sum_{l=0}^{N-1} h \,
	\bigg[ \Erw \bigg( \bigg| \sum_{k=1}^m \sum_{i=1}^s \sum_{r=1}^d \sum_{j=1}^s
	s^2 \, \czwei^4 \, d \, h^2 \, \cc^4 \, (1+ \| X_{t_l} \|^2 ) \bigg|^{\frac{p}{2}}
	\bigg) \bigg]^{\frac{2}{p}} \bigg)^{\frac{p}{2}}
	\nonumber \\
	&\leq \Big( \frac{p}{\sqrt{p-1}} \Big)^p \bigg( \sum_{l=0}^{N-1} h \, m \, s^4 \,
	d^2 \, \czwei^4 \, h^2 \, c^4 \, \big[ \Erw \big( 2^{\frac{p}{2}-1} + 2^{\frac{p}{2}-1}
	\big( \sup_{t_0 \leq t \leq T} \| X_t \|^2 \big)^{\frac{p}{2}} \big) \big]^{\frac{2}{p}}
	\big)^{\frac{p}{2}}
	\nonumber \\
	%
	&\leq \Big( \frac{p}{\sqrt{p-1}} \Big)^p \, (T-t_0)^{\frac{p}{2}} \, (m \, s^4 \, d^2 \,
	\czwei^4 \, \cc^4)^{\frac{p}{2}} \, \big[ 2^{\frac{p}{2}-1} + 2^{\frac{p}{2}-1} \,
	\Erw \big( \sup_{t_0 \leq t \leq T} \| X_t \|^p \big) \big] \, h^p
	\nonumber \\
	&\leq \Big( \frac{p}{\sqrt{p-1}} \Big)^p \, (T-t_0)^{\frac{p}{2}} \, (m \, s^4 \, d^2 \,
	\czwei^4 \, \cc^4)^{\frac{p}{2}} \, \big[ 2^{\frac{p}{2}-1} + 2^{\frac{p}{2}-1} \,
	\ccp (1+ \Erw( \|X_{t_0} \|^p)) \big] \, h^p \, .
	\nonumber
\end{align}
%
%
%
Next, consider term \eqref{Proof:MainThm:Teil-B2-2b}. With
Lemma~\ref{Lem:Ij-Iij-Moment-estimate}, Lemma~\ref{Lem:H0-Xtl-estimate} and
Lemma~\ref{Lem:Lp-bound-SDE-sol} it follows that
%
\begin{align}
	&\Erw \bigg( \sup_{0 \leq n \leq N} \bigg\|
	\sum_{l=0}^{n-1} \sum_{k=1}^m \sum_{i=1}^s \beta_i^{(1)} \Ii_{(k),l}
	\sum_{r=1}^d \frac{\partial}{\partial x_r} b^k(t_l + c_i^{(1)} h, X_{t_l})
	\nonumber \\
	&\quad \times 
	\sum_{j=1}^s A_{i,j}^{(1)} \int_0^1 \sum_{v=1}^d \frac{\partial}{\partial x_v}
	a^r(t_l + c_j^{(0)} h, X_{t_l} + u(H_{j,l}^{(0),X_{t_l}}-X_{t_l})) \,
	( {H_{j,l}^{(0),X_{t_l}} }^v - X_{t_l}^v) \, \mathrm{d}u \, h \bigg\|^p \bigg)
	\nonumber \\
	&\leq \Erw \bigg( \sup_{0 \leq n \leq N} \bigg(
	\sum_{l=0}^{n-1} \sum_{k=1}^m \sum_{i=1}^s | \beta_i^{(1)} | \sum_{r=1}^d
	\sum_{j=1}^s | A_{i,j}^{(1)} | \sum_{v=1}^d 
	\Big\| \frac{\partial}{\partial x_r} b^k(t_l + c_i^{(1)} h, X_{t_l})
	\nonumber \\
	&\quad \times
	\int_0^1 \frac{\partial}{\partial x_v} 
	a^r(t_l + c_j^{(0)} h, X_{t_l} + u(H_{j,l}^{(0),X_{t_l}}-X_{t_l})) \,
	( {H_{j,l}^{(0),X_{t_l}} }^v - X_{t_l}^v) \, \mathrm{d}u \Big\| \, h \, | \Ii_{(k),l} | 
	\bigg)^p \bigg)
	\nonumber \\
	&\leq \Erw \bigg( \bigg( \sum_{l=0}^{N-1} \sum_{k=1}^m \sum_{i=1}^s
	\sum_{r=1}^d \sum_{j=1}^s \sum_{v=1}^d \czwei^2 
	\Big\| \frac{\partial}{\partial x_r} b^k(t_l + c_i^{(1)} h, X_{t_l}) \Big\|
	\nonumber \\
	&\quad \times
	\int_0^1 \Big| \frac{\partial}{\partial x_v} 
	a^r(t_l + c_j^{(0)} h, X_{t_l} + u(H_{j,l}^{(0),X_{t_l}}-X_{t_l})) \Big| \,
	| {H_{j,l}^{(0),X_{t_l}} }^v - X_{t_l}^v | \, \mathrm{d}u \, h \, | \Ii_{(k),l} |  \bigg)^p
	\bigg)
	\nonumber \\
	&\leq \Erw \bigg( \bigg( \sum_{l=0}^{N-1} \sum_{k=1}^m s \, d \sum_{j=1}^s
	\sum_{v=1}^d \czwei^2 \, \cc^2 \, \| H_{j,l}^{(0),X_{t_l}} - X_{t_l} \| \, h \, 
	| \Ii_{(k),l} | \bigg)^p \bigg)
	\nonumber \\
	%
	&\leq (s \, d^2 \, \czwei^2 \, \cc^2 )^p \, N^{p-1} \sum_{l=0}^{N-1} h^p
	m^{p-1} \sum_{k=1}^m \sum_{j=1}^s \Erw \big( | \Ii_{(k),l} |^p \big) \, 
	\Erw \big( \| H_{j,l}^{(0),X_{t_l}} - X_{t_l} \|^p \big)
	\nonumber \\
	&\leq (m \, s \, d^2 \, \czwei^2 \, \cc^2 )^p \, N^{p-1} \sum_{l=0}^{N-1} h^p \, 
	(p-1)^p \, h^{\frac{p}{2}} \, \cMHdetInc (1+ \Erw( \|X_{t_l} \|^p)) \, h^p
	\nonumber \\
	&\leq (m \, s \, d^2 \, \czwei^2 \, \cc^2 )^p \, (T-t_0)^p \, (p-1)^p \, h^{\frac{p}{2}} \,
	\cMHdetInc (1+ \ccp (1+ \Erw( \|X_{t_0} \|^p)) ) \, h^p \, .
	\nonumber
\end{align}
%
%
%
Now, consider term \eqref{Proof:MainThm:Teil-B2-2c}. Firstly, we split this
term into four terms that are estimated separately. Then, we get
%
\begin{align}
	&\Erw \bigg( \sup_{0 \leq n \leq N} \bigg\|
	\sum_{l=0}^{n-1} \sum_{k=1}^m \sum_{i=1}^s \beta_i^{(1)} \Ii_{(k),l}
	\sum_{r=1}^d \frac{\partial}{\partial x_r} b^k(t_l + c_i^{(1)} h, X_{t_l})
	\nonumber \\
	&\quad \times 
	\sum_{j=1}^{i-1} \sum_{q=1}^m B_{i,j}^{(1)} \, 
	b^{r,q}(t_l + c_j^{(1)} h, X_{t_l}) \, \Ii_{(q,k),l} \bigg\|^p \bigg)
	\nonumber \\
	&= \Erw \bigg( \sup_{0 \leq n \leq N} \bigg\|
	\sum_{l=0}^{n-1} \sum_{k=1}^m \sum_{i=1}^s \beta_i^{(1)} \Ii_{(k),l}
	\sum_{r=1}^d \bigg[ \frac{\partial}{\partial x_r} b^k(t_l, X_{t_l})
	\nonumber \\
	&\quad + \int_0^1 \frac{\partial^2}{\partial x_r \partial t} 
	b^k(t_l + u (c_i^{(1)} h), X_{t_l}) \, c_i^{(1)} h \, \mathrm{d}u \bigg]
	\nonumber \\
	&\quad \times
	\sum_{j=1}^{i-1} \sum_{q=1}^m B_{i,j}^{(1)} \bigg[ b^{r,q}(t_l, X_{t_l})
	+ \int_0^1 \frac{\partial}{\partial t} b^{r,q}(t_l + u(c_j^{(1)} h), X_{t_l}) \, 
	c_j^{(1)} h \, \mathrm{d}u \bigg] \, \Ii_{(q,k),l} \bigg\|^p \bigg)
	\nonumber \\
	%
	&\leq 4^{p-1} \, \Erw \bigg( \sup_{0 \leq n \leq N} \bigg\|
	\sum_{l=0}^{n-1} \sum_{k=1}^m \sum_{i=1}^s \beta_i^{(1)} 
	\sum_{j=1}^{i-1} B_{i,j}^{(1)} \, \Ii_{(k),l}
	\sum_{r=1}^d \frac{\partial}{\partial x_r} b^k(t_l, X_{t_l})
	\sum_{q=1}^m b^{r,q}(t_l, X_{t_l})
	\, \Ii_{(q,k),l} \bigg\|^p \bigg)
	\label{Proof:MainThm:Teil-B2-2c-1} \\
	&\quad + 4^{p-1} \, \Erw \bigg( \sup_{0 \leq n \leq N} \bigg\|
	\sum_{l=0}^{n-1} \sum_{k=1}^m \sum_{i=1}^s \beta_i^{(1)} \Ii_{(k),l}
	\sum_{r=1}^d \frac{\partial}{\partial x_r} b^k(t_l, X_{t_l})
	\nonumber \\
	&\quad \times
	\sum_{j=1}^{i-1} \sum_{q=1}^m B_{i,j}^{(1)} 
	\int_0^1 \frac{\partial}{\partial t} b^{r,q}(t_l + u(c_j^{(1)} h), X_{t_l}) \, 
	c_j^{(1)} h \, \mathrm{d}u \, \Ii_{(q,k),l} \bigg\|^p \bigg)
	\label{Proof:MainThm:Teil-B2-2c-2} \\
	&\quad + 4^{p-1} \, \Erw \bigg( \sup_{0 \leq n \leq N} \bigg\|
	\sum_{l=0}^{n-1} \sum_{k=1}^m \sum_{i=1}^s \beta_i^{(1)} \Ii_{(k),l}
	\sum_{r=1}^d \int_0^1 \frac{\partial^2}{\partial x_r \partial t} 
	b^k(t_l + u (c_i^{(1)} h), X_{t_l}) \, c_i^{(1)} h \, \mathrm{d}u
	\nonumber \\
	&\quad \times
	\sum_{j=1}^{i-1} \sum_{q=1}^m B_{i,j}^{(1)} \, b^{r,q}(t_l, X_{t_l})
	 \, \Ii_{(q,k),l} \bigg\|^p \bigg)
	 \label{Proof:MainThm:Teil-B2-2c-3} \\
	 &\quad + 4^{p-1} \, \Erw \bigg( \sup_{0 \leq n \leq N} \bigg\|
	 \sum_{l=0}^{n-1} \sum_{k=1}^m \sum_{i=1}^s \beta_i^{(1)} \Ii_{(k),l}
	 \sum_{r=1}^d \int_0^1 \frac{\partial^2}{\partial x_r \partial t} 
	 b^k(t_l + u (c_i^{(1)} h), X_{t_l}) \, c_i^{(1)} h \, \mathrm{d}u
	 \nonumber \\
	 &\quad \times
	 \sum_{j=1}^{i-1} \sum_{q=1}^m B_{i,j}^{(1)}
	 \int_0^1 \frac{\partial}{\partial t} b^{r,q}(t_l + u(c_j^{(1)} h), X_{t_l}) \, 
	 c_j^{(1)} h \, \mathrm{d}u \, \Ii_{(q,k),l} \bigg\|^p \bigg)
	 \label{Proof:MainThm:Teil-B2-2c-4}
\end{align}
%
%
%
Term~\eqref{Proof:MainThm:Teil-B2-2c-1} coincides with 
term~\eqref{Proof:MainThm:Teil-A2-1i} except for a constant factor and can be 
estimated exactly in the same way as term~\eqref{Proof:MainThm:Teil-A2-1i}.
%
%
%
Next, we consider term \eqref{Proof:MainThm:Teil-B2-2c-2}. 
With~\eqref{Assumption-a-bk:Bound-derivative-1}, 
\eqref{Assumption-Bound-derivative-1t-and-2t},
Lemma~\ref{Lem:Lp-bound-SDE-sol} and Lemma~\ref{Lem:Ij-Iij-Moment-estimate} 
we get
%
\begin{align}
	&\Erw \bigg( \sup_{0 \leq n \leq N} \bigg\|
	\sum_{l=0}^{n-1} \sum_{k=1}^m \sum_{i=1}^s \beta_i^{(1)} \Ii_{(k),l}
	\sum_{r=1}^d \frac{\partial}{\partial x_r} b^k(t_l, X_{t_l})
	\nonumber \\
	&\quad \times
	\sum_{j=1}^{i-1} \sum_{q=1}^m B_{i,j}^{(1)} 
	\int_0^1 \frac{\partial}{\partial t} b^{r,q}(t_l + u(c_j^{(1)} h), X_{t_l}) \, 
	c_j^{(1)} h \, \mathrm{d}u \, \Ii_{(q,k),l} \bigg\|^p \bigg)
	\nonumber \\
	&\leq \Erw \bigg( \sup_{0 \leq n \leq N} \bigg(
	\sum_{l=0}^{n-1} \sum_{k,q=1}^m \sum_{i=1}^s | \beta_i^{(1)} | \, | \Ii_{(k),l} |
	\sum_{r=1}^d \sum_{j=1}^{i-1} | B_{i,j}^{(1)} | \, 
	\Big\| \frac{\partial}{\partial x_r} b^k(t_l, X_{t_l}) \Big\|
	\nonumber \\
	&\quad \times
	\int_0^1 \Big| \frac{\partial}{\partial t} b^{r,q}(t_l + u(c_j^{(1)} h), X_{t_l}) \Big| \, 
	|c_j^{(1)}| \, h \, \mathrm{d}u \, | \Ii_{(q,k),l} | \bigg)^p \bigg)
	\nonumber \\
	%
	&\leq \Erw \bigg( \bigg(
	\sum_{l=0}^{N-1} \sum_{k,q=1}^m \sum_{i=1}^s \sum_{r=1}^d \sum_{j=1}^{i-1}
	\czwei^3 \, | \Ii_{(k),l} | \, | \Ii_{(q,k),l} | \, \cc ^2 (1+ \| X_{t_l} \|) \, h \bigg)^p \bigg)
	\nonumber \\
	&\leq N^{p-1} \sum_{l=0}^{N-1} h^p \, m^{2(p-1)} \sum_{k,q=1}^m
	(s^2 \, d \, \czwei^3 \, \cc^2)^p \, \Erw \big( | \Ii_{(k),l} |^p \, | \Ii_{(q,k),l} |^p \big) \,
	\Erw \big( 2^{p-1} (1+ \| X_{t_l} \|^p) \big)
	\nonumber \\
	%
	&\leq (T-t_0)^p \, (m^2 \, s^2 \, d \, \czwei^3 \, \cc^2)^p \, 2^{p-1} \,
	(1 + \ccp (1+ \Erw( \|X_{t_0} \|^p)) ) \,
	\frac{(2p -1)^{2p}}{2^{p/2}} \, h^{\frac{3p}{2}} \, .
	\nonumber
\end{align}
%
%
%
We analyse term \eqref{Proof:MainThm:Teil-B2-2c-3}. Under the assumption that
$X_{t_0} \in L^{2p}(\Omega)$ and 
with~\eqref{Assumption-a-bk:lin-growth},
\eqref{Assumption-Bound-derivative-2tx-bk},
Lemma~\ref{Lem:Lp-bound-SDE-sol} and Lemma~\ref{Lem:Ij-Iij-Moment-estimate} 
we get
%
\begin{align}
	&\Erw \bigg( \sup_{0 \leq n \leq N} \bigg\|
	\sum_{l=0}^{n-1} \sum_{k=1}^m \sum_{i=1}^s \beta_i^{(1)} \Ii_{(k),l}
	\sum_{r=1}^d \int_0^1 \frac{\partial^2}{\partial x_r \partial t} 
	b^k(t_l + u (c_i^{(1)} h), X_{t_l}) \, c_i^{(1)} h \, \mathrm{d}u
	\nonumber \\
	&\quad \times
	\sum_{j=1}^{i-1} \sum_{q=1}^m B_{i,j}^{(1)} \, b^{r,q}(t_l, X_{t_l})
	\, \Ii_{(q,k),l} \bigg\|^p \bigg)
	\nonumber \\
	%
	&\leq \Erw \bigg( \sup_{0 \leq n \leq N} \bigg(
	\sum_{l=0}^{n-1} \sum_{k=1}^m \sum_{i=1}^s | \beta_i^{(1)} | \, | \Ii_{(k),l} |
	\sum_{r=1}^d |c_i^{(1)}| \, h \sum_{j=1}^{i-1} \sum_{q=1}^m |B_{i,j}^{(1)}| \,
	|\Ii_{(q,k),l}|
	\nonumber \\
	&\quad \times
	\int_0^1 \Big\| \frac{\partial^2}{\partial x_r \partial t} 
	b^k(t_l + u (c_i^{(1)} h), X_{t_l}) \Big\| \, | b^{r,q}(t_l, X_{t_l}) | \, \mathrm{d}u
	\bigg)^p \bigg)
	\nonumber \\
	%
	&\leq \Erw \bigg( \bigg(
	\sum_{l=0}^{N-1} \sum_{k=1}^m \sum_{i=1}^s \sum_{r=1}^d \sum_{j=1}^{i-1}
	\sum_{q=1}^m \czwei^3 \, h \, | \Ii_{(k),l} | \, |\Ii_{(q,k),l}| \, 
	\cc^2 \, ( 1 + \| X_{t_l} \| )^2 \bigg)^p \bigg)
	\nonumber \\
	%
	&\leq N^{p-1} \sum_{l=0}^{N-1} h^p \, m^{2(p-1)} \sum_{k,q=1}^m
	(s^2 \, d \, \czwei^3 \, \cc^2)^p \,
	\Erw \big( | \Ii_{(k),l} |^p \, |\Ii_{(q,k),l}|^p \big) \,
	2^{2p-1} \big( 1 + \Erw \big( \| X_{t_l} \|^{2p} \big) \big)
	\nonumber \\
	&\leq (T-t_0)^p \, (m^2 \, s^2 \, d \, \czwei^3 \, \cc^2)^p \,
	\frac{(2p -1)^{2p} }{2^{p/2}} \, 
	2^{2p-1} \big( 1 + \ccp (1+ \Erw( \|X_{t_0} \|^{2p})) \big) \, h^{\frac{3p}{2}} \, .
	\nonumber
\end{align}
%
%
%
For term \eqref{Proof:MainThm:Teil-B2-2c-4}, we get with $X_{t_0} \in L^{2p}(\Omega)$,
\eqref{Assumption-Bound-derivative-2tx-bk}, 
\eqref{Assumption-Bound-derivative-1t-and-2t},
Lemma~\ref{Lem:Lp-bound-SDE-sol} and 
Lemma~\ref{Lem:Ij-Iij-Moment-estimate} that
%
\begin{align}
	&\Erw \bigg( \sup_{0 \leq n \leq N} \bigg\|
	\sum_{l=0}^{n-1} \sum_{k=1}^m \sum_{i=1}^s \beta_i^{(1)} \Ii_{(k),l}
	\sum_{r=1}^d \int_0^1 \frac{\partial^2}{\partial x_r \partial t} 
	b^k(t_l + u (c_i^{(1)} h), X_{t_l}) \, c_i^{(1)} h \, \mathrm{d}u
	\nonumber \\
	&\quad \times
	\sum_{j=1}^{i-1} \sum_{q=1}^m B_{i,j}^{(1)}
	\int_0^1 \frac{\partial}{\partial t} b^{r,q}(t_l + u(c_j^{(1)} h), X_{t_l}) \, 
	c_j^{(1)} h \, \mathrm{d}u \, \Ii_{(q,k),l} \bigg\|^p \bigg)
	\nonumber \\
	%
	&\leq \Erw \bigg( \sup_{0 \leq n \leq N} \bigg(
	\sum_{l=0}^{n-1} \sum_{k=1}^m \sum_{i=1}^s |\beta_i^{(1)}| \, |\Ii_{(k),l}|
	\sum_{r=1}^d |c_i^{(1)}| \, h \sum_{j=1}^{i-1} \sum_{q=1}^m |B_{i,j}^{(1)}|
	\, |c_j^{(1)}| \, h \, |\Ii_{(q,k),l}|
	\nonumber \\
	&\quad \times
	\int_0^1 \int_0^1 \Big\| \frac{\partial^2}{\partial x_r \partial t} 
	b^k(t_l + u_1 (c_i^{(1)} h), X_{t_l}) \Big\| \, 
	\Big| \frac{\partial}{\partial t} b^{r,q}(t_l + u_2 (c_j^{(1)} h), X_{t_l}) \Big|
	\, \mathrm{d}u_1 \, \mathrm{d}u_2 \bigg)^p \bigg)
	\nonumber \\
	%
	&\leq \Erw \bigg( \bigg(
	\sum_{l=0}^{N-1} \sum_{k=1}^m \sum_{i=1}^s \czwei^4 \, h^2 \, |\Ii_{(k),l}|
	\sum_{r=1}^d \sum_{j=1}^{i-1} \sum_{q=1}^m |\Ii_{(q,k),l}|
	\int_0^1 \int_0^1 \cc^2 \, ( 1 + \| X_{t_l} \| )^2
	\, \mathrm{d}u_1 \, \mathrm{d}u_2 \bigg)^p \bigg)
	\nonumber \\
	&\leq N^{p-1} \sum_{l=0}^{N-1} h^{2p} \, m^{2(p-1)} \sum_{k,q=1}^m 
	(s^2 \, \czwei^4 \, d \, \cc^2)^p \, \Erw \big( |\Ii_{(k),l}|^p \, |\Ii_{(q,k),l}|^p \big)
	\, 2^{2p-1} (1 + \Erw ( \| X_{t_l} \|^{2p} ) )
	\nonumber \\
	%
	&\leq (T-t_0)^p \, (m^2 \, s^2 \, \czwei^4 \, d \, \cc^2)^p \, 2^{2p-1} \,
	\frac{(2p -1)^{2p}}{2^{p/2}} \, h^{\frac{3p}{2}} \,
	(1 + \ccp (1+ \Erw( \|X_{t_0} \|^{2p})) ) \, h^p \, .
	\nonumber
\end{align}
%
%
%
Term \eqref{Proof:MainThm:Teil-B2-2d} is considered next. 
With~\eqref{Assumption-a-bk:Bound-derivative-1},
Lemma~\ref{Lem:Ij-Iij-H-Z-estimate} and
Lemma~\ref{Lem:Lp-bound-SDE-sol} we get
%
\begin{align}
	&\Erw \bigg( \sup_{0 \leq n \leq N} \bigg\|
	\sum_{l=0}^{n-1} \sum_{k=1}^m \sum_{i=1}^s \beta_i^{(1)} \Ii_{(k),l}
	\sum_{r=1}^d \frac{\partial}{\partial x_r} b^k(t_l + c_i^{(1)} h, X_{t_l})
	\sum_{j=1}^{i-1} \sum_{q=1}^m B_{i,j}^{(1)} 
	\nonumber \\
	&\quad \times
	\int_0^1 \sum_{v=1}^d \frac{\partial}{\partial x_v}
	b^{r,q}(t_l + c_j^{(1)} h, X_{t_l} + u(H_{j,l}^{(q),X_{t_l}}-X_{t_l})) \, 
	( {H_{j,l}^{(q),X_{t_l}} }^v -X_{t_l}^v) \, \mathrm{d}u \, \Ii_{(q,k),l} \bigg\|^p \bigg)
	\nonumber \\
	&\leq \Erw \bigg( \sup_{0 \leq n \leq N} \bigg(
	\sum_{l=0}^{n-1} \sum_{k=1}^m \sum_{i=1}^s |\beta_i^{(1)}| \, |\Ii_{(k),l}|
	\sum_{r=1}^d \sum_{j=1}^{i-1} \sum_{q=1}^m |B_{i,j}^{(1)} | \, |\Ii_{(q,k),l}| 
	\Big\| \frac{\partial}{\partial x_r} b^k(t_l + c_i^{(1)} h, X_{t_l}) \Big\|
	\nonumber \\
	&\quad \times
	\int_0^1 \sum_{v=1}^d \Big| \frac{\partial}{\partial x_v}
	b^{r,q}(t_l + c_j^{(1)} h, X_{t_l} + u(H_{j,l}^{(q),X_{t_l}}-X_{t_l})) \Big| \, 
	| {H_{j,l}^{(q),X_{t_l}} }^v -X_{t_l}^v | \, \mathrm{d}u \bigg)^p \bigg)
	\nonumber \\
	&\leq \Erw \bigg( \bigg( \sum_{l=0}^{N-1} \sum_{k=1}^m s \, \czwei^2 \, d \, 
	\cc^2 \, |\Ii_{(k),l}| \sum_{j=1}^{s-1} \sum_{q=1}^m |\Ii_{(q,k),l}|  \, d \,
	\| H_{j,l}^{(q),X_{t_l}} -X_{t_l} \|  \bigg)^p \bigg)
	\nonumber \\
	&\leq N^{p-1} \sum_{l=0}^{N-1} m^{2(p-1)} \sum_{k,q=1}^m (s \, \czwei^2
	\, d^2 \, \cc^2)^p \, s^{p-1} \, \sum_{j=1}^{s-1} \Erw \big( |\Ii_{(k),l}|^p \, 
	|\Ii_{(q,k),l}|^p \, \| H_{j,l}^{(q),X_{t_l}} -X_{t_l} \|^p \big)
	\nonumber \\
	%
	&\leq N^{p-1} \sum_{l=0}^{N-1} (m^2 \, s^2 \, \czwei^2 \, d^2 \, \cc^2)^p \,
	\cMH (1+ \Erw( \| X_{t_l} \|^{p}) ) \, h^{\frac{p}{2} + 2p}
	\nonumber \\
	&\leq (T-t_0)^p \, (m^2 \, s^2 \, \czwei^2 \, d^2 \, \cc^2)^p \,
	\cMH (1+ \ccp (1+ \Erw( \|X_{t_0} \|^p)) ) \, h^{\frac{3p}{2}} \, .
	\nonumber
\end{align}
Thus, we have completed the estimate for term~\eqref{Proof:MainThm:Teil-B2-2}.
%
%
%
Next, we proceed with term \eqref{Proof:MainThm:Teil-B2-3}. One observes that
term \eqref{Proof:MainThm:Teil-B2-3} coincides with term 
\eqref{Proof:MainThm:Teil-B2-2} if $X_{t_l}$ is replaced by $Y_l$ in 
\eqref{Proof:MainThm:Teil-B2-2}. Therefore,
term \eqref{Proof:MainThm:Teil-B2-3} can be estimated in the same way as
term \eqref{Proof:MainThm:Teil-B2-2} if we apply
Proposition~\ref{Prop:Lp-bound-Approximation} instead of 
Lemma~\ref{Lem:Lp-bound-SDE-sol}.
%
%
%
We proceed with term \eqref{Proof:MainThm:Teil-B2-4}. For 
$X_{t_0} \in L^{2p}(\Omega)$ and 
with~\eqref{Assumption-a-bk:Bound-derivative2-x},
Lemma~\ref{Lem:Ij-Iij-H-Z-estimate} and 
Lemma~\ref{Lem:Lp-bound-SDE-sol} we get
%
\begin{align}
	&\Erw \bigg( \sup_{0 \leq n \leq N} \bigg\| 
	\sum_{l=0}^{n-1} \sum_{k=1}^m \sum_{i=1}^s \beta_i^{(1)} \Ii_{(k),l}
	\int_0^1 \sum_{q,r=1}^d \frac{\partial^2}{\partial x_q \partial x_r}
	b^k(t_l + c_i^{(1)} h, X_{t_l} + u(H_{i,l}^{(k),X_{t_l}} - X_{t_l}))
	\nonumber \\
	&\quad \times ( {H_{i,l}^{(k),X_{t_l}} }^r - X_{t_l}^r ) \, 
	( {H_{i,l}^{(k),X_{t_l}} }^q - X_{t_l}^q ) \, (1-u) \, \mathrm{d}u \bigg\|^p \bigg)
	\nonumber \\
	&\leq \Erw \bigg( \sup_{0 \leq n \leq N} \bigg(
	\sum_{l=0}^{n-1} \sum_{k=1}^m \sum_{i=1}^s |\beta_i^{(1)}| 
	|\Ii_{(k),l}|
	\int_0^1 \sum_{q,r=1}^d \Big\| \frac{\partial^2}{\partial x_q \partial x_r}
	b^k(t_l + c_i^{(1)} h, X_{t_l} + u(H_{i,l}^{(k),X_{t_l}} - X_{t_l})) \Big\|
	\nonumber \\
	&\quad \times
	| {H_{i,l}^{(k),X_{t_l}} }^r - X_{t_l}^r | \, 
	| {H_{i,l}^{(k),X_{t_l}} }^q - X_{t_l}^q | \, (1-u) \, \mathrm{d}u \bigg)^p \bigg)
	\nonumber \\
	&\leq \Erw \bigg( \bigg( 
	\sum_{l=0}^{N-1} \sum_{k=1}^m \sum_{i=1}^s \czwei \, |\Ii_{(k),l}|
	\int_0^1 (1-u) \, \mathrm{d}u \sum_{q,r=1}^d \cc \,
	\| H_{i,l}^{(k),X_{t_l}} - X_{t_l} \|^2 \bigg)^p \bigg)
	\nonumber \\
	&\leq N^{p-1} \sum_{l=0}^{N-1} m^{p-1} \sum_{k=1}^m s^{p-1} \sum_{i=1}^s
	\czwei^p \, 2^{-p} \, d^{2p} \, \cc^p \,
	\Erw \big( |\Ii_{(k),l}|^p \, \| H_{i,l}^{(k),X_{t_l}} - X_{t_l} \|^{2p} \big)
	\nonumber \\
	%
	&\leq N^{p-1} \sum_{l=0}^{N-1} (m \, s \, \czwei \, 2^{-1} \, d^2 \, \cc)^p \,
	\cMH (1+ \Erw( \| X_{t_l} \|^{2p}) ) \, h^{\frac{p}{2} + 2p}
	\nonumber \\
	&\leq (T-t_0)^p \, (m \, s \, \czwei \, 2^{-1} \, d^2 \, \cc)^p \,
	\cMH (1+ \ccp (1+ \Erw( \|X_{t_0} \|^{2p})) ) \, h^{\frac{3p}{2}} \, .
	\nonumber
\end{align}
%
%
%
Term~\eqref{Proof:MainThm:Teil-B2-5} coincides with 
term~\eqref{Proof:MainThm:Teil-B2-4} if $X_{t_l}$ is replaced by $Y_l$ in
\eqref{Proof:MainThm:Teil-B2-4}.
Therefore, \eqref{Proof:MainThm:Teil-B2-5} can be estimated the same way
as term \eqref{Proof:MainThm:Teil-B2-4} if
Proposition~\ref{Prop:Lp-bound-Approximation} is applied instead of
Lemma~\ref{Lem:Lp-bound-SDE-sol}.
%
%
%
Next, we consider term~\eqref{Proof:MainThm:Teil-B3} which we split into
four terms that are estimated separately in the following. So, we get for
term~\eqref{Proof:MainThm:Teil-B3} that
%
\begin{align}
	&\Erw \bigg( \sup_{0 \leq n \leq N} \bigg\| 
	\sum_{k=1}^m \sum_{l=0}^{n-1} \sum_{i=1}^s \beta_i^{(2)}
	\big( b^k(t_l + c_i^{(1)} h, H_{i,l}^{(k),X_{t_l}})
	- b^k(t_l + c_i^{(1)} h, H_{i,l}^{(k),Y_l}) \big) \bigg\|^p \bigg)
	\nonumber \\
	&= \Erw \bigg( \sup_{0 \leq n \leq N} \bigg\| 
	\sum_{k=1}^m \sum_{l=0}^{n-1} \sum_{i=1}^s \beta_i^{(2)}
	\bigg( b^k(t_l + c_i^{(1)} h, X_{t_l}) + \sum_{r=1}^d \frac{\partial}{\partial x_r}
	b^k(t_l + c_i^{(1)} h, X_{t_l}) \, ( {H_{i,l}^{(k),X_{t_l}} }^r - X_{t_l}^r )
	\nonumber \\
	&\quad
	+ \int_0^1 \sum_{q,r=1}^d \frac{\partial^2}{\partial x_q \partial x_r}
	b^k(t_l + c_i^{(1)} h, X_{t_l} + u(H_{i,l}^{(k),X_{t_l}}-X_{t_l}) )
	\nonumber \\
	&\quad \times
	( {H_{i,l}^{(k),X_{t_l}} }^q - X_{t_l}^q ) \, ( {H_{i,l}^{(k),X_{t_l}} }^r - X_{t_l}^r )
	\, (1-u) \, \mathrm{d}u
	\nonumber \\
	&\quad
	- \bigg( b^k(t_l + c_i^{(1)} h, Y_l) + \sum_{r=1}^d \frac{\partial}{\partial x_r}
	b^k(t_l + c_i^{(1)} h, Y_l) \, ( {H_{i,l}^{(k),Y_l}}^r-Y_l^r )
	\nonumber \\
	&\quad
	+ \int_0^1 \sum_{q,r=1}^d \frac{\partial^2}{\partial x_q \partial x_r}
	b^k(t_l + c_i^{(1)} h, Y_l + u(H_{i,l}^{(k),Y_l}-Y_l) )
	\nonumber \\
	&\quad \times
	( {H_{i,l}^{(k),Y_l}}^q-Y_l^q ) \, ( {H_{i,l}^{(k),Y_l}}^r-Y_l^r ) \,
	(1-u) \, \mathrm{d}u \bigg) \bigg) \bigg\|^p \bigg)
	\nonumber \\
	%
	&\leq 4^{p-1} \, \Erw \bigg( \sup_{0 \leq n \leq N} \bigg\|
	\sum_{k=1}^m \sum_{l=0}^{n-1} \sum_{i=1}^s \beta_i^{(2)}
	\big( b^k(t_l + c_i^{(1)} h, X_{t_l}) - b^k(t_l + c_i^{(1)} h, Y_l) \big) 
	\bigg\|^p \bigg)
	\label{Proof:MainThm:Teil-B3-1} \\
	&\quad
	+ 4^{p-1} \, \Erw \bigg( \sup_{0 \leq n \leq N} \bigg\|
	\sum_{k=1}^m \sum_{l=0}^{n-1} \sum_{i=1}^s \beta_i^{(2)}
	\bigg( \sum_{r=1}^d \frac{\partial}{\partial x_r}
	b^k(t_l + c_i^{(1)} h, X_{t_l}) \, ( {H_{i,l}^{(k),X_{t_l}} }^r - X_{t_l}^r )
	\nonumber \\
	&\quad
	- \sum_{r=1}^d \frac{\partial}{\partial x_r}
	b^k(t_l + c_i^{(1)} h, Y_l) \, ( {H_{i,l}^{(k),Y_l}}^r-Y_l^r )
	\bigg) \bigg\|^p \bigg)
	\label{Proof:MainThm:Teil-B3-2} \\
	&\quad
	+ 4^{p-1} \, \Erw \bigg( \sup_{0 \leq n \leq N} \bigg\|
	\sum_{k=1}^m \sum_{l=0}^{n-1} \sum_{i=1}^s \beta_i^{(2)}
	\int_0^1 \sum_{q,r=1}^d \frac{\partial^2}{\partial x_q \partial x_r}
	b^k(t_l + c_i^{(1)} h, X_{t_l} + u(H_{i,l}^{(k),X_{t_l}}-X_{t_l}) )
	\nonumber \\
	&\quad \times
	( {H_{i,l}^{(k),X_{t_l}} }^q - X_{t_l}^q ) \, ( {H_{i,l}^{(k),X_{t_l}} }^r - X_{t_l}^r )
	\, (1-u) \, \mathrm{d}u \bigg\|^p \bigg)
	\label{Proof:MainThm:Teil-B3-3} \\
	&\quad
	+ 4^{p-1} \, \Erw \bigg( \sup_{0 \leq n \leq N} \bigg\|
	\sum_{k=1}^m \sum_{l=0}^{n-1} \sum_{i=1}^s \beta_i^{(2)}
	\int_0^1 \sum_{q,r=1}^d \frac{\partial^2}{\partial x_q \partial x_r}
	b^k(t_l + c_i^{(1)} h, Y_l + u(H_{i,l}^{(k),Y_l}-Y_l) )
	\nonumber \\
	&\quad \times
	( {H_{i,l}^{(k),Y_l}}^q-Y_l^q ) \, ( {H_{i,l}^{(k),Y_l}}^r-Y_l^r ) \,
	(1-u) \, \mathrm{d}u \bigg\|^p \bigg) \, .
	\label{Proof:MainThm:Teil-B3-4}
\end{align}
%
%
%
Firstly, we consider term \eqref{Proof:MainThm:Teil-B3-1}. Under the assumption
that $\sum_{i=1}^s \beta_i^{(2)} = 0$ and $\sum_{i=1}^s \beta_i^{(2)} c_i^{(1)} =0$
and with~\eqref{Assumption-Bound-derivative-1t-and-2t},
Lemma~\ref{Lem:Lp-bound-SDE-sol}
and Proposition~\ref{Prop:Lp-bound-Approximation} we calculate 
%
\begin{align}
	&\Erw \bigg( \sup_{0 \leq n \leq N} \bigg\|
	\sum_{k=1}^m \sum_{l=0}^{n-1} \sum_{i=1}^s \beta_i^{(2)}
	\big( b^k(t_l + c_i^{(1)} h, X_{t_l}) - b^k(t_l + c_i^{(1)} h, Y_l) \big) 
	\bigg\|^p \bigg)
	\nonumber \\
	&= \Erw \bigg( \sup_{0 \leq n \leq N} \bigg\|
	\sum_{k=1}^m \sum_{l=0}^{n-1} \sum_{i=1}^s \beta_i^{(2)} 
	\bigg( b^k(t_l, X_{t_l}) + \frac{\partial}{\partial t} b^k(t_l, X_{t_l}) \, c_i^{(1)} h
	\nonumber \\
	&\quad
	+ \int_0^1 \frac{\partial^2}{\partial t^2} b^k(t_l + u \, c_i^{(1)} h, X_{t_l})
	\, ( c_i^{(1)} h )^2 \, (1-u) \, \mathrm{d}u
	- b^k(t_l, Y_l) - \frac{\partial}{\partial t} b^k(t_l, Y_l) \, c_i^{(1)} h
	\nonumber \\
	&\quad
	- \int_0^1 \frac{\partial^2}{\partial t^2} b^k(t_l + u \, c_i^{(1)} h, Y_l)
	\, ( c_i^{(1)} h )^2 \, (1-u) \, \mathrm{d}u \bigg) \bigg\|^p \bigg)
	\nonumber \\
	&= \Erw \bigg( \sup_{0 \leq n \leq N} \bigg\|
	\sum_{k=1}^m \sum_{l=0}^{n-1} \sum_{i=1}^s \beta_i^{(2)} 
	\int_0^1 \Big( \frac{\partial^2}{\partial t^2} b^k(t_l + u \, c_i^{(1)} h, X_{t_l})
	\nonumber \\
	&\quad
	-\frac{\partial^2}{\partial t^2} b^k(t_l + u \, c_i^{(1)} h, Y_l) \Big)
	\, ( c_i^{(1)} h )^2 \, (1-u) \, \mathrm{d}u \bigg) \bigg\|^p \bigg)
	\nonumber \\
	%
	&\leq \Erw \bigg( \sup_{0 \leq n \leq N} \bigg(
	\sum_{k=1}^m \sum_{l=0}^{n-1} \sum_{i=1}^s |\beta_i^{(2)} |
	\int_0^1 \Big( 
	\Big\| \frac{\partial^2}{\partial t^2} b^k(t_l + u \, c_i^{(1)} h, X_{t_l}) \Big\|
	\nonumber \\
	&\quad
	+ \Big\| \frac{\partial^2}{\partial t^2} b^k(t_l + u \, c_i^{(1)} h, Y_l) \Big) \Big\|
	\Big) \, ( c_i^{(1)} h )^2 \, (1-u) \, \mathrm{d}u \bigg) \bigg)^p \bigg)
	\nonumber \\
	%
	&\leq \Erw \bigg( \bigg(
	\sum_{k=1}^m \sum_{l=0}^{N-1} \sum_{i=1}^s \czwei^3 \,
	\big( \cc \, (1+\| X_{t_l} \|) + \cc \, (1+\| Y_l \|) \big) \, h^2
	\int_0^1 (1-u) \, \mathrm{d}u \bigg)^p \bigg)
	\nonumber \\
	&\leq N^{p-1} \sum_{l=0}^{N-1} h^{2p} \, (m \, s \, \czwei^3 \, \cc \, 2^{-1})^p
	\, \Erw \big( 3^{p-1} \big( 2^p + \| X_{t_l} \|^p + \| Y_l \|^p \big) \big)
	\nonumber \\
	&\leq (T-t_0)^p \, (m \, s \, \czwei^3 \, \cc \, 2^{-1})^p \, 3^{p-1} 
	(2^p + \ccp (1+ \Erw( \|X_{t_0} \|^p)) 
	+ \cYMB (1+ \Erw ( \| Y_0 \|^p ) ) ) \, h^p \, .
	\nonumber
\end{align}
%
%
%
The next term to be analysed is \eqref{Proof:MainThm:Teil-B3-2}. For a detailed
estimate, we split term \eqref{Proof:MainThm:Teil-B3-2} into three parts:
%
\begin{align}
	&\Erw \bigg( \sup_{0 \leq n \leq N} \bigg\|
	\sum_{k=1}^m \sum_{l=0}^{n-1} \sum_{i=1}^s \beta_i^{(2)}
	\bigg( \sum_{r=1}^d \frac{\partial}{\partial x_r}
	b^k(t_l + c_i^{(1)} h, X_{t_l}) \, ( {H_{i,l}^{(k),X_{t_l}} }^r - X_{t_l}^r )
	\nonumber \\
	&\quad
	- \sum_{r=1}^d \frac{\partial}{\partial x_r}
	b^k(t_l + c_i^{(1)} h, Y_l) \, ( {H_{i,l}^{(k),Y_l}}^r-Y_l^r )
	\bigg) \bigg\|^p \bigg)
	\nonumber \\
	&= \Erw \bigg( \sup_{0 \leq n \leq N} \bigg\|
	\sum_{k=1}^m \sum_{l=0}^{n-1} \sum_{i=1}^s \beta_i^{(2)}
	\bigg( \bigg( \sum_{r=1}^d \frac{\partial}{\partial x_r} b^k(t_l, X_{t_l})
	\nonumber \\
	&\quad 
	+ \int_0^1 \sum_{r=1}^d \frac{\partial^2}{\partial x_r \partial t}
	b^k(t_l + u \, c_i^{(1)} h, X_{t_l}) \, c_i^{(1)} h \, \mathrm{d}u \bigg)
	\nonumber \\
	&\quad \times
	\bigg( \sum_{j=1}^s A_{i,j}^{(1)} \, a^r(t_l + c_j^{(0)} h, H_{j,l}^{(0),X_{t_l}}) \, h
	+ \sum_{j=1}^{i-1} \sum_{j_1 = 1}^m B_{i,j}^{(1)} \, 
	b^{r,j_1}(t_l + c_j^{(1)} h, H_{j,l}^{(j_1),X_{t_l}}) \, \Ii_{(j_1,k),l} \bigg)
	\nonumber \\
	&\quad
	- \bigg( \sum_{r=1}^d \frac{\partial}{\partial x_r} b^k(t_l, Y_l)
	+ \int_0^1 \sum_{r=1}^d \frac{\partial^2}{\partial x_r \partial t}
	b^k(t_l + u \, c_i^{(1)} h, Y_l) \, c_i^{(1)} h \, \mathrm{d}u \bigg)
	\nonumber \\
	&\quad \times
	\bigg( \sum_{j=1}^s A_{i,j}^{(1)} \, a^r(t_l + c_j^{(0)} h, H_{j,l}^{(0),Y_l}) \, h
	+ \sum_{j=1}^{i-1} \sum_{j_1 = 1}^m B_{i,j}^{(1)} \, 
	b^{r,j_1}(t_l + c_j^{(1)} h, H_{j,l}^{(j_1),Y_l}) \, \Ii_{(j_1,k),l} \bigg) \bigg) 
	\bigg\|^p \bigg)
	\nonumber \\
	%
	&\leq 3^{p-1} \, \Erw \bigg( \sup_{0 \leq n \leq N} \bigg\|
	\sum_{l=0}^{n-1} \sum_{k=1}^m \sum_{i=1}^s \beta_i^{(2)}
	\bigg( \sum_{r=1}^d \frac{\partial}{\partial x_r} b^k(t_l, X_{t_l})
	\bigg( \sum_{j=1}^s A_{i,j}^{(1)} \, a^r(t_l + c_j^{(0)} h, H_{j,l}^{(0),X_{t_l}}) \, h
	\nonumber \\
	&\quad
	+ \sum_{j=1}^{i-1} \sum_{j_1 = 1}^m B_{i,j}^{(1)} \, 
	b^{r,j_1}(t_l + c_j^{(1)} h, H_{j,l}^{(j_1),X_{t_l}}) \, \Ii_{(j_1,k),l} \bigg)
	- \sum_{r=1}^d \frac{\partial}{\partial x_r} b^k(t_l, Y_l)
	\nonumber \\
	&\quad \times	
	\bigg( \sum_{j=1}^s A_{i,j}^{(1)} \, a^r(t_l + c_j^{(0)} h, H_{j,l}^{(0),Y_l}) \, h
	+ \sum_{j=1}^{i-1} \sum_{j_1 = 1}^m B_{i,j}^{(1)} \, 
	b^{r,j_1}(t_l + c_j^{(1)} h, H_{j,l}^{(j_1),Y_l}) \, \Ii_{(j_1,k),l} \bigg) \bigg) 
	\bigg\|^p \bigg)
	\label{Proof:MainThm:Teil-B3-2a} \\
	&\quad + 3^{p-1} \, \Erw \bigg( \sup_{0 \leq n \leq N} \bigg\|
	\sum_{l=0}^{n-1} \sum_{k=1}^m \sum_{i=1}^s \beta_i^{(2)}
	\bigg( \int_0^1 \sum_{r=1}^d \frac{\partial^2}{\partial x_r \partial t}
	b^k(t_l + u \, c_i^{(1)} h, X_{t_l}) \, c_i^{(1)} h \, \mathrm{d}u \bigg)
	\nonumber \\
	&\quad \times
	\bigg( \sum_{j=1}^s A_{i,j}^{(1)} \, a^r(t_l + c_j^{(0)} h, H_{j,l}^{(0),X_{t_l}}) \, h
	+ \sum_{j=1}^{i-1} \sum_{j_1 = 1}^m B_{i,j}^{(1)} \, 
	b^{r,j_1}(t_l + c_j^{(1)} h, H_{j,l}^{(j_1),X_{t_l}}) \, \Ii_{(j_1,k),l} \bigg) 
	\bigg\|^p \bigg)
	\label{Proof:MainThm:Teil-B3-2b} \\
	&\quad + 3^{p-1} \, \Erw \bigg( \sup_{0 \leq n \leq N} \bigg\|
	\sum_{l=0}^{n-1} \sum_{k=1}^m \sum_{i=1}^s \beta_i^{(2)}
	\bigg( \int_0^1 \sum_{r=1}^d \frac{\partial^2}{\partial x_r \partial t}
	b^k(t_l + u \, c_i^{(1)} h, Y_l) \, c_i^{(1)} h \, \mathrm{d}u \bigg)
	\nonumber \\
	&\quad \times
	\bigg( \sum_{j=1}^s A_{i,j}^{(1)} \, a^r(t_l + c_j^{(0)} h, H_{j,l}^{(0),Y_l}) \, h
	+ \sum_{j=1}^{i-1} \sum_{j_1 = 1}^m B_{i,j}^{(1)} \, 
	b^{r,j_1}(t_l + c_j^{(1)} h, H_{j,l}^{(j_1),Y_l}) \, \Ii_{(j_1,k),l} \bigg) 
	\bigg\|^p \bigg) \, .
	\label{Proof:MainThm:Teil-B3-2c}
\end{align}
%
%
%
Now, we start considering term~\eqref{Proof:MainThm:Teil-B3-2a}. As the first
step, we derive a Taylor expansion such that \eqref{Proof:MainThm:Teil-B3-2a}
can be written as
%
\begin{align}
	&\Erw \bigg( \sup_{0 \leq n \leq N} \bigg\|
	\sum_{l=0}^{n-1} \sum_{k=1}^m \sum_{i=1}^s \beta_i^{(2)}
	\bigg( \sum_{r=1}^d \frac{\partial}{\partial x_r} b^k(t_l, X_{t_l})
	\bigg( \sum_{j=1}^s A_{i,j}^{(1)} \, a^r(t_l + c_j^{(0)} h, H_{j,l}^{(0),X_{t_l}}) \, h
	\nonumber \\
	&\quad
	+ \sum_{j=1}^{i-1} \sum_{j_1 = 1}^m B_{i,j}^{(1)} \, 
	b^{r,j_1}(t_l + c_j^{(1)} h, H_{j,l}^{(j_1),X_{t_l}}) \, \Ii_{(j_1,k),l} \bigg)
	- \sum_{r=1}^d \frac{\partial}{\partial x_r} b^k(t_l, Y_l)
	\nonumber \\
	&\quad \times	
	\bigg( \sum_{j=1}^s A_{i,j}^{(1)} \, a^r(t_l + c_j^{(0)} h, H_{j,l}^{(0),Y_l}) \, h
	+ \sum_{j=1}^{i-1} \sum_{j_1 = 1}^m B_{i,j}^{(1)} \, 
	b^{r,j_1}(t_l + c_j^{(1)} h, H_{j,l}^{(j_1),Y_l}) \, \Ii_{(j_1,k),l} \bigg) \bigg) 
	\bigg\|^p \bigg)
	\nonumber \\
	%
	&= \Erw \bigg( \sup_{0 \leq n \leq N} \bigg\|
	\sum_{l=0}^{n-1} \sum_{k=1}^m \sum_{i=1}^s \beta_i^{(2)}
	\bigg( \sum_{r=1}^d \frac{\partial}{\partial x_r} b^k(t_l, X_{t_l})
	\bigg( \sum_{j=1}^s A_{i,j}^{(1)} \, h
	\bigg[ a^r(t_l + c_j^{(0)} h, X_{t_l}) 
	\nonumber \\
	&\quad + \int_0^1 \sum_{q=1}^d \frac{\partial}{\partial x_q}
	a^r(t_l + c_j^{(0)} h, X_{t_l} + u(H_{j,l}^{(0),X_{t_l}}-X_{t_l}))
	\, ( {H_{j,l}^{(0),X_{t_l}} }^q-X_{t_l}^q) \, \mathrm{d}u \bigg]
	\nonumber \\
	&\quad + \sum_{j=1}^{i-1} \sum_{j_1=1}^m B_{i,j}^{(1)} \, \Ii_{(j_1,k),l}
	\bigg[ b^{r,j_1}(t_l + c_j^{(1)} h, X_{t_l})
	\nonumber \\
	&\quad + \int_0^1 \sum_{q=1}^d \frac{\partial}{\partial x_q}
	b^{r,j_1}(t_l + c_j^{(1)} h, X_{t_l} + u(H_{j,l}^{(j_1),X_{t_l}}-X_{t_l}))
	\, ( {H_{j,l}^{(j_1),X_{t_l}} }^q-X_{t_l}^q ) \, \mathrm{d}u \bigg] \bigg)
	\nonumber \\
	&\quad 
	- \sum_{r=1}^d \frac{\partial}{\partial x_r} b^k(t_l, Y_l)
	\bigg( \sum_{j=1}^s A_{i,j}^{(1)} \, h \bigg[ a^r(t_l + c_j^{(0)} h, Y_l)
	\nonumber \\
	&\quad + \int_0^1 \sum_{q=1}^d \frac{\partial}{\partial x_q} 
	a^r(t_l + c_j^{(0)} h, Y_l + u (H_{j,l}^{(0),Y_l}-Y_l)) \,
	( {H_{j,l}^{(0),Y_l} }^q-Y_l^q) \, \mathrm{d}u \bigg]
	\nonumber \\
	&\quad
	+ \sum_{j=1}^{i-1} \sum_{j_1 = 1}^m B_{i,j}^{(1)} \, \Ii_{(j_1,k),l}
	\bigg[ b^{r,j_1}(t_l + c_j^{(1)} h, Y_l)
	\nonumber \\
	&\quad
	+ \int_0^1 \sum_{q=1}^d \frac{\partial}{\partial x_q}
	b^{r,j_1}(t_l + c_j^{(1)} h, Y_l + u(H_{j,l}^{(j_1),Y_l}-Y_l)) \,
	( {H_{j,l}^{(j_1),Y_l} }^q-Y_l^q) \, \mathrm{d}u \bigg] \bigg) \bigg) 
	\bigg\|^p \bigg)
	\nonumber \\
	%
	&= \Erw \bigg( \sup_{0 \leq n \leq N} \bigg\|
	\sum_{l=0}^{n-1} \sum_{k=1}^m \sum_{i=1}^s \beta_i^{(2)}
	\bigg( \sum_{r=1}^d \frac{\partial}{\partial x_r} b^k(t_l, X_{t_l})
	\bigg( \sum_{j=1}^s A_{i,j}^{(1)} \, h \, a^r(t_l, X_{t_l}) 
	\nonumber \\
	&\quad
	+ \sum_{j=1}^s A_{i,j}^{(1)} \, h \int_0^1 \frac{\partial}{\partial t}
	a^r(t_l + u \, c_j^{(0)} h, X_{t_l}) \, c_j^{(0)} h \, \mathrm{d}u
	\nonumber \\
	&\quad
	+ \sum_{j=1}^s A_{i,j}^{(1)} \, h \int_0^1 \sum_{q=1}^d \frac{\partial}{\partial x_q} 
	a^r(t_l + c_j^{(0)} h, X_{t_l} + u(H_{j,l}^{(0),X_{t_l}}-X_{t_l}))
	\, ( {H_{j,l}^{(0),X_{t_l}} }^q-X_{t_l}^q) \, \mathrm{d}u
	\nonumber \\
	&\quad
	+ \sum_{j=1}^{i-1} \sum_{j_1=1}^m B_{i,j}^{(1)} \, \Ii_{(j_1,k),l} 
	b^{r,j_1}(t_l, X_{t_l}) 
	\nonumber \\
	&\quad
	+ \sum_{j=1}^{i-1} \sum_{j_1=1}^m B_{i,j}^{(1)} \, \Ii_{(j_1,k),l} 
	\int_0^1 \frac{\partial}{\partial t} b^{r,j_1}(t_l + u \, c_j^{(1)} h, X_{t_l}) \, 
	c_j^{(1)} h \, \mathrm{d}u
	+ \sum_{j=1}^{i-1} \sum_{j_1=1}^m B_{i,j}^{(1)} \, \Ii_{(j_1,k),l} 
	\nonumber \\
	&\quad \times 
	\int_0^1 \sum_{q=1}^d \frac{\partial}{\partial x_q}
	b^{r,j_1}(t_l + c_j^{(1)} h, X_{t_l} + u(H_{j,l}^{(j_1),X_{t_l}}-X_{t_l}))
	\, ( {H_{j,l}^{(j_1),X_{t_l}} }^q-X_{t_l}^q ) \, \mathrm{d}u \bigg)
	\nonumber \\
	&\quad
	- \sum_{r=1}^d \frac{\partial}{\partial x_r} b^k(t_l, Y_l)
	\bigg( \sum_{j=1}^s A_{i,j}^{(1)} \, h \, a^r(t_l, Y_l)
	+ \sum_{j=1}^s A_{i,j}^{(1)} \, h \int_0^1 \frac{\partial}{\partial t}
	a^r(t_l + u \, c_j^{(0)} h, Y_l) \, c_j^{(0)} h \, \mathrm{d}u
	\nonumber \\
	&\quad
	+ \sum_{j=1}^s A_{i,j}^{(1)} \, h \int_0^1 \sum_{q=1}^d \frac{\partial}{\partial x_q} 
	a^r(t_l + c_j^{(0)} h, Y_l + u (H_{j,l}^{(0),Y_l}-Y_l)) \,
	( {H_{j,l}^{(0),Y_l} }^q-Y_l^q) \, \mathrm{d}u
	\nonumber \\
	&\quad
	+ \sum_{j=1}^{i-1} \sum_{j_1 = 1}^m B_{i,j}^{(1)} \, \Ii_{(j_1,k),l} \, b^{r,j_1}(t_l, Y_l)
	\nonumber \\
	&\quad
	+ \sum_{j=1}^{i-1} \sum_{j_1 = 1}^m B_{i,j}^{(1)} \, \Ii_{(j_1,k),l}
	\int_0^1 \frac{\partial}{\partial t} b^{r,j_1}(t_l + u \, c_j^{(1)} h, Y_l) \, c_j^{(1)} h
	\, \mathrm{d}u
	+ \sum_{j=1}^{i-1} \sum_{j_1 = 1}^m B_{i,j}^{(1)} \, \Ii_{(j_1,k),l}
	\nonumber \\
	&\quad \times
	\int_0^1 \sum_{q=1}^d \frac{\partial}{\partial x_q}
	b^{r,j_1}(t_l + c_j^{(1)} h, Y_l + u(H_{j,l}^{(j_1),Y_l}-Y_l)) \,
	( {H_{j,l}^{(j_1),Y_l} }^q-Y_l^q) \, \mathrm{d}u \bigg) \bigg) \bigg\|^p \bigg)
	\nonumber
\end{align}
In the next step, we estimate the terms from the Taylor expansion of
\eqref{Proof:MainThm:Teil-B3-2a}. Thus, with the
assumption $\sum_{i=1}^s \beta_i^{(2)} \sum_{j=1}^s A_{i,j}^{(1)} = 0$, 
with~\eqref{Assumption-a-bk:Bound-derivative-1}
and~\eqref{Assumption-Bound-derivative-1t-and-2t}
it follows that
%
\begin{align}
	&\Erw \bigg( \sup_{0 \leq n \leq N} \bigg\|
	\sum_{l=0}^{n-1} \sum_{k=1}^m \sum_{i=1}^s \beta_i^{(2)}
	\bigg( \sum_{r=1}^d \frac{\partial}{\partial x_r} b^k(t_l, X_{t_l})
	\bigg( \sum_{j=1}^s A_{i,j}^{(1)} \, a^r(t_l + c_j^{(0)} h, H_{j,l}^{(0),X_{t_l}}) \, h
	\nonumber \\
	&\quad
	+ \sum_{j=1}^{i-1} \sum_{j_1 = 1}^m B_{i,j}^{(1)} \, 
	b^{r,j_1}(t_l + c_j^{(1)} h, H_{j,l}^{(j_1),X_{t_l}}) \, \Ii_{(j_1,k),l} \bigg)
	- \sum_{r=1}^d \frac{\partial}{\partial x_r} b^k(t_l, Y_l)
	\nonumber \\
	&\quad \times	
	\bigg( \sum_{j=1}^s A_{i,j}^{(1)} \, a^r(t_l + c_j^{(0)} h, H_{j,l}^{(0),Y_l}) \, h
	+ \sum_{j=1}^{i-1} \sum_{j_1 = 1}^m B_{i,j}^{(1)} \, 
	b^{r,j_1}(t_l + c_j^{(1)} h, H_{j,l}^{(j_1),Y_l}) \, \Ii_{(j_1,k),l} \bigg) \bigg) 
	\bigg\|^p \bigg)
	\nonumber \\
	%
	&\leq 10^{p-1} \, \Erw \bigg( \sup_{0 \leq n \leq N} \bigg\|
	\sum_{l=0}^{n-1} \sum_{k=1}^m \sum_{i=1}^s \beta_i^{(2)}
	\sum_{j=1}^s A_{i,j}^{(1)} \, h  
	\nonumber \\
	&\quad \times
	\sum_{r=1}^d \Big( \frac{\partial}{\partial x_r}
	b^k(t_l, X_{t_l}) \, a^r(t_l, X_{t_l}) - \frac{\partial}{\partial x_r} b^k(t_l, Y_l)
	\, a^r(t_l, Y_l) \Big) \bigg\|^p \bigg)
	\nonumber \\
	&\quad + 10^{p-1} \, \Erw \bigg( \sup_{0 \leq n \leq N} \bigg(
	\sum_{l=0}^{n-1} h^2 \sum_{k=1}^m \sum_{i=1}^s |\beta_i^{(2)}|
	\sum_{j=1}^s |A_{i,j}^{(1)}| \, |c_j^{(0)}| \sum_{r=1}^d
	\Big\| \frac{\partial}{\partial x_r} b^k(t_l, X_{t_l}) \Big\|
	\nonumber \\
	&\quad \times
	\int_0^1 \Big| \frac{\partial}{\partial t}
	a^r(t_l + u \, c_j^{(0)} h, X_{t_l}) \Big| \, \mathrm{d}u \bigg)^p \bigg)
	\nonumber \\
	&\quad + 10^{p-1} \, \Erw \bigg( \sup_{0 \leq n \leq N} \bigg(
	\sum_{l=0}^{n-1} h \sum_{k=1}^m \sum_{i=1}^s |\beta_i^{(2)}|
	\sum_{j=1}^s |A_{i,j}^{(1)}| \sum_{q,r=1}^d
	\Big\| \frac{\partial}{\partial x_r} b^k(t_l, X_{t_l}) \Big\|
	\nonumber \\
	&\quad \times
	\int_0^1 \Big| \frac{\partial}{\partial x_q}
	a^r(t_l + c_j^{(0)} h, X_{t_l} + u(H_{j,l}^{(0),X_{t_l}}-X_{t_l})) \Big|
	\, | {H_{j,l}^{(0),X_{t_l}} }^q-X_{t_l}^q| \, \mathrm{d}u \bigg)^p \bigg)
	\nonumber \\
	&\quad + 10^{p-1} \, \Erw \bigg( \sup_{0 \leq n \leq N} \bigg\|
	\sum_{l=0}^{n-1} \sum_{j_1,k=1}^m \sum_{i=1}^s \beta_i^{(2)}
	\sum_{j=1}^{i-1} B_{i,j}^{(1)} \, \Ii_{(j_1,k),l} 
	\nonumber \\
	&\quad \times
	\bigg(
	\sum_{r=1}^d \frac{\partial}{\partial x_r} b^k(t_l, X_{t_l}) \, b^{r,j_1}(t_l, X_{t_l}) 
	- \sum_{r=1}^d \frac{\partial}{\partial x_r} b^k(t_l, Y_l) \, b^{r,j_1}(t_l, Y_l) \bigg)
	\bigg\|^p \bigg)
	\nonumber \\
	&\quad + 10^{p-1} \, \Erw \bigg( \sup_{0 \leq n \leq N} \bigg(
	\sum_{l=0}^{n-1} \sum_{j_1,k=1}^m \sum_{i=1}^s |\beta_i^{(2)}|
	\sum_{j=1}^{i-1} |B_{i,j}^{(1)}| \, |\Ii_{(j_1,k),l} |
	\nonumber \\
	&\quad \times
	\sum_{r=1}^d \Big\| \frac{\partial}{\partial x_r} b^k(t_l, X_{t_l}) \Big\|
	\int_0^1 \Big| \frac{\partial}{\partial t} b^{r,j_1}(t_l + u \, c_j^{(1)} h, X_{t_l}) \Big| 
	\, |c_j^{(1)}| \, h \, \mathrm{d}u \bigg)^p \bigg)
	\nonumber \\
	&\quad + 10^{p-1} \, \Erw \bigg( \sup_{0 \leq n \leq N} \bigg(
	\sum_{l=0}^{n-1} \sum_{j_1,k=1}^m \sum_{i=1}^s |\beta_i^{(2)}|
	\sum_{j=1}^{i-1} |B_{i,j}^{(1)}| \, |\Ii_{(j_1,k),l} |
	\sum_{q,r=1}^d \Big\| \frac{\partial}{\partial x_r} b^k(t_l, X_{t_l}) \Big\|
	\nonumber \\
	&\quad \times
	\int_0^1 \Big| \frac{\partial}{\partial x_q}
	b^{r,j_1}(t_l + c_j^{(1)} h, X_{t_l} + u(H_{j,l}^{(j_1),X_{t_l}}-X_{t_l})) \Big| \,
	| {H_{j,l}^{(j_1),X_{t_l}} }^q-X_{t_l}^q | \, \mathrm{d}u \bigg)^p \bigg)
	\nonumber \\
	&\quad + 10^{p-1} \, \Erw \bigg( \sup_{0 \leq n \leq N} \bigg(
	\sum_{l=0}^{n-1} \sum_{k=1}^m \sum_{i=1}^s |\beta_i^{(2)}|
	\sum_{j=1}^s |A_{i,j}^{(1)}| \, |c_j^{(0)}| \, h^2 \sum_{r=1}^d 
	\Big\| \frac{\partial}{\partial x_r} b^k(t_l, Y_l) \Big\|
	\nonumber \\
	&\quad \times
	\int_0^1 \Big| \frac{\partial}{\partial t}
	a^r(t_l + u \, c_j^{(0)} h, Y_l) \Big| \, \mathrm{d}u \bigg)^p \bigg)
	\nonumber \\
	&\quad + 10^{p-1} \, \Erw \bigg( \sup_{0 \leq n \leq N} \bigg(
	\sum_{l=0}^{n-1} \sum_{k=1}^m \sum_{i=1}^s |\beta_i^{(2)}|
	\sum_{j=1}^s |A_{i,j}^{(1)}| \, h \sum_{q,r=1}^d 
	\Big\| \frac{\partial}{\partial x_r} b^k(t_l, Y_l) \Big\|
	\nonumber \\
	&\quad \times
	\int_0^1 \Big| \frac{\partial}{\partial x_q} 
	a^r(t_l + c_j^{(0)} h, Y_l + u (H_{j,l}^{(0),Y_l}-Y_l)) \Big| \,
	| {H_{j,l}^{(0),Y_l} }^q-Y_l^q| \, \mathrm{d}u \bigg)^p \bigg)
	\nonumber \\
	&\quad + 10^{p-1} \, \Erw \bigg( \sup_{0 \leq n \leq N} \bigg(
	\sum_{l=0}^{n-1} \sum_{j_1,k=1}^m \sum_{i=1}^s |\beta_i^{(2)}|
	\sum_{j=1}^{i-1} |B_{i,j}^{(1)}| \, |c_j^{(1)}| \, |\Ii_{(j_1,k),l}| \, h
	\nonumber \\
	&\quad \times
	\sum_{r=1}^d \Big\| \frac{\partial}{\partial x_r} b^k(t_l, Y_l) \Big\|
	\int_0^1 \Big| \frac{\partial}{\partial t} b^{r,j_1}(t_l + u \, c_j^{(1)} h, Y_l) \Big|
	\, \mathrm{d}u \bigg)^p \bigg)
	\nonumber \\
	&\quad + 10^{p-1} \, \Erw \bigg( \sup_{0 \leq n \leq N} \bigg(
	\sum_{l=0}^{n-1} \sum_{j_1,k=1}^m \sum_{i=1}^s |\beta_i^{(2)}|
	\sum_{j=1}^{i-1} |B_{i,j}^{(1)}| \, |\Ii_{(j_1,k),l}|
	\sum_{q,r=1}^d \Big\| \frac{\partial}{\partial x_r} b^k(t_l, Y_l) \Big\|
	\nonumber \\
	&\quad \times
	\int_0^1 \Big| \frac{\partial}{\partial x_q}
	b^{r,j_1}(t_l + c_j^{(1)} h, Y_l + u(H_{j,l}^{(j_1),Y_l}-Y_l)) \Big| \,
	| {H_{j,l}^{(j_1),Y_l} }^q-Y_l^q| \, \mathrm{d}u \bigg)^p \bigg)
	\nonumber \\
	%
	&\leq 10^{p-1} \, N^{p-1} \sum_{l=0}^{N-1} h^{2p} \, (m \, s \, \czwei^3 \, \cc)^p
	\, s^{p-1} \sum_{j=1}^s d^{p-1} \sum_{r=1}^d \Erw \bigg( \bigg( \int_0^1 
	\cc \, (1+\| X_{t_l} \|) \, \mathrm{d}u \bigg)^p \bigg)
	\label{Proof:MainThm:Teil-B3-2a1} \\
	&\quad + 10^{p-1} \, N^{p-1} \sum_{l=0}^{N-1} h^p \, 
	(m \, s \, \czwei^2 \, \cc^2 \, d)^p \, s^{p-1} \sum_{j=1}^s d^{p-1} \sum_{q=1}^d
	\Erw \bigg( \bigg( \int_0^1 \| H_{j,l}^{(0),X_{t_l}}-X_{t_l} \| \, \mathrm{d}u 
	\bigg)^p \bigg)
	\label{Proof:MainThm:Teil-B3-2a2} \\
	%
	%
	&\quad + 10^{p-1} \, \Erw \bigg( \sup_{0 \leq n \leq N} \bigg\|
	\sum_{l=0}^{n-1} \sum_{k=1}^m \sum_{j_1=1}^m 
	\sum_{i=1}^s \beta_i^{(2)} \sum_{j=1}^{i-1} B_{i,j}^{(1)}
	\nonumber \\
	&\quad \times
	\sum_{r=1}^d \Big( b^{r,j_1}(t_l, X_{t_l}) \, \frac{\partial}{\partial x_r} b^k(t_l, X_{t_l})
	- b^{r,j_1}(t_l, Y_l) \, \frac{\partial}{\partial x_r} b^k(t_l, Y_l) \Big)
	\, \Ii_{(j_1,k),l} \bigg\|^p \bigg)
	\label{Proof:MainThm:Teil-B3-2a3} \\
	&\quad + 10^{p-1} \, N^{p-1} \sum_{l=0}^{N-1} h^p \, (s \, \czwei^3 \, \cc)^p
	\, m^{2(p-1)} \sum_{j_1,k=1}^m s^{p-1} \sum_{j=1}^{s-1} d^{p-1} \sum_{r=1}^d
	\Erw \big( |\Ii_{(j_1,k),l}|^p \big)
	\nonumber \\
	&\quad \times
	\Erw \bigg( \bigg( \int_0^1 \cc \, (1+\| X_{t_l} \|) \, \mathrm{d}u \bigg)^p \bigg)
	\label{Proof:MainThm:Teil-B3-2a4} \\
	&\quad + 10^{p-1} \, N^{p-1} \sum_{l=0}^{N-1} m^{2(p-1)} \sum_{j_1,k=1}^m 
	(s \, \czwei^2)^p \, s^{p-1} \sum_{j=1}^{s-1} \cc^{2p} \, d^p 
	\nonumber \\
	&\quad \times 
	d^{p-1} \sum_{q=1}^d \Erw \bigg( |\Ii_{(j_1,k),l}|^p \bigg( \int_0^1 
	\| H_{j,l}^{(j_1),X_{t_l}}-X_{t_l} \| \, \mathrm{d}u \bigg)^p \bigg)
	\label{Proof:MainThm:Teil-B3-2a5} \\
	&\quad + 10^{p-1} \, N^{p-1} \sum_{l=0}^{N-1} h^{2p} \, (m \, s^2 \, \czwei^3 
	\,\cc \, d)^p \, \Erw \bigg( \bigg( \int_0^1 \cc \, (1+\| Y_l \|) \, \mathrm{d}u 
	\bigg)^p \bigg)
	\label{Proof:MainThm:Teil-B3-2a6} \\
	&\quad + 10^{p-1} \, N^{p-1} \sum_{l=0}^{N-1} h^p \, (m \, s \, \czwei^2 \, \cc^2)^p
	\, s^{p-1} \sum_{j=1}^s d^p \, d^{p-1} \sum_{q=1}^d 
	\Erw \bigg( \bigg( \int_0^1 \| H_{j,l}^{(0),Y_l} -Y_l \| \, \mathrm{d}u \bigg)^p \bigg)
	\label{Proof:MainThm:Teil-B3-2a7} \\
	&\quad + 10^{p-1} \, N^{p-1} \sum_{l=0}^{N-1} h^p \, (s^2 \, \czwei^3 \, \cc \, d)^p
	\, m^{2(p-1)} \sum_{j_1,k=1}^m \Erw \bigg( |\Ii_{(j_1,k),l}|^p
	\bigg( \int_0^1 \cc \, (1+\| Y_l \|) \, \mathrm{d}u \bigg)^p \bigg)
	\label{Proof:MainThm:Teil-B3-2a8} \\
	&\quad + 10^{p-1} \, N^{p-1} \sum_{l=0}^{N-1} m^{2(p-1)} \sum_{j_1,k=1}^m
	(s \, \czwei^2 \, \cc^2 \, d)^p 
	\nonumber \\
	&\quad \times 
	d^{p-1} \sum_{q=1}^d s^{p-1} \sum_{j=1}^{s-1}
	\Erw \bigg( |\Ii_{(j_1,k),l}|^p \bigg( \int_0^1 
	\| H_{j,l}^{(j_1),Y_l}-Y_l \| \, \mathrm{d}u \bigg)^p \bigg) \, .
	\label{Proof:MainThm:Teil-B3-2a9}
\end{align}
Now, we consider the third summand~\eqref{Proof:MainThm:Teil-B3-2a3}.
With \eqref{Assumption-bDb-Lip} and Lemma~\ref{Lem:Ij-Iij-Moment-estimate}
we calculate
%
\begin{align}
	&\Erw \bigg( \sup_{0 \leq n \leq N} \bigg\|
	\sum_{l=0}^{n-1} \sum_{k=1}^m \sum_{j_1=1}^m 
	\sum_{i=1}^s \beta_i^{(2)} \sum_{j=1}^{i-1} B_{i,j}^{(1)}
	\nonumber \\
	&\quad \times
	\sum_{r=1}^d \Big( b^{r,j_1}(t_l, X_{t_l}) \, \frac{\partial}{\partial x_r} b^k(t_l, X_{t_l})
	- b^{r,j_1}(t_l, Y_l) \, \frac{\partial}{\partial x_r} b^k(t_l, Y_l) \Big)
	\, \Ii_{(j_1,k),l} \bigg\|^p \bigg)
	\nonumber \\
	&\leq \Erw \bigg( \sup_{0 \leq n \leq N} \bigg(
	\sum_{l=0}^{n-1} \sum_{k=1}^m \sum_{j_1=1}^m 
	\sum_{i=1}^s | \beta_i^{(2)} | \sum_{j=1}^{i-1} | B_{i,j}^{(1)} |
	\nonumber \\
	&\quad \times
	\Big\| \sum_{r=1}^d \Big( b^{r,j_1}(t_l, X_{t_l}) \, \frac{\partial}{\partial x_r} b^k(t_l, X_{t_l})
	- b^{r,j_1}(t_l, Y_l) \, \frac{\partial}{\partial x_r} b^k(t_l, Y_l) \Big\|
	\, | \Ii_{(j_1,k),l} | \bigg)^p \bigg)
	\nonumber \\
	&\leq N^{p-1} \sum_{l=0}^{N-1} m^{p-1} \sum_{k=1}^m m^{p-1} \sum_{j_1=1}^m
	s^{2p} \, \czwei^{2p} \, \cc^p \, \Erw ( \| X_{t_l} - Y_l \|^p) \, \Erw ( | \Ii_{(j_1,k),l} |^p )
	\nonumber \\
	&\leq N^{p-1} \sum_{l=0}^{N-1} m^{2p} \, s^{2p} \, \czwei^{2p} \, \cc^p \, 
	\Erw ( \sup_{0 \leq q \leq l} \| X_{t_q} - Y_q \|^p) \, 
	\Big( \frac{\max \{2,p \}}{\sqrt{2}} \Big)^p \, h^p
	\nonumber \\
	&\leq (T-t_0)^{p-1} \, m^{2p} \, s^{2p} \, \czwei^{2p} \, \cc^p \, 
	\Big( \frac{\max \{2,p \}}{\sqrt{2}} \Big)^p
	\sum_{l=0}^{N-1} h \, \Erw ( \sup_{0 \leq q \leq l} \| X_{t_q} - Y_q \|^p) \, .
	\label{Proof:MainThm:Teil-B3-2a3-A}
\end{align}

Now, we estimate \eqref{Proof:MainThm:Teil-B3-2a} by taking into account
\eqref{Proof:MainThm:Teil-B3-2a1}--\eqref{Proof:MainThm:Teil-B3-2a9}
together with \eqref{Proof:MainThm:Teil-B3-2a3-A}. Then, we calculate with
Lemma~\ref{Lem:Lp-bound-SDE-sol}, Lemma~\ref{Lem:H0-Xtl-estimate},
Lemma~\ref{Lem:Ij-Iij-Moment-estimate},
Lemma~\ref{Lem:Ij-Iij-H-Z-estimate} and
Proposition~\ref{Prop:Lp-bound-Approximation} that
%
\begin{align}
	&\Erw \bigg( \sup_{0 \leq n \leq N} \bigg\|
	\sum_{l=0}^{n-1} \sum_{k=1}^m \sum_{i=1}^s \beta_i^{(2)}
	\bigg( \sum_{r=1}^d \frac{\partial}{\partial x_r} b^k(t_l, X_{t_l})
	\bigg( \sum_{j=1}^s A_{i,j}^{(1)} \, a^r(t_l + c_j^{(0)} h, H_{j,l}^{(0),X_{t_l}}) \, h
	\nonumber \\
	&\quad
	+ \sum_{j=1}^{i-1} \sum_{j_1 = 1}^m B_{i,j}^{(1)} \, 
	b^{r,j_1}(t_l + c_j^{(1)} h, H_{j,l}^{(j_1),X_{t_l}}) \, \Ii_{(j_1,k),l} \bigg)
	- \sum_{r=1}^d \frac{\partial}{\partial x_r} b^k(t_l, Y_l)
	\nonumber \\
	&\quad \times	
	\bigg( \sum_{j=1}^s A_{i,j}^{(1)} \, a^r(t_l + c_j^{(0)} h, H_{j,l}^{(0),Y_l}) \, h
	+ \sum_{j=1}^{i-1} \sum_{j_1 = 1}^m B_{i,j}^{(1)} \, 
	b^{r,j_1}(t_l + c_j^{(1)} h, H_{j,l}^{(j_1),Y_l}) \, \Ii_{(j_1,k),l} \bigg) \bigg) 
	\bigg\|^p \bigg)
	\nonumber \\
	%
	&\leq 10^{p-1} \, (T-t_0)^p \, (m \, s^2 \, \czwei^3 \, \cc^2 \, d)^p \,
	2^{p-1} \, (1+ \ccp (1+ \Erw ( \|X_{t_0} \|^p )) ) \, h^p
	\nonumber \\
	&\quad + 10^{p-1} \, (T-t_0)^p \, (m \, s^2 \, \czwei^2 \, \cc^2 \, d^2)^p \,
	\cMHdetInc (1+ \Erw( \| X_{t_l} \|^p)) \, h^p
	\nonumber \\
	&\quad + 10^{p-1} \, (T-t_0)^{p-1} \, m^{2p} \, \cc^p \, 
	\Big( \frac{\max \{2,p \}}{\sqrt{2}} \Big)^p
	\sum_{l=0}^{N-1} h \, \Erw ( \sup_{0 \leq q \leq l} \| X_{t_q} - Y_q \|^p)
	\nonumber \\
	&\quad + 10^{p-1} \, (T-t_0)^p \, (s^2 \, \czwei^3 \, \cc^2 \, m^2 \, d)^p \, 
	\Big( \frac{p-1}{\sqrt{2}} \Big)^p \, h^p \, 
	2^{p-1} (1+ \ccp (1+ \Erw ( \|X_{t_0} \|^p )) )
	\nonumber \\
	&\quad + 10^{p-1} \, N^{p-1} \sum_{l=0}^{N-1} 
	(m^2 \, s^2 \, \czwei^2 \, \cc^2 \, d^2)^p \, \cMH (1+ \Erw( \| X_{t_l} \|^p) ) \, h^{2p}
	\nonumber \\
	&\quad + 10^{p-1} \, (T-t_0)^p \, (m \, s^2 \, \czwei^3
	\, \cc^2 \, d)^p \, 2^{p-1} (1 + \cYMB (1+ \Erw ( \| Y_0 \|^p )) ) \,  h^p
	\nonumber \\
	&\quad + 10^{p-1} \, N^{p-1} \sum_{l=0}^{N-1} h^p \, 
	(m \, s^2 \, \czwei^2 \, \cc^2 \, d^2)^p \, \cMHdetInc (1+ \Erw( \| Y_l \|^p)) \, h^p
	\nonumber \\
	&\quad + 10^{p-1} \, N^{p-1} \sum_{l=0}^{N-1} h^p \, 
	(m^2 \, s^2 \, \czwei^3 \, \cc^2 \, d)^p \, \Big( \frac{p-1}{\sqrt{2}} \Big)^p \, h^p
	\, 2^{p-1} (1+ \cYMB (1+ \Erw ( \| Y_0 \|^p )) )
	\nonumber \\
	&\quad + 10^{p-1} \, N^{p-1} \sum_{l=0}^{N-1} (m^2 \, s^2 \, \czwei^2 \, \cc^2
	\, d^2)^p \, \cMH (1+ \Erw( \| Y_l \|^p) ) \, h^{2p}
	\nonumber \\
	%
	&\leq 10^{p-1} \, (T-t_0)^p \big( 
	(m \, s^2 \, \czwei^3 \, \cc^2 \, d)^p \, 2^{p-1} \, (1+ \ccp (1+ \Erw ( \|X_{t_0} \|^p )) )
	\nonumber \\
	&\quad + (m \, s^2 \, \czwei^2 \, \cc^2 \, d^2)^p \,
	\cMHdetInc (1+ \ccp (1+ \Erw ( \|X_{t_0} \|^p )) )
	\nonumber \\
	&\quad + (s^2 \, \czwei^3 \, \cc^2 \, m^2 \, d)^p \, (p-1)^p
	\, 2^{\frac{p}{2}-1} (1+ \ccp (1+ \Erw ( \|X_{t_0} \|^p )) )
	\nonumber \\
	&\quad + (m^2 \, s^2 \, \czwei^2 \, \cc^2 \, d^2)^p \, \cMH (1+ 
	\ccp (1+ \Erw ( \|X_{t_0} \|^p )) )
	\nonumber \\
	&\quad + (m \, s^2 \, \czwei^3 \, \cc^2 \, d)^p \, 2^{p-1}
	(1 + \cYMB (1+ \Erw ( \| Y_0 \|^p )) )
	\nonumber \\
	&\quad + (m \, s^2 \, \czwei^2 \, \cc^2 \, d^2)^p \, \cMHdetInc (1+ 
	\cYMB (1+ \Erw ( \| Y_0 \|^p )) )
	\nonumber \\
	&\quad + (m^2 \, s^2 \, \czwei^3 \, \cc^2 \, d)^p \, (p-1)^p
	\, 2^{\frac{p}{2}-1} (1+ \cYMB (1+ \Erw ( \| Y_0 \|^p )) )
	\nonumber \\
	&\quad + (m^2 \, s^2 \, \czwei^2 \, \cc^2 \, d^2)^p \, \cMH (1+ \Erw( \| Y_l \|^p) )
	\big) \, h^p
	\nonumber \\
	&\quad + 10^{p-1} \, (T-t_0)^{p-1} \, m^{2p} \, \cc^p \, 
	\Big( \frac{\max \{2,p \}}{\sqrt{2}} \Big)^p
	\sum_{l=0}^{N-1} h \, \Erw ( \sup_{0 \leq q \leq l} \| X_{t_q} - Y_q \|^p) \, .
	\nonumber
\end{align}
%
%
%
Next, we consider term \eqref{Proof:MainThm:Teil-B3-2b}. With $X_{t_0} \in 
L^{2p}(\Omega)$, \eqref{Assumption-Bound-derivative-2tx-bk}, H\"older's inequality,
Lemma~\ref{Lem:Lp-bound-SDE-sol} and
Lemma~\ref{Lem:Ij-Iij-H-Z-estimate} we get
%
\begin{align}
	&\Erw \bigg( \sup_{0 \leq n \leq N} \bigg\|
	\sum_{l=0}^{n-1} \sum_{k=1}^m \sum_{i=1}^s \beta_i^{(2)}
	\bigg( \int_0^1 \sum_{r=1}^d \frac{\partial^2}{\partial x_r \partial t}
	b^k(t_l + u \, c_i^{(1)} h, X_{t_l}) \, c_i^{(1)} h \, \mathrm{d}u \bigg)
	\nonumber \\
	&\quad \times
	\bigg( \sum_{j=1}^s A_{i,j}^{(1)} \, a^r(t_l + c_j^{(0)} h, H_{j,l}^{(0),X_{t_l}}) \, h
	+ \sum_{j=1}^{i-1} \sum_{j_1 = 1}^m B_{i,j}^{(1)} \, 
	b^{r,j_1}(t_l + c_j^{(1)} h, H_{j,l}^{(j_1),X_{t_l}}) \, \Ii_{(j_1,k),l} \bigg) 
	\bigg\|^p \bigg)
	\nonumber \\
	%
	&\leq \Erw \bigg( \sup_{0 \leq n \leq N} \bigg(
	\sum_{l=0}^{n-1} \sum_{k=1}^m \sum_{i=1}^s |\beta_i^{(2)}|
	\int_0^1 \sum_{r=1}^d \Big\| \frac{\partial^2}{\partial x_r \partial t}
	b^k(t_l + u \, c_i^{(1)} h, X_{t_l}) \Big\| \, |c_i^{(1)}| \, h \, \mathrm{d}u
	\nonumber \\
	&\quad \times \| H_{i,l}^{(k),X_{t_l}}-X_{t_l} \| \bigg)^p \bigg)
	\nonumber \\
	&\leq \Erw \bigg( \bigg( \sum_{l=0}^{N-1} \sum_{k=1}^m \sum_{i=1}^s
	\czwei^2 \, h \, d \, \cc (1+ \| X_{t_l} \|) \, \| H_{i,l}^{(k),X_{t_l}}-X_{t_l} \| 
	\bigg)^p \bigg)
	\nonumber \\
	%
	&\leq N^{p-1} \sum_{l=0}^{N-1} h^p \, m^{p-1} \sum_{k=1}^m s^{p-1}
	\sum_{i=1}^s (\czwei^2 \, d \, \cc)^p \, ( \Erw ( (1+ \| X_{t_l} \|)^{2p} ) )^{\frac{1}{2}}
	\, ( \Erw ( \| H_{i,l}^{(k),X_{t_l}}-X_{t_l} \|^{2p} ) )^{\frac{1}{2}}
	\nonumber \\
	&\leq N^{p-1} \sum_{l=0}^{N-1} h^p \, m^{p-1} \sum_{k=1}^m s^{p-1}
	\sum_{i=1}^s (\czwei^2 \, d \, \cc)^p \, 2^{p-\frac{1}{2}} 
	( 1 + \ccp (1+ \Erw ( \|X_{t_0} \|^{2p} )) )^{\frac{1}{2}}
	\nonumber \\
	&\quad \times 
	(\cMH (1+ \Erw( \| X_{t_l} \|^{2p}) ) )^{\frac{1}{2}} \, h^p
	\nonumber \\
	%
	&\leq (T-t_0)^p \, (m \, s \, \czwei^2 \, d \, \cc)^p \, 2^{p-\frac{1}{2}}
	\sqrt{\cMH} (1+ \ccp (1+ \Erw ( \|X_{t_0} \|^{2p} ) ) ) \, h^p \, .
	\nonumber 
\end{align}
%
%
%
Term \eqref{Proof:MainThm:Teil-B3-2c} can be estimated in a similar way as
term \eqref{Proof:MainThm:Teil-B3-2b}. Under the assumption that $X_{t_0}
\in L^{2p}(\Omega)$, with \eqref{Assumption-Bound-derivative-2tx-bk}, 
H\"older's inequality, Proposition~\ref{Prop:Lp-bound-Approximation}
and Lemma~\ref{Lem:Ij-Iij-H-Z-estimate} we get
%
\begin{align}
	&\Erw \bigg( \sup_{0 \leq n \leq N} \bigg\|
	\sum_{l=0}^{n-1} \sum_{k=1}^m \sum_{i=1}^s \beta_i^{(2)}
	\bigg( \int_0^1 \sum_{r=1}^d \frac{\partial^2}{\partial x_r \partial t}
	b^k(t_l + u \, c_i^{(1)} h, Y_l) \, c_i^{(1)} h \, \mathrm{d}u \bigg)
	\nonumber \\
	&\quad \times
	\bigg( \sum_{j=1}^s A_{i,j}^{(1)} \, a^r(t_l + c_j^{(0)} h, H_{j,l}^{(0),Y_l}) \, h
	+ \sum_{j=1}^{i-1} \sum_{j_1 = 1}^m B_{i,j}^{(1)} \, 
	b^{r,j_1}(t_l + c_j^{(1)} h, H_{j,l}^{(j_1),Y_l}) \, \Ii_{(j_1,k),l} \bigg) 
	\bigg\|^p \bigg)
	\nonumber \\
	&\leq \Erw \bigg( \sup_{0 \leq n \leq N} \bigg(
	\sum_{l=0}^{n-1} \sum_{k=1}^m \sum_{i=1}^s |\beta_i^{(2)}|
	\int_0^1 \sum_{r=1}^d \Big\| \frac{\partial^2}{\partial x_r \partial t}
	b^k(t_l + u \, c_i^{(1)} h, Y_l) \Big\| \, |c_i^{(1)}| \, h \, \mathrm{d}u
	\nonumber \\
	&\quad \times \| H_{i,l}^{(k),Y_l}-Y_l \| \bigg)^p \bigg)
	\nonumber \\
	&\leq \Erw \bigg( \bigg( \sum_{l=0}^{N-1} \sum_{k=1}^m \sum_{i=1}^s
	\czwei^2 \, h \, d \, \cc (1+ \| Y_l \|) \, \| H_{i,l}^{(k),Y_l}-Y_l \| 
	\bigg)^p \bigg)
	\nonumber \\
	%
	&\leq N^{p-1} \sum_{l=0}^{N-1} h^p \, m^{p-1} \sum_{k=1}^m s^{p-1}
	\sum_{i=1}^s (\czwei^2 \, d \, \cc)^p \, ( \Erw ( (1+ \| Y_l \|)^{2p} ) )^{\frac{1}{2}}
	\, ( \Erw ( \| H_{i,l}^{(k),Y_l}-Y_l \|^{2p} ) )^{\frac{1}{2}}
	\nonumber \\
	&\leq N^{p-1} \sum_{l=0}^{N-1} h^p \, m^{p-1} \sum_{k=1}^m s^{p-1}
	\sum_{i=1}^s (\czwei^2 \, d \, \cc)^p \, 2^{p-\frac{1}{2}} 
	( 1 + \cYMB (1+ \Erw ( \| Y_0 \|^{2p} )) )^{\frac{1}{2}}
	\nonumber \\
	&\quad \times 
	(\cMH (1+ \Erw( \| Y_l \|^{2p}) ) )^{\frac{1}{2}} \, h^p
	\nonumber \\
	%
	&\leq (T-t_0)^p \, (m \, s \, \czwei^2 \, d \, \cc)^p \, 2^{p-\frac{1}{2}}
	\sqrt{\cMH} (1+ \cYMB (1+ \Erw ( \| Y_0 \|^{2p} )) ) \, h^p \, .
	\nonumber
\end{align}
%
%
%
The next term to be estimated is \eqref{Proof:MainThm:Teil-B3-3}. Assuming
that $X_{t_0} \in L^{2p}(\Omega)$, 
with~\eqref{Assumption-a-bk:Bound-derivative2-x},
Lemma~\ref{Lem:Ij-Iij-H-Z-estimate}
and Lemma~\ref{Lem:Lp-bound-SDE-sol} we get
%
\begin{align}
	&\Erw \bigg( \sup_{0 \leq n \leq N} \bigg\|
	\sum_{k=1}^m \sum_{l=0}^{n-1} \sum_{i=1}^s \beta_i^{(2)}
	\int_0^1 \sum_{q,r=1}^d \frac{\partial^2}{\partial x_q \partial x_r}
	b^k(t_l + c_i^{(1)} h, X_{t_l} + u(H_{i,l}^{(k),X_{t_l}}-X_{t_l}) )
	\nonumber \\
	&\quad \times
	( {H_{i,l}^{(k),X_{t_l}} }^q - X_{t_l}^q ) \, ( {H_{i,l}^{(k),X_{t_l}} }^r - X_{t_l}^r )
	\, (1-u) \, \mathrm{d}u \bigg\|^p \bigg)
	\nonumber \\
	%
	&\leq \Erw \bigg( \sup_{0 \leq n \leq N} \bigg(
	\sum_{l=0}^{n-1} \sum_{k=1}^m \sum_{i=1}^s |\beta_i^{(2)}|
	\sum_{q,r=1}^d \int_0^1 \Big\| \frac{\partial^2}{\partial x_q \partial x_r}
	b^k(t_l + c_i^{(1)} h, X_{t_l} + u(H_{i,l}^{(k),X_{t_l}}-X_{t_l}) ) \Big\|
	\nonumber \\
	&\quad \times
	\| H_{i,l}^{(k),X_{t_l}} - X_{t_l} \| \, \| H_{i,l}^{(k),X_{t_l}} - X_{t_l} \|
	\, (1-u) \, \mathrm{d}u \bigg)^p \bigg)
	\nonumber \\
	&\leq N^{p-1} \sum_{l=0}^{N-1} m^{p-1} \sum_{k=1}^m s^{p-1} \sum_{i=1}^s
	\czwei^p \, d^{2p} \bigg( \int_0^1 1-u \, \mathrm{d}u \bigg)^p \, \cc^p \,
	\Erw \big( \| H_{i,l}^{(k),X_{t_l}} - X_{t_l} \|^{2p} \big)
	\nonumber \\
	%
	&\leq N^{p-1} \sum_{l=0}^{N-1} (m \, s \, \czwei \, d^2 \, \cc)^p \, 2^{-p} \,
	\cMH (1+ \Erw( \| X_{t_l} \|^{2p}) ) \, h^{2p}
	\nonumber \\
	&\leq (T-t_0)^p \, (m \, s \, \czwei \, d^2 \, \cc)^p \, 2^{-p}
	\cMH (1+ \ccp (1+ \Erw ( \|X_{t_0} \|^{2p} )) ) \, h^p \, .
	\nonumber
\end{align}
%
%
%
Finally, we consider term \eqref{Proof:MainThm:Teil-B3-4} which can be
estimated in a similar way as term \eqref{Proof:MainThm:Teil-B3-3}.
Under the assumption that $X_{t_0} \in L^{2p}(\Omega)$, 
with~\eqref{Assumption-a-bk:Bound-derivative2-x},
Lemma~\ref{Lem:Ij-Iij-H-Z-estimate}
and Proposition~\ref{Prop:Lp-bound-Approximation} we get
%
\begin{align}
	&\Erw \bigg( \sup_{0 \leq n \leq N} \bigg\|
	\sum_{k=1}^m \sum_{l=0}^{n-1} \sum_{i=1}^s \beta_i^{(2)}
	\int_0^1 \sum_{q,r=1}^d \frac{\partial^2}{\partial x_q \partial x_r}
	b^k(t_l + c_i^{(1)} h, Y_l + u(H_{i,l}^{(k),Y_l}-Y_l) )
	\nonumber \\
	&\quad \times
	( {H_{i,l}^{(k),Y_l}}^q-Y_l^q ) \, ( {H_{i,l}^{(k),Y_l}}^r-Y_l^r ) \,
	(1-u) \, \mathrm{d}u \bigg\|^p \bigg)
	\nonumber \\
	%
	&\leq \Erw \bigg( \sup_{0 \leq n \leq N} \bigg(
	\sum_{l=0}^{n-1} \sum_{k=1}^m \sum_{i=1}^s |\beta_i^{(2)}|
	\sum_{q,r=1}^d \int_0^1 \Big\| \frac{\partial^2}{\partial x_q \partial x_r}
	b^k(t_l + c_i^{(1)} h, Y_l + u(H_{i,l}^{(k),Y_l}-Y_l) ) \Big\|
	\nonumber \\
	&\quad \times
	\| H_{i,l}^{(k),Y_l} - Y_l \| \, \| H_{i,l}^{(k),Y_l} - Y_l \|
	\, (1-u) \, \mathrm{d}u \bigg)^p \bigg)
	\nonumber \\
	&\leq N^{p-1} \sum_{l=0}^{N-1} m^{p-1} \sum_{k=1}^m s^{p-1} \sum_{i=1}^s
	\czwei^p \, d^{2p} \bigg( \int_0^1 1-u \, \mathrm{d}u \bigg)^p \, \cc^p \,
	\Erw \big( \| H_{i,l}^{(k),Y_l} - Y_l \|^{2p} \big)
	\nonumber \\
	%
	&\leq N^{p-1} \sum_{l=0}^{N-1} (m \, s \, \czwei \, d^2 \, \cc)^p \, 2^{-p} \,
	\cMH (1+ \Erw( \| Y_l \|^{2p}) ) \, h^{2p}
	\nonumber \\
	&\leq (T-t_0)^p \, (m \, s \, \czwei \, d^2 \, \cc)^p \, 2^{-p}
	\cMH (1+ \cYMB (1+ \Erw ( \| Y_0 \|^{2p} )) ) \, h^p \, .
	\nonumber
\end{align}
%
%
%
Summarizing, we derived an estimate for the second summand 
in~\eqref{Proof:MainThm:Teil-A-B} of the form
\begin{align}
	\Erw \big( \sup_{0 \leq n \leq N} \| Z_n-Y_n \|^p \big)
	&\leq \cZY \, \bigg( h^p + \sum_{l=0}^{N-1} h \,
	\Erw \big( \sup_{0 \leq q \leq l} \| X_{t_q}-Y_q \|^p \big) \bigg) \, .
	\label{Proof:MainThm:Teil-B-Estimate}
\end{align}
Taking together \eqref{Proof:MainThm:Teil-A-Estimate} and
\eqref{Proof:MainThm:Teil-B-Estimate}, we get for
\eqref{Proof:MainThm:Teil-A-B} that
\begin{align}
	\Erw \big( \sup_{0 \leq n \leq N} \| X_{t_n}-Y_n \|^p \big)
	%
	&\leq 2^{p-1} \, (\cXZ + \cZY) \, h^p + 2^{p-1} \, \cZY
	\sum_{l=0}^{N-1} h \,
	\Erw \big( \sup_{0 \leq q \leq l} \| X_{t_q}-Y_q \|^p \big) \, .
	\label{Proof:MainThm:Teil-A-B-All-Estimate}
\end{align}
Then, the assertion follows by applying Gronwall's lemma.
\end{proof}
\subsection{Proof of Convergence for the SRK Method for Stratonovich SDEs}
\label{Sub:Sec:Proof-Convergence-SRK-method-Strato}
Here, we give a proof for Theorem~\ref{Sec:Main-Result:Thm-Konv-SRK-allg-Strato}.
For the proof of convergence for the SRK method~\eqref{SRK-method} with $\Iihat_{(l,k),n}
= \Ji_{(l,k),n}$ for a Stratonovich SDE of the form~\eqref{SDE-Integral-form}, the idea is to
represent the solution process $X$ as the solution of the corresponding It{\^o} 
SDE~\eqref{SDE-Strato-to-Ito} with corrected drift as in~\eqref{SDE-Strato-to-Ito-KorrekturDrift}.
Since Lemma~\ref{Lem:Lp-bound-SDE-sol}--\ref{Lem:Ij-Iij-H0-estimate} 
and Proposition~\ref{Prop:Lp-bound-Approximation}
are valid for both, the It{\^o} and the Stratonovich SDE case only with slightly different
constants, most of the estimates in the proof for 
Theorem~\ref{Sec:Main-Result:Thm-Konv-SRK-allg} carry over to the Stratonovich setting.
Thus, we mainly follow the lines of the proof for Theorem~\ref{Sec:Main-Result:Thm-Konv-SRK-allg}.
However, some estimates are different and need to be considered in a different way.
Therefore, we only detail the terms that need to be handled in a different way in the following.
\begin{proof}[Proof of Theorem~\ref{Sec:Main-Result:Thm-Konv-SRK-allg-Strato}.]
	We follow the lines of the proof for Theorem~\ref{Sec:Main-Result:Thm-Konv-SRK-allg}
	and only discuss terms that need to be handled in a significantly different way.
	Therefore we represent the solution of Stratonovich SDE~\eqref{SDE-Integral-form} as 
	the solution of the It{\^o} SDE~\eqref{SDE-Strato-to-Ito}
	with corrected drift $\underline{a}$ defined in~\eqref{SDE-Strato-to-Ito-KorrekturDrift}.
	The SRK method~\eqref{SRK-method} for Stratonovich SDEs uses the stages values
	\eqref{Proof:MainThm:Stages-H-0-Gamma}--\eqref{Proof:MainThm:Stages-H-k-Gamma}
	with $\Iihat_{(r,k),l} = \Ji_{(r,k),l}$.
	%
	
	%
	The first terms that need to be considered 
	are~\eqref{Proof:MainThm:Teil-A1}--\eqref{Proof:MainThm:Teil-A2} where we add
	the Stratonovich corrections and get
	%
	\begin{align}
		&\Erw \big( \sup_{0 \leq n \leq N} \| X_{t_n}-Z_n \|^p \big) 
		\nonumber \\
		&= \Erw \bigg( \sup_{0 \leq n \leq N} \bigg\| \sum_{l=0}^{n-1} \int_{t_l}^{t_{l+1}}
		\underline{a}(s,X_s) - \sum_{i=1}^s \alpha_i a(t_l+c_i^{(0)} h, H_{i,l}^{(0),X_{t_l}})
		\, \mathrm{d}s \nonumber \\
		&\quad + \sum_{k=1}^m \sum_{l=0}^{n-1} 
		\bigg( \int_{t_l}^{t_{l+1}} b^k(s,X_s) \, \mathrm{d}W_s^k
		- \sum_{i=1}^s \big( \beta_i^{(1)} \Ii_{(k),l} + \beta_i^{(2)} \big)
		b^k(t_l+c_i^{(1)} h, H_{i,l}^{(k),X_{t_l}} ) \bigg\|^p \bigg) \nonumber \\
		&\leq 2^{p-1} \Erw \bigg( \sup_{0 \leq n \leq N} \bigg\| \sum_{l=0}^{n-1}
		\int_{t_l}^{t_{l+1}} a(s,X_s) - \sum_{i=1}^s \alpha_i 
		a(t_l+c_i^{(0)} h, H_{i,l}^{(0),X_{t_l}} ) \, \mathrm{d}s \bigg\|^p \bigg) 
		\label{Proof:MainThm:Teil-A1-Strato} \\
		&\quad + 2^{p-1} \Erw \bigg( \sup_{0 \leq n \leq N} \bigg\| \sum_{k=1}^m
		\sum_{l=0}^{n-1} \bigg( \int_{t_l}^{t_{l+1}} b^k(s,X_s) \, \mathrm{d}W_s^k 
		+ \frac{1}{2} \int_{t_l}^{t_{l+1}} \sum_{q=1}^d b^{q,k}(s,X_s) \,
		\frac{\partial}{\partial x_q} b^k(s,X_s) \, \mathrm{d}s
		\nonumber \\
		&\quad 
		- \sum_{i=1}^s \big( \beta_i^{(1)} \Ii_{(k),l} + \beta_i^{(2)} \big)
		b^k(t_l+c_i^{(1)} h, H_{i,l}^{(k),X_{t_l}}) \bigg) \bigg\|^p \bigg) \, .
		\label{Proof:MainThm:Teil-A2-Strato}
	\end{align}
	%
	%
	%
	Considering~\eqref{Proof:MainThm:Teil-A1-Strato}, the corresponding 
	terms~\eqref{Proof:MainThm:Teil-A1-1} 
	and~\eqref{Proof:MainThm:Teil-A1-3}--\eqref{Proof:MainThm:Teil-A1-6} remain the same. 
	Only term~\eqref{Proof:MainThm:Teil-A1-2} becomes for the Stratonovich case
	%
	%
	\begin{align}
		&\Erw \bigg( \sup_{0 \leq n \leq N} \bigg\| \sum_{l=0}^{n-1}
		\int_{t_l}^{t_{l+1}} \sum_{r=1}^d \frac{\partial}{\partial x_r} a(t_l,X_{t_l}) 
		\, (X_s^r-X_{t_l}^r) \, \mathrm{d}s \bigg\|^p \bigg) 
		\nonumber \\
		&\leq 2^{p-1} \Erw \bigg( \sup_{0 \leq n \leq N} \bigg\| \sum_{l=0}^{n-1}
		\int_{t_l}^{t_{l+1}} \sum_{r=1}^d \frac{\partial}{\partial x_r} a(t_l,X_{t_l})
		\int_{t_l}^s \underline{a}^r(u,X_u) \, \mathrm{d}u \, \mathrm{d}s \bigg\|^p \bigg)
		\label{Proof:MainThm:Teil-A1-2a-Strato} \\
		&\quad + 2^{p-1} \Erw \bigg( \sup_{0 \leq n \leq N} \bigg\| \sum_{l=0}^{n-1}
		\int_{t_l}^{t_{l+1}} \sum_{r=1}^d \frac{\partial}{\partial x_r} a(t_l,X_{t_l})
		\sum_{j=1}^m \int_{t_l}^s b^{r,j}(u,X_u) \, \mathrm{d}W_u^j \, \mathrm{d}s
		\bigg\|^p \bigg) \, .
		\label{Proof:MainThm:Teil-A1-2b-Strato}
	\end{align}
	Thus, as for term~\eqref{Proof:MainThm:Teil-A1-2a} we now get
	with~\eqref{Assumption-a-bk:lin-growth} and~\eqref{Assumption-a-bk:Bound-derivative-1}
	for~\eqref{Proof:MainThm:Teil-A1-2a-Strato} that
	%
	%
	\begin{align*}
		&\Erw \bigg( \sup_{0 \leq n \leq N} \bigg\| \sum_{l=0}^{n-1}
		\int_{t_l}^{t_{l+1}} \sum_{r=1}^d \frac{\partial}{\partial x_r} a(t_l,X_{t_l})
		\int_{t_l}^s \underline{a}^r(u,X_u) \, \mathrm{d}u \, \mathrm{d}s \bigg\|^p \bigg) 
		\nonumber \\
		&\leq (T-t_0)^{p-1} \int_{t_0}^{t_N} \cc^p \, d^{p-1} \sum_{r=1}^d \Erw \bigg(
		\bigg( \int_{\lfloor s \rfloor}^s | \underline{a}^r(u,X_u) | \, \mathrm{d}u  \bigg)^p \bigg) 
		\, \mathrm{d}s 
		\nonumber \\
		&\leq (T-t_0)^{p-1} \int_{t_0}^{t_N} \cc^p \, d^p \Erw \big(
		\cc^p \big(1 + \sup_{t_0 \leq t \leq T} \| X_t \| \big)^p \big) 
		\bigg( \int_{\lfloor s \rfloor}^s \, \mathrm{d}u  \bigg)^p \, \mathrm{d}s 
		\nonumber \\
		%
		&\leq (T-t_0)^p \, (\cc^2 \, d)^p \big( 2^{p-1} + 2^{p-1} \ccp \big(1+
		\Erw \big( \| X_{t_0} \|^p \big) \big) \, h^p \, .
	\end{align*}
	%
	%
	%
	Term~\eqref{Proof:MainThm:Teil-A1-2b-Strato} remains the same
	as~\eqref{Proof:MainThm:Teil-A1-2b} and can be estimated in exactly the same
	way as~\eqref{Proof:MainThm:Teil-A1-2b}.
	%
	%

	For term~\eqref{Proof:MainThm:Teil-A2-Strato}, we calculate analogously
	as for term~\eqref{Proof:MainThm:Teil-A2} with Taylor expansion
	%
	%
	\begin{align}
		&\Erw \bigg( \sup_{0 \leq n \leq N} \bigg\| \sum_{k=1}^m
		\sum_{l=0}^{n-1} \bigg( \int_{t_l}^{t_{l+1}} b^k(s,X_s) \, \mathrm{d}W_s^k 
		+ \frac{1}{2} \int_{t_l}^{t_{l+1}} \sum_{q=1}^d b^{q,k}(s,X_s) \,
		\frac{\partial}{\partial x_q} b^k(s,X_s) \, \mathrm{d}s
		\nonumber \\		
		&\quad - \sum_{i=1}^s \big( \beta_i^{(1)} \Ii_{(k),l} + \beta_i^{(2)} \big)
		b^k(t_l+c_i^{(1)} h, H_{i,l}^{(k),X_{t_l}}) \bigg\|^p \bigg) 
		\nonumber \\
		&= \Erw \bigg( \sup_{0 \leq n \leq N} \bigg\| \sum_{l=0}^{n-1} \bigg[
		\sum_{k=1}^m \int_{t_l}^{t_{l+1}} b^k(t_l,X_{t_l}) \, \mathrm{d}W_s^k
		+ \sum_{k=1}^m \int_{t_l}^{t_{l+1}} \sum_{r=1}^d \frac{\partial}{\partial x_r}
		b^k(t_l,X_{t_l}) \, (X_s^r-X_{t_l}^r) \, \mathrm{d}W_s^k 
		\nonumber \\
		&\quad + \sum_{k=1}^m \int_{t_l}^{t_{l+1}} \int_0^1 \sum_{q,r=1}^d 
		\frac{\partial^2}{\partial x_q \partial x_r} b^k(t_l, X_{t_l}+u(X_s-X_{t_l})) 
		\nonumber \\
		&\quad \times (X_s^q-X_{t_l}^q) \, (X_s^r-X_{t_l}^r) \, (1-u) 
		\, \mathrm{d}u \, \mathrm{d}W_s^k 
		\nonumber \\
		&\quad + \sum_{k=1}^m \int_{t_l}^{t_{l+1}} \int_0^1 \frac{\partial}{\partial t}
		b^k(t_l+u(s-t_l),X_s) \, (s-t_l) \, \mathrm{d}u \, \mathrm{d}W_s^k 
		\nonumber \\
		%
		&\quad + \frac{1}{2} \sum_{k=1}^m \int_{t_l}^{t_{l+1}} \sum_{q=1}^d b^{q,k}(t_l, X_{t_l}) \,
		\frac{\partial}{\partial x_q} b^k(t_l,X_{t_l}) \, \mathrm{d}s
		\nonumber \\
		&\quad + \frac{1}{2} \sum_{k=1}^m \int_{t_l}^{t_{l+1}} \int_0^1 \sum_{q=1}^d
		\frac{\partial}{\partial t} 
		\Big( b^{q,k} \cdot \frac{\partial}{\partial x_q} b^k \Big) (t_l+u(s-t_l),X_s) \, (s-t_l)
		\, \mathrm{d}u \, \mathrm{d}s
		\nonumber \\
		&\quad + \frac{1}{2} \sum_{k=1}^m \int_{t_l}^{t_{l+1}} \sum_{q,r=1}^d 
		\frac{\partial}{\partial x_r} \Big( b^{q,k} \cdot \frac{\partial}{\partial x_q} b^k \Big)(t_l,X_{t_l}) 
		\, (X_s^r-X_{t_l}^r) \, \mathrm{d}s
		\nonumber \\
		&\quad + \frac{1}{2} \sum_{k=1}^m \int_{t_l}^{t_{l+1}} \int_0^1 \sum_{q,r,j=1}^d 
		\frac{\partial}{\partial x_j} \Big( \frac{\partial}{\partial x_r} \Big( b^{q,k} \cdot
		\frac{\partial}{\partial x_q} b^k \Big) \Big) (t_l,X_{t_l} + u(X_s-X_{t_l}))
		\nonumber \\
		&\quad \times (X_s^r-X_{t_l}^r) \, (X_s^j-X_{t_l}^j) \, (1-u) \, \mathrm{d}u \, \mathrm{d}s
		\nonumber \\
		%
		&\quad - \sum_{k=1}^m \sum_{i=1}^s ( \beta_i^{(1)} \, \Ii_{(k),l} + \beta_i^{(2)} )
		\, b^k(t_l,X_{t_l}) 
		\nonumber \\
		&\quad -\sum_{k=1}^m \sum_{i=1}^s ( \beta_i^{(1)} \, \Ii_{(k),l} + \beta_i^{(2)} )
		\sum_{r=1}^d \frac{\partial}{\partial x_r} b^k(t_l,X_{t_l}) 
		\, ( {H_{i,l}^{(k),X_{t_l}} }^r - X_{t_l}^r) 
		\nonumber \\
		&\quad -\sum_{k=1}^m \sum_{i=1}^s ( \beta_i^{(1)} \, \Ii_{(k),l} + \beta_i^{(2)} )
		\, \bigg( \frac{\partial}{\partial t} b^k(t_l,X_{t_l}) \, c_i^{(1)} h 
		\nonumber \\
		&\quad + \int_0^1 \frac{\partial^2}{\partial t^2} 
		b^k(t_l+u \, c_i^{(1)} \, h, H_{i,l}^{(k),X_{t_l}}) \, (c_i^{(1)} h)^2 \, (1-u) 
		\, \mathrm{d}u \bigg) 
		\nonumber \\
		&\quad -\sum_{k=1}^m \sum_{i=1}^s ( \beta_i^{(1)} \, \Ii_{(k),l} + \beta_i^{(2)} )
		\int_0^1 \sum_{r=1}^d \frac{\partial^2}{\partial t \partial x_r}
		b^k(t_l,X_{t_l} +u (H_{i,l}^{(k),X_{t_l}} -X_{t_l})) 
		\nonumber \\
		&\quad \times (c_i^{(1)} h) \, ( {H_{i,l}^{(k),X_{t_l}}}^r -X_{t_l}^r ) \, \mathrm{d}u 
		\nonumber \\
		&\quad -\sum_{k=1}^m \sum_{i=1}^s ( \beta_i^{(1)} \, \Ii_{(k),l} + \beta_i^{(2)} )
		\int_0^1 \sum_{q,r=1}^d \frac{\partial^2}{\partial x_q \partial x_r}
		b^k(t_l,X_{t_l} + u (H_{i,l}^{(k),X_{t_l}} - X_{t_l})) 
		\nonumber \\
		&\quad \times ({ H_{i,l}^{(k),X_{t_l}} }^q - X_{t_l}^q) 
		\, ( {H_{i,l}^{(k),X_{t_l}} }^r - X_{t_l}^r) \, (1-u) \, \mathrm{d}u 
		\bigg] \bigg\|^p \bigg) \, .
		\nonumber
	\end{align}
	%
	%
	As a result of this, term~\eqref{Proof:MainThm:Teil-A2-Strato} can now be estimated
	in a similar way as in the proof for Theorem~\ref{Sec:Main-Result:Thm-Konv-SRK-allg}.
	Adding the additional terms from the Taylor expansion given above, we obtain instead
	of~\eqref{Proof:MainThm:Teil-A2-1} the term
	%
	\begin{align}	
		&\Erw \bigg( \sup_{0 \leq n \leq \mathbb{N}} \bigg\| 
		\sum_{l=0}^{n-1} \bigg( \sum_{k=1}^m \int_{t_l}^{t_{l+1}} \sum_{r=1}^d 
		\frac{\partial}{\partial x_r} b^k(t_l,X_{t_l}) \, (X_s^r-X_{t_l}^r) \, \mathrm{d}W_s^k 
		\nonumber \\
		%
		&\quad + \frac{1}{2} \sum_{k=1}^m \int_{t_l}^{t_{l+1}} \sum_{q=1}^d b^{q,k}(t_l, X_{t_l}) \,
		\frac{\partial}{\partial x_q} b^k(t_l,X_{t_l}) \, \mathrm{d}s
		\nonumber \\
		%
		&\quad -\sum_{k=1}^m \sum_{i=1}^s ( \beta_i^{(1)} \, \Ii_{(k),l} + \beta_i^{(2)} )
		\sum_{r=1}^d \frac{\partial}{\partial x_r} b^k(t_l,X_{t_l}) 
		\, ( {H_{i,l}^{(k),X_{t_l}} }^r - X_{t_l}^r) \bigg) \bigg\|^p \bigg)
		\label{Proof:MainThm:Teil-A2-1-Strato}
	\end{align}
	and the three additional remainder terms
	%
	\begin{align}	
		%
		&\Erw \bigg( \sup_{0 \leq n \leq \mathbb{N}} \bigg\| \sum_{l=0}^{n-1}
		\frac{1}{2} \sum_{k=1}^m \int_{t_l}^{t_{l+1}} \int_0^1 \sum_{q=1}^d
		\frac{\partial}{\partial t} 
		\Big( b^{q,k} \cdot \frac{\partial}{\partial x_q} b^k \Big) (t_l+u(s-t_l),X_s) \, (s-t_l)
		\, \mathrm{d}u \, \mathrm{d}s \bigg\|^p \bigg)
		\label{Proof:MainThm:Teil-A2-7-Strato} \\
		&+ \Erw \bigg( \sup_{0 \leq n \leq \mathbb{N}} \bigg\| \sum_{l=0}^{n-1}
		\frac{1}{2} \sum_{k=1}^m \int_{t_l}^{t_{l+1}} \sum_{q,r=1}^d 
		\frac{\partial}{\partial x_r} \Big( b^{q,k} \cdot \frac{\partial}{\partial x_q} b^k \Big)(t_l,X_{t_l}) 
		\, (X_s^r-X_{t_l}^r) \, \mathrm{d}s \bigg\|^p \bigg)
		\label{Proof:MainThm:Teil-A2-8-Strato} \\
		&+ \Erw \bigg( \sup_{0 \leq n \leq \mathbb{N}} \bigg\| \sum_{l=0}^{n-1}
		\frac{1}{2} \sum_{k=1}^m \int_{t_l}^{t_{l+1}} \int_0^1 \sum_{q,r,j=1}^d 
		\frac{\partial}{\partial x_j} \Big( \frac{\partial}{\partial x_r} \Big( b^{q,k} \cdot
		\frac{\partial}{\partial x_q} b^k \Big) \Big) (t_l,X_{t_l} + u(X_s-X_{t_l}))
		\nonumber \\
		&\quad \times (X_s^r-X_{t_l}^r) \, (X_s^j-X_{t_l}^j) \, (1-u) \, \mathrm{d}u \, \mathrm{d}s
		\bigg\|^p \bigg) \, .
		\label{Proof:MainThm:Teil-A2-9-Strato}
		%
	\end{align}
	The remaining terms are the same as the corresponding 
	terms~\eqref{Proof:MainThm:Teil-A2-2}--\eqref{Proof:MainThm:Teil-A2-6} and their
	estimates can be calculated in exactly the same way as
	for~\eqref{Proof:MainThm:Teil-A2-2}--\eqref{Proof:MainThm:Teil-A2-6}.
	%
	%
	Thus, we only need to consider
	terms~\eqref{Proof:MainThm:Teil-A2-1-Strato}--\eqref{Proof:MainThm:Teil-A2-9-Strato}
	in the following.
	%

	Using assumptions~\eqref{Assumption-a-bk:Bound-derivative-1}, 
	\eqref{Assumption-Bound-derivative-1t-and-2t}
	and~\eqref{Assumption-LinGrow-derivative-tx-bk-bk-Strato} we can
	estimate term~\eqref{Proof:MainThm:Teil-A2-7-Strato} in the same way as
	term~\eqref{Proof:MainThm:Teil-A1-4}.
	%
	
	Term~\eqref{Proof:MainThm:Teil-A2-8-Strato} can be estimated similar to
	term~\eqref{Proof:MainThm:Teil-A1-2}, however with some differences. 
	For~\eqref{Proof:MainThm:Teil-A2-8-Strato} we calculate
	%
	\begin{align}
		&\Erw \bigg( \sup_{0 \leq n \leq \mathbb{N}} \bigg\| \sum_{l=0}^{n-1}
		\frac{1}{2} \sum_{k=1}^m \int_{t_l}^{t_{l+1}} \sum_{q,r=1}^d 
		\frac{\partial}{\partial x_r} \Big( b^{q,k} \cdot \frac{\partial}{\partial x_q} b^k \Big)(t_l,X_{t_l}) 
		\, (X_s^r-X_{t_l}^r) \, \mathrm{d}s \bigg\|^p \bigg)
		\nonumber \\
		&= \Erw \bigg( \sup_{0 \leq n \leq N} \bigg\| \sum_{l=0}^{n-1}
		\frac{1}{2} \sum_{k=1}^m \int_{t_l}^{t_{l+1}} \sum_{q,r=1}^d 
		\frac{\partial}{\partial x_r} \Big( b^{q,k} \cdot \frac{\partial}{\partial x_q} b^k \Big)(t_l,X_{t_l}) 
		\, \bigg( \int_{t_l}^s \underline{a}^r(u,X_u) \, \mathrm{d}u
		\nonumber \\
		&\quad + \sum_{j=1}^m \int_{t_l}^s
		b^{r,j}(u,X_u) \, \mathrm{d}W_u^j \bigg) \mathrm{d}s \bigg\|^p \bigg)
		\nonumber \\
		&\leq 2^{p-1} \Erw \bigg( \sup_{0 \leq n \leq N} \bigg\| \sum_{l=0}^{n-1}
		\frac{1}{2} \sum_{k=1}^m \int_{t_l}^{t_{l+1}} \sum_{q,r=1}^d 
		\frac{\partial}{\partial x_r} \Big( b^{q,k} \cdot \frac{\partial}{\partial x_q} b^k \Big)(t_l,X_{t_l}) 
		\int_{t_l}^s \underline{a}^r(u,X_u) \, \mathrm{d}u \, \mathrm{d}s \bigg\|^p \bigg)
		\label{Proof:MainThm:Teil-A2-8a-Strato} \\
		&\quad + 2^{p-1} \Erw \bigg( \sup_{0 \leq n \leq N} \bigg\| \sum_{l=0}^{n-1}
		\frac{1}{2} \sum_{k=1}^m \int_{t_l}^{t_{l+1}} \sum_{q,r=1}^d 
		\frac{\partial}{\partial x_r} \Big( b^{q,k} \cdot \frac{\partial}{\partial x_q} b^k \Big)(t_l,X_{t_l})
		\nonumber \\
		&\quad \times \sum_{j=1}^m \int_{t_l}^s b^{r,j}(u,X_u) \, \mathrm{d}W_u^j \, \mathrm{d}s
		\bigg\|^p \bigg) \, .
		\label{Proof:MainThm:Teil-A2-8b-Strato}
	\end{align}
	%
	%
	Term~\eqref{Proof:MainThm:Teil-A2-8a-Strato} needs to be estimated in a different
	way compared to term~\eqref{Proof:MainThm:Teil-A1-2a} because we do not assume
	that $\tfrac{\partial}{\partial x_r} \big( b^{q,k} \cdot 
	\tfrac{\partial}{\partial x_q} b^k \big)(t,x)$ is bounded. However, with
	assumption~\eqref{Assumption-a-bk:lin-growth} and 
	\eqref{Assumption-a-bk:Bound-derivative-1} it follows that the function
	$\tfrac{\partial}{\partial x_r} \big( b^{q,k} \cdot \tfrac{\partial}{\partial x_q} b^k \big)$
	fulfills a linear growth condition.
	Thus, for $p \geq 2$, $X_{t_0} \in L^{2p}(\Omega)$ and with H\"older's inequality,
	\eqref{Assumption-a-bk:lin-growth}, \eqref{Assumption-a-bk:Bound-derivative-1}
	and
	Lemma~\ref{Lem:Lp-bound-SDE-sol} we get for~\eqref{Proof:MainThm:Teil-A2-8a-Strato}
	that
	%
	\begin{align*}
		&\Erw \bigg( \sup_{0 \leq n \leq N} \bigg\| \sum_{l=0}^{n-1}
		\sum_{k=1}^m \int_{t_l}^{t_{l+1}} \sum_{q,r=1}^d 
		\frac{\partial}{\partial x_r} \Big( b^{q,k} \cdot \frac{\partial}{\partial x_q} b^k \Big)(t_l,X_{t_l}) 
		\int_{t_l}^s \underline{a}^r(u,X_u) \, \mathrm{d}u \, \mathrm{d}s \bigg\|^p \bigg) 
		\nonumber \\
		&= \Erw \bigg( \sup_{0 \leq n \leq N} \bigg\|
		\int_{t_0}^{t_n} \sum_{k=1}^m \sum_{q,r=1}^d 
		\frac{\partial}{\partial x_r} \Big( b^{q,k} \cdot \frac{\partial}{\partial x_q} b^k \Big)
		(\lfloor s \rfloor,X_{\lfloor s \rfloor})
		\int_{\lfloor s \rfloor}^s \underline{a}^r(u,X_u) \, \mathrm{d}u \, \mathrm{d}s \bigg\|^p \bigg) 
		\nonumber \\
		&\leq \Erw \bigg( \sup_{0 \leq n \leq N} \bigg(
		\int_{t_0}^{t_n} \bigg\| \sum_{k=1}^m \sum_{q,r=1}^d 
		\frac{\partial}{\partial x_r} \Big( b^{q,k} \cdot \frac{\partial}{\partial x_q} b^k \Big)
		(\lfloor s \rfloor,X_{\lfloor s \rfloor})
		\int_{\lfloor s \rfloor}^s \underline{a}^r(u,X_u) \, \mathrm{d}u  \bigg\| \, \mathrm{d}s 
		\bigg)^p \bigg) 
		\nonumber \\
		&\leq \Erw \bigg( \int_{t_0}^{t_N} \bigg\| \sum_{k=1}^m \sum_{q,r=1}^d 
		\frac{\partial}{\partial x_r} \Big( b^{q,k} \cdot \frac{\partial}{\partial x_q} b^k \Big)
		(\lfloor s \rfloor,X_{\lfloor s \rfloor})
		\int_{\lfloor s \rfloor}^s \underline{a}^r(u,X_u) 
		\, \mathrm{d}u  \bigg\|^p \, \mathrm{d}s \, (t_N-t_0)^{p-1} \bigg) 
		\nonumber \\
		%
		&\leq (T-t_0)^{p-1} \Erw \bigg( \int_{t_0}^{t_N} \bigg( \sum_{k=1}^m \sum_{q,r=1}^d
		\bigg\| \frac{\partial}{\partial x_r} \Big( b^{q,k} \cdot \frac{\partial}{\partial x_q} b^k \Big)
		(\lfloor s \rfloor,X_{\lfloor s \rfloor}) \bigg\|
		 \int_{\lfloor s \rfloor}^s 
		| \underline{a}^r(u,X_u) | \, \mathrm{d}u  \bigg)^p \, \mathrm{d}s \bigg) 
		\nonumber \\
		&\leq (T-t_0)^{p-1} \Erw \bigg( \int_{t_0}^{t_N} \bigg( \sum_{k=1}^m \sum_{q,r=1}^d
		\Big( \Big| \frac{\partial}{\partial x_r} b^{q,k} (\lfloor s \rfloor,X_{\lfloor s \rfloor}) \Big| \,
		\Big\| \frac{\partial}{\partial x_q} b^k(\lfloor s \rfloor,X_{\lfloor s \rfloor}) \Big\|
		\nonumber \\
		&\quad + | b^{q,k} (\lfloor s \rfloor,X_{\lfloor s \rfloor}) | \, 
		\Big\| \frac{\partial^2}{\partial x_r \, \partial x_q} b^k(\lfloor s \rfloor,X_{\lfloor s \rfloor}) \Big\|
		\int_{\lfloor s \rfloor}^s 
		| \underline{a}^r(u,X_u) | \, \mathrm{d}u  \bigg)^p \, \mathrm{d}s \bigg) 
		\nonumber \\
		%
		&\leq (T-t_0)^{p-1} \Erw \bigg( m^p \, d^{2p} \, \cc^{3p}
		\big( 1 + (1+ \sup_{t_0 \leq t \leq T} \| X_t \|) \big)^p (1 + \sup_{t_0 \leq t \leq T} \| X_t \|)^p
		\int_{t_0}^{t_N} \bigg( \int_{\lfloor s \rfloor}^s \, \mathrm{d}u \bigg)^p \, \mathrm{d}s \bigg) 
		\nonumber \\
		&\leq (T-t_0)^p \, (m \, \cc^3 \, d^2)^p 2^{2p-2} \Big[ 
		1 + \Erw (\sup_{t_0 \leq t \leq T} \|X_t\|^p)
		\nonumber \\
		&\quad + 1 + 2 \, \Erw (\sup_{t_0 \leq t \leq T} \|X_t\|^p) 
		+ \Erw (\sup_{t_0 \leq t \leq T} \|X_t\|^{2p}) \Big] \, h^p
		\nonumber \\
		&\leq (T-t_0)^p \, (m \, \cc^3 \, d^2)^p 2^{2p-2} \big[ 2 + 3 \, \ccp ( 1+ \Erw( \|X_{t_0} \|^p) )
		+ \ccp ( 1+ \Erw( \|X_{t_0} \|^{2p}) ) \big] \, h^p
	\end{align*}
	%

	Next, we consider term~\eqref{Proof:MainThm:Teil-A2-8b-Strato}. Analogously 
	to~\eqref{Proof:MainThm:Teil-A1-2b} we	consider $(M_n)_{n \in \{0,\ldots,N\}}$ 
	with $M_0=0$ and
	%
	\begin{align*}
		M_n := \sum_{l=0}^{n-1}
		\frac{1}{2} \int_{t_l}^{t_{l+1}} \sum_{k=1}^m \sum_{q,r=1}^d 
		\frac{\partial}{\partial x_r} \Big( b^{q,k} \cdot \frac{\partial}{\partial x_q} b^k \Big)(t_l,X_{t_l})
		\sum_{j=1}^m \int_{t_l}^s b^{r,j}(u,X_u) \, \mathrm{d}W_u^j \, \mathrm{d}s
	\end{align*}
	for $n=1, \ldots, N$ that is a discrete time martingale w.r.t.\ the filtration 
	$(\mathcal{F}_{t_n})_{n \in \{0, \ldots, N\}}$.
	Then, for $p \geq 2$ it follows with Burkholder's inequality (see, e.g.,
	\cite{Burk88} or \cite[Prop.~2.1, Prop.~2.2]{PlRoe21}), H\"older's
	inequality, It{\^o} isometry and Lemma~\ref{Lem:Lp-bound-SDE-sol} that
	%
	\begin{align*}
		&\Erw \bigg( \sup_{0 \leq n \leq N} \bigg\| \sum_{l=0}^{n-1}
		\frac{1}{2} \int_{t_l}^{t_{l+1}} \sum_{k=1}^m \sum_{q,r=1}^d 
		\frac{\partial}{\partial x_r} \Big( b^{q,k} \cdot \frac{\partial}{\partial x_q} b^k \Big)(t_l,X_{t_l})
		\sum_{j=1}^m \int_{t_l}^s b^{r,j}(u,X_u) \, \mathrm{d}W_u^j \, \mathrm{d}s
		\bigg\|^p \bigg) \\
		&\leq \Big( \frac{p^2}{p-1} \Big)^{\frac{p}{2}} \bigg( \sum_{l=0}^{N-1} \bigg(
		\Erw \bigg( \bigg\|
		\frac{1}{2} \int_{t_l}^{t_{l+1}} \sum_{k=1}^m \sum_{q,r=1}^d 
		\frac{\partial}{\partial x_r} \Big( b^{q,k} \cdot \frac{\partial}{\partial x_q} b^k \Big)(t_l,X_{t_l})
		\nonumber \\
		&\quad \times \sum_{j=1}^m \int_{t_l}^s b^{r,j}(u,X_u) \, \mathrm{d}W_u^j
		\, \mathrm{d}s \bigg\|^p \bigg) \bigg)^{\frac{2}{p}} \bigg)^{\frac{p}{2}} \\
		%
		&\leq \Big( \frac{p^2}{p-1} \Big)^{\frac{p}{2}} \bigg( \sum_{l=0}^{N-1} \bigg(
		\Erw \bigg( \int_{t_l}^{t_{l+1}} \bigg\| \frac{1}{2} \sum_{k=1}^m \sum_{q,r=1}^d 
		\frac{\partial}{\partial x_r} \Big( b^{q,k} \cdot \frac{\partial}{\partial x_q} b^k \Big)(t_l,X_{t_l})
		\nonumber \\
		&\quad \times \sum_{j=1}^m \int_{t_l}^s b^{r,j}(u,X_u) \, \mathrm{d}W_u^j
		\bigg\|^p \mathrm{d}s \, h^{p-1} \bigg) \bigg)^{\frac{2}{p}} \bigg)^{\frac{p}{2}} \\
		&\leq \Big( \frac{p^2}{p-1} \Big)^{\frac{p}{2}} \bigg( \sum_{l=0}^{N-1} 
		h^{2-\frac{2}{p}} \bigg( \int_{t_l}^{t_{l+1}} \frac{1}{2^p} m^{p-1} \sum_{k=1}^m
		d^{2p-2} \sum_{q,r=1}^d \Erw \bigg( \Big\| 
		\frac{\partial}{\partial x_r} \Big( b^{q,k} \cdot \frac{\partial}{\partial x_q} b^k \Big)(t_l,X_{t_l})
		\Big\|^p
		\nonumber \\
		&\quad \times \bigg| \sum_{j=1}^m \int_{t_l}^s b^{r,j}(u,X_u) \, \mathrm{d}W_u^j \bigg|^p
		\bigg) \, \mathrm{d}s \bigg)^{\frac{2}{p}} \bigg)^{\frac{p}{2}} \\
		&\leq \Big( \frac{p^2}{p-1} \Big)^{\frac{p}{2}} \bigg( \sum_{l=0}^{N-1} 
		h^{2-\frac{2}{p}} \bigg( \int_{t_l}^{t_{l+1}} \frac{1}{2^p} m^{p-1} \sum_{k=1}^m
		d^{2p-2} \sum_{q,r=1}^d \bigg[ \Erw \bigg( \Big\| 
		\frac{\partial}{\partial x_r} \Big( b^{q,k} \cdot \frac{\partial}{\partial x_q} b^k \Big)(t_l,X_{t_l})
		\Big\|^{2p} \bigg) 
		\nonumber \\
		&\quad \times \Erw \bigg( 
		\bigg| \sum_{j=1}^m \int_{t_l}^s b^{r,j}(u,X_u) \, \mathrm{d}W_u^j \bigg|^{2p} \bigg)
		\bigg]^{\frac{1}{2}} \, \mathrm{d}s \bigg)^{\frac{2}{p}} \bigg)^{\frac{p}{2}} \\
		&\leq \Big( \frac{p^2}{p-1} \Big)^{\frac{p}{2}} \bigg( \sum_{l=0}^{N-1} 
		h^{2-\frac{2}{p}} \bigg( \int_{t_l}^{t_{l+1}} \frac{1}{2^p} m^{p-1} \sum_{k=1}^m
		d^{2p-2} \sum_{q,r=1}^d \big[ \Erw ( (\cc^2 + \cc^2 
		(1 + \sup_{t_0 \leq t \leq T} \| X_t \| ) )^{2p} ) \big]^{\frac{1}{2}}
		\nonumber \\
		&\quad \times \bigg[ (2p-1) \int_{t_l}^s \bigg( \Erw \bigg(
		\bigg| \sum_{j=1}^m | b^{r,j}(u,X_u) |^2 \bigg|^p \bigg) \bigg)^{\frac{1}{p}} \, \mathrm{d}u
		\bigg]^{\frac{p}{2}} \, \mathrm{d}s \bigg)^{\frac{2}{p}} \bigg)^{\frac{p}{2}} \\
		&\leq \Big( \frac{p^2}{p-1} \Big)^{\frac{p}{2}} \bigg( \sum_{l=0}^{N-1} 
		h^{2-\frac{2}{p}} \bigg( \int_{t_l}^{t_{l+1}} \frac{1}{2^p} m^{p-1} \sum_{k=1}^m
		d^{2p-2} \sum_{q,r=1}^d \cc^{2p} \, \big[ \Erw ( (2 +  
		\sup_{t_0 \leq t \leq T} \| X_t \| ) )^{2p} ) \big]^{\frac{1}{2}}
		\nonumber \\
		&\quad \times \bigg[ (2p-1) \int_{t_l}^s \bigg( \Erw \bigg(
		m^{p-1} \sum_{j=1}^m ( \cc (1 + \sup_{t_0 \leq t \leq T} \| X_t \| ) )^{2p} \bigg) 
		\bigg)^{\frac{1}{p}} \, \mathrm{d}u \bigg]^{\frac{p}{2}} \, \mathrm{d}s \bigg)^{\frac{2}{p}}
		\bigg)^{\frac{p}{2}} \\
		&\leq \Big( \frac{p^2}{p-1} \Big)^{\frac{p}{2}} \bigg( \sum_{l=0}^{N-1} 
		h^{2-\frac{2}{p}} \bigg( 
		\frac{1}{2^p} m^{p} \, d^{2p} \, \cc^{2p} \, 2^{\frac{2p-1}{2}} \, \big[ 2^{2p} 
		+ \Erw ( \sup_{t_0 \leq t \leq T} \| X_t \|^{2p} ) \big]^{\frac{1}{2}}
		\nonumber \\
		&\quad \times (2p-1)^{\frac{p}{2}} \, m^{\frac{p}{2}} \, \cc^p \, 2^{\frac{2p-1}{2}} \,
		\big( 1 + \Erw \big( \sup_{t_0 \leq t \leq T} \| X_t \|^{2p} \big) 
		\big)^{\frac{1}{2}}
		\int_{t_l}^{t_{l+1}}  \bigg( \int_{t_l}^s \, \mathrm{d}u \bigg)^{\frac{p}{2}} \, \mathrm{d}s \bigg)^{\frac{2}{p}} \bigg)^{\frac{p}{2}} \\
		&\leq \Big( \frac{p^2}{p-1} \Big)^{\frac{p}{2}} \Big( \frac{1}{2} \, m \, \cc^3 \, d^2 \Big)^p \,
		2^{\frac{2p-1}{2}} \, 2^{\frac{2p-1}{2}} \, (2p-1)^{\frac{p}{2}} \, m^{\frac{p}{2}} \, 
		\big[ 2^{2p} + \ccp (1 + \Erw ( \| X_{t_0} \|^{2p} ) ) \big]^{\frac{1}{2}}
		\nonumber \\
		&\quad \times \big[ 1 + \ccp (1 + \Erw ( \| X_{t_0} \|^{2p} ) ) \big]^{\frac{1}{2}} \, 
		\bigg( \sum_{l=0}^{N-1} h^{2-\frac{2}{p}} \, h^{1+ \frac{2}{p}} \bigg)^{\frac{p}{2}} \\
		&\leq \Big( \frac{p^2}{p-1} \Big)^{\frac{p}{2}} \Big( \frac{1}{2^2} \, m^3 \, \cc^6 \, 
		d^4 \, (2p-1) \Big)^{\frac{p}{2}} \,
		2^{2p-1} \, \big[ 2^{2p} + \ccp (1 + \Erw ( \| X_{t_0} \|^{2p} ) ) \big]^{\frac{1}{2}}
		\nonumber \\
		&\quad \times \big[ 1 + \ccp (1 + \Erw ( \| X_{t_0} \|^{2p} ) ) \big]^{\frac{1}{2}} \, 
		(T-t_0)^{\frac{p}{2}} \, h^p \, .
	\end{align*}
	%

	Applying assumption~\eqref{Assumption-LinGrow-derivative-tx-bk-bk-Strato}
	we can estimate term~\eqref{Proof:MainThm:Teil-A2-9-Strato} in a similar way as
	term~\eqref{Proof:MainThm:Teil-A1-3}. With H\"older's
	inequality, \eqref{Assumption-a-bk:Bound-derivative2-x},
	Lemma~\ref{Lem:Lp-bound-SDE-sol} and Lemma~\ref{Lem:Xs-Xtl-estimate} follows
	for $X_{t_0} \in L^{4p}(\Omega)$ that
	%
	\begin{align*}
		&\Erw \bigg( \sup_{0 \leq n \leq \mathbb{N}} \bigg\| \sum_{l=0}^{n-1}
		\frac{1}{2} \sum_{k=1}^m \int_{t_l}^{t_{l+1}} \int_0^1 \sum_{q,r,j=1}^d 
		\frac{\partial}{\partial x_j} \Big( \frac{\partial}{\partial x_r} \Big( b^{q,k} \cdot
		\frac{\partial}{\partial x_q} b^k \Big) \Big) (t_l,X_{t_l} + u(X_s-X_{t_l}))
		\nonumber \\
		&\quad \times (X_s^r-X_{t_l}^r) \, (X_s^j-X_{t_l}^j) \, (1-u) \, \mathrm{d}u \, \mathrm{d}s
		\bigg\|^p \bigg) \\
		&\leq \Erw \bigg( \sup_{0 \leq n \leq \mathbb{N}} \bigg( \sum_{l=0}^{n-1}
		\frac{1}{2} \sum_{k=1}^m \int_{t_l}^{t_{l+1}} \int_0^1 \sum_{q,r,j=1}^d 
		\Big\| \frac{\partial}{\partial x_j} \Big( \frac{\partial}{\partial x_r} \Big( b^{q,k} \cdot
		\frac{\partial}{\partial x_q} b^k \Big) \Big) (t_l,X_{t_l} + u(X_s-X_{t_l})) \Big\|
		\nonumber \\
		&\quad \times |X_s^r-X_{t_l}^r| \, |X_s^j-X_{t_l}^j| \, (1-u) \, \mathrm{d}u \, \mathrm{d}s
		\bigg)^p \bigg) \\
		%
		%
		&\leq \Erw \bigg( \sup_{0 \leq n \leq N} \bigg( \sum_{l=0}^{n-1} 
		\int_{t_l}^{t_{l+1}} \int_0^1 d^2 \cc \, (1 + \| X_{t_l} \| + u \, \| X_s-X_{t_l} \| ) \, 
		\| X_s-X_{t_l} \|^2 \, (1-u)
		\, \mathrm{d}u \, \mathrm{d}s \bigg)^p \bigg) \\
		&\leq ( d^2 \cc )^p 
		\Erw \bigg( \bigg( \int_{t_0}^{t_{N}} \| X_s-X_{\lfloor s \rfloor} \|^2 + 
		\| X_{\lfloor s \rfloor} \| \, \| X_s-X_{\lfloor s \rfloor} \|^2 
		+ \| X_s-X_{\lfloor s \rfloor} \|^3 \, \mathrm{d}s \bigg)^p \bigg) \\
		&\leq ( d^2 \cc )^p \, 3^{p-1} \bigg( 
		\Erw \bigg( \bigg( \int_{t_0}^{t_{N}} \| X_s-X_{\lfloor s \rfloor} \|^2 \, \mathrm{d}s 
		\bigg)^p \bigg) + \Erw \bigg( \bigg( \int_{t_0}^{t_{N}} 
		\| X_{\lfloor s \rfloor} \| \, \| X_s-X_{\lfloor s \rfloor} \|^2 \, \mathrm{d}s 
		\bigg)^p \bigg) \\
		&\quad + \Erw \bigg( \bigg( \int_{t_0}^{t_{N}}
		\| X_s-X_{\lfloor s \rfloor} \|^3 \, \mathrm{d}s \bigg)^p \bigg) \bigg) \\
		%
		&\leq ( d^2 \cc )^p \, 3^{p-1} \bigg(
		\int_{t_0}^{t_{N}} \Erw \big( \| X_s-X_{\lfloor s \rfloor} \|^{2p} \big) 
		\, \mathrm{d}s \, (t_N-t_0)^{p-1} \\
		&\quad +  \int_{t_0}^{t_{N}} \Erw \big( 
		\| X_{\lfloor s \rfloor} \|^p \, \| X_s-X_{\lfloor s \rfloor} \|^{2p} \big) 
		\, \mathrm{d}s \, (t_N-t_0)^{p-1} \\
		&\quad + \int_{t_0}^{t_{N}} \Erw \big( \| X_s-X_{\lfloor s \rfloor} \|^{3p} \big) 
		\, \mathrm{d}s \, (t_N-t_0)^{p-1} \bigg) \\
		%
		%
		&\leq ( d^2 \cc )^p \, 3^{p-1} \bigg(
		(T-t_0)^{p} \, \cMXinc (1+ \Erw( \|X_{t_0} \|^{2p})) \, h^p \\
		&\quad +  \int_{t_0}^{t_{N}} \big[ \Erw \big( 
		\| X_{\lfloor s \rfloor} \|^{2p} \big) \, 
		\Erw \big( \| X_s-X_{\lfloor s \rfloor} \|^{4p} \big) \big]^{\frac{1}{2}} \, \mathrm{d}s 
		\, (t_N-t_0)^{p-1} \\
		&\quad + (T-t_0)^{p} \, \cMXinc (1+ \Erw( \|X_{t_0} \|^{3p})) \, h^{\frac{3}{2} p}
		 \bigg) \\
		%
		&\leq ( d^2 \cc )^p \, 3^{p-1} \, (T-t_0)^{p} \big( 
		\cMXinc (1+ \Erw( \|X_{t_0} \|^{3p}))
		+ \big( \ccp (1+ \Erw( \|X_{t_0} \|^{2p})) \, 
		\cMXinc (1+ \Erw( \|X_{t_0} \|^{4p})) \big)^{\frac{1}{2}} \\
		&\quad + \cMXinc (1+ \Erw( \|X_{t_0} \|^{3p})) \, h^{\frac{1}{2} p} \big)
		\, h^p \, .
	\end{align*}
	%
		
	Now we consider term~\eqref{Proof:MainThm:Teil-A2-1-Strato}.
	We estimate term~\eqref{Proof:MainThm:Teil-A2-1-Strato} analogously
	to~\eqref{Proof:MainThm:Teil-A2-1} and apply some Taylor expansion to obtain
	%
	%
	\begin{align} \label{Proof:MainThm:Teil-A2-1-Replaced-Strato}
		&\Erw \bigg( \sup_{0 \leq n \leq \mathbb{N}} \bigg\| 
		\sum_{l=0}^{n-1} \bigg( \sum_{k=1}^m \int_{t_l}^{t_{l+1}} \sum_{r=1}^d 
		\frac{\partial}{\partial x_r} b^k(t_l,X_{t_l}) \, (X_s^r-X_{t_l}^r) \, \mathrm{d}W_s^k 
		\nonumber \\
		%
		&\quad + \frac{1}{2} \sum_{k=1}^m \int_{t_l}^{t_{l+1}} \sum_{q=1}^d b^{q,k}(t_l, X_{t_l}) \,
		\frac{\partial}{\partial x_q} b^k(t_l,X_{t_l}) \, \mathrm{d}s
		\nonumber \\
		%
		&\quad -\sum_{k=1}^m \sum_{i=1}^s ( \beta_i^{(1)} \, \Ii_{(k),l} + \beta_i^{(2)} )
		\sum_{r=1}^d \frac{\partial}{\partial x_r} b^k(t_l,X_{t_l}) 
		\, ( {H_{i,l}^{(k),X_{t_l}} }^r - X_{t_l}^r) \bigg) \bigg\|^p \bigg) 
		\nonumber \\ 
		&= \Erw \bigg( \sup_{0 \leq n \leq \mathbb{N}} \bigg\| \sum_{l=0}^{n-1} 
		\sum_{k=1}^m \int_{t_l}^{t_{l+1}} \sum_{r=1}^d 
		\frac{\partial}{\partial x_r} b^k(t_l,X_{t_l}) \bigg[ \int_{t_l}^s \underline{a}^r(u,X_u) 
		\, \mathrm{d}u + \sum_{k_2=1}^m \int_{t_l}^s b^{r,k_2}(t_l,X_{t_l}) 
		\, \mathrm{d}W_u^{k_2}
		\nonumber \\
		&\quad + \sum_{k_2=1}^m \int_{t_l}^s \int_0^1 \sum_{q=1}^d
		\frac{\partial}{\partial x_q} b^{r,k_2}(t_l, X_{t_l}+v(X_u-X_{t_l})) \, (X_u^q-X_{t_l}^q)
		\, \mathrm{d}v \, \mathrm{d}W_u^{k_2}
		\nonumber \\
		&\quad + \sum_{k_2=1}^m \int_{t_l}^s \int_0^1
		\frac{\partial}{\partial t} b^{r,k_2}(t_l + v(s-t_l),X_u) \, (s-t_l) \, \mathrm{d}v
		\, \mathrm{d}W_u^{k_2} \bigg] \, \mathrm{d}W_s^k 
		\nonumber \\
		%
		&\quad + \frac{1}{2} \sum_{l=0}^{n-1} \sum_{k=1}^m \int_{t_l}^{t_{l+1}} \sum_{q=1}^d 
		b^{q,k}(t_l, X_{t_l}) \,
		\frac{\partial}{\partial x_q} b^k(t_l,X_{t_l}) \, \mathrm{d}s
		\nonumber \\
		%
		&\quad - \sum_{l=0}^{n-1} \sum_{k=1}^m \sum_{i=1}^s ( \beta_i^{(1)} \Ii_{(k),l}
		+ \beta_i^{(2)} ) \sum_{r=1}^d \frac{\partial}{\partial x_r} b^k(t_l,X_{t_l}) \bigg[
		\sum_{j=1}^s A_{i,j}^{(1)} \, a^r(t_l,X_{t_l}) \, h
		\nonumber \\
		&\quad + \sum_{j=1}^s A_{i,j}^{(1)} \, h \int_0^1 \sum_{q=1}^d 
		\frac{\partial}{\partial x_q} a^r(t_l,X_{t_l} +u(H_{j,l}^{(0),X_{t_l}}-X_{t_l})) 
		\, ( {H_{j,l}^{(0),X_{t_l}} }^q-X_{t_l}^q) \, \mathrm{d}u
		\nonumber \\
		&\quad + \sum_{j=1}^s A_{i,j}^{(1)} \, h \int_0^1 \frac{\partial}{\partial t} 
		a^r(t_l+u \, c_j^{(0)} \, h,H_{j,l}^{(0),X_{t_l}}) \, c_j^{(0)} \, h \, \mathrm{d}u
		\nonumber \\
		&\quad + \sum_{j=1}^{i-1} \sum_{k_2=1}^m B_{i,j}^{(1)} \, b^{r,k_2}(t_l,X_{t_l})
		\, \Big( \Ii_{(k_2,k),l} + \ind_{\{k_2=k\}} \, \frac{1}{2} h \Big)
		\nonumber \\
		&\quad + \sum_{j=1}^{i-1} \sum_{k_2=1}^m B_{i,j}^{(1)} \int_0^1 \sum_{q=1}^d
		\frac{\partial}{\partial x_q} b^{r,k_2}(t_l, X_{t_l}+u(H_{j,l}^{(k_2),X_{t_l}}-X_{t_l}))
		\, ( {H_{j,l}^{(k_2),X_{t_l}} }^q-X_{t_l}^q) \, \mathrm{d}u \, \Ji_{(k_2,k),l}
		\nonumber \\
		&\quad + \sum_{j=1}^{i-1} \sum_{k_2=1}^m B_{i,j}^{(1)} \int_0^1
		\frac{\partial}{\partial t} b^{r,k_2}(t_l+u \, c_j^{(1)} \, h, H_{j,l}^{(k_2),X_{t_l}})
		\, c_j^{(1)} \, h \, \mathrm{d}u \, \Ji_{(k_2,k),l}
		\bigg] \bigg\|^p \bigg)
	\end{align}
	where we utilize that $\Ji_{(k_2,k),l} = \Ii_{(k_2,k),l}$ for $k_2 \neq k$ and 
	$\Ji_{(k,k),l} = \Ii_{(k,k),l} + \frac{1}{2} h$.
	%
	Here, we need that $\sum_{i=1}^s \beta_i^{(2)} \sum_{j=1}^{i-1} B_{i,j}^{(1)} =1$
	and especially the additional Stratonovich correction term from the solution process can
	be canceled with the corresponding term from the approximation. Together with
	assumption $\sum_{i=1}^s \beta_i^{(2)} \sum_{j=1}^s A_{i,j}^{(1)} =0$ we get for
	term~\eqref{Proof:MainThm:Teil-A2-1-Strato} that

	%
	\begin{align}
		&\Erw \bigg( \sup_{0 \leq n \leq \mathbb{N}} \bigg\| 
		\sum_{l=0}^{n-1} \bigg( \sum_{k=1}^m \int_{t_l}^{t_{l+1}} \sum_{r=1}^d 
		\frac{\partial}{\partial x_r} b^k(t_l,X_{t_l}) \, (X_s^r-X_{t_l}^r) \, \mathrm{d}W_s^k 
		\nonumber \\
		%
		&\quad + \frac{1}{2} \sum_{k=1}^m \int_{t_l}^{t_{l+1}} \sum_{q=1}^d b^{q,k}(t_l, X_{t_l}) \,
		\frac{\partial}{\partial x_q} b^k(t_l,X_{t_l}) \, \mathrm{d}s
		\nonumber \\
		%
		&\quad -\sum_{k=1}^m \sum_{i=1}^s ( \beta_i^{(1)} \, \Ii_{(k),l} + \beta_i^{(2)} )
		\sum_{r=1}^d \frac{\partial}{\partial x_r} b^k(t_l,X_{t_l}) 
		\, ( {H_{i,l}^{(k),X_{t_l}} }^r - X_{t_l}^r) \bigg) \bigg\|^p \bigg) 
		\nonumber \\ 
		&\leq 10^{p-1} \Erw \bigg( \sup_{0 \leq n \leq N} \bigg\| \sum_{l=0}^{n-1} 
		\sum_{k=1}^m \int_{t_l}^{t_{l+1}} \sum_{r=1}^d \frac{\partial}{\partial x_r}
		b^k(t_l,X_{t_l}) 
		\int_{t_l}^s \underline{a}^r(u,X_u) \, \mathrm{d}u \, \mathrm{d}W_s^k \bigg\|^p \bigg)
		\label{Proof:MainThm:Teil-A2-1a-Strato} \\
		&\quad + 10^{p-1} \Erw \bigg( \sup_{0 \leq n \leq N} \bigg\| \sum_{l=0}^{n-1} 
		\sum_{k=1}^m \int_{t_l}^{t_{l+1}} \sum_{r=1}^d \frac{\partial}{\partial x_r}
		b^k(t_l,X_{t_l}) 
		\nonumber \\
		&\quad \times
		\sum_{k_2=1}^m \int_{t_l}^s \int_0^1 \sum_{q=1}^d \frac{\partial}{\partial x_q}
		b^{r,k_2}(t_l, X_{t_l} + v(X_u-X_{t_l})) \, (X_u^q-X_{t_l}^q) \, \mathrm{d}v
		\, \mathrm{d}W_u^{k_2} \, \mathrm{d}W_s^k \bigg\|^p \bigg)
		\label{Proof:MainThm:Teil-A2-1b-Strato} \\
		&\quad + 10^{p-1} \Erw \bigg( \sup_{0 \leq n \leq N} \bigg\| \sum_{l=0}^{n-1} 
		\sum_{k=1}^m \int_{t_l}^{t_{l+1}} \sum_{r=1}^d \frac{\partial}{\partial x_r}
		b^k(t_l,X_{t_l})
		\nonumber \\
		&\quad \times
		\sum_{k_2=1}^m \int_{t_l}^s \int_0^1 \frac{\partial}{\partial t} 
		b^{r,k_2}(t_l + v (s-t_l),X_u) \, (s-t_l) \, \mathrm{d}v \, \mathrm{d}W_u^{k_2}
		\, \mathrm{d}W_s^k \bigg\|^p \bigg)
		\label{Proof:MainThm:Teil-A2-1c-Strato} \\
		&\quad + 10^{p-1} \Erw \bigg( \sup_{0 \leq n \leq N} \bigg\| \sum_{l=0}^{n-1}
		\sum_{k=1}^m \sum_{i=1}^s \beta_i^{(1)} \, \Ii_{(k),l} \sum_{r=1}^d
		\frac{\partial}{\partial x_r} b^k(t_l,X_{t_l}) \sum_{j=1}^s A_{i,j}^{(1)} \,
		a^r(t_l,X_{t_l}) \, h \bigg\|^p \bigg)
		\label{Proof:MainThm:Teil-A2-1d-Strato} \\
		&\quad + 10^{p-1} \Erw \bigg( \sup_{0 \leq n \leq N} \bigg\| \sum_{l=0}^{n-1}
		\sum_{k=1}^m \sum_{i=1}^s (\beta_i^{(1)} \, \Ii_{(k),l} + \beta_i^{(2)} )
		\sum_{r=1}^d \frac{\partial}{\partial x_r} b^k(t_l,X_{t_l})
		\nonumber \\
		&\quad \times
		\sum_{j=1}^s A_{i,j}^{(1)} \, h \int_0^1 \sum_{q=1}^d \frac{\partial}{\partial x_q}
		a^r(t_l,X_{t_l} + u(H_{j,l}^{(0),X_{t_l}}-X_{t_l})) \, 
		( {H_{j,l}^{(0),X_{t_l}} }^q-X_{t_l}^q) \, \mathrm{d}u \bigg\|^p \bigg)
		\label{Proof:MainThm:Teil-A2-1e-Strato} \\
		%
		&\quad + 10^{p-1} \Erw \bigg( \sup_{0 \leq n \leq N} \bigg\| \sum_{l=0}^{n-1}
		\sum_{k=1}^m \sum_{i=1}^s (\beta_i^{(1)} \, \Ii_{(k),l} + \beta_i^{(2)} )
		\sum_{r=1}^d \frac{\partial}{\partial x_r} b^k(t_l,X_{t_l})
		\nonumber \\
		&\quad \times \sum_{j=1}^s A_{i,j}^{(1)} \, h \int_0^1 \frac{\partial}{\partial t}
		a^r(t_l + u \, c_j^{(0)} \, h, H_{j,l}^{(0),X_{t_l}}) \, c_j^{(0)} \, h \, \mathrm{d}u
		\bigg\|^p \bigg)
		\label{Proof:MainThm:Teil-A2-1f-Strato} \\
		&\quad + 10^{p-1} \Erw \bigg( \sup_{0 \leq n \leq N} \bigg\| \sum_{l=0}^{n-1}
		\sum_{k=1}^m \sum_{i=1}^s \beta_i^{(1)} \, \Ii_{(k),l} \sum_{r=1}^d 
		\frac{\partial}{\partial x_r} b^k(t_l,X_{t_l})
		\nonumber \\
		&\quad \times 
		\sum_{j=1}^{i-1} \sum_{k_2=1}^m B_{i,j}^{(1)} \, b^{r,k_2}(t_l,X_{t_l}) 
		\, \Ii_{(k_2,k),l} \bigg\|^p \bigg)
		\label{Proof:MainThm:Teil-A2-1i-Strato} \\
		%
		&\quad + 10^{p-1} \Erw \bigg( \sup_{0 \leq n \leq N} \bigg\| \sum_{l=0}^{n-1}
		\sum_{k=1}^m \sum_{i=1}^s \beta_i^{(1)} \, \Ii_{(k),l} \sum_{r=1}^d 
		\frac{\partial}{\partial x_r} b^k(t_l,X_{t_l})
		\nonumber \\
		&\quad \times 
		\sum_{j=1}^{i-1} \sum_{k_2=1}^m B_{i,j}^{(1)} \, b^{r,k_2}(t_l,X_{t_l}) 
		\, \ind_{\{k_2=k\}} \, \frac{1}{2} h \bigg\|^p \bigg)
		\label{Proof:MainThm:Teil-A2-1i-2-Strato} \\
		%
		&\quad + 10^{p-1} \Erw \bigg( \sup_{0 \leq n \leq N} \bigg\| \sum_{l=0}^{n-1}
		\sum_{k=1}^m \sum_{i=1}^s (\beta_i^{(1)} \, \Ii_{(k),l} + \beta_i^{(2)} )
		\sum_{r=1}^d \frac{\partial}{\partial x_r} b^k(t_l,X_{t_l})
		\sum_{j=1}^{i-1} \sum_{k_2=1}^m B_{i,j}^{(1)} 
		\nonumber \\
		&\quad \times \int_0^1
		\sum_{q=1}^d \frac{\partial}{\partial x_q} 
		b^{r,k_2}(t_l,X_{t_l}+ u(H_{j,l}^{(k_2),X_{t_l}}-X_{t_l})) \,
		( {H_{j,l}^{(k_2),X_{t_l}} }^q-X_{t_l}^q) \, \mathrm{d}u \, \Ji_{(k_2,k),l}
		\bigg\|^p \bigg)
		\label{Proof:MainThm:Teil-A2-1g-Strato} \\
		&\quad + 10^{p-1} \Erw \bigg( \sup_{0 \leq n \leq N} \bigg\| \sum_{l=0}^{n-1}
		\sum_{k=1}^m \sum_{i=1}^s (\beta_i^{(1)} \, \Ii_{(k),l} + \beta_i^{(2)} )
		\sum_{r=1}^d \frac{\partial}{\partial x_r} b^k(t_l,X_{t_l})
		\nonumber \\
		&\quad \times \sum_{j=1}^{i-1} \sum_{k_2=1}^m B_{i,j}^{(1)} \int_0^1
		\frac{\partial}{\partial t} b^{r,k_2}(t_l +u \, c_j^{(1)} \, h, H_{j,l}^{(k_2),X_{t_l}})
		\, c_j^{(1)} \, h \, \mathrm{d}u \, \Ji_{(k_2,k),l} \bigg\|^p \bigg) 
		\label{Proof:MainThm:Teil-A2-1h-Strato} \, .
	\end{align}

	Assumptions~\eqref{Assumption-a-bk:lin-growth} and 
	\eqref{Assumption-a-bk:Bound-derivative-1} imply that
	$\tfrac{\partial}{\partial x_r} \big( b^{q,k} \cdot \tfrac{\partial}{\partial x_q} b^k \big)$
	and thus $\underline{a}$ fulfill a linear growth condition. Therefore, 	
	term~\eqref{Proof:MainThm:Teil-A2-1a-Strato} can be estimated
	as~\eqref{Proof:MainThm:Teil-A2-1a}. Further,
	Terms~\eqref{Proof:MainThm:Teil-A2-1b-Strato}--\eqref{Proof:MainThm:Teil-A2-1i-Strato} 
	agree with
	terms~\eqref{Proof:MainThm:Teil-A2-1b}--\eqref{Proof:MainThm:Teil-A2-1i} and can be
	treated with the same estimates.
	The same applies to terms~\eqref{Proof:MainThm:Teil-A2-1g-Strato} 
	and~\eqref{Proof:MainThm:Teil-A2-1h-Strato} that can be estimated the same way as
	terms~\eqref{Proof:MainThm:Teil-A2-1g} and~\eqref{Proof:MainThm:Teil-A2-1h}.

	It remains to consider the new term~\eqref{Proof:MainThm:Teil-A2-1i-2-Strato}, that
	can be estimated for $p \geq 2$ with Burkholder's inequality (see, e.g.,
	\cite{Burk88} or \cite[Prop.~2.1, Prop.~2.2]{PlRoe21}), 
	\eqref{Assumption-a-bk:Bound-derivative-1}, \eqref{Assumption-a-bk:lin-growth}
	and Lemma~\ref{Lem:Lp-bound-SDE-sol} by
	\begin{align}
		&\Erw \bigg( \sup_{0 \leq n \leq N} \bigg\| \sum_{l=0}^{n-1}
		\sum_{k=1}^m \sum_{i=1}^s \beta_i^{(1)} \, \Ii_{(k),l} \sum_{r=1}^d 
		\frac{\partial}{\partial x_r} b^k(t_l,X_{t_l})
		\sum_{j=1}^{i-1} \sum_{k_2=1}^m B_{i,j}^{(1)} \, b^{r,k_2}(t_l,X_{t_l}) 
		\, \ind_{\{k_2=k\}} \, \frac{1}{2} h \bigg\|^p \bigg)
		\nonumber \\
		&\leq \bigg( \frac{p^2}{p-1} \bigg)^{\frac{p}{2}} \bigg( \sum_{l=0}^{n-1}
		\bigg( \Erw \bigg( \bigg\| \sum_{k=1}^m \int_{t_l}^{t_{l+1}} \sum_{i=1}^s 
		\beta_i^{(1)} \sum_{r=1}^d \frac{\partial}{\partial x_r} b^k(t_l,X_{t_l}) 
		\nonumber \\
		&\quad \times \sum_{j=1}^{i-1} B_{i,j}^{(1)} \, b^{r,k}(t_l,X_{t_l}) 
		\, \frac{1}{2} h \, \mathrm{d}W_u^k \bigg\|^p \bigg) \bigg)^{\frac{2}{p}} 
		\bigg)^{\frac{p}{2}} 
		\nonumber \\
		&\leq \bigg( \frac{p^2}{p-1} \bigg)^{\frac{p}{2}} \bigg( \sum_{l=0}^{n-1}
		(p-1) \int_{t_l}^{t_{l+1}} \bigg( \Erw \bigg( \Big| \sum_{k=1}^m 
		\Big\| \sum_{i=1}^s \beta_i^{(1)} \sum_{r=1}^d \frac{\partial}{\partial x_r} b^k(t_l,X_{t_l})
		\nonumber \\
		&\quad \times \sum_{j=1}^{i-1} B_{i,j}^{(1)} \, b^{r,k}(t_l,X_{t_l}) 
		\, \frac{1}{2} h \Big\|^2 \Big|^{\frac{p}{2}} \bigg) \bigg)^{\frac{2}{p}} \, \mathrm{d}u
		\bigg)^{\frac{p}{2}}
		\nonumber \\
		&\leq p^p \bigg( \sum_{l=0}^{n-1}
		\int_{t_l}^{t_{l+1}} \mathrm{d}u \, \bigg( \Erw \bigg( m^{\frac{p}{2}-1} \sum_{k=1}^m
		s^{p-1} \sum_{i=1}^s d^{p-1} \sum_{r=1}^d s^{p-1} \sum_{j=1}^{i-1}
		\Big( | \beta_i^{(1)} | 
		\nonumber \\
		&\quad \times \Big\| \frac{\partial}{\partial x_r} b^k(t_l,X_{t_l}) \Big\|
		\, | B_{i,j}^{(1)} | \, | b^{r,k}(t_l,X_{t_l}) | \, \frac{1}{2} \, h \Big)^p \bigg) \bigg)^{\frac{2}{p}}
		\bigg)^{\frac{p}{2}}
		\nonumber \\
		&\leq p^p \, \frac{1}{2^p} \, (\cc^2 \, \czwei^2 \, m^{\frac{1}{2}} \, s^2 \, d)^p
		\bigg( \sum_{l=0}^{n-1} h^3 \big( \Erw \big( 2^{p-1} + 2^{p-1}
		\sup_{t_0 \leq t \leq T} \| X_t \|^p \big) \big)^{\frac{2}{p}} \bigg)^{\frac{p}{2}}
		\nonumber \\
		&\leq \frac{1}{2} \, p^p \, (\cc^2 \, \czwei^2 \, m^{\frac{1}{2}} \, s^2 \, d)^p \, 
		(T-t_0)^{\frac{p}{2}} \, ( 1 + \ccp (1 + \Erw( \| X_{t_0} \|^p ) ) ) \, h^p \, .
		\label{Proof:MainThm:Teil-A2-1i-2-Estimate-Strato}
	\end{align}

	As a result of these calculations, we get estimate~\eqref{Proof:MainThm:Teil-A-Estimate}
	with a different constant in case of the Stratonovich setting.

	For the estimation of the second summand in term~\eqref{Proof:MainThm:Teil-A-B}, the
	proof in the Stratonovich setting follows the lines of the proof for
	Theorem~\ref{Sec:Main-Result:Thm-Konv-SRK-allg}. Again, we get exactly
	terms~\eqref{Proof:MainThm:Teil-B1}--\eqref{Proof:MainThm:Teil-B3} that need to
	be estimated. These terms can be estimated in the same way where
	the only difference to the corresponding
	parts of the proof for Theorem~\ref{Sec:Main-Result:Thm-Konv-SRK-allg} is that
	$\Ii_{(i,j),n}$ is replaced by $\Ji_{(i,j),n}$. 
	However, there is one term~\eqref{Proof:MainThm:Teil-B2-2c-1} that needs to be 
	estimated in a slightly different way. Instead of term~\eqref{Proof:MainThm:Teil-B2-2c-1} 
	we get exactly the two
	terms~\eqref{Proof:MainThm:Teil-A2-1i-Strato} 
	and~\eqref{Proof:MainThm:Teil-A2-1i-2-Strato}. Again, 
	term~\eqref{Proof:MainThm:Teil-A2-1i-Strato} coincides with
	term~\eqref{Proof:MainThm:Teil-A2-1i} and term~\eqref{Proof:MainThm:Teil-A2-1i-2-Strato}
	is estimated by~\eqref{Proof:MainThm:Teil-A2-1i-2-Estimate-Strato}.
	Thus, we finally obtain estimate~\eqref{Proof:MainThm:Teil-B-Estimate} with a
	different constant.

	Taking into account both estimates \eqref{Proof:MainThm:Teil-A-Estimate} and
	\eqref{Proof:MainThm:Teil-B-Estimate} as in~\eqref{Proof:MainThm:Teil-A-B-All-Estimate}
	together with Gronwall's lemma completes the proof.
\end{proof}
\subsection{Proof of Proposition~\ref{Sec:Main-Result:Prop-Konv-Approx-IterIntegrals}}
\label{Sub:Sec:Proof-Prop-Konv-Approx-IterIntegrals}
Firstly, we need the following auxiliary result which corresponds to
Lemma~\ref{Lem:Ij-Iij-Moment-estimate} that holds for $\Iitilde_{(i,j),t,t+h}$
with different constants as well.
\begin{lem} \label{Tilde-Ij-Iij-Moment-estimate}
	Let $h>0$ and $i,j,k \in \{1, \ldots, m\}$. Then, for $p \geq 1$ there exists some
	constant $\cIierror = \cIierror(p) >0$ such that it holds
	\begin{align} \label{Lem:Tilde-Ij-Iij-Moment-estimate-eq1}
		\| \Iitilde_{(i,j),t,t+h} \|_{L^p(\Omega)} 
		\leq \Big( \cIierror(p) + \frac{\max\{2,p\}}{\sqrt{2}} \Big) \, h
	\end{align}
	and for $p_1, p_2 \geq 1$ it holds
	\begin{align} \label{Lem:Tilde-Ij-Iij-Moment-estimate-eq2}
		\Erw \big( | \Ii_{(k),t,t+h} |^{p_1} | \Iitilde_{(i,j),t,t+h} |^{p_2} \big)
		\leq (2p_1 -1)^{p_1} \Big( \cIierror(2p_2) + \frac{2p_2}{\sqrt{2}} \Big)^{p_2} \, 
		h^{\frac{p_1}{2} + p_2} \, .
	\end{align}
\end{lem}
\begin{proof}
	With Lemma~\ref{Lem:Ij-Iij-Moment-estimate}
	and~\eqref{Sec:Main-Result:Prop-Konv-Approx-IterIntegrals-eqn} it follows that
	\begin{align*}
		\| \Iitilde_{(i,j),t,t+h} \|_{L^p(\Omega)} 
		&\leq \| \Iihat_{(i,j),t,t+h} - \Iitilde_{(i,j),t,t+h} \|_{L^p(\Omega)} 
		+ \| \Iihat_{(i,j),t,t+h} \|_{L^p(\Omega)} \\
		&\leq \cIierror(p) \, h + \frac{\max\{2,p\}}{\sqrt{2}} \, h
	\end{align*}
	because $\max \{2,p\}-1 < \max \{2,p\}-1 + \frac{1}{\sqrt{2}} <  \max \{2,p\}$.
	Further, with Cauchy-Schwarz inequality and Lemma~\ref{Lem:Ij-Iij-Moment-estimate}
	it holds
	\begin{align*}
		\Erw \big( | \Ii_{(k),t,t+h} |^{p_1} | \Iitilde_{(i,j),t,t+h} |^{p_2} \big)
		&\leq \big[ \Erw \big( | \Ii_{(k),t,t+h} |^{2 p_1} \big) \big]^{\frac{1}{2}} \,
		\big[ \Erw \big( | \Iitilde_{(i,j),t,t+h} |^{2 p_2} \big) \big]^{\frac{1}{2}} \\
		&\leq (2 p_1 -1)^{p_1} \, \Big( \cIierror(2 p_2) \, h 
		+ \frac{\max\{2,2p_2\}}{\sqrt{2}} \Big)^{p_2} \, h^{\frac{p_1}{2} + p_2}
	\end{align*}
	which completes the proof.
\end{proof}
\noindent
Note that with Lemma~\ref{Tilde-Ij-Iij-Moment-estimate} it directly follows under the
assumptions of Proposition~\ref{Sec:Main-Result:Prop-Konv-Approx-IterIntegrals} that
Lemma~\ref{Lem:Ij-Iij-H-Z-estimate} and Lemma~\ref{Lem:Ij-Iij-Hk-estimate} remain valid
with some corresponding constants if $\Iihat_{(i,j),t,t+h}$ is replaced by $\Iitilde_{(i,j),t,t+h}$.
If $\tilde{Y}_n$ denotes the approximation by the SRK method~\eqref{SRK-method}
with~\eqref{SRK-method-Ito-Strato-Iihat}
when $\Iihat_{(i,j),t,t+h}$ is replaced by $\Iitilde_{(i,j),t,t+h}$
in the stages~\eqref{SRK-method-stage-Hk}, then
Proposition~\ref{Prop:Lp-bound-Approximation} holds for the approximation
$\tilde{Y}_n$ as well, i.e., $(\tilde{Y}_n)_{n \in \{0,1, \ldots, N\}}$, $N \in \mathbb{N}$,
has uniformly bounded $p$-th moments. Obviously, Lemma~\ref{Lem:H0-Xtl-estimate}
and Lemma~\ref{Lem:Ij-Iij-H0-estimate} remain valid as well. All proofs for these auxiliary
results remain the same with slightly different constants resulting from applying
Lemma~\ref{Tilde-Ij-Iij-Moment-estimate} instead of Lemma~\ref{Lem:Ij-Iij-Moment-estimate}.
\begin{proof}[Proof of Proposition~\ref{Sec:Main-Result:Prop-Konv-Approx-IterIntegrals}]
	Let $N \in \mathbb{N}$ and $h = \frac{T-t_0}{N}$. Then, let $\tilde{Y}_0=Y_0$ and define 
	for $n \in \{0,1, \ldots, N\}$ the approximations
	\begin{align*}
		\tilde{Z}_n &= \tilde{Y}_0 + \sum_{l=0}^{n-1} \sum_{i=1}^s
		\alpha_i \, a(t_l+c_i^{(0)} h, H_{i,l}^{(0),\tilde{Y}_l} ) \, h \\
		&\quad + \sum_{k=1}^m \sum_{l=0}^{n-1} \sum_{i=1}^s \big( \beta_i^{(1)}
		\Ii_{(k),l} + \beta_i^{(2)} \big) \, b^k(t_l+c_i^{(1)} h, H_{i,l}^{(k),\tilde{Y}_l}) \, , \\
		\tilde{Y}_n &= \tilde{Y}_0 + \sum_{l=0}^{n-1} \sum_{i=1}^s \alpha_i \,
		a(t_l+c_i^{(0)} h, \tilde{H}_{i,l}^{(0),\tilde{Y}_l}) \, h \\
		&\quad + \sum_{k=1}^m \sum_{l=0}^{n-1} \sum_{i=1}^s \big( \beta_i^{(1)}
		\Ii_{(k),l} + \beta_i^{(2)} \big) \, b^k(t_l+c_i^{(1)} h, \tilde{H}_{i,l}^{(k),\tilde{Y}_l}) \, .
	\end{align*}
	with stages $H_{i,l}^{(0),\tilde{Y}_l}$ and $H_{i,l}^{(k),\tilde{Y}_l}$ defined in
	\eqref{Proof:MainThm:Stages-H-0-Gamma} and \eqref{Proof:MainThm:Stages-H-k-Gamma},
	respectively, and with
	\begin{align*}
		\tilde{H}_{i,l}^{(0),\tilde{Y}_l} &= \tilde{Y}_l + \sum_{j=1}^s A_{i,j}^{(0)} \,
		a(t_l+c_j^{(0)} h, \tilde{H}_{j,l}^{(0), \tilde{Y}_l}) \, h \\
		\tilde{H}_{i,l}^{(k),\tilde{Y}_l} &= \tilde{Y}_l + \sum_{j=1}^s A_{i,j}^{(1)} \,
		a(t_l+c_j^{(0)} h, \tilde{H}_{j,l}^{(0), \tilde{Y}_l}) \, h 
		+ \sum_{r=1}^m \sum_{j=1}^{i-1} B_{i,j}^{(1)} \,
		b^r(t_l+c_j^{(1)} h, \tilde{H}_{j,l}^{(r), \tilde{Y}_l}) \, \Iitilde_{(r,k),l}
	\end{align*}
	for $l=0, \ldots, N-1$, $k \in \{1, \ldots, m\}$ and $i \in \{1, \ldots, s\}$. Here, 
	$(\tilde{Y}_n)_{n \in \{0,1, \ldots, N\}}$ denotes the approximation by the SRK
	method~\eqref{SRK-method} using the approximate random variables 
	$\Iihat_{(r,k),l} = \Iitilde_{(r,k),l}$ in the stages \eqref{SRK-method-stage-Hk}.
	\\ \\
	Let $p \geq 2$, let $\tilde{Y}_0 = Y_0 = X_{t_0}$ and consider
	%
	\begin{align}
		\Erw \big( \sup_{0 \leq n \leq N} \| Y_n - \tilde{Y}_n \|^p \big)
		&\leq 2^{p-1} \Erw \big( \sup_{0 \leq n \leq N} \| Y_n - \tilde{Z}_n \|^p \big)
		+ 2^{p-2} \Erw \big( \sup_{0 \leq n \leq N} \| \tilde{Z}_n - \tilde{Y}_n \|^p \big) \, .
		\label{Proof:ApproxInt:First-split}
	\end{align}
	%
	The first summand on the right hand side of~\eqref{Proof:ApproxInt:First-split}
	can be estimated exactly the same way as the
	second summand on the right hand side of \eqref{Proof:MainThm:Teil-A-B} if $X_{t_l}$ 
	is replaced by $\tilde{Y}_l$ in the proof for estimate~\eqref{Proof:MainThm:Teil-B-Estimate}.
	Therefore, we get
	\begin{align}
		\Erw \big( \sup_{0 \leq n \leq N} \| Y_n-\tilde{Z}_n \|^p \big)
		&\leq \cYtildeZ \, \bigg( h^p + \sum_{l=0}^{N-1} h \,
		\Erw \big( \sup_{0 \leq q \leq l} \| Y_q - \tilde{Y}_q \|^p \big) \bigg) \, .
		\label{Proof:ApproxInt:Teil-A-Estimate}
	\end{align}
	%
	Next, we consider the second summand on the right hand side
	of~\eqref{Proof:ApproxInt:First-split}. Taking into account that
	$\tilde{H}_{i,l}^{(0),\tilde{Y}_l} = H_{i,l}^{(0),\tilde{Y}_l}$ it follows that
	%
	\begin{align}
		&\Erw \big( \sup_{0 \leq n \leq N} \| \tilde{Z}_n - \tilde{Y}_n \|^p \big)
		\nonumber \\
		&= \Erw \bigg( \sup_{0 \leq n \leq N} \bigg\| \sum_{l=0}^{n-1} \sum_{k=1}^m
		\sum_{i=1}^s
		\big( \beta_i^{(1)} \Ii_{(k),l} + \beta_i^{(2)} \big) 
		\big[ b^k(t_l+c_i^{(1)} h, {H}_{i,l}^{(k),\tilde{Y}_l})
		- b^k(t_l+c_i^{(1)} h, \tilde{H}_{i,l}^{(k),\tilde{Y}_l}) \big] \bigg\|^p \bigg) 
		\nonumber \\
		&= \Erw \bigg( \sup_{0 \leq n \leq N} \bigg\| \sum_{l=0}^{n-1} \sum_{k=1}^m
		\sum_{i=1}^s
		\big( \beta_i^{(1)} \Ii_{(k),l} + \beta_i^{(2)} \big) 
		\bigg[ 
		b^k(t_l+c_i^{(1)} h, \tilde{Y}_l) - b^k(t_l+c_i^{(1)} h, \tilde{Y}_l)
		\nonumber \\
		&\quad + \sum_{r=1}^d \frac{\partial}{\partial x_r} b^k(t_l+c_i^{(1)} h, \tilde{Y}_l)
		\big( { H_{i,l}^{(k),\tilde{Y}_l} }^r - \tilde{Y}_l^r \big)
		\nonumber \\
		&\quad + \int_0^1 \sum_{q,r=1}^d \frac{\partial^2}{\partial x_q \partial x_r}
		b^k(t_l+c_i^{(1)} h, \tilde{Y}_l + u ( H_{i,l}^{(k),\tilde{Y}_l} - \tilde{Y}_l )) \,
		\big( { H_{i,l}^{(k),\tilde{Y}_l} }^q - \tilde{Y}_l^q \big) \,
		\big( { H_{i,l}^{(k),\tilde{Y}_l} }^r - \tilde{Y}_l^r \big) \, 
		\nonumber \\
		&\quad \times (1-u) \, \mathrm{d}u
		- \sum_{r=1}^d \frac{\partial}{\partial x_r} b^k(t_l+c_i^{(1)} h, \tilde{Y}_l)
		\, \big( \tilde{H}_{i,l}^{(k),\tilde{Y}_l} {\vphantom{ \tilde{H}_{i,l}^{(k),
		\tilde{Y}_l} }}^r - \tilde{Y}_l^r \big)
		\nonumber \\
		&\quad - \int_0^1 \sum_{q,r=1}^d \frac{\partial^2}{\partial x_q \partial x_r}
		b^k(t_l+c_i^{(1)} h, \tilde{Y}_l + u ( \tilde{H}_{i,l}^{(k),\tilde{Y}_l} - \tilde{Y}_l )) \,
		\big( \tilde{H}_{i,l}^{(k),\tilde{Y}_l} {\vphantom{ \tilde{H}_{i,l}^{(k),\tilde{Y}_l} }}^q
		- \tilde{Y}_l^q \big) \,
		\big( \tilde{H}_{i,l}^{(k),\tilde{Y}_l} {\vphantom{ \tilde{H}_{i,l}^{(k),\tilde{Y}_l} }}^r 
		- \tilde{Y}_l^r \big) \,
		\nonumber \\
		&\quad \times (1-u) \, \mathrm{d}u \bigg] \bigg\|^p \bigg)
		\nonumber \\
		%
		&\leq 3^{p-1} \, \Erw \bigg( \sup_{0 \leq n \leq N} \bigg\| \sum_{l=0}^{n-1}
		\sum_{k=1}^m
		\sum_{i=1}^s \big( \beta_i^{(1)} \Ii_{(k),l} + \beta_i^{(2)} \big)
		\nonumber \\
		&\quad \times
		\sum_{r=1}^d \frac{\partial}{\partial x_r} b^k(t_l+c_i^{(1)} h, \tilde{Y}_l)
		\big( { H_{i,l}^{(k),\tilde{Y}_l} }^r
		-\tilde{H}_{i,l}^{(k),\tilde{Y}_l} {\vphantom{ \tilde{H}_{i,l}^{(k),\tilde{Y}_l} }}^r \big) 
		\bigg\|^p \bigg)
		\label{Proof:ApproxInt:Teil-IB-1} \\
		&\quad + 3^{p-1} \, \Erw \bigg( \sup_{0 \leq n \leq N} \bigg\| \sum_{l=0}^{n-1}
		\sum_{k=1}^m
		\sum_{i=1}^s \big( \beta_i^{(1)} \Ii_{(k),l} + \beta_i^{(2)} \big)
		\nonumber \\
		&\quad \times \int_0^1 \sum_{q,r=1}^d \frac{\partial^2}{\partial x_q \partial x_r}
		b^k(t_l+c_i^{(1)} h, \tilde{Y}_l + u ( H_{i,l}^{(k),\tilde{Y}_l} - \tilde{Y}_l )) \,
		\big( { H_{i,l}^{(k),\tilde{Y}_l} }^q - \tilde{Y}_l^q \big) \,
		\big( { H_{i,l}^{(k),\tilde{Y}_l} }^r - \tilde{Y}_l^r \big) \, 
		\nonumber \\
		&\quad \times (1-u) \, \mathrm{d}u \bigg\|^p \bigg)
		\label{Proof:ApproxInt:Teil-IB-2} \\
		&\quad + 3^{p-1} \, \Erw \bigg( \sup_{0 \leq n \leq N} \bigg\| \sum_{l=0}^{n-1}
		\sum_{k=1}^m
		\sum_{i=1}^s \big( \beta_i^{(1)} \Ii_{(k),l} + \beta_i^{(2)} \big)
		\nonumber \\
		&\quad \times \int_0^1 \sum_{q,r=1}^d \frac{\partial^2}{\partial x_q \partial x_r}
		b^k(t_l+c_i^{(1)} h, \tilde{Y}_l + u ( \tilde{H}_{i,l}^{(k),\tilde{Y}_l} - \tilde{Y}_l )) \,
		\big( \tilde{H}_{i,l}^{(k),\tilde{Y}_l} {\vphantom{ \tilde{H}_{i,l}^{(k),\tilde{Y}_l} }}^q
		- \tilde{Y}_l^q \big) \,
		\big( \tilde{H}_{i,l}^{(k),\tilde{Y}_l} {\vphantom{ \tilde{H}_{i,l}^{(k),\tilde{Y}_l} }}^r 
		- \tilde{Y}_l^r \big) \,
		\nonumber \\
		&\quad \times (1-u) \, \mathrm{d}u \bigg\|^p \bigg) \, .
		\label{Proof:ApproxInt:Teil-IB-3}
	\end{align}
	First, term~\eqref{Proof:ApproxInt:Teil-IB-1} is estimated applying Taylor 
	expansion and with $\tilde{H}_{i,l}^{(0),\tilde{Y}_l} = H_{i,l}^{(0),\tilde{Y}_l}$ as
	%
	\begin{align}
		&\Erw \bigg( \sup_{0 \leq n \leq N} \bigg\| \sum_{l=0}^{n-1} \sum_{k=1}^m
		\sum_{i=1}^s \big( \beta_i^{(1)} \Ii_{(k),l} + \beta_i^{(2)} \big)
		\sum_{r=1}^d \frac{\partial}{\partial x_r} b^k(t_l+c_i^{(1)} h, \tilde{Y}_l)
		\big( { H_{i,l}^{(k),\tilde{Y}_l} }^r
		-\tilde{H}_{i,l}^{(k),\tilde{Y}_l} {\vphantom{ \tilde{H}_{i,l}^{(k),\tilde{Y}_l} }}^r \big) 
		\bigg\|^p \bigg)
		\nonumber \\
		&= \Erw \bigg( \sup_{0 \leq n \leq N} \bigg\| \sum_{l=0}^{n-1} \sum_{k=1}^m
		\sum_{i=1}^s \big( \beta_i^{(1)} \Ii_{(k),l} + \beta_i^{(2)} \big)
		\sum_{r=1}^d \frac{\partial}{\partial x_r} b^k(t_l+c_i^{(1)} h, \tilde{Y}_l) \, 
		\nonumber \\
		&\quad \times \bigg(
		\tilde{Y}_l + \sum_{j=1}^s A_{i,j}^{(1)} \,
		a^r(t_l+c_j^{(0)} h, H_{j,l}^{(0), \tilde{Y}_l}) \, h 
		+ \sum_{q=1}^m \sum_{j=1}^{i-1} B_{i,j}^{(1)} \,
		b^{r,q}(t_l+c_j^{(1)} h, H_{j,l}^{(q), \tilde{Y}_l}) \, \Iihat_{(q,k),l}
		\nonumber \\
		&\quad -\tilde{Y}_l - \sum_{j=1}^s A_{i,j}^{(1)} \,
		a^r(t_l+c_j^{(0)} h, \tilde{H}_{j,l}^{(0), \tilde{Y}_l}) \, h 
		- \sum_{q=1}^m \sum_{j=1}^{i-1} B_{i,j}^{(1)} \,
		b^{r,q}(t_l+c_j^{(1)} h, \tilde{H}_{j,l}^{(q), \tilde{Y}_l}) \, \Iitilde_{(q,k),l}
		\bigg) \bigg\|^p \bigg)
		\nonumber \\
		&= \Erw \bigg( \sup_{0 \leq n \leq N} \bigg\| \sum_{l=0}^{n-1} \sum_{k=1}^m
		\sum_{i=1}^s \big( \beta_i^{(1)} \Ii_{(k),l} + \beta_i^{(2)} \big)
		\sum_{r=1}^d \frac{\partial}{\partial x_r} b^k(t_l+c_i^{(1)} h, \tilde{Y}_l) \, 
		\nonumber \\
		&\quad \times \bigg( \sum_{q=1}^m \sum_{j=1}^{i-1} B_{i,j}^{(1)} \,
		b^{r,q}(t_l+c_j^{(1)} h, H_{j,l}^{(q), \tilde{Y}_l}) \, \Iihat_{(q,k),l}
		\nonumber \\
		&\quad - \sum_{q=1}^m \sum_{j=1}^{i-1} B_{i,j}^{(1)} \,
		b^{r,q}(t_l+c_j^{(1)} h, \tilde{H}_{j,l}^{(q), \tilde{Y}_l}) \, \Iitilde_{(q,k),l}
		\bigg) \bigg\|^p \bigg)
		\nonumber \\
		%
		&= \Erw \bigg( \sup_{0 \leq n \leq N} \bigg\| \sum_{l=0}^{n-1} \sum_{k=1}^m
		\sum_{i=1}^s \big( \beta_i^{(1)} \Ii_{(k),l} + \beta_i^{(2)} \big)
		\sum_{r=1}^d \frac{\partial}{\partial x_r} b^k(t_l+c_i^{(1)} h, \tilde{Y}_l) \, 
		\nonumber \\
		&\quad \times \bigg( \sum_{q=1}^m \sum_{j=1}^{i-1} B_{i,j}^{(1)} \bigg(
		b^{r,q}(t_l+c_j^{(1)} h, \tilde{Y}_l) \, \Iihat_{(q,k),l}
		\nonumber \\
		&\quad + \int_0^1 \sum_{r_2=1}^d \frac{\partial}{\partial x_{r_2}}
		b^{r,q}(t_l+c_j^{(1)} h, \tilde{Y}_l + u ( H_{j,l}^{(q), \tilde{Y}_l} - \tilde{Y}_l )) \,
		\big( {H_{j,l}^{(q),\tilde{Y}_l}}^{r_2} - \tilde{Y}_l^{r_2} \big) \, \mathrm{d}u 
		\, \Iihat_{(q,k),l}
		\nonumber \\
		&\quad - b^{r,q}(t_l+c_j^{(1)} h, \tilde{Y}_l) \, \Iitilde_{(q,k),l}
		- \int_0^1 \sum_{r_2=1}^d \frac{\partial}{\partial x_{r_2}}
		b^{r,q}(t_l+c_j^{(1)} h, \tilde{Y}_l + u ( \tilde{H}_{j,l}^{(q), \tilde{Y}_l} 
		- \tilde{Y}_l )) \,
		\nonumber \\
		&\quad \times \big( \tilde{H}_{j,l}^{(q),\tilde{Y}_l} 
		{\vphantom{ \tilde{H}_{j,l}^{(q),\tilde{Y}_l} }}^{r_2} 
		- \tilde{Y}_l^{r_2} \big) \, \mathrm{d}u \, \Iitilde_{(q,k),l} \bigg) \bigg) \bigg\|^p
		\bigg)
		\nonumber \\
		&\leq 3^{p-1} \, \Erw \bigg( \sup_{0 \leq n \leq N} \bigg\| \sum_{l=0}^{n-1}
		\sum_{k=1}^m
		\sum_{i=1}^s \big( \beta_i^{(1)} \Ii_{(k),l} + \beta_i^{(2)} \big)
		\sum_{r=1}^d \frac{\partial}{\partial x_r} b^k(t_l+c_i^{(1)} h, \tilde{Y}_l) \, 
		\nonumber \\
		&\quad \times \sum_{q=1}^m \sum_{j=1}^{i-1} B_{i,j}^{(1)} \,
		b^{r,q}(t_l+c_j^{(1)} h, \tilde{Y}_l) \, \big( \Iihat_{(q,k),l} - \Iitilde_{(q,k),l} \big)
		\bigg\|^p \bigg)
		\label{Proof:ApproxInt:Teil-IB-1a} \\
		&\quad + 3^{p-1} \, \Erw \bigg( \sup_{0 \leq n \leq N} \bigg\| \sum_{l=0}^{n-1}
		\sum_{k=1}^m
		\sum_{i=1}^s \big( \beta_i^{(1)} \Ii_{(k),l} + \beta_i^{(2)} \big)
		\sum_{r=1}^d \frac{\partial}{\partial x_r} b^k(t_l+c_i^{(1)} h, \tilde{Y}_l) \,
		\sum_{q=1}^m \sum_{j=1}^{i-1} B_{i,j}^{(1)}
		\nonumber \\
		&\quad \times 
		\int_0^1 \sum_{r_2=1}^d \frac{\partial}{\partial x_{r_2}}
		b^{r,q}(t_l+c_j^{(1)} h, \tilde{Y}_l + u ( H_{j,l}^{(q), \tilde{Y}_l} - \tilde{Y}_l )) \,
		\big( {H_{j,l}^{(q),\tilde{Y}_l}}^{r_2} - \tilde{Y}_l^{r_2} \big) \, \mathrm{d}u 
		\, \Iihat_{(q,k),l}
		\bigg\|^p \bigg)
		\label{Proof:ApproxInt:Teil-IB-1b} \\
		&\quad + 3^{p-1} \, \Erw \bigg( \sup_{0 \leq n \leq N} \bigg\| 
		\sum_{l=0}^{n-1} \sum_{k=1}^m
		\sum_{i=1}^s \big( \beta_i^{(1)} \Ii_{(k),l} + \beta_i^{(2)} \big)
		\sum_{r=1}^d \frac{\partial}{\partial x_r} b^k(t_l+c_i^{(1)} h, \tilde{Y}_l) \, 
		\sum_{q=1}^m \sum_{j=1}^{i-1} B_{i,j}^{(1)}
		\nonumber \\
		&\quad \times 
		\int_0^1 \sum_{r_2=1}^d \frac{\partial}{\partial x_{r_2}}
		b^{r,q}(t_l+c_j^{(1)} h, \tilde{Y}_l + u ( \tilde{H}_{j,l}^{(q), \tilde{Y}_l} 
		- \tilde{Y}_l )) \,
		\big( \tilde{H}_{j,l}^{(q),\tilde{Y}_l} 
		{\vphantom{ \tilde{H}_{j,l}^{(q),\tilde{Y}_l} }}^{r_2} 
		- \tilde{Y}_l^{r_2} \big) \, \mathrm{d}u \, \Iitilde_{(q,k),l} \bigg\|^p \bigg) \, .
		\label{Proof:ApproxInt:Teil-IB-1c}
	\end{align}
	%
	Next, under the assumption of 
	Proposition~\ref{Sec:Main-Result:Prop-Konv-Approx-IterIntegrals},
	term~\eqref{Proof:ApproxInt:Teil-IB-1a} can be estimated using
	Burkholder's inequality, see, e.g.,
	\cite{Burk88} or~\cite[Prop.~2.1]{PlRoe21}, 
	with~\eqref{Assumption-a-bk:lin-growth},
	\eqref{Assumption-a-bk:Bound-derivative-1}
	and Proposition~\ref{Prop:Lp-bound-Approximation} for $\tilde{Y}_n$ 
	such that
	%
	\begin{align}
		&\Erw \bigg( \sup_{0 \leq n \leq N} \bigg\| \sum_{l=0}^{n-1} \sum_{k=1}^m
		\sum_{i=1}^s \big( \beta_i^{(1)} \Ii_{(k),l} + \beta_i^{(2)} \big)
		\sum_{r=1}^d \frac{\partial}{\partial x_r} b^k(t_l+c_i^{(1)} h, \tilde{Y}_l) \, 
		\nonumber \\
		&\quad \times \sum_{q=1}^m \sum_{j=1}^{i-1} B_{i,j}^{(1)} \,
		b^{r,q}(t_l+c_j^{(1)} h, \tilde{Y}_l) \, \big( \Iihat_{(q,k),l} - \Iitilde_{(q,k),l} \big)
		\bigg\|^p \bigg)
		\nonumber \\
		%
		&\leq 2^{p-1} \, \Erw \bigg( \sup_{0 \leq n \leq N} \bigg( \sum_{l=0}^{n-1}
		\sum_{k=1}^m
		\sum_{i=1}^s | \beta_i^{(1)} | \, | \Ii_{(k),l} |  \sum_{r=1}^d \Big\|
		\frac{\partial}{\partial x_r} b^k(t_l+c_i^{(1)} h, \tilde{Y}_l) \Big\|
		\nonumber \\
		&\quad \times \sum_{q=1}^m \sum_{j=1}^{i-1} | B_{i,j}^{(1)} | \, 
		| b^{r,q}(t_l+c_j^{(1)} h, \tilde{Y}_l) | \, \big| \Iihat_{(q,k),l} - \Iitilde_{(q,k),l} \big|
		\bigg)^p \bigg)
		\nonumber \\
		&\quad + 2^{p-1} \Big( \frac{p}{\sqrt{p-1}} \Big)^p \bigg( \sum_{l=0}^{N-1} \bigg[
		\Erw \bigg( \bigg\| \sum_{k=1}^m \sum_{i=1}^s \beta_i^{(2)} \sum_{r=1}^d
		\frac{\partial}{\partial x_r} b^k(t_l+c_i^{(1)} h, \tilde{Y}_l)
		\nonumber \\
		&\quad \times \sum_{q=1}^m \sum_{j=1}^{i-1} B_{i,j}^{(1)} \, 
		b^{r,q}(t_l+c_j^{(1)} h, \tilde{Y}_l) \, \big( \Iihat_{(q,k),l} - \Iitilde_{(q,k),l} \big)
		\bigg\|^p \bigg) \bigg]^{\frac{2}{p}} \bigg)^{\frac{p}{2}}
		\nonumber \\
		&\leq 2^{p-1} \, \Erw \bigg( \bigg( \sum_{l=0}^{N-1} \sum_{k=1}^m
		\sum_{i=1}^s \czwei \, | \Ii_{(k),l} |  \sum_{r=1}^d \cc \sum_{q=1}^m
		\sum_{j=1}^{i-1}
		\czwei \, \cc (1+ \| \tilde{Y}_l \| ) \, \big| \Iihat_{(q,k),l} - \Iitilde_{(q,k),l} \big|
		\bigg)^p \bigg)
		\nonumber \\
		&\quad + 2^{p-1} \Big( \frac{p}{\sqrt{p-1}} \Big)^p \bigg( \sum_{l=0}^{N-1} \bigg[
		\Erw \bigg( m^{p-1} \sum_{k=1}^m s^{p-1} \sum_{i=1}^s \czwei^p \, d^{p-1}
		\sum_{r=1}^d \Big\| \frac{\partial}{\partial x_r} b^k(t_l+c_i^{(1)} h, \tilde{Y}_l)
		\Big\|^p
		\nonumber \\
		&\quad \times s^{p-1} \sum_{j=1}^{i-1} m^{p-1} \sum_{q=1}^m \czwei^p \, 
		| b^{r,q}(t_l+c_j^{(1)} h, \tilde{Y}_l) |^p \, \big| \Iihat_{(q,k),l} - \Iitilde_{(q,k),l}
		\big|^p 
		\bigg) \bigg]^{\frac{2}{p}} \bigg)^{\frac{p}{2}}
		\nonumber \\
		%
		&\leq 2^{p-1} \, N^{p-1} \sum_{l=0}^{N-1} m^{p-1} \sum_{k=1}^m (s^2 \, 
		\czwei \, d \, \cc)^p
		\, m^{p-1} \sum_{q=1}^m (\czwei \, \cc)^p \, \Erw \big( | \Ii_{(k),l} |^p 
		(1 + \| \tilde{Y}_l \| )^p
		\, \big| \Iihat_{(q,k),l} - \Iitilde_{(q,k),l} \big|^p \big)
		\nonumber \\
		&\quad + 2^{p-1} \Big( \frac{p}{\sqrt{p-1}} \Big)^p \bigg( \sum_{l=0}^{N-1} \bigg[ 
		m^{p-1} \sum_{k=1}^m s^{p-1} \sum_{i=1}^s \czwei^p \, d^{p-1} 
		\sum_{r=1}^d \cc^p \,
		s^{p-1} \sum_{j=1}^{i-1} m^{p-1} \sum_{q=1}^m (\czwei \, \cc)^p \,
		\nonumber \\
		&\quad \times 
		\Erw \big( (1+ \| \tilde{Y}_l \| )^p \, \big| \Iihat_{(q,k),l} - \Iitilde_{(q,k),l} \big|^p \big)
		\bigg]^{\frac{2}{p}} \bigg)^{\frac{p}{2}}
		\nonumber \\
		&\leq 2^{p-1} \, N^{p-1} \sum_{l=0}^{N-1} (s^2 \, \czwei^2 \, d \, \cc^2)^p 
		\, m^{2p-2}
		\sum_{k,q=1}^m 2^{p-1} \big(1+ \Erw \big( \| \tilde{Y}_l \|^p \big) \big) \, 
		\Erw \big( | \Ii_{(k),l} |^p \, \big| \Iihat_{(q,k),l} - \Iitilde_{(q,k),l} \big|^p \big)
		\nonumber \\
		&\quad + 2^{p-1} \Big( \frac{p}{\sqrt{p-1}} \Big)^p \bigg( \sum_{l=0}^{N-1} \bigg[
		(s^2 \, \czwei^2 \, d \, \cc^2)^p \, m^{2p-2} \sum_{k,q=1}^m 2^{p-1}
		\big(1+ \Erw \big( \| \tilde{Y}_l \|^p \big) \big) \, 
		\nonumber \\
		&\quad \times
		\Erw \big( \big| \Iihat_{(q,k),l} - \Iitilde_{(q,k),l} \big|^p \big) \bigg]^{\frac{2}{p}}
		\bigg)^{\frac{p}{2}}
		\nonumber \\
		&\leq 2^{2p-2} \, N^{p-1} \sum_{l=0}^{N-1} (s^2 \, \czwei^2 \, d \, \cc^2)^p 
		\, m^{2p-2}
		\big( 1 + \cYMBtilde \big( 1 + \Erw \big( \| \tilde{Y}_0 \|^p \big) \big) \big)
		\nonumber \\
		&\quad \times \sum_{k,q=1}^m \big[ \Erw \big( | \Ii_{(k),l} |^{2p} \big)
		\big]^{\frac{1}{2}}
		\big[ \Erw \big( \big| \Iihat_{(q,k),l} - \Iitilde_{(q,k),l} \big|^{2p} \big)
		\big]^{\frac{1}{2}}
		\nonumber \\
		&\quad + 2^{2p-2} \Big( \frac{p}{\sqrt{p-1}} \Big)^p (s^2 \, \czwei^2 \, 
		d \, \cc^2)^p
		m^{2p-2} \big( 1 + \cYMBtilde \big( 1 + \Erw \big( \| \tilde{Y}_0 \|^p \big) \big) \big)
		\nonumber \\
		&\quad \times \bigg( \sum_{l=0}^{N-1} \bigg[ \sum_{k,q=1}^m 
		\Erw \big( \big| \Iihat_{(q,k),l} - \Iitilde_{(q,k),l} \big|^p \big) \bigg]^{\frac{2}{p}}
		\bigg)^{\frac{p}{2}}
		\nonumber \\
		%
		&\leq 2^{2p-2} \, (s^2 \, \czwei^2 \, d \, \cc^2 \, m^2)^p \,
		\big( 1 + \cYMBtilde \big( 1 + \Erw \big( \| \tilde{Y}_0 \|^p \big) \big) \big) \,
		N^{p-1} \sum_{l=0}^{N-1} (2p-1)^p \, \cIierror^p \, h^{2p}
		\nonumber \\
		&\quad + 2^{2p-2} \Big( \frac{p}{\sqrt{p-1}} \Big)^p (s^2 \, \czwei^2 \, d \, 
		\cc^2 \, m^2)^p \,
		\big( 1 + \cYMBtilde \big( 1 + \Erw \big( \| \tilde{Y}_0 \|^p \big) \big) \big) \,
		\bigg( \sum_{l=0}^{N-1} \big[ \cIierror^p \, h^{\frac{3p}{2}} \big]^{\frac{2}{p}} 
		\bigg)^{\frac{p}{2}}
		\nonumber \\
		&= 2^{2p-2} \, (s^2 \, \czwei^2 \, d \, \cc^2 \, m^2)^p \,
		\big( 1 + \cYMBtilde \big( 1 + \Erw \big( \| \tilde{Y}_0 \|^p \big) \big) \big) \,
		(2p-1)^p \, \cIierror^p \, (T-t_0)^p \, h^p
		\nonumber \\
		&\quad + 2^{2p-2} \Big( \frac{p}{\sqrt{p-1}} \Big)^p (s^2 \, \czwei^2 \, d \, 
		\cc^2 \, m^2)^p \,
		\big( 1 + \cYMBtilde \big( 1 + \Erw \big( \| \tilde{Y}_0 \|^p \big) \big) \big) \,
		\cIierror^p \, (T-t_0)^{\frac{p}{2}} \, h^p \,
		\label{Proof:ApproxInt:Teil-IB-1a-Estimate}
	\end{align}
	%
	Next, term~\eqref{Proof:ApproxInt:Teil-IB-1b} has to be estimated. 
	With~\eqref{Assumption-a-bk:Bound-derivative-1},
	Lemma~\ref{Lem:Ij-Iij-H-Z-estimate} and
	Proposition~\ref{Prop:Lp-bound-Approximation} it follows
	%
	\begin{align}
		&\Erw \bigg( \sup_{0 \leq n \leq N} \bigg\| \sum_{l=0}^{n-1} \sum_{k=1}^m
		\sum_{i=1}^s \big( \beta_i^{(1)} \Ii_{(k),l} + \beta_i^{(2)} \big)
		\sum_{r=1}^d \frac{\partial}{\partial x_r} b^k(t_l+c_i^{(1)} h, \tilde{Y}_l)
		\sum_{q=1}^m \sum_{j=1}^{i-1} B_{i,j}^{(1)}
		\nonumber \\
		&\quad \times 
		\int_0^1 \sum_{r_2=1}^d \frac{\partial}{\partial x_{r_2}}
		b^{r,q}(t_l+c_j^{(1)} h, \tilde{Y}_l + u ( H_{j,l}^{(q), \tilde{Y}_l} - \tilde{Y}_l )) 
		\, \big( {H_{j,l}^{(q),\tilde{Y}_l}}^{r_2} - \tilde{Y}_l^{r_2} \big) \, \mathrm{d}u
		\, \Iihat_{(q,k),l} \bigg\|^p \bigg)
		\nonumber \\
		&\leq \Erw \bigg( \sup_{0 \leq n \leq N} \bigg( \sum_{l=0}^{n-1} \sum_{k=1}^m
		\sum_{i=1}^s \big( |\beta_i^{(1)}| \, |\Ii_{(k),l}| + |\beta_i^{(2)}| \big)
		\sum_{r=1}^d \Big\| \frac{\partial}{\partial x_r} b^k(t_l+c_i^{(1)} h, \tilde{Y}_l) \Big\|
		\sum_{q=1}^m \sum_{j=1}^{i-1} |B_{i,j}^{(1)}|
		\nonumber \\
		&\quad \times 
		\int_0^1 \sum_{r_2=1}^d \Big| \frac{\partial}{\partial x_{r_2}}
		b^{r,q}(t_l+c_j^{(1)} h, \tilde{Y}_l + u ( H_{j,l}^{(q), \tilde{Y}_l} - \tilde{Y}_l )) \Big|
		\, \big| {H_{j,l}^{(q),\tilde{Y}_l}}^{r_2} - \tilde{Y}_l^{r_2} \big| \, \mathrm{d}u \,
		|\Iihat_{(q,k),l}| \bigg)^p \bigg)
		\nonumber \\
		&\leq N^{p-1} \sum_{l=0}^{N-1} (m \, s)^{2p-2} \sum_{k,q=1}^m \sum_{i=1}^s
		\sum_{j=1}^{i-1} (\czwei \, \cc \, d)^{2p} \, 2^{p-1} \, 
		\Erw \big( \big( 1 + |\Ii_{(k),l}|^p \big) \, |\Iihat_{(q,k),l}|^p \, 
		\big\| H_{j,l}^{(q),\tilde{Y}_l} - \tilde{Y}_l \big\|^p \big)
		\nonumber \\
		%
		&\leq (m \, s \, d \, \czwei \, \cc)^{2p} \, 2^{p-1} \, N^{p-1} \sum_{l=0}^{N-1}
		\cMH \big( 1 + \Erw \big( \| \tilde{Y}_l \|^p \big) \big) \, 
		\big( h^{2p} + h^{\frac{p}{2} + 2p} \big)
		\nonumber \\
		&\leq (m \, s \, d \, \czwei \, \cc)^{2p} \, 2^{p-1} \, (T-t_0)^p \,
		\cMH \big( 1 + \cYMBtilde \big(1+ \Erw \big( \| \tilde{Y}_0 \|^p \big) \big) \big) \,
		\big(1 + h^{\frac{p}{2}} \big) \, h^p \, .
		\label{Proof:ApproxInt:Teil-IB-1b-Estimate}
	\end{align}
	%
	Term~\eqref{Proof:ApproxInt:Teil-IB-1c} can be estimated in the same way as 
	term~\eqref{Proof:ApproxInt:Teil-IB-1b} if $H_{j,l}^{(q), \tilde{Y}_l}$ and
	$\Iihat_{(q,k),l}$
	are replaced by $\tilde{H}_{j,l}^{(q), \tilde{Y}_l}$ and $\Iitilde_{(q,k),l}$,
	respectively. Thus, it holds
	%
	\begin{align}
		&\Erw \bigg( \sup_{0 \leq n \leq N} \bigg\| \sum_{l=0}^{n-1} \sum_{k=1}^m
		\sum_{i=1}^s \big( \beta_i^{(1)} \Ii_{(k),l} + \beta_i^{(2)} \big)
		\sum_{r=1}^d \frac{\partial}{\partial x_r} b^k(t_l+c_i^{(1)} h, \tilde{Y}_l) \, 
		\sum_{q=1}^m \sum_{j=1}^{i-1} B_{i,j}^{(1)} 
		\nonumber \\
		&\quad \times
		\int_0^1 \sum_{r_2=1}^d \frac{\partial}{\partial x_{r_2}}
		b^{r,q}(t_l+c_j^{(1)} h, \tilde{Y}_l + u ( \tilde{H}_{j,l}^{(q), \tilde{Y}_l} 
		- \tilde{Y}_l )) \,
		\big( \tilde{H}_{j,l}^{(q),\tilde{Y}_l} 
		{\vphantom{ \tilde{H}_{j,l}^{(q),\tilde{Y}_l} }}^{r_2} 
		- \tilde{Y}_l^{r_2} \big) \, \mathrm{d}u \, \Iitilde_{(q,k),l} \bigg\|^p \bigg)
		\nonumber \\
		&\leq (m \, s \, d \, \czwei \, \cc)^{2p} \, 2^{p-1} \, (T-t_0)^p \,
		\cMHtilde \big( 1 + \cYMBtilde \big(1+ \Erw \big( \| \tilde{Y}_0 \|^p \big) \big) \big) \,
		\big(1 + h^{\frac{p}{2}} \big) \, h^p \, .
		\label{Proof:ApproxInt:Teil-IB-1c-Estimate}
	\end{align}
	%
	Next, term~\eqref{Proof:ApproxInt:Teil-IB-2} has to be considered. 
	For this term it follows with~\eqref{Assumption-a-bk:Bound-derivative2-x},
	Lemma~\ref{Lem:Ij-Iij-H-Z-estimate} and 
	Proposition~\ref{Prop:Lp-bound-Approximation} that
	%
	\begin{align}
		&\Erw \bigg( \sup_{0 \leq n \leq N} \bigg\| \sum_{l=0}^{n-1} \sum_{k=1}^m
		\sum_{i=1}^s \big( \beta_i^{(1)} \Ii_{(k),l} + \beta_i^{(2)} \big)
		\nonumber \\
		&\quad \times \int_0^1 \sum_{q,r=1}^d \frac{\partial^2}{\partial x_q \partial x_r}
		b^k(t_l+c_i^{(1)} h, \tilde{Y}_l + u ( H_{i,l}^{(k),\tilde{Y}_l} - \tilde{Y}_l )) \,
		\big( { H_{i,l}^{(k),\tilde{Y}_l} }^q - \tilde{Y}_l^q \big) \,
		\big( { H_{i,l}^{(k),\tilde{Y}_l} }^r - \tilde{Y}_l^r \big) \, 
		\nonumber \\
		&\quad \times (1-u) \, \mathrm{d}u \bigg\|^p \bigg)
		\nonumber \\
		&\leq \Erw \bigg( \sup_{0 \leq n \leq N} \bigg( \sum_{l=0}^{n-1} \sum_{k=1}^m
		\sum_{i=1}^s \big( |\beta_i^{(1)}| \, |\Ii_{(k),l}| + |\beta_i^{(2)}| \big)
		\nonumber \\
		&\quad \times \int_0^1 \sum_{q,r=1}^d \Big\| \frac{\partial^2}{\partial x_q 
		\partial x_r}
		b^k(t_l+c_i^{(1)} h, \tilde{Y}_l + u ( H_{i,l}^{(k),\tilde{Y}_l} - \tilde{Y}_l )) \Big\| \,
		\big| { H_{i,l}^{(k),\tilde{Y}_l} }^q - \tilde{Y}_l^q \big| \,
		\big| { H_{i,l}^{(k),\tilde{Y}_l} }^r - \tilde{Y}_l^r \big| \, 
		\nonumber \\
		&\quad \times (1-u) \, \mathrm{d}u \bigg)^p \bigg)
		\nonumber \\
		&\leq \Erw \bigg( \bigg( \sum_{l=0}^{N-1} \sum_{k=1}^m \sum_{i=1}^s 
		\czwei ( |\Ii_{(k),l}| + 1 )
		\int_0^1 \sum_{q,r=1}^d \cc \, \big\| H_{i,l}^{(k),\tilde{Y}_l} - \tilde{Y}_l \big\|^2 \,
		(1-u) \, \mathrm{d}u \bigg)^p \bigg)
		\nonumber \\
		&\leq N^{p-1} \sum_{l=0}^{N-1} m^{p-1} \sum_{k=1}^m s^{p-1} \sum_{i=1}^s 
		(\czwei \, \cc \, d^2)^p \Big( \int_0^1 1-u \, \mathrm{d}u \Big)^p \,
		\Erw \big( \big( ( |\Ii_{(k),l}| + 1 ) \, \big\| H_{i,l}^{(k),\tilde{Y}_l} - \tilde{Y}_l \big\|^2
		\big)^p \big)
		\nonumber \\
		%
		&\leq N^{p-1} \sum_{l=0}^{N-1} m^{p-1} \sum_{k=1}^m s^{p-1} \sum_{i=1}^s 
		(\czwei \, \cc \, d^2)^p \, 2^{-p} \, 
		\nonumber \\
		&\quad \times 2^{p-1} 
		\big( \Erw \big( \big\| H_{i,l}^{(k),\tilde{Y}_l} - \tilde{Y}_l \big\|^{2p} \big)
		+ \Erw \big( |\Ii_{(k),l}|^p \, \big\| H_{i,l}^{(k),\tilde{Y}_l} - \tilde{Y}_l 
		\big\|^{2p} \big) \big)
		\nonumber \\
		&\leq N^{p-1} \sum_{l=0}^{N-1} m^{p-1} \sum_{k=1}^m s^{p-1} \sum_{i=1}^s 
		(\czwei \, \cc \, d^2)^p \, \frac{1}{2} \, \big( \cMH
		\big( 1 + \Erw \big( \| \tilde{Y}_l \|^{2p} \big) \big) \, h^{2p}
		+ \cMH \big( 1 + \Erw \big( \| \tilde{Y}_l \|^{2p} \big) \big) \, 
		h^{\frac{p}{2} + 2p} \big)
		\nonumber \\
		&\leq N^{p-1} \sum_{l=0}^{N-1} (m \, s \, \czwei \, \cc \, d^2)^p \, 
		\frac{1}{2} \, \cMH 
		\big( 1 + \cYMBtilde \big( 1 + \Erw \big( \| \tilde{Y}_0 \|^{2p} \big) \big) \big)
		\big( 1 + h^{\frac{p}{2}} \big) \, h^{2p}
		\nonumber \\
		&\leq (T-t_0)^p \, (m \, s \, \czwei \, \cc \, d^2)^p \, \frac{1}{2} \, \cMH
		\big( 1 + \cYMBtilde \big( 1 + \Erw \big( \| \tilde{Y}_0 \|^{2p} \big) \big) \big)
		\big( 1 + h^{\frac{p}{2}} \big) \, h^p \, .
		\label{Proof:ApproxInt:Teil-IB-2-Estimate}
	\end{align}
	%
	The last term~\eqref{Proof:ApproxInt:Teil-IB-3} can be estimated in the 
	same way as 
	term~\eqref{Proof:ApproxInt:Teil-IB-2} if $H_{i,l}^{(k),\tilde{Y}_l}$ is replaced by
	$\tilde{H}_{i,l}^{(k),\tilde{Y}_l}$. Then,
	with~\eqref{Assumption-a-bk:Bound-derivative2-x},
	Lemma~\ref{Lem:Ij-Iij-H-Z-estimate} and
	Proposition~\ref{Prop:Lp-bound-Approximation}
	it follows
	%
	\begin{align}
		&\Erw \bigg( \sup_{0 \leq n \leq N} \bigg\| \sum_{l=0}^{n-1} \sum_{k=1}^m
		\sum_{i=1}^s \big( \beta_i^{(1)} \Ii_{(k),l} + \beta_i^{(2)} \big)
		\nonumber \\
		&\quad \times \int_0^1 \sum_{q,r=1}^d \frac{\partial^2}{\partial x_q \partial x_r}
		b^k(t_l+c_i^{(1)} h, \tilde{Y}_l + u ( \tilde{H}_{i,l}^{(k),\tilde{Y}_l} - \tilde{Y}_l )) \,
		\big( \tilde{H}_{i,l}^{(k),\tilde{Y}_l} {\vphantom{ \tilde{H}_{i,l}^{(k),\tilde{Y}_l} }}^q
		- \tilde{Y}_l^q \big) \,
		\big( \tilde{H}_{i,l}^{(k),\tilde{Y}_l} {\vphantom{ \tilde{H}_{i,l}^{(k),\tilde{Y}_l} }}^r 
		- \tilde{Y}_l^r \big) \,
		\nonumber \\
		&\quad \times (1-u) \, \mathrm{d}u \bigg\|^p \bigg)
		\nonumber \\
		&\leq (T-t_0)^p \, (m \, s \, \czwei \, \cc \, d^2)^p \, \frac{1}{2} \, \cMHtilde 
		\big( 1 + \cYMBtilde \big( 1 + \Erw \big( \| \tilde{Y}_0 \|^{2p} \big) \big) \big)
		\big( 1 + h^{\frac{p}{2}} \big) \, h^p \, .
		\label{Proof:ApproxInt:Teil-IB-3-Estimate}
	\end{align}
	%
	As a result of the estimates
	\eqref{Proof:ApproxInt:Teil-IB-1a-Estimate}--\eqref{Proof:ApproxInt:Teil-IB-3-Estimate}
	there exists some constant $\cYtildeZtilde >0$ such that
	\begin{align}
		\Erw \big( \sup_{0 \leq n \leq N} \| \tilde{Z}_n - \tilde{Y}_n \|^p \big)
		&\leq \cYtildeZtilde \, \big( 1 + \Erw \big( \| \tilde{Y}_0 \|^p \big) 
		+ \Erw \big( \| \tilde{Y}_0 \|^{2p} \big) \big) \, h^p \, .
		\label{Proof:ApproxInt:Teil-B-Estimate}
	\end{align}
	Finally, the two estimates~\eqref{Proof:ApproxInt:Teil-A-Estimate}
	and~\eqref{Proof:ApproxInt:Teil-B-Estimate} together with~\eqref{Proof:ApproxInt:First-split}
	result in
	\begin{align}
		\Erw \big( \sup_{0 \leq n \leq N} \| Y_n - \tilde{Y}_n \|^p \big)
		&\leq 2^{p-1} \, \cYtildeZ \, \bigg( h^p + \sum_{l=0}^{N-1} h \,
		\Erw \big( \sup_{0 \leq q \leq l} \| Y_q - \tilde{Y}_q \|^p \big) \bigg)
		\nonumber \\
		&\quad + 2^{p-1} \, \cYtildeZtilde \, \big( 1 + \Erw \big( \| \tilde{Y}_0 \|^p \big) 
		+ \Erw \big( \| \tilde{Y}_0 \|^{2p} \big) \big) \, h^p
		\nonumber
	\end{align}
	and the assertion follows by applying Gronwall's lemma.
\end{proof}
\subsection{Proof of Convergence for the SRK Method for It{\^o} SDEs with 
	Commutative Noise}
\label{Sub:Sec:Proof-Convergence-CommNoise}
Firstly, observe that due to~\eqref{Sec:SRK-method-CommNoise:RVs} it follows
$\IiC_{(i,j),n} = \frac{1}{2} ( \Ii_{(i,j),n} + \Ii_{(j,i),n})$ for all $i,j \in \{1, \ldots, m\}$.
Analogously to Lemma~\ref{Lem:Ij-Iij-Moment-estimate} it hold exactly the
same estimates for $\IiC_{(i,j),n}$.
%
\begin{lem} \label{Lem:DachC-Ij-Iij-Moment-estimate}
	Let $h>0$ and $i,j,k \in \{1, \ldots, m\}$. Then, for $p \geq 1$ it holds
	\begin{align}
		\| \IiC_{(i,j),t,t+h} \|_{L^p(\Omega)} \leq \frac{\max \{ 2,p \}-1}{\sqrt{2}} \, h
	\end{align}
	and for $p_1, p_2 \geq 1$ it holds
	\begin{align}
		\Erw \big( | \Ii_{(k),t,t+h} |^{p_1} | \IiC_{(i,j),t,t+h} |^{p_2} \big)
		\leq \frac{(2p_1 -1)^{p_1} (2p_2 -1)^{p_2}}{2^{p_2/2}} \, 
		h^{\frac{p_1}{2} + p_2} \, .
	\end{align}
\end{lem}
\begin{proof}
	For any $i,j \in \{1, \ldots, m\}$ and any $p \geq 2$ it holds that
	\begin{align*}
		\| \IiC_{(i,j),t,t+h} \|_{L^p(\Omega)} 
		&= \Big\| \frac{1}{2} ( \Ii_{(i,j),t,t+h} + \Ii_{(j,i),t,t+h} ) \Big\|_{L^p(\Omega)} \\
		&\leq \frac{1}{2} \big( \| \Ii_{(i,j),t,t+h} \|_{L^p(\Omega)} 
		+ \| \Ii_{(i,j),t,t+h} \|_{L^p(\Omega)} \big)
	\end{align*}
	and the results follows with Lemma~\ref{Lem:Ij-Iij-Moment-estimate}.
	Let additionally $k \in \{1, \ldots, m\}$ and $p_1, p_2 \geq 1$. Then,
	\begin{align*}
		&\Erw \big( | \Ii_{(k),t,t+h} |^{p_1} | \IiC_{(i,j),t,t+h} |^{p_2} \big) \\
		&= \Erw \Big( | \Ii_{(k),t,t+h} |^{p_1} 
		\Big| \frac{1}{2} (\Ii_{(i,j),t,t+h} + \Ii_{(j,i),t,t+h}) \Big|^{p_2} \Big) \\
		&\leq \Erw \Big( | \Ii_{(k),t,t+h} |^{p_1} \, \frac{1}{2^{p_2}} \, 2^{p_2-1} \,
		\big( |  \Ii_{(i,j),t,t+h} |^{p_2} + | \Ii_{(j,i),t,t+h}) |^{p_2} \big) \Big) \\
		&= \frac{1}{2} \big( \Erw \big( | \Ii_{(k),t,t+h} |^{p_1} | \Ii_{(i,j),t,t+h} |^{p_2} \big)
		+ \Erw \big( | \Ii_{(k),t,t+h} |^{p_1} | \Ii_{(j,i),t,t+h} |^{p_2} \big) \big)
	\end{align*} 
	and the assertion follows directly with Lemma~\ref{Lem:Ij-Iij-Moment-estimate}.
\end{proof}

As a result of Lemma~\ref{Lem:DachC-Ij-Iij-Moment-estimate} it follows that
Lemma~\ref{Lem:Ij-Iij-H-Z-estimate}, Lemma~\ref{Lem:Ij-Iij-Hk-estimate}
and Lemma~\ref{Lem:Ij-Iij-H0-estimate} are also valid with the same constants
if we replace $\Ii_{(i,j),t,t+h}$ by $\IiC_{(i,j),t,t+h}$. Further, 
Proposition~\ref{Prop:Lp-bound-Approximation} remains also valid if 
$\Ii_{(i,j),t,t+h}$ is replaced by $\IiC_{(i,j),t,t+h}$ in the SRK method. Thus, it
remains to consider the proof of convergence for
Theorem~\ref{Sec:Main-Result:Thm-Konv-CommNoise}. 
%
%
%
\begin{proof}[Proof of Theorem~\ref{Sec:Main-Result:Thm-Konv-CommNoise}]
The SRK method~\eqref{SRK-method}
with~\eqref{Sec:SRK-method-CommNoise:RVs}
for commutative noise coincides with the SRK method~\eqref{SRK-method}
with~\eqref{SRK-method-Ito-Strato-Iihat}
if $\Ii_{(i,j),n}$ is replaced by $\IiC_{(i,j),n}$. Therefore, most of the proof for
Theorem~\ref{Sec:Main-Result:Thm-Konv-SRK-allg} carries directly over to the 
proof for
Theorem~\ref{Sec:Main-Result:Thm-Konv-CommNoise}, except for a few
estimates that are detailed in the following: First of all, due to the 
commutativity condition~\eqref{Sec:SRK-Method-CommNoise-Cond} it follows
%
\begin{align} \label{Proof:MainThmCommNoise:Identity-1}
	\sum_{k_1,k_2=1}^m \sum_{r=1}^d \frac{\partial}{\partial x^r} b^{k_1}(t,x) \,
	b^{r,k_2}(t,x) \, \IiC_{(k_2,k_1),n}
	&= \sum_{k_1,k_2=1}^m \sum_{r=1}^d \frac{\partial}{\partial x^r} b^{k_1}(t,x) \,
	b^{r,k_2}(t,x) \, \Ii_{(k_2,k_1),n}
\end{align}
for any $(t,x) \in \mathbb{R}_+ \times \mathbb{R}^d$. 
Now, we consider the terms that need to be handled in a different way compared
to the original proof for Theorem~\ref{Sec:Main-Result:Thm-Konv-SRK-allg}.
%
The first term that needs to be considered is term~\eqref{Proof:MainThm:Teil-A2-1}.
This term contains the summand
%
\begin{align}
	&\sum_{l=0}^{n-1} \sum_{k=1}^m \int_{t_l}^{t_{l+1}} \sum_{r=1}^d 
	\frac{\partial}{\partial x_r} b^k(t_l,X_{t_l}) \bigg[ \sum_{k_2=1}^m \int_{t_l}^u
	b^{r,k_2}(t_l,X_{t_l}) \, \mathrm{d}W_v^{k_2} \bigg] \, \mathrm{d}W_u^k
	\nonumber \\
	&- \sum_{l=0}^{n-1} \sum_{k=1}^m \sum_{i=1}^s ( \beta_i^{(1)} \Ii_{(k),l}
	+ \beta_i^{(2)} ) \sum_{r=1}^d \frac{\partial}{\partial x_r} b^k(t_l,X_{t_l}) 
	\bigg[ \sum_{j=1}^{i-1} \sum_{k_2=1}^m B_{i,j}^{(1)} \, b^{r,k_2}(t_l,X_{t_l})
	\, \IiC_{(k_2,k),l} \bigg] 
	\nonumber \\
	&= -\sum_{l=0}^{n-1}
	\sum_{k=1}^m \sum_{i=1}^s \beta_i^{(1)} \, \Ii_{(k),l} \sum_{r=1}^d 
	\frac{\partial}{\partial x_r} b^k(t_l,X_{t_l})
	\sum_{j=1}^{i-1} \sum_{k_2=1}^m B_{i,j}^{(1)} \, b^{r,k_2}(t_l,X_{t_l}) 
	\, \IiC_{(k_2,k),l}
	\label{Proof:MainThmCommNoise:Term-K2.1i-native}
\end{align}
that has to be treated slightly different in the case of $\IiC_{(i,j),l}$ instead of
$\Ii_{(i,j),l}$. In order to 
calculate~\eqref{Proof:MainThmCommNoise:Term-K2.1i-native}
we have applied condition $\sum_{i=1}^s \beta_i^{(2)} \sum_{j=1}^{i-1}
B_{i,j}^{(1)}=1$ and \eqref{Proof:MainThmCommNoise:Identity-1}.
With $\IiC_{(k_2,k),l} = \frac{1}{2} ( \Ii_{(k_2,k),l} + \Ii_{(k,k_2),l})$ we
estimate~\eqref{Proof:MainThmCommNoise:Term-K2.1i-native} by
%
\begin{align}
	&\Erw \bigg( \sup_{0 \leq n \leq N} \bigg\| \sum_{l=0}^{n-1}
	\sum_{k=1}^m \sum_{i=1}^s \beta_i^{(1)} \, \Ii_{(k),l} \sum_{r=1}^d 
	\frac{\partial}{\partial x_r} b^k(t_l,X_{t_l})
	\sum_{j=1}^{i-1} \sum_{k_2=1}^m B_{i,j}^{(1)} \, b^{r,k_2}(t_l,X_{t_l}) 
	\, \IiC_{(k_2,k),l} \bigg\|^p \bigg)
	\nonumber \\
	&\leq \frac{1}{2^p} \Erw \bigg( \sup_{0 \leq n \leq N} \bigg\| \sum_{l=0}^{n-1}
	\sum_{k=1}^m \sum_{i=1}^s \beta_i^{(1)} \sum_{r=1}^d 
	\frac{\partial}{\partial x_r} b^k(t_l,X_{t_l})
	\sum_{j=1}^{i-1} \sum_{k_2=1}^m B_{i,j}^{(1)} \, b^{r,k_2}(t_l,X_{t_l}) 
	\Ii_{(k),l} \Ii_{(k_2,k),l} \bigg\|^p \bigg) 
	\label{Proof:MainThmCommNoise:Term-K2.1i} \\
	&\quad + \frac{1}{2^p} \Erw \bigg( \sup_{0 \leq n \leq N} \bigg\| \sum_{l=0}^{n-1}
	\sum_{k=1}^m \sum_{i=1}^s \beta_i^{(1)} \sum_{r=1}^d 
	\frac{\partial}{\partial x_r} b^k(t_l,X_{t_l})
	\sum_{j=1}^{i-1} \sum_{k_2=1}^m B_{i,j}^{(1)} \, b^{r,k_2}(t_l,X_{t_l}) 
	\Ii_{(k),l} \Ii_{(k,k_2),l} \bigg\|^p \bigg) .
	\label{Proof:MainThmCommNoise:Term-K2.1i-2}
\end{align}
Term~\eqref{Proof:MainThmCommNoise:Term-K2.1i} can be estimated exactly like 
term~\eqref{Proof:MainThm:Teil-A2-1i}. Considering
term~\eqref{Proof:MainThmCommNoise:Term-K2.1i-2} and taking into account
$\Ii_{(k),l} \Ii_{(k,k_2),l} = \Ii_{(k,k_2,k),l} + 2 \Ii_{(k,k,k_2),l} + \Ii_{(0,k_2),l}$ results in
%
\begin{align}
	&\Erw \bigg( \sup_{0 \leq n \leq N} \bigg\| \sum_{l=0}^{n-1}
	\sum_{k=1}^m \sum_{i=1}^s \beta_i^{(1)} \sum_{r=1}^d 
	\frac{\partial}{\partial x_r} b^k(t_l,X_{t_l})
	\sum_{j=1}^{i-1} \sum_{k_2=1}^m B_{i,j}^{(1)} \, b^{r,k_2}(t_l,X_{t_l}) 
	\Ii_{(k),l} \Ii_{(k,k_2),l} \bigg\|^p \bigg)
	\nonumber \\
	&\leq 3^{p-1} \Erw \bigg( \sup_{0 \leq n \leq N} \bigg\| \sum_{l=0}^{n-1}
	\sum_{k=1}^m \sum_{i=1}^s \beta_i^{(1)} \sum_{r=1}^d 
	\frac{\partial}{\partial x_r} b^k(t_l,X_{t_l})
	\sum_{j=1}^{i-1} \sum_{k_2=1}^m B_{i,j}^{(1)} \, b^{r,k_2}(t_l,X_{t_l}) 
	\Ii_{(k,k_2,k),l} \bigg\|^p \bigg)
	\label{Proof:MainThmCommNoise:Term-K2.1i-2a} \\
	&\quad + 3^{p-1} \Erw \bigg( \sup_{0 \leq n \leq N} \bigg\| \sum_{l=0}^{n-1}
	\sum_{k=1}^m \sum_{i=1}^s \beta_i^{(1)} \sum_{r=1}^d 
	\frac{\partial}{\partial x_r} b^k(t_l,X_{t_l})
	\sum_{j=1}^{i-1} \sum_{k_2=1}^m B_{i,j}^{(1)} \, b^{r,k_2}(t_l,X_{t_l}) 
	\, 2 \Ii_{(k,k,k_2),l} \bigg\|^p \bigg)
	\label{Proof:MainThmCommNoise:Term-K2.1i-2b} \\
	&\quad + 3^{p-1} \Erw \bigg( \sup_{0 \leq n \leq N} \bigg\| \sum_{l=0}^{n-1}
	\sum_{k=1}^m \sum_{i=1}^s \beta_i^{(1)} \sum_{r=1}^d 
	\frac{\partial}{\partial x_r} b^k(t_l,X_{t_l})
	\sum_{j=1}^{i-1} \sum_{k_2=1}^m B_{i,j}^{(1)} \, b^{r,k_2}(t_l,X_{t_l}) 
	\, \Ii_{(0,k_2),l} \bigg\|^p \bigg) .
	\label{Proof:MainThmCommNoise:Term-K2.1i-2c}
\end{align}
Here, term~\eqref{Proof:MainThmCommNoise:Term-K2.1i-2a} and 
term~\eqref{Proof:MainThmCommNoise:Term-K2.1i-2b} can be estimated in 
the same way as term~\eqref{Proof:MainThm:Teil-A2-1i-1} and
term~\eqref{Proof:MainThm:Teil-A2-1i-2},
respectively. Only term~\eqref{Proof:MainThmCommNoise:Term-K2.1i-2c} 
needs to be
estimated in a different way. We calculate with Burkholder's inequality, see, e.g.,
\cite{Burk88} or \cite[Prop.~2.1 \& 2.2]{PlRoe21},
with~\eqref{Assumption-a-bk:lin-growth}, \eqref{Assumption-a-bk:Bound-derivative-1}
and Lemma~\ref{Lem:Lp-bound-SDE-sol} that
%
\begin{align*}
	&\Erw \bigg( \sup_{0 \leq n \leq N} \bigg\| \sum_{l=0}^{n-1}
	\sum_{k=1}^m \sum_{i=1}^s \beta_i^{(1)} \sum_{r=1}^d 
	\frac{\partial}{\partial x_r} b^k(t_l,X_{t_l})
	\sum_{j=1}^{i-1} \sum_{k_2=1}^m B_{i,j}^{(1)} \, b^{r,k_2}(t_l,X_{t_l}) 
	\, \Ii_{(0,k_2),l} \bigg\|^p \bigg) \\
	&= \Erw \bigg( \sup_{0 \leq n \leq N} \bigg\| \sum_{l=0}^{n-1} \sum_{k,k_2=1}^m 
	\int_{t_l}^{t_{l+1}} \int_{t_l}^u \sum_{i=1}^s \beta_i^{(1)} \sum_{j=1}^{i-1}
	B_{i,j}^{(1)} 
	\sum_{r=1}^d \frac{\partial}{\partial x_r} b^k(t_l,X_{t_l}) \, b^{r,k_2}(t_l,X_{t_l}) 
	\, \mathrm{d}v \, \mathrm{d}W_u^{k_2} \bigg\|^p \bigg) \\
	&\leq \Big( \frac{p}{\sqrt{p-1}} \Big)^p \bigg( \sum_{l=0}^{N-1} \bigg[ 
	\Erw \bigg( \bigg\| \sum_{k,k_2=1}^m \int_{t_l}^{t_{l+1}} \int_{t_l}^u \sum_{i=1}^s 
	\beta_i^{(1)} \sum_{j=1}^{i-1} B_{i,j}^{(1)} \\
	&\quad \times \sum_{r=1}^d \frac{\partial}{\partial x_r}
	b^k(t_l,X_{t_l}) \, b^{r,k_2}(t_l,X_{t_l}) \, \mathrm{d}v \, \mathrm{d}W_u^{k_2} 
	\bigg\|^p \bigg) \bigg]^{\frac{2}{p}} \bigg)^{\frac{p}{2}} \\
	&\leq \Big( \frac{p}{\sqrt{p-1}} \Big)^p \bigg( \sum_{l=0}^{N-1} (p-1)
	\int_{t_l}^{t_{l+1}}
	\bigg[ \Erw \bigg( \bigg\| \sum_{k_2=1}^m \bigg\| \sum_{k=1}^m \int_{t_l}^u
	\sum_{i=1}^s \beta_i^{(1)} \sum_{j=1}^{i-1} B_{i,j}^{(1)} \\
	&\quad \times \sum_{r=1}^d \frac{\partial}{\partial x_r}
	b^k(t_l,X_{t_l}) \, b^{r,k_2}(t_l,X_{t_l}) \, \mathrm{d}v \bigg\|^2
	\bigg\|^{\frac{p}{2}} \bigg)
	\bigg]^{\frac{2}{p}} \, \mathrm{d}u \bigg)^{\frac{p}{2}} \\
	&\leq \Big( \frac{p}{\sqrt{p-1}} \Big)^p \bigg( \sum_{l=0}^{N-1} (p-1)
	\sum_{k,k_2=1}^m
	m \int_{t_l}^{t_{l+1}} (u-t_l)^2 \, \mathrm{d}u \, s^4 \, \czwei^4 \, d \\
	&\quad \times \sum_{r=1}^d
	\Big[ \Erw \Big( \Big\| \frac{\partial}{\partial x_r}
	b^k(t_l,X_{t_l}) \Big\|^p \, \big| b^{r,k_2}(t_l,X_{t_l}) \big|^p \Big) \Big]^{\frac{2}{p}} 
	\Big)^{\frac{p}{2}} \\
	&\leq \Big( \frac{p}{\sqrt{p-1}} \Big)^p \bigg( \sum_{l=0}^{N-1} \frac{1}{3}
	h^3 \, (p-1)
	\, m^3 \, s^4 \, \czwei^4 \, d^2 \, \cc^2 \, \big[ \Erw \big( (1+ \| X_{t_l} \| )^p \big)
	\big]^{\frac{2}{p}} \bigg)^{\frac{p}{2}} \\
	&\leq \Big( \frac{p}{\sqrt{p-1}} \Big)^p \, (T-t_0)^{\frac{p}{2}} \, 
	\Big( \frac{1}{3} \, (p-1)
	\, m^3 \, s^4 \, \czwei^4 \, d^2 \, \cc^2 \Big)^{\frac{p}{2}} \, 2^{p-1} \,
	\big( 1 + \ccp \big( 1 + \Erw ( \| X_{t_0} \|^p ) \big) \big) \, h^p \, .
\end{align*}
Next, term~\eqref{Proof:MainThm:Teil-B2-2c-1} needs to be considered as well.
If $\Ii_{(q,k),l}$ is replaced by $\IiC_{(q,k),l}$ in~\eqref{Proof:MainThm:Teil-B2-2c-1}
then exactly the same estimates as
for~\eqref{Proof:MainThmCommNoise:Term-K2.1i-native} apply.
The last term that needs to be considered is~\eqref{Proof:MainThm:Teil-B3-2a3}
with $\Ii_{(q,k),l}$ replaced by $\IiC_{(q,k),l}$.
Here, this term can be rewritten with~\eqref{Proof:MainThmCommNoise:Identity-1}
such that
%
\begin{align}
	&\Erw \bigg( \sup_{0 \leq n \leq N} \bigg\|
	\sum_{l=0}^{n-1} \sum_{k=1}^m \sum_{j_1=1}^m 
	\sum_{i=1}^s \beta_i^{(2)} \sum_{j=1}^{i-1} B_{i,j}^{(1)}
	\nonumber \\
	&\quad \times
	\sum_{r=1}^d \Big( \frac{\partial}{\partial x_r} b^k(t_l, X_{t_l}) \, 
	b^{r,j_1}(t_l, X_{t_l}) 
	- \frac{\partial}{\partial x_r} b^k(t_l, Y_l) \, b^{r,j_1}(t_l, Y_l) \Big)
	\, \IiC_{(j_1,k),l} \bigg\|^p \bigg) \\
	&= \Erw \bigg( \sup_{0 \leq n \leq N} \bigg\|
	\sum_{l=0}^{n-1} \sum_{k=1}^m \sum_{j_1=1}^m 
	\sum_{i=1}^s \beta_i^{(2)} \sum_{j=1}^{i-1} B_{i,j}^{(1)}
	\nonumber \\
	&\quad \times
	\sum_{r=1}^d \Big( \frac{\partial}{\partial x_r} b^k(t_l, X_{t_l}) \, 
	b^{r,j_1}(t_l, X_{t_l}) 
	- \frac{\partial}{\partial x_r} b^k(t_l, Y_l) \, b^{r,j_1}(t_l, Y_l) \Big)
	\, \Ii_{(j_1,k),l} \bigg\|^p \bigg)
\end{align}
and thus coincides with~\eqref{Proof:MainThm:Teil-B3-2a3}. Therefore, it 
can be estimated in the same way as~\eqref{Proof:MainThm:Teil-B3-2a3}.
\end{proof}
\subsection{Proof of Convergence for the SRK Method for Stratonovich SDEs with 
	Commutative Noise}
\label{Sub:Sec:Proof-Convergence-CommNoise-Srato}
Again, observe that due to~\eqref{Sec:SRK-method-CommNoise:RVs-Strato} it follows
$\JiC_{(i,j),n} = \frac{1}{2} ( \Ji_{(i,j),n} + \Ji_{(j,i),n})$ for all $i,j \in \{1, \ldots, m\}$.
Now, the same corresponding estimates as in Lemma~\ref{Lem:Ij-Iij-Moment-estimate}
hold for $\IiC_{(i,j),n}$ as well.
%
\begin{lem} \label{Lem:DachC-Ij-Iij-Moment-estimate-Strato}
	Let $h>0$ and $i,j,k \in \{1, \ldots, m\}$. Then, for $p \geq 1$ it holds
	\begin{align}
		\| \JiC_{(i,j),t,t+h} \|_{L^p(\Omega)} \leq \frac{\max \{ 2,p \}-1 
		+ \frac{1}{\sqrt{2}}}{\sqrt{2}} \, h
	\end{align}
	and for $p_1, p_2 \geq 1$ it holds
	\begin{align}
		\Erw \big( | \Ii_{(k),t,t+h} |^{p_1} | \JiC_{(i,j),t,t+h} |^{p_2} \big)
		\leq \frac{(2p_1 -1)^{p_1} (2p_2 -1 + \frac{1}{\sqrt{2}})^{p_2}}{2^{p_2/2}} \, 
		h^{\frac{p_1}{2} + p_2} \, .
	\end{align}
\end{lem}
\begin{proof}
	The proof follows the lines of the proof for Lemma~\ref{Lem:DachC-Ij-Iij-Moment-estimate}
	using the estimates from Lemma~\ref{Lem:Ij-Iij-Moment-estimate}.
\end{proof}

Due to Lemma~\ref{Lem:DachC-Ij-Iij-Moment-estimate-Strato} we get that
Lemma~\ref{Lem:Ij-Iij-H-Z-estimate}, Lemma~\ref{Lem:Ij-Iij-Hk-estimate}
and Lemma~\ref{Lem:Ij-Iij-H0-estimate} remain valid if we replace 
$\Ji_{(i,j),t,t+h}$ by $\JiC_{(i,j),t,t+h}$. Moreover, 
Proposition~\ref{Prop:Lp-bound-Approximation} still holds if 
$\Ji_{(i,j),t,t+h}$ is replaced by $\JiC_{(i,j),t,t+h}$ in the SRK 
method~\eqref{SRK-method}. Thus, we only have to consider the proof of 
convergence for Theorem~\ref{Sec:Main-Result:Thm-Konv-CommNoise-Strato}.
%
%
%
\begin{proof}[Proof of Theorem~\ref{Sec:Main-Result:Thm-Konv-CommNoise-Strato}]
	The proof mainly follows the lines of the proof for
	Theorem~\ref{Sec:Main-Result:Thm-Konv-SRK-allg-Strato} (and thus the proof for
	Theorem~\ref{Sec:Main-Result:Thm-Konv-SRK-allg})
	by substituting $\Ji_{(i,j),t,t+h}$ with $\JiC_{(i,j),t,t+h}$.
	However, there are a few terms that need to treated in a different way as specified
	in the following. First of all, note that with the relationship
	%
	\begin{align*}
		\JiC_{(k_2,k),l} = \frac{1}{2} \Ii_{(k_2),l} \, \Ii_{(k),l} 
		= \frac{1}{2} ( \Ji_{(k_2,k),l} + \Ji_{(k,k_2),l} )
		= \frac{1}{2} ( \Ii_{(k_2,k),l} + \Ii_{(k,k_2),l} ) + \ind_{\{k_2=k\}} \, \frac{1}{2} h \, .
	\end{align*}
	and the commutativity condition~\eqref{Sec:SRK-Method-CommNoise-Cond} 
	it holds that
	%
	\begin{equation}
	\begin{split} \label{Proof:MainThmCommNoise:Identity-1-Strato}
		&\sum_{k_1,k_2=1}^m \sum_{r=1}^d \frac{\partial}{\partial x^r} b^{k_1}(t,x) \,
		b^{r,k_2}(t,x) \, \JiC_{(k_2,k_1),n} \\
		&= \sum_{k_1,k_2=1}^m \sum_{r=1}^d \frac{\partial}{\partial x^r} b^{k_1}(t,x) \,
		b^{r,k_2}(t,x) \, \Big( \Ii_{(k_2,k_1),n} + \ind_{\{k_2=k\}} \, \frac{1}{2} h \Big)
	\end{split}
	\end{equation}
	for $(t,x) \in \mathbb{R}_+ \times \mathbb{R}^d$. 

	Looking at inequality~\eqref{Proof:MainThm:Teil-A-B}, we start to estimate the first 
	summand on the right hand side. Then, the first term that needs to be treated 
	differently is~\eqref{Proof:MainThm:Teil-A2-1-Strato}.
	Estimating term~\eqref{Proof:MainThm:Teil-A2-1-Strato} results 
	in~\eqref{Proof:MainThm:Teil-A2-1-Replaced-Strato} where each of the terms
	$(\Ii_{(k_2,k),l} + \ind_{\{k_2=k\}} \, \frac{1}{2} h )$ and $\Ji_{(k_2,k),l}$
	is now replaced by $\JiC_{(k_2,k),l}$. Note that with $\sum_{i=1}^s \beta_i^{(2)} 
	\sum_{j=1}^{i-1} B_{i,j}^{(1)} =1$ and~\eqref{Proof:MainThmCommNoise:Identity-1-Strato}
	we thus calculate
	%
	\begin{align}
		&\sum_{l=0}^{n-1} \sum_{k=1}^m \int_{t_l}^{t_{l+1}} \sum_{r=1}^d 
		\frac{\partial}{\partial x_r} b^k(t_l,X_{t_l}) \bigg[ \sum_{k_2=1}^m \int_{t_l}^u
		b^{r,k_2}(t_l,X_{t_l}) \, \mathrm{d}W_v^{k_2} \bigg] \, \mathrm{d}W_u^k
		\nonumber \\
		&\quad + \frac{1}{2} \sum_{l=0}^{n-1} \sum_{k=1}^m \int_{t_l}^{t_{l+1}} \sum_{q=1}^d 
		b^{q,k}(t_l, X_{t_l}) \,
		\frac{\partial}{\partial x_q} b^k(t_l,X_{t_l}) \, \mathrm{d}s
		\nonumber \\
		&- \sum_{l=0}^{n-1} \sum_{k=1}^m \sum_{i=1}^s ( \beta_i^{(1)} \Ii_{(k),l}
		+ \beta_i^{(2)} ) \sum_{r=1}^d \frac{\partial}{\partial x_r} b^k(t_l,X_{t_l}) 
		\bigg[ \sum_{j=1}^{i-1} \sum_{k_2=1}^m B_{i,j}^{(1)} \, b^{r,k_2}(t_l,X_{t_l})
		\, \JiC_{(k_2,k),l} \bigg] 
		\nonumber \\
		&= -\sum_{l=0}^{n-1}
		\sum_{k=1}^m \sum_{i=1}^s \beta_i^{(1)} \, \Ii_{(k),l} \sum_{r=1}^d 
		\frac{\partial}{\partial x_r} b^k(t_l,X_{t_l})
		\sum_{j=1}^{i-1} \sum_{k_2=1}^m B_{i,j}^{(1)} \, b^{r,k_2}(t_l,X_{t_l}) 
		\, \JiC_{(k_2,k),l}
		\label{Proof:MainThmCommNoise:Term-K2.1i-native-Strato}
	\end{align}
	in~\eqref{Proof:MainThm:Teil-A2-1-Replaced-Strato}. Now, if we rewrite 
	$\JiC_{(k_2,k),l} = \frac{1}{2} ( \Ii_{(k_2,k),l} + \Ii_{(k,k_2),l} ) + \ind_{\{k_2=k\}} 
	\, \frac{1}{2} h$ in term~\eqref{Proof:MainThmCommNoise:Term-K2.1i-native-Strato}
	and if we apply triangle inequality then we can estimate the resulting three terms
	in the same way as~\eqref{Proof:MainThm:Teil-A2-1i-Strato} 
	and~\eqref{Proof:MainThm:Teil-A2-1i-2-Strato} (see
	also~\eqref{Proof:MainThm:Teil-A2-1i-2-Estimate-Strato}), respectively.

	It remains to estimate the terms that correspond to~\eqref{Proof:MainThm:Teil-A2-1g-Strato}
	and~\eqref{Proof:MainThm:Teil-A2-1h-Strato} if $\Ji_{(k_2,k),t,t+h}$ is
	replaced by $\JiC_{(k_2,k),t,t+h} = \frac{1}{2} ( \Ji_{(k_2,k),l} + \Ji_{(k,k_2),l} )$.
	However, with triangle inequality these terms can be estimated 
	like~\eqref{Proof:MainThm:Teil-A2-1g-Strato} 
	and~\eqref{Proof:MainThm:Teil-A2-1h-Strato}, i.~e., in the same way as 
	terms~\eqref{Proof:MainThm:Teil-A2-1g} and~\eqref{Proof:MainThm:Teil-A2-1h}.
	All remaining terms stay unchanged. Thus, we obtain 
	estimate~\eqref{Proof:MainThm:Teil-A-Estimate}
	with a different constant in case of the Stratonovich setting with commutative noise.

	In order to estimate the second summand in term~\eqref{Proof:MainThm:Teil-A-B}, the
	proof in the Stratonovich setting with commutative noise follows the lines of the proof for
	Theorem~\ref{Sec:Main-Result:Thm-Konv-SRK-allg} with some additional estimates 
	from the proof of Theorem~\ref{Sec:Main-Result:Thm-Konv-CommNoise}. 
	Once more, we get the corresponding
	terms~\eqref{Proof:MainThm:Teil-B1}--\eqref{Proof:MainThm:Teil-B3} that can
	be estimated in the same way with the only difference that
	$\Ii_{(i,j),n}$ is replaced by $\JiC_{(i,j),n}$. 

	However, the term that correspondents to~\eqref{Proof:MainThm:Teil-B2-2c-1} 
	needs to be estimated in a slightly different way now. In the current proof, we 
	obtain~\eqref{Proof:MainThm:Teil-B2-2c-1} with $\JiC_{(q,k),n}$ instead of
	$\Ii_{(q,k),n}$. Rewriting $\JiC_{(q,k),t,t+h} = \frac{1}{2} ( \Ii_{(q,k),l} + \Ii_{(k,q),l} ) 
	+ \ind_{\{q=k\}} \, \frac{1}{2} h$
	and applying triangle inequality results exactly in the two 
	terms~\eqref{Proof:MainThmCommNoise:Term-K2.1i}
	and~\eqref{Proof:MainThmCommNoise:Term-K2.1i-2}
	as well as term~\eqref{Proof:MainThm:Teil-A2-1i-2-Strato} with a different 
	constant factor. Thus, these terms can be estimated as in the proofs for
	Theorem~\ref{Sec:Main-Result:Thm-Konv-CommNoise} and for
	Theorem~\ref{Sec:Main-Result:Thm-Konv-SRK-allg-Strato}, respectively.
	As a result of this, we finally get estimate~\eqref{Proof:MainThm:Teil-B-Estimate} 
	with a different constant.

	Taking into account both estimates \eqref{Proof:MainThm:Teil-A-Estimate} and
	\eqref{Proof:MainThm:Teil-B-Estimate} as in~\eqref{Proof:MainThm:Teil-A-B-All-Estimate}
	and applying Gronwall's lemma completes the proof.
\end{proof}
%
%
%
%
%
\bibliographystyle{plainurl}
\bibliography{BibPaper}
%
%
%
\end{document}